\documentclass[11pt]{article}

\usepackage[utf8]{inputenc}
\usepackage[margin=1in]{geometry}
\usepackage{amsmath, amsthm, amssymb, asymptote, array, bbm, enumitem, float, mathrsfs, todonotes, dsfont}

\usepackage[linktocpage=true]{hyperref}
\usepackage{xcolor}
\hypersetup{
  colorlinks   = true, %Colours links instead of ugly boxes
  urlcolor     = blue, %Colour for external hyperlinks
  linkcolor    = blue, %Colour of internal links
  citecolor   = red %Colour of citations
}

\usetikzlibrary{spy}

\newtheorem{theorem}{Theorem}[section]
\newtheorem{lemma}[theorem]{Lemma}
\newtheorem{proposition}[theorem]{Proposition}

\newtheorem{corollary}[theorem]{Corollary}

\theoremstyle{definition}
\newtheorem{definition}[theorem]{Definition}

\newtheorem{example}[theorem]{Example}

\theoremstyle{remark}
\newtheorem{remark}[theorem]{Remark}

\numberwithin{equation}{section}

\makeatletter
\newcommand\RedeclareMathOperator{%
  \@ifstar{\def\rmo@s{m}\rmo@redeclare}{\def\rmo@s{o}\rmo@redeclare}%
}
\newcommand\rmo@redeclare[2]{%
  \begingroup \escapechar\m@ne\xdef\@gtempa{{\string#1}}\endgroup
  \expandafter\@ifundefined\@gtempa
     {\@latex@error{\noexpand#1undefined}\@ehc}%
     \relax
  \expandafter\rmo@declmathop\rmo@s{#1}{#2}}
\newcommand\rmo@declmathop[3]{%
  \DeclareRobustCommand{#2}{\qopname\newmcodes@#1{#3}}%
}
\@onlypreamble\RedeclareMathOperator
\makeatother

\newcommand{\mathsc}[1]{{\normalfont\textsc{#1}}}

\newcommand{\f}{\frac}
\newcommand{\lt}{\left}
\newcommand{\rt}{\right}
\newcommand{\la}{\langle}
\newcommand{\ra}{\rangle}
\newcommand{\eps}{{\varepsilon}}

\DeclareMathOperator*{\E}{{\mathbb{E}}}
\RedeclareMathOperator*{\P}{{\mathbb{P}}}
\DeclareMathOperator*{\Var}{{\mathrm{Var}}}
\newcommand{\ind}[1]{\mathds{1}\lt\{#1\rt\}}

\newcommand{\diff}[1]{{\mathrm{d}#1}}
\newcommand{\rn}[2]{\f{\diff{#1}}{\diff{#2}}}

\newcommand{\cB}{{\mathcal{B}}}
\newcommand{\cC}{{\mathcal{C}}}
\newcommand{\cD}{{\mathcal{D}}}
\newcommand{\cG}{{\mathcal{G}}}
\newcommand{\cL}{{\mathcal{L}}}
\newcommand{\cM}{{\mathcal{M}}}
\newcommand{\cN}{{\mathcal{N}}}
\newcommand{\cS}{{\mathcal{S}}}
\newcommand{\cX}{{\mathcal{X}}}

\newcommand{\bE}{{\mathbb{E}}}
\newcommand{\bN}{{\mathbb{N}}}
\newcommand{\bR}{{\mathbb{R}}}
\newcommand{\bS}{{\mathbb{S}}}

\newcommand{\sE}{{\mathscr{E}}}
\newcommand{\sL}{{\mathscr{L}}}

\newcommand{\diag}{{\mathrm{diag}}}
\newcommand{\etr}{{\mathrm{etr}}}
\newcommand{\HS}{{\mathrm{HS}}}
\newcommand{\norm}[1]{\lt\|#1\rt\|}
\newcommand{\op}{{\mathrm{op}}}
\newcommand{\rank}{{\mathrm{rank}}}
\newcommand{\spec}{{\mathrm{spec}}}
\newcommand{\Tr}{{\mathrm{Tr}}}

\newcommand{\Nl}{{N_{<v}}}
\newcommand{\Nv}{{N(v)}}
\newcommand{\downN}{{N^{\downarrow}}}
\newcommand{\downNv}{{\downN(v)}}
\renewcommand{\deg}{{\mathsf{deg}}}
\newcommand{\ddeg}{{\deg^{\downarrow}}}
\newcommand{\degvl}{{\deg_{<v}}}

\newcommand{\degv}{{\deg(v)}}
\newcommand{\degu}{{\deg(u)}}
\newcommand{\ddegv}{{\ddeg(v)}}
\newcommand{\dvli}{{\degvl(i)}}
\newcommand{\dvlj}{{\degvl(j)}}
\newcommand{\dvlk}{{\degvl(k)}}
\newcommand{\dvll}{{\degvl(\ell)}}
\newcommand{\dvlij}{{\degvl(i,j)}}
\newcommand{\dvlkl}{{\degvl(k,\ell)}}
\newcommand{\Hom}{{\mathrm{Hom}}}
\newcommand{\Num}{{\mathsf{num}}}
\newcommand{\oNum}{{\overrightarrow{\Num}}}
\newcommand{\maxdeg}{{\mathsf{maxdeg}}}
\newcommand{\vst}{{v^*}}
\newcommand{\oK}{{\overrightarrow{K}}}
\newcommand{\oP}{{\overrightarrow{P}}}

\newcommand{\chisq}{{\chi^2}}
\newcommand{\KL}{{\mathsf{KL}}}
\newcommand{\TV}{{\mathsf{TV}}}

\newcommand{\Sop}{{S_{\op}}}
\newcommand{\Sopc}{{S_{\op}^c}}
\newcommand{\Str}{{S_{\Tr}}}
\newcommand{\Sdet}{{S_{\det}}}
\newcommand{\Syv}{{S_{Y,v}}}
\newcommand{\Syvc}{{S_{Y,v}^c}}
\newcommand{\Sze}{{S_{Z,e}}}
\newcommand{\Szec}{{S_{Z,e}^c}}
\newcommand{\Szv}{{S_{Z,v}}}
\newcommand{\Szvc}{{S_{Z,v}^c}}
\newcommand{\Svop}{{S^v_{\op}}}
\newcommand{\Svtr}{{S^v_{\Tr}}}
\newcommand{\Svdet}{{S^v_{\det}}}
\newcommand{\Tdet}{{T_{\det}}}
\newcommand{\Tdetc}{{T_{\det}^c}}

\newcommand{\muvl}{{\mu_{<v}}}
\newcommand{\muv}{{\mu_v}}
\newcommand{\muvls}{{\mu_{<v}^S}}
\newcommand{\muvlsv}{{\mu_{<v}^{S^v}}}
\newcommand{\muvs}{{\mu_v^S}}
\newcommand{\muvsv}{{\mu_v^{S^v}}}
\newcommand{\muvlt}{{\mu_{<v}^T}}
\newcommand{\muvt}{{\mu_v^T}}

\newcommand{\nuv}{{\nu_v}}
\newcommand{\muvltd}{{\mu_{<v}^\Tdet}}
\newcommand{\muvtd}{{\mu_v^\Tdet}}

\newcommand{\Xvl}{{X_{<v}}}
\newcommand{\Xvlt}{{X_{<v}^\top}}
\newcommand{\Xpvl}{{X'_{<v}}}

\newcommand{\Xn}{{X_{\Nv}}}
\newcommand{\Xnt}{{X_{\Nv}^\top}}
\newcommand{\Xbn}{{X_{\downNv}}}
\newcommand{\Xbnt}{{X_{\downNv}^\top}}
\newcommand{\Xpbn}{{X'_{\downNv}}}
\newcommand{\Xpbnt}{{{X'}_{\downNv}^\top}}
\newcommand{\XL}{{X_L}}
\newcommand{\XR}{{X_R}}
\newcommand{\XRt}{{X_R^\top}}

\newcommand{\Xone}[1][]{{X^{(1)}_{#1}}}
\newcommand{\Xonei}{{\Xone[i]}}
\newcommand{\Xonej}{{\Xone[j]}}
\newcommand{\Xonek}{{\Xone[k]}}
\newcommand{\Xonel}{{\Xone[\ell]}}
\newcommand{\XoneR}{{\Xone[R]}}
\newcommand{\Xonevl}{{\Xone[<v]}}
\newcommand{\Xtwo}[1][]{{X^{(2)}_{#1}}}
\newcommand{\Xtwoi}{{\Xtwo[i]}}
\newcommand{\Xtwoj}{{\Xtwo[j]}}
\newcommand{\Xtwok}{{\Xtwo[k]}}
\newcommand{\Xtwol}{{\Xtwo[\ell]}}
\newcommand{\XtwoR}{{\Xtwo[R]}}
\newcommand{\Xtwovl}{{\Xtwo[<v]}}
\newcommand{\Xr}[1][]{{X^{(r)}_{#1}}}
\newcommand{\Xri}{{\Xr[i]}}
\newcommand{\Xrj}{{\Xr[j]}}
\newcommand{\Xrk}{{\Xr[k]}}
\newcommand{\Xrl}{{\Xr[\ell]}}
\newcommand{\Xrvl}{{\Xr[<v]}}
\newcommand{\Xrn}{{\Xr[\Nv]}}
\newcommand{\Xrbn}{{\Xr[\downNv]}}
\newcommand{\XrR}{{\Xr[R]}}

\newcommand{\Uone}[1][]{{U^{(1)}_{#1}}}
\newcommand{\Utwo}[1][]{{U^{(2)}_{#1}}}
\newcommand{\Ur}[1][]{{U^{(r)}_{#1}}}
\newcommand{\Wone}[1][]{{W^{(1)}_{#1}}}
\newcommand{\Wtwo}[1][]{{W^{(2)}_{#1}}}
\newcommand{\Wr}[1][]{{W^{(r)}_{#1}}}

\newcommand{\Mk}{{M_k}}
\newcommand{\Mkt}{{M_k^\top}}
\newcommand{\Mkinv}{{\lt(\Mkt\Mk\rt)^{-1}}}
\newcommand{\Mkproj}{{\Mk \Mkinv \Mkt}}
\newcommand{\MkT}{{\lt(\Mkt\Mk\rt)_T^*}}
\newcommand{\MkprojT}{{\Mk \MkT \Mkt}}
\newcommand{\Ml}{{M_\ell}}
\newcommand{\Mlt}{{M_\ell^\top}}
\newcommand{\Mlinv}{{\lt(\Mlt\Ml\rt)^{-1}}}
\newcommand{\Mlproj}{{\Ml \Mlinv \Mlt}}
\newcommand{\MlT}{{\lt(\Mlt\Ml\rt)_T^*}}
\newcommand{\MlprojT}{{\Ml \MlT \Mlt}}

\newcommand{\In}{{I_{\Nv}}}
\newcommand{\Ibn}{{I_{\downNv}}}
\newcommand{\Deln}{{\Delta_{\Nv}}}
\newcommand{\Delbn}{{\Delta_{\downNv}}}
\newcommand{\Delbnsq}{{\Delta_{\downNv}^2}}
\newcommand{\Delone}[1][]{{\Delta^{(1)}_{#1}}}
\newcommand{\Delonen}{{\Delone[\Nv]}}
\newcommand{\Delonebn}{{\Delone[\downNv]}}
\newcommand{\Deltwo}[1][]{{\Delta^{(2)}_{#1}}}
\newcommand{\Deltwon}{{\Deltwo[\Nv]}}
\newcommand{\Deltwobn}{{\Deltwo[\downNv]}}
\newcommand{\Delr}[1][]{{\Delta^{(r)}_{#1}}}
\newcommand{\Delrn}{{\Delr[\Nv]}}
\newcommand{\Delrbn}{{\Delr[\downNv]}}

\newcommand{\IP}{{\mathrm{IP}}}
\newcommand{\GS}{{\mathsc{gs}}}
\newcommand{\cCgs}[1][]{{\cC^{\mathrm{GS}}_{#1}}}
\newcommand{\cCtckl}{{\cC^{\mathrm{TC}}_{k,\ell}}}
\newcommand{\cNvdips}{{\cN_{v,d}^{\IP}(S)}}
\newcommand{\cNvdipr}{{\cN_{v,d}^{\IP}(\bR^{d\times (v-1)})}}
\newcommand{\cMvd}{{\cM_{v,d}}}

\newcommand{\eZ}[1]{{e_Z\lt(#1\rt)}}
\newcommand{\sigY}[1]{{\sigma^2_Y\lt(#1\rt)}}
\newcommand{\sigZ}[1]{{\sigma^2_Z\lt(#1\rt)}}

\newcommand{\blambda}{{\bar \lambda}}
\newcommand{\Ber}{{\mathrm{Ber}}}
\newcommand{\Ctr}{{C_{\Tr}}}
\newcommand{\GOE}{{\mathrm{GOE}}}
\newcommand{\polylog}{{\mathrm{polylog}}}
\newcommand{\Sym}{{\mathrm{Sym}}}
\newcommand{\RGG}{{\mathsc{rgg}}}

\begin{document}

\title{De Finetti-Style Results for Wishart Matrices: \\ Combinatorial Structure and Phase Transitions}

\author{Matthew Brennan
\and
Guy Bresler\thanks{Massachusetts Institute of Technology. Department of EECS. Email: \texttt{guy@mit.edu}. Supported by MIT-IBM Watson AI Lab and NSF CAREER award CCF-1940205.}
\and
Brice Huang\thanks{Massachusetts Institute of Technology. Department of EECS. Email: \texttt{bmhuang@mit.edu}. Supported by NSF Graduate Research Scholarship 1745302, a Siebel Scholarship, and NSF TRIPODS award 1740751.}}

\date{\today}

\maketitle

\begin{center}
\textit{Dedicated to the memory of Matthew Brennan.}
\end{center}

\begin{abstract}
A recent line of work has studied the relationship between the Wishart matrix $X^\top X$, where $X\in \mathbb{R}^{d\times n}$ has i.i.d. standard Gaussian entries, and the corresponding Gaussian matrix with independent entries above the diagonal. Jiang and Li~\cite{JL15} and Bubeck et al.~\cite{BDER16} showed that these two matrix ensembles converge in total variation whenever $d/n^3\to \infty$, and~\cite{BDER16} showed this to be sharp. 
In this paper we aim to identify the precise threshold for $d$ in terms of $n$ for \emph{subsets} of Wishart matrices to converge in total variation to independent Gaussians.
It turns out that the combinatorial structure of the revealed entries, viewed as the adjacency matrix of a graph $G$, characterizes the distance from fully independent. 
Specifically, we show that the threshold for $d$ depends on the number of various small subgraphs in $G$. 
So, even when the number of revealed entries is fixed, the threshold can vary wildly depending on their configuration. Convergence of masked Wishart to independent Gaussians thus inherently involves an interplay between both probabilistic and combinatorial phenomena.
Our results determine the sharp threshold for a large family of $G$, including Erd\H{o}s-R\'enyi $G\sim \cG(n,p)$ at all values $p\gtrsim n^{-2}\polylog(n)$.
Our proof techniques are both combinatorial and information theoretic, which together allow us to carefully unravel the dependencies in the masked Wishart ensemble. 
\end{abstract}

\tableofcontents
\newpage

\section{Introduction}
\label{sec:intro}

It is a classical fact that the projection to a $k$ dimensional subspace of a random point from the unit sphere in $n$-dimensional Euclidean space has, in the limit $n\to \infty$ with $k$ growing slowly enough, a Gaussian distribution with covariance $I_k/n$ (see \cite{DF87} for the history). Diaconis and Freedman \cite{DF87} proved a finite version of the result, showing a bound on the total variation of the order $k/n$. While one can derive from this a de Finetti style representation theorem, one may also simply interpret the result as showing that dependence in a certain multivariate distribution is diminished under projection, with a precise bound on the distance to a suitable independent distribution. 

For exchangeable distributions, such as the uniform distribution on the Euclidean $n$-sphere, the choice of which $k$ coordinates to retain from the original $n$ is by definition immaterial. This is no longer the case for distributions satisfying only partial exchangeability. Consider for example the centered and normalized Wishart distribution $W(n,d)$ over $n\times n$ matrices, defined as the law of the matrix $W = d^{-1/2} (X^\top X -dI_n)$ where $X$ is a $d\times n$ matrix of i.i.d. standard Gaussian variables. The Wishart distribution satisfies \emph{joint exchangeability}, meaning that the distribution of $\{W_{\sigma(i),\sigma(j)}\}$ is the same as that of $\{W_{i,j}\}$ for all permutations $\sigma:[n]\to [n]$. Notably, different subsets of the same cardinality can possess differing amount of dependence. For example, the entries $(W_{1,2}, W_{3,4}, W_{5,6})$ are jointly independent while $(W_{1,2}, W_{2,3}, W_{1,3})$ are not. In this paper we study the role played by the structure of revealed entries. 

The degrees of freedom parameter $d$ controls the amount of dependence among the entries of the Wishart distribution. It follows from multidimensional CLTs that as $d\to \infty$ and $n$ is constant, the Wishart distribution $W(n,d)$ converges in total variation to the Gaussian Orthogonal Ensemble $\GOE(n)$ distribution which has independent Gaussian entries on and above the diagonal. Jiang and Li \cite{JL15} and Bubeck et al. \cite{BDER16} showed that for jointly varying $n$ and $d$, this convergence happens precisely when $d/n^3 \to \infty$.
We aim to characterize the precise threshold for $d$ in terms of $n$ for \emph{subsets} of Wishart matrices $W(n,d)$ to converge to corresponding subsets of $\GOE(n)$ matrices. 

It turns out that the combinatorial structure of the set of revealed entries, when viewed as the adjacency matrix of a graph $G$, characterizes the distance from fully independent. As we show in this paper, the difference between subsets of the same cardinality can be dramatic. For instance, revealing a random subset of half the entries leaves the phase transition unchanged at $d\asymp n^3$, while revealing the upper right and lower left $n/2$ by $n/2$ submatrices changes the phase transition to the much smaller $d\asymp n^2$.\footnote{The phase transition for this case was independently identified by Bubeck \cite{Bub20}.} The difference between these is that the former subset of entries corresponds to $G$ having on the order of $n^3$ triangles, while the latter subset of entries corresponds to $G$ being bipartite and hence triangle-free, in which case the number of 4-cycles plays a leading role. In general, we will show that the threshold for $d$ depends on the number of various small subgraphs in $G$. We emphasize that the dependence of this threshold on subgraph counts is not a product of our techniques, but an intrinsic property of subsets of the Wishart ensemble. Our results specialize to yield sharp results for Erd\H{o}s-R\'enyi masks $G\sim \cG(n,p)$ at all values $p\gtrsim n^{-2}\polylog(n)$, with different subgraphs being dominant depending on the sparsity. 

A number of recent papers have proved CLTs for random matrices, including for the Wishart ensemble \cite{JL15,BDER16,Jia06,Lev18,Mik20,RR19,Ste19}. These papers and others are discussed in the next subsection on related work.
The present paper continues this theme, but introduces a substantial generalization. From the technical perspective, the projected setting we consider adds nontrivial combinatorial structure to a phenomenon studied so far entirely with analytic techniques. These techniques often make use of properties that are brittle to projection onto an arbitrary subset of entries, such as explicit densities and characterizations of random matrices' eigenvalue distributions. Instead, our techniques are combinatorial and information theoretic; by carefully unraveling the dependencies in the masked Wishart ensemble, we make the phase transition tractable to analyze and obtain quite sharp results. The techniques developed in this paper are also useful in solving an open problem of Eldan and Mikulincer \cite{EM16} on the phase transition associated to detecting the anisotropic random geometric graph. These latter results are presented in a forthcoming paper.

\paragraph{Organization.} 
In the remainder of this section we overview related work and collect some useful notation.
We state our main convergence theorems in Section~\ref{sec:tv-upper-bounds} and show nearly matching converse results demonstrating the tightness of our theorems in Section~\ref{sec:tv-lower-bounds}.
Section~\ref{sec:er-masks} specializes our results to Erd\H{o}s-R\'enyi and bipartite Erd\H{o}s-R\'enyi masks for which, for all such graphs with more than $\polylog(n)$ expected edges, we precisely characterize the asymptotic threshold.
We provide a detailed technical overview in Section~\ref{sec:technical-overview} which describes the main ideas in the proofs of our convergence theorems.
Reading Sections~\ref{sec:intro} through \ref{sec:technical-overview} gives an accurate summary of our work.

Later sections contain proofs, structured as follows.
Sections~\ref{sec:general-ub-main-argument} through \ref{sec:general-ub-refined-higher-order-terms} prove Theorem~\ref{thm:general-ub}, our main convergence theorem.
In Section~\ref{sec:general-ub-main-argument}, we show a weaker variant of this theorem, whose proof contains the key ideas in the proof of Theorem~\ref{thm:general-ub}. 
In Sections~\ref{sec:general-ub-refined-linear-term} and \ref{sec:general-ub-refined-higher-order-terms} we improve the techniques from Section~\ref{sec:general-ub-main-argument} to prove Theorem~\ref{thm:general-ub}.
Section~\ref{sec:bipartite-ub-proof} proves Theorem~\ref{thm:bipartite-ub}, our main convergence theorem for bipartite masks.
Section~\ref{sec:tv-lower-bounds-proofs} proves the converses in Section~\ref{sec:tv-lower-bounds}.

\subsection{Related Work}
As mentioned above, there is a growing literature on CLTs for random matrices.
Chatterjee and Meckes \cite{CS08} proved a general multidimensional CLT via Stein's method.
As noted by Bubeck and Ganguly in \cite{BG16}, \cite[Theorem 7]{CS08} shows that the centered and rescaled Wishart ensemble $W(n,d)$ converges to $\GOE(n)$ in Wasserstein distance if $d/n^6 \to \infty$.
Jiang and Li~\cite{JL15} and Bubeck et al.~\cite{BDER16} showed this convergence holds in total variation if $d/n^3\to \infty$, and~\cite{BDER16} showed a matching converse whereby the total variation tends to $1$ if $d/n^3\to 0$. 
These sharp convergence results used the explicit expressions for Wishart and GOE densities. 
Bubeck and Ganguly~\cite{BG16} generalized the convergence result to the case of matrices $W=d^{-1/2} (X^\top X - \diag(X^\top X))$ with $X$ having i.i.d. entries from some log-concave measure, in which case the lack of an explicit density requires new techniques. Their approach uses tensorization of the Kullback-Leibler divergence, which we also do in this paper. However, as discussed in detail in Section~\ref{subsec:technical-overview-naive-comparison}, this tensorization does not by itself provide the precise control of dependencies needed in our setting with masked entries.
Racz and Richey~\cite{RR19} computed the total variation between $W(n,d)$ and $\GOE(n)$ in the limit $d/n^3\to c\in (0,\infty)$, making use of the fact that both the Wishart and GOE ensembles have explicit densities. Their expression easily evaluates to $0$ or $1$ in the limit $c\to \infty$ or $c\to 0$, respectively.
Chetelat and Wells~\cite{CW19} showed a remarkable countable sequence of phase transitions for the Wishart ensemble. For each $K\in \bN$, they defined an explicit density $f_K$ and showed that if $n^{K+3} / d^{K+1}\to 0$, then the normalized Wishart distribution converges in total variation to $f_K$. Their approach is based on a new variation of the Fourier transform that they introduce, applied to the densities under consideration.
Generalizing in a completely different direction, Mikulincer~\cite{Mik20} and Nourdin and Zheng~\cite{NZ18} showed CLTs for tensor analogues of the Wishart distribution, and for Wishart matrices as above where $X\in \bR^{n\times d}$ has correlated entries.

Understanding the relationship between the Wishart and GOE ensembles at different parameter values is intimately related to the analogous question for high-dimensional random geometric graphs versus Erd\H{o}s-R\'enyi. 
The most commonly studied high-dimensional random geometric graph $\RGG(n,p,d)$ associates to each node $i\in[n]$ a point $X_i$ uniformly sampled on the sphere $\bS^{d-1}$, and includes edge $(i,j)$ if $\langle X_i, X_j\rangle$ exceeds a threshold $t_{p,d}$ which is chosen so that each edge has marginal $\Ber(p)$.
When $p=1/2$ the two random graph models can be instantiated by thresholding the entries of Wishart and GOE matrices. Devroye et al. \cite{DLU11} applied a multivariate CLT to show that each of the $2^{\binom{n}{2}}$ terms in the summation over graphs in the total variation expression tend to zero for $d\gg n^7$, but this then requires $d\gg n^7 2^{\binom{n}{2}}$ for the total variation between the two random graph models to tend to zero. They also showed that the clique number of the random geometric graph is close to that of the Erd\H{o}s-R\'enyi graph whenever $d\gg \log^3 n$.  
Bubeck et al. 
\cite{BDER16} derived the sharp $d\asymp n^3$ threshold for total variation convergence of the two graphs from their result for Wishart versus GOE matrices. They also conjectured that as $p$ decreases, i.e., the graphs become sparser, convergence occurs at smaller values of $d$ than $n^3$. If $p=c/n$, they conjectured that the threshold occurs at $d\asymp \log^3 n$. 
Brennan et al.~\cite{BBN20} proved that indeed the threshold decreases as the graphs become sparser, with $\RGG(n,p,d)$ and $\cG(n,p)$ converging in total variation whenever $d=\tilde \omega(n^3 p, n^{7/2} p^2)$. Their methods combined information inequalities with probabilistic coupling arguments. They also showed sharp results for comparison of random intersection graphs with Erd\H{o}s-R\'enyi, as well as more general convergence results for matrices of intersections between families of random sets and Poisson random matrices. 
Eldan and Mikulincer \cite{EM16} studied the question of convergence of anisotropic random geometric graphs, and obtained lower and upper bounds on when the convergence occurs in terms of the dimension parameter. They left open the question of determining the threshold; as noted above, in a forthcoming paper we use techniques related to those developed in the present paper to determine the threshold precisely. 

Several papers have proved convergence results between the upper left $p\times q$ submatrix of a matrix uniformly sampled from the orthogonal group $O(n)$ and a $p\times q$ matrix of independent Gaussians.
Diaconis et al. showed that convergence occurs if $p=q=o(n^{1/3})$. 
Jiang \cite{Jia06} improved this result to $p=q=o(\sqrt{n})$, and moreover, showed that convergence does not occur if both $p$ and $q$ are $\Omega(\sqrt{n})$.
Jiang and Ma \cite{JM19} and
Stewart \cite{Ste19} generalized this result, showing convergence to Gaussian for any $p,q$ such that $pq=o(n)$. Both \cite{Jia06} and \cite{JM19} also consider distances other than total variation. 
The analogous question of when submatrices of random \textit{unitary} matrices converge to independent Gaussians is also relevant to quantum physics, as discussed in \cite{AA13,AA14,Lev18}.

\subsection{Notation}
For a positive integer $n$, let $[n] = \{1,\dots,n\}$.
Throughout this paper, let $G$ be a simple graph on $[n]$.
Let $E(G)$ denote the edge set of $G$. 
For simplicity of notation, we let $G$ also denote its vertex set.
Let $A_G$ denote the adjacency matrix of $G$.
For $v\in G$, let $N(v)$ denote the set of vertices adjacent to $v$ in $G$, and let $\deg(v) = |N(v)|$ denote the degree of $v$.
Let $G[v]$ denote the induced subgraph of $G$ on $[v]$.
For $i\in G$, let $\Nl(i) = N(i) \cap [v-1]$ denote the set of neighbors of $i$ in $[v-1]$, and let $\degvl(i) = |\Nl(i)|$. 
Let $\downNv = \Nl(v)$ and $\ddegv = |\downNv|$.

Throughout this paper, all quantities other than $n$, except where stated, are functions of $n$.
For example, $G = G_n$ and $d = d_n$ are the graph and degree of freedom parameter associated with $n$.
The asymptotic notation $f(n)\gg g(n)$ means $\lim_{n\to\infty} f(n)/g(n) = \infty$.
Similarly, $f(n) \gtrsim g(n)$ means $\liminf_{n\to\infty} f(n)/g(n) > 0$.
The asymptotic notations $f(n) \ll g(n)$ and $f(n) \lesssim g(n)$ are defined symmetrically.
The notation $f(n) \asymp g(n)$ means $f(n) \gtrsim g(n)$ and $f(n) \lesssim g(n)$.

The total variation distance between two probability measures $\mu$ and $\nu$ on the same space is denoted by $\TV(\mu,\nu) = \frac12\|\mu - \nu\|_1$. 
Similarly, Kullback-Leibler divergence and $\chisq$ divergence are denoted $\KL(\mu \parallel \nu)$ and $\chisq(\mu, \nu)$. These are defined in Section~\ref{sec:technical-overview}.

Throughout this paper, $\cN(0, \Sigma)$ denotes a jointly Gaussian vector with mean $0$ and covariance $\Sigma$.
Moreover, $\cN(0, \Sigma)^{\otimes k}$ denotes a matrix with $k$ i.i.d. columns, which are each a sample from $\cN(0, \Sigma)$.
Let $\chisq(d)$ denote a sample from a $\chisq$ distribution with $d$ degrees of freedom.

For two matrices $A$ and $B$ of the same dimensions, $A\odot B$ denotes the matrix Schur (or Hadamard) product, given by $(A\odot B)_{i,j} = A_{i,j} B_{i,j}$.
For a square matrix $A$, $\spec(A)$ denotes the set of eigenvalues of $A$, including multiplicity.
The Kronecker delta function is denoted by $\delta_{i,j} = \ind{i=j}$.  

\subsection*{Acknowledgements}
We are greatly indebted to Dheeraj Nagaraj for many helpful discussions throughout this work and for ideas that led to the proofs of Lemmas~\ref{lem:general-ub-kl-conditioning} and \ref{lem:dheeraj-method-deg2}. 
BH is also grateful to Mehtaab Sawhney and Dan Mikulincer for helpful conversations over the course of this work.
This work was done in part while the authors were participating in the Probability, Geometry, and Computation in High Dimensions program at the Simons Institute for the Theory of Computing in Fall 2020.

\section{Convergence of Subsets of Wishart and GOE in Total Variation}
\label{sec:tv-upper-bounds}

In this section, we formally introduce the models we study and state our main results.
We begin by defining the Wishart and Gaussian Orthogonal Ensemble (GOE) matrices.

\begin{definition}[Wishart and GOE matrices]
    \label{defn:wishart-goe}
    For positive integers $n,d$, let $W(n,d)$ denote the law of a centered and normalized isotropic Wishart matrix, i.e. the law of $d^{-1/2} (X^\top X - dI_n)$ where $X\in \bR^{d\times n}$ has i.i.d. standard Gaussian entries.
    Let $M(n)$ denote the law of a sample from an $n\times n$ GOE matrix, i.e. the symmetric random matrix with standard Gaussian off-diagonal entries, $\cN(0,2)$-distributed diagonal entries, and mutually independent entries on or above the diagonal.
\end{definition}

Let $G$ be a graph on $[n]$ with adjacency matrix $A_G$.
The objects of study of this paper are the masked Wishart and GOE matrices, which are samples from the Wishart and GOE matrices restricted to the entries $(i,j)$ corresponding to edges $(i,j)\in E(G)$.
Formally, they are defined as follows.

\begin{definition}[Masked Wishart and GOE matrices]
    Let $W(G,d)$ denote the law of $A_G \odot M$ where $M\sim W(|G|,d)$ and $\odot$ denotes the matrix Schur product.
    Similarly, let $M(G)$ denote the law of $A_G\odot M$ where $M\sim M(|G|)$.
\end{definition}

In this definition, the Schur product functions as a mask that preserves the entries $(i,j)$ of the Wishart and GOE matrices where $(i,j)\in E(G)$ and deletes the remaining entries.
Note that the diagonal entries of the Wishart and GOE matrices are necessarily deleted.

The object of this paper is to study asymptotic conditions on $n$, $G$, and $d$ under which we have $\TV(W(G,d), M(G)) \to 0$ or $\TV(W(G,d), M(G)) \to 1$.
In this section, we will focus on asymptotic conditions under which $\TV(W(G,d), M(G)) \to 0$, and we will study conditions under which $\TV(W(G,d), M(G)) \to 1$ in Section~\ref{sec:tv-lower-bounds}.
Throughout this paper, $G$ and $d$ are implicitly functions of $n$.
That is, $G$ implicitly denotes a sequence of graphs $(G_n)_{n\in \bN}$, where $G_n$ is a graph on $[n]$.
Similarly, $d$ implicitly denotes a sequence $(d_n)_{n\in \bN}$.
For simplicity of notation, we will typically denote $G_n$ and $d_n$ by $G$ and $d$.

When $G = K_n$, $W(G,d)$ and $M(G)$ are the ordinary Wishart and GOE matrices $W(n,d)$ and $M(n)$ with diagonal entries removed.
This case is well understood: there are many proofs \cite{BDER16, BG16, JL15, RR19} in the literature that, with or without diagonal entries, $W(n,d)$ and $M(n)$ converge to total variation distance $0$ if $d \gg n^3$ and diverge to total variation distance $1$ if $d \ll n^3$.
Thus, there is a sharp phase transition at $d\asymp n^3$.
We will recover this fact as a special case of our results for arbitrary $G$.
Our results will characterize sharp phase transitions for a large family of graph sequences $G$, which will include, as we will see in Section~\ref{sec:er-masks}, all typical instantiations of Erd\H{o}s-R\'enyi $G$ with more than polylogarithmically many expected edges.

\subsection{Subgraph Counts}

The criteria under which we show $\TV(W(G,d), M(G)) \to 0$ and $\TV(W(G,d), M(G)) \to 1$ in this paper depend on counts of small subgraphs in $G$ .
To state our results formally, we first introduce a notion of subgraph count. 
\begin{definition}
    \label{defn:subgraph-count}
    For a fixed graph $H$, let $\Num_G(H)$ denote the number of subgraphs $G'\subseteq G$ isomorphic to $H$.
    We do not require $G'$ to be an induced subgraph of $G$. 
    Unlike in the notation $\Hom_G(H)$ used in the combinatorics literature, we do require that $G'$ is a non-degenerate copy of $H$, i.e. each vertex of $H$ must correspond to a distinct vertex of $G'$.
    To reduce notational clutter, let $\Num_G(H_1,\ldots,H_k) = \sum_{i=1}^k \Num_G(H_i)$.
\end{definition}
We now define the subgraphs $H$ whose count $\Num_G(H)$ will appear in this paper.
Let $E$ denote the graph consisting of two vertices connected by an edge.
For $r\ge 2$, let $P_r$ denote the path with $r$ edges and $r+1$ vertices.
For $r\ge 3$, let $C_r$ denote the cycle of length $r$.
For $r,s\ge 1$, let $K_{r,s}$ denote the $(r,s)$-complete bipartite graph.
The remaining graphs whose count will appear in this paper are enumerated in Figure~\ref{fig:subgraph-counts}.

\begin{figure}[h!]
    \centering
    \includegraphics{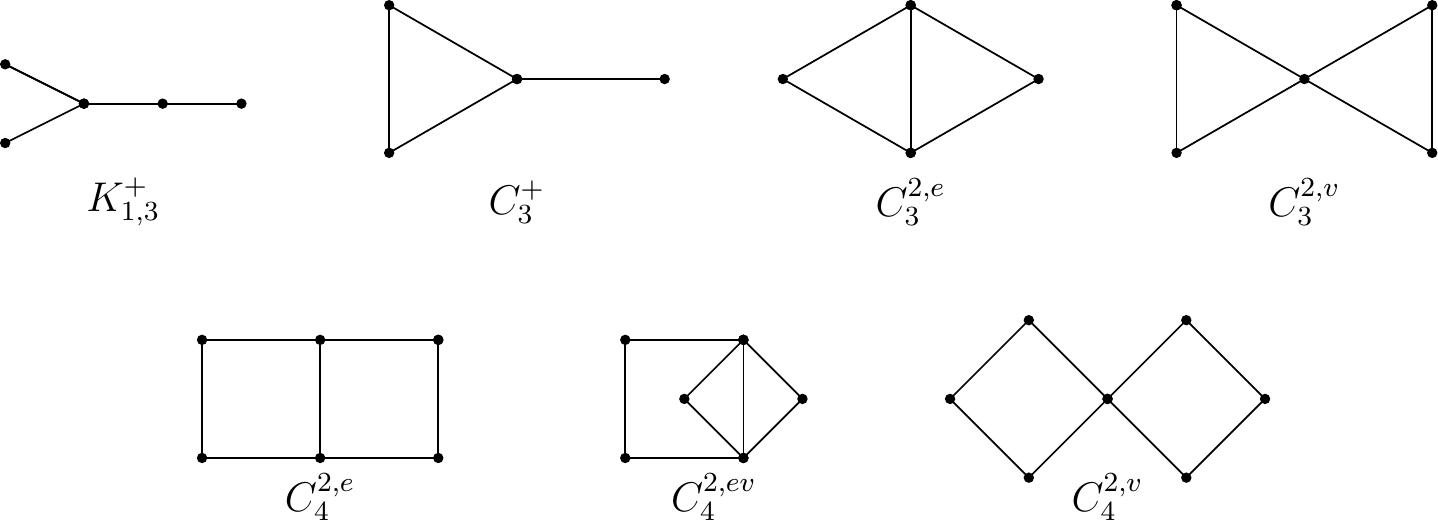}
    \caption{Graphs whose count appears in this paper.}
    \label{fig:subgraph-counts}
\end{figure}

Some of our results study the case when $G$ is bipartite. 
In this setting, the criteria under which we show $\TV(W(G,d), M(G)) \to 0$ or $1$ depend on the above subgraph counts and certain oriented subgraph counts, which we now introduce.
When $G$ is bipartite, we may assign an orientation to $G$: we partition the vertices of $G$ into left-vertices $V_L(G)$ and right-vertices $V_R(G)$ such that $E(G) \subseteq V_L(G) \times V_R(G)$.
When $G$ is clear from context, we refer to these sets as $V_L$ and $V_R$.
We now define a notion of oriented subgraph count with respect to the orientation $(V_L, V_R)$.

\begin{definition}
    For a fixed bipartite graph $H$, also equipped with an orientation $(V_L(H), V_R(H))$, let $\oNum_G(H)$ denote the number of subgraphs $G'\subseteq G$ isomorphic to $H$, such that vertices of $G'$ corresponding to $V_L(H)$ are in $V_L(G)$ and vertices of $G'$ corresponding to $V_R(H)$ are in $V_R(G)$.
    As in Definition~\ref{defn:subgraph-count}, we do not require $G'$ to be an induced subgraph, but do require that $G'$ is not degenerate.
    For notational simplicity, let $\oNum_G(H_1,\ldots,H_k) = \sum_{i=1}^k \oNum_G(H_i)$.
\end{definition}
We emphasize that the oriented subgraph counts $\oNum_G(H)$ are defined only with respect to a fixed orientation $(V_L, V_R)$ of $G$, and that this orientation may not be unique, even up to interchanging $V_L$ and $V_R$.

Let us define the bipartite graphs $H$ whose oriented count $\oNum_G(H)$ will appear in this paper.
For $r,s\ge 1$, let $\oK_{r,s}$ denote the $(r,s)$-complete bipartite graph, whose left-vertices and right-vertices are the sides of the bipartition with $r$ and $s$ vertices, respectively.
The remaining bipartite graph we will need is the oriented 4-path $\oP_4$, depicted in Figure~\ref{fig:oriented-subgraph-counts}.

\begin{figure}[h!]
    \centering
    % \begin{asy}
    %     unitsize(1.0cm);
    %     pair LABEL_SHIFT = (2, 0.5);
    %     pair CENTERING_SHIFT = (2, 1.5);
    %     void dot_shift(pair p, pair dif) {
    %         dot(p + dif + CENTERING_SHIFT);
    %     }
    %     void draw_shift(path p, pair dif) {
    %         draw(shift(dif + CENTERING_SHIFT) * (p));
    %     }
    %     void label_shift(string s, pair dif) {
    %         label(s, dif + LABEL_SHIFT);
    %     }
    %     void dot_path(string s, pair[] p, pair dif) {
    %         int n = p.length;
    %         label_shift(s, dif);
    %         for (int i=0; i < n-1; ++i) {
    %             dot_shift(p[i], dif);
    %             draw_shift(p[i]--p[i+1], dif);
    %         }
    %         dot_shift(p[n-1], dif);
    %     }
    %     // P4
    %     pair P4_SHIFT = (0, 0);
    %     pair[] P4_VERTICES = {(0.75, 0.5), (-0.75, 0.25), (0.75, 0), (-0.75, -0.25), (0.75, -0.5)};
    %     dot_path("$\overrightarrow{P}_4$", P4_VERTICES, P4_SHIFT);
    % \end{asy}
    \includegraphics{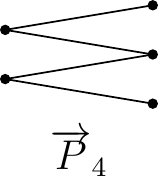}
    \caption{
        The oriented 4-path.
        The left-vertices and right-vertices of this graph are drawn, respectively, on the left and right.
    }
    \label{fig:oriented-subgraph-counts}
\end{figure}

\subsection{Result for General Masks}

Our main result below identifies conditions under which $\TV(W(G,d), M(G)) \to 0$.
\begin{theorem}
    \label{thm:general-ub}
    Suppose the following asymptotic inequalities hold:
    \begin{eqnarray}
        \label{eq:general-ub-hypothesis-triangles}
        d &\gg& \Num_G(C_3), \\
        \label{eq:general-ub-hypothesis-4cycles}
        d^2 &\gg& \Num_G(C_4, P_2, E), \\
        \label{eq:general-ub-hypothesis-k18}
        d^4 &\gg& \Num_G(K_{1,8}) + \log^8 n\, \Num_G(K_{1,4}, E).
    \end{eqnarray}
    Then, $\TV(W(G,d), M(G)) \to 0$ as $n\to \infty$.
\end{theorem}

Note that when $G=K_n$, this theorem recovers the $d \gg n^3$ threshold at which $W(n,d)$ and $M(n)$ (without diagonal entries) converge in total variation.

Conditions (\ref{eq:general-ub-hypothesis-triangles}) and (\ref{eq:general-ub-hypothesis-4cycles}) are sharp, in the sense that they have matching TV lower bounds.
Theorems~\ref{thm:deg3-lb} and \ref{thm:deg4-lb} below show that in the presence of a mild condition, if $d \ll \Num_G(C_3)$ or $d^2 \ll \Num_G(C_4, P_2, E)$, then $\TV(W(G,d), M(G)) \to 1$.
Moreover, for a large family of $G$ -- including, as we will see in Theorem~\ref{thm:er-mask}, typical samples from all Erd\H{o}s-R\'enyi graphs with more than $\polylog(n)$ expected edges -- (\ref{eq:general-ub-hypothesis-k18}) is implied by one of (\ref{eq:general-ub-hypothesis-triangles}) and (\ref{eq:general-ub-hypothesis-4cycles}).
Thus, for all $G$ in this family, Theorem~\ref{thm:general-ub} identifies the correct asymptotic threshold under which $\TV(W(G,d), M(G)) \to 0$.

Condition (\ref{eq:general-ub-hypothesis-k18}) is a product of our methods, and is in general not sharp.
In graphs where this condition dictates the threshold given by the theorem, Theorem~\ref{thm:general-ub} will be suboptimal.
For example, for $G = K_{n/2, n/2}$, Theorem~\ref{thm:general-ub} gives that $\TV(W(G,d), M(G)) \to 0$ when $d \gg n^{9/4}$.
As we will see below in Theorem~\ref{thm:bipartite-ub}, we in fact have $\TV(W(G,d), M(G)) \to 0$ when $d \gg n^2$, and Theorem~\ref{thm:deg4-lb} below implies that this is the correct threshold.

We will prove Theorem~\ref{thm:general-ub} in Sections~\ref{sec:general-ub-main-argument} through \ref{sec:general-ub-refined-higher-order-terms}.
A detailed outline of this proof will be given in Section~\ref{sec:technical-overview}; we sketch here the main ideas.
The key challenge in upper bounding $\TV(W(G,d), M(G))$ is unraveling the intricate dependencies among the entries of $W(G,d)$, and we will devise information theoretic techniques to do so.

The first idea of our proof is to isolate the information contribution of each vertex of $G$.
We first use Pinsker's Inequality to pass from TV distance to KL divergence, to take advantage of KL divergence's tensorization properties.
We consider an iterative construction of the masked Wishart matrix $W(G,d)$, where we set the i.i.d. latent vectors $X_1,\ldots,X_n \sim \cN(0,I_d)$ one by one; thus, after the $v$th step we observe the upper-left $v\times v$ submatrix of $W(G,d)$.
By KL tensorization, we can write $\KL(W(G,d) \parallel M(G))$ as a sum of $n$ averaged KL divergences, where the $v$th summand is the information contribution of adding vertex $v$. 
We will bound these summands separately. 

The second idea is to recognize each of these summands as a KL divergence between a mixture of Gaussians and a Gaussian.
Indeed, the entries of $W(G,d)$ revealed in the $v$th step are a subset of entries of $d^{-1/2} \Xvl^\top X_v$, where $\Xvl = (X_1,\ldots,X_{v-1})$; this is jointly Gaussian conditioned on $\Xvl$.
The corresponding entries of $M(G)$ are, of course, Gaussian.
Now, $\chisq$ divergence is amenable to mixtures via the second moment method.
So, after truncating on a high probability event $S^v\in \sigma(\Xvl)$ to ensure integrability, we bound each summand by passing to $\chisq$ divergence.

After applying the second moment method, it remains to bound a coupled exponentiated overlap
\begin{equation}
    \label{eq:tv-ubs-exp-overlap}
    \E
    \exp\lt(
        \f{1}{2d^2}
        \sum_{i,j\in \downNv}
        \lt(
            \la \Xonei, \Xonej\ra - d\delta_{i,j}
        \rt)
        \lt(
            \la \Xtwoi, \Xtwoj\ra - d\delta_{i,j}
        \rt)
    \rt).
\end{equation}
Here, the expectation is over $(\Xonevl, \Xtwovl)$ where $\Xonevl$ is sampled from $\cN(0, I_d)^{\otimes v-1}$ conditioned on $\Xonevl\in S$ and $\Xtwovl$ is an independent copy of $\Xonevl$ conditioned further on $\la \Xonei, \Xonej \ra = \la \Xtwoi, \Xtwoj \ra$ for all $(i,j)\in E(G[v-1])$.
Due to the complex dependencies in the coupling of $\Xonevl$ and $\Xtwovl$, this expectation is difficult to evaluate or bound.
Controlling this overlap is our main technical contribution; the bulk of Sections~\ref{sec:general-ub-main-argument} through \ref{sec:general-ub-refined-higher-order-terms} is dedicated to this task.

The third idea in our proof is to expand the exponential in (\ref{eq:tv-ubs-exp-overlap}) into multiple terms and apply convexity in a different way for each term.
We will show by convexity that, for each term in this expansion, taking a stronger coupling over $(\Xonevl, \Xtwovl)$ can only increase that term's expectation.
We will devise a tailored stronger coupling to each term, which allows us to tractably estimate each term while still attaining a reasonably sharp bound. 
Combining these bounds yields Theorem~\ref{thm:general-ub}.
This term-by-term convexity argument is necessary: we will see that a global convexity argument, which upper bounds (\ref{eq:tv-ubs-exp-overlap}) with a single stronger coupling for the entire expression, does not capture the true dependence of $\TV(W(G,d), M(G))$ on $G$. 

We remark that, while our KL tensorization step is reminiscent of the approach of \cite{BG16}, this approach cannot by itself optimally determine the threshold at which $\TV(W(G,d), M(G)) \to 0$, as we will see in Section~\ref{subsec:technical-overview-naive-comparison}.
The approach of \cite{BG16} is equivalent to estimating (\ref{eq:tv-ubs-exp-overlap}) with the global coupling $\Xonevl = \Xtwovl$, which as discussed above is suboptimal.
This underscores the importance of our term-by-term convexity argument to deriving the thresholds in Theorem~\ref{thm:general-ub}.

\subsection{Result for Bipartite Masks}
We also study the case where $G$ is bipartite, because in this case our methods give especially sharp results.
For bipartite $G$, we can forgo the KL tensorization step and pass to $\chisq$ divergence directly, applying the second moment method with all the $X_v$ for $v$ on one side of $G$ as latent randomness.
Like for general $G$, the second moment method leaves the task of bounding the expectation of an exponentiated overlap.
However, in this setting, the two latent random matrices in the exponentiated overlap are fully independent.
Thus we have an expectation over only i.i.d. Gaussians, which allows a sharp analysis.
We derive the following theorem identifying conditions under which $W(G,d)$, and $M(G)$ converge in total variation for bipartite $G$.
We will prove this theorem in Section~\ref{sec:bipartite-ub-proof}.

\begin{theorem}
    \label{thm:bipartite-ub}
    Let $G$ be a bipartite graph with a fixed orientation $(V_L, V_R)$.
    Suppose the following four asymptotic inequalities hold.
    \begin{eqnarray}
        \label{eq:bipartite-ub-hypothesis-4cycles}
        d^2 &\gg& \Num_G(C_4, P_2, E), \\
        \label{eq:bipartite-ub-hypothesis-k14}
        d^3 &\gg& \oNum_G(\oK_{1,4}) + \Num_G(E) \log^3 n, \\
        \label{eq:bipartite-ub-hypothesis-d8}
        d^8 &\gg& \oNum_G(\oK_{1,3})^2\oNum_G(\oK_{2,4}) + 
        \Num_G(E)^2 \oNum_G(\oK_{1,4}, \oK_{2,4}) \log^4 n, \\
        \label{eq:bipartite-ub-hypothesis-d9}
        d^9 &\gg& \Num_G(E)^2 \oNum_G(\oP_4) \log^4 n.
    \end{eqnarray}
    Then, $\TV(W(G,d), M(G)) \to 0$ as $n\to \infty$.
\end{theorem}

In this theorem, the condition (\ref{eq:bipartite-ub-hypothesis-4cycles}) is sharp, and is matched by Theorem~\ref{thm:deg4-lb}.
For a large family of bipartite $G$ -- including, as we will see in Theorem~\ref{thm:bip-er-mask}, typical samples from all bipartite Erd\H{o}s-R\'enyi graphs with more than $\polylog(n)$ expected edges, even when one side of the graph is much larger than the other -- the remaining three conditions are implied by (\ref{eq:bipartite-ub-hypothesis-4cycles}).

Note that there may be many possible choices of the orientation $(V_L, V_R)$ of $G$.
To deduce the conclusion that $\TV(W(G,d), M(G)) \to 0$, we only need the hypotheses of Theorem~\ref{thm:bipartite-ub} to hold for one valid choice of orientation, and we may choose the orientation such that Theorem~\ref{thm:bipartite-ub} is strongest.
We will exploit this fact when we study random bipartite graphs $G$ in Theorem~\ref{thm:bip-er-mask}.

\section{Divergence of Subsets of Wishart and GOE in Total Variation}
\label{sec:tv-lower-bounds}

In this section, we identify asymptotic conditions on $n, G, d$ under which $\TV(W(G,d), M(G)) \to 1$.
These results function as converses to the results in Section~\ref{sec:tv-upper-bounds}.
We will see that for many graphs $G$, the results in this section give thresholds that asymptotically match the thresholds derived in Section~\ref{sec:tv-upper-bounds}.

Define $C_3(G)$, the set of 3-cycles in $G$, by
\[
    C_3(G) =
    \lt\{
    (i,j,k):
    (i,j), (j,k), (k,i) \in E(G),
    i < j < k
    \rt\}.
\]
Define the \textit{degree 3 statistic} $\kappa_3 : \bR^{n\times n} \to \bR$ by
\[
    \kappa_3(M) =
    \sum_{(i,j,k)\in C_3(G)}
    M_{i,j}
    M_{j,k}
    M_{k,i}.
\]
This is the restriction of the 3-cycles statistic $\sum_{1\le i<j<k\le n} M_{i,j} M_{j,k} M_{k,i}$ analyzed in \cite{BDER16}, which separates the ordinary Wishart and GOE matrices $W(n,d)$ and $M(n)$ to total variation $1$ when $d \ll n^3$, to the masked setting.
The degree 3 statistic yields the following criterion for TV divergence.

\begin{theorem}
    \label{thm:deg3-lb}
    Suppose the following two asymptotic inequalities hold:
    \begin{eqnarray}
        \label{eq:deg3-lb-hypothesis-triangles}
        d &\ll& \Num_G(C_3), \\
        \label{eq:deg3-lb-hypothesis-regularity}
        \Num_G(C_3^{2,e}, C_3^{2,v}) &\ll& \Num_G(C_3)^2.
    \end{eqnarray}
    Then, $\TV(\kappa_3(W(G,d)), \kappa_3(M(G))) \to 1$ as $n\to \infty$.
    In particular, $\TV(W(G,d), M(G)) \to 1$.
\end{theorem}

We think of (\ref{eq:deg3-lb-hypothesis-triangles}) as the main condition in this theorem.
The condition (\ref{eq:deg3-lb-hypothesis-regularity}) is a mild regularity condition satisfied by all non-pathological graphs.
To see this, note that $\Num_G(C_3^{2,e}, C_3^{2,v})$ counts pairs of 3-cycles in $G$ intersecting in at least one or two vertices, while $\Num_G(C_3)^2$ counts pairs of 3-cycles in $G$ without restriction; for most $G$, most pairs of 3-cycles in $G$ will not intersect.

Recall that $E(G)$ is the set of edges in $G$.
We can analogously define the sets of 4-cycles and 2-paths in $G$ by
\begin{eqnarray*}
    C_4(G)
    &=&
    \lt\{
    (i,j,k,\ell) :
    (i,j), (j,k), (k,\ell), (\ell, i) \in E(G),
    i < \min(j,k,\ell), j<\ell
    \rt\}, \\
    P_2(G)
    &=&
    \lt\{
    (i,j,k) :
    (i,j), (j,k)\in E(G), i<k
    \rt\}.
\end{eqnarray*}
The inequalities among the indices in these conditions ensure that each 4-cycle and 2-path is included exactly once.
Define the \textit{degree 4 statistic} $\kappa_4 : \bR^{n\times n} \to \bR$ by
\[
    \kappa_4(M) =
    \kappa_4^{C_4}(M) +
    \kappa_4^{P_2}(M) +
    \kappa_4^{E}(M),
\]
where the \textit{4-cycles statistic} $\kappa_4^{C_4}$, \textit{2-paths statistic} $\kappa_4^{P_2}$, and \textit{edges statistic} $\kappa_4^{E}$ are defined by
\begin{eqnarray*}
    \kappa_4^{C_4}(M)
    &=&
    \sum_{(i,j,k,\ell)\in C_4(G)}
    M_{i,j}
    M_{j,k}
    M_{k,\ell}
    M_{\ell,i}, \\
    \kappa_4^{P_2}(M)
    &=&
    \sum_{(i,j,k)\in P_2(G)}
    (M_{i,j}^2-1)
    (M_{j,k}^2-1), \\
    \kappa_4^{E}(M)
    &=&
    \sum_{(i,j)\in E(G)}
    (M_{i,j}^4-6M_{i,j}^2+3).
\end{eqnarray*}
Let us motivate this choice of statistic.
If we expand the likelihood ratio $L = \rn{W(G,d)}{M(G)}$ in the orthonormal basis of Hermite polynomials in the entries of $M(G)$, the expansion up to degree $3$ is $L \approx 1 + \f{1}{\sqrt{d}} \kappa_3(M)$.
This explains why $\kappa_3(M)$ is a natural statistic: it is the lowest degree nontrivial term in the expansion of the likelihood ratio with respect to the Hermite basis.
Since, by the Neyman-Pearson lemma, the likelihood ratio test is the most powerful test between two distributions, it is reasonable to expect a low degree proxy for the likelihood ratio to be powerful as well.
Extending this heuristic reasoning, the Fourier expansion of $L$ up to degree $4$ is
\[
    L \approx 1 +
    \f{1}{\sqrt{d}} \kappa_3(M) +
    \f{1}{d} \kappa_4^{C_4}(M) +
    \f{2}{d} \kappa_4^{P_2}(M) +
    \f{6}{d} \kappa_4^{E}(M).
\]
Thus, the degree 4 term of the Fourier expansion is $\kappa_4(M)$, up to constant factors on the three constituent terms of $\kappa_4(M)$.
So, when the statistic $\kappa_3$ is not powerful enough to test between $W(G,d)$ and $M(G)$ -- as is the case for bipartite $G$, where $\kappa_3(M)$ is identically zero, or more generally $G$ with few 3-cycles -- it is reasonable to consider $\kappa_4$ next.
This statistic yields the following criterion for TV divergence.

\begin{theorem}
    \label{thm:deg4-lb}
    Suppose the following two asymptotic inequalities hold.
    \begin{eqnarray}
        \label{eq:deg4-lb-hypothesis-4cycs}
        d^2 &\ll& \Num_G(C_4, P_2, E), \\
        \label{eq:deg4-lb-hypothesis-regularity}
        \Num_G(K_{1,4}, K_{2,4}, C_4^{2,e}, C_4^{2,v}) &\ll& \Num_G(C_4, P_2, E)^2.
    \end{eqnarray}
    Then, $\TV(\kappa_4(W(G,d)), \kappa_4(M(G))) \to 1$ as $n\to \infty$.
    In particular, $\TV(W(G,d), M(G)) \to 1$.
\end{theorem}
Like above, (\ref{eq:deg4-lb-hypothesis-4cycs}) is this theorem's main condition and (\ref{eq:deg4-lb-hypothesis-regularity}) is a mild regularity condition.

Finally, define the \textit{longest row statistic} $\kappa_r$ as follows.
Recall that for $v\in G$, $N(v)$ and $\degv$ denote the neighborhood and degree of $v$.
Let $\maxdeg(G) = \max_{v\in G} \degv$ be the maximal degree of $G$.
Let $v^*\in G$ be the vertex with maximal degree, breaking ties in an arbitrary but deterministic way (for example, the maximal-degree vertex with smallest label).
Let
\[
    \kappa_r(M)
    =
    \f{1}{\maxdeg(G)}
    \sum_{i\in N(v^*)}
    M_{v^*,i}^2.
\]
The longest row statistic yields the following criterion for TV divergence.
\begin{theorem}
    \label{thm:maxdeg-lb}
    Suppose that $d \ll \maxdeg(G)$.
    Then, $\TV(\kappa_r(W(G,d)), \kappa_r(M(G))) \to 1$ as $n\to \infty$.
    In particular, $\TV(W(G,d), M(G)) \to 1$.
\end{theorem}
This result is usually less powerful than Theorem~\ref{thm:deg4-lb}.
However, it will be useful in the proof of Theorem~\ref{thm:bip-er-mask} below, when the smaller side of the random bipartite graph has $O(1)$ vertices.
For such graphs, the regularity condition (\ref{eq:deg4-lb-hypothesis-regularity}) fails to hold, and we will use Theorem~\ref{thm:maxdeg-lb} to establish the threshold for TV divergence.

We will prove these results in Section~\ref{sec:tv-lower-bounds-proofs}.
The proofs of Theorems~\ref{thm:deg3-lb} and \ref{thm:deg4-lb} echo the proofs of the TV lower bounds in \cite{BDER16,BBN20}: we will compute the mean and variance of the statistics $\kappa_3(M)$ and $\kappa_4(M)$ for $M\sim W(G,d)$ and $M\sim M(G)$ and show these statistics' distributions separate to total variation $1$ by Chebyshev's Inequality.
To prove Theorem~\ref{thm:maxdeg-lb}, we will characterize the distributions of the statistic $\kappa_r(M)$ for $M\sim W(G,d)$ and $M\sim M(G)$.
In both cases the statistic concentrates around $1$, but the fluctuations are larger in the former case.
We will show an anticoncentration result for $\kappa_r(M)$ when $M\sim W(G,d)$ and a concentration result for $\kappa_r(M)$ when $M\sim M(G)$.
This will imply that the statistics' distributions separate to total variation $1$.

\begin{remark}
    We can also consider a masked version of the analogous problem for random geometric and Erd\H{o}s-R\'enyi graphs.
    In this problem, we observe the presence or absence of a subset of edges of $\RGG(n,p,d)$ and $\cG(n,p)$, and we want to identify asymptotic conditions under which samples from these two models converge and diverge in total variation.

    By a data processing argument like that of \cite[Section 5]{BDER16}, the criteria we derive under which $\TV(W(G,d), M(G)) \to 0$ yield analogous criteria for convergence of masked samples of $\cG(n,p)$ and $\RGG(n,p,d)$.
    However, the criteria we derive under which $\TV(W(G,d), M(G)) \to 1$ do not all correspond to analogous criteria for divergence of masked samples of $\cG(n,p)$ and $\RGG(n,p,d)$.
    This contrasts sharply with the non-masked setting, where for constant $p$, $\TV(W(n,d), M(n))\to 1$ and $\TV(\cG(n,p), \RGG(n,p,d)) \to 1$ occur at the same asymptotic threshold and are witnessed by analogous statistics based on 3-cycles.
    
    The TV divergence criteria witnessed by the statistics $\kappa_3$ and $\kappa_4^{C_4}$ do carry over to the random graphs setting: their analogues are the signed 3-cycles statistic
    \[
        \omega_3(M)
        =
        \sum_{(i,j,k)\in C_3(G)}
        (M_{i,j} - p)(M_{j,k} - p)(M_{i,k} - p)
    \]
    and signed 4-cycles statistic
    \[
        \omega_4(M)
        =
        \sum_{(i,j,k,\ell)\in C_4(G)}
        (M_{i,j} - p)(M_{j,k} - p)(M_{k,\ell} - p)(M_{\ell,i} - p).
    \]
    However, the statistics $\kappa_4^{P_2}$, $\kappa_4^{E}$, and $\kappa_r$ do not have analogues.
    These statistics are expressions of degree more than $1$ in the entries of $W(G,d)$ and $M(G)$, which encode high-degree information not present in the binary edge indicators in $\RGG(n,p,d)$ or $\cG(n,p)$.
    If $\Num_G(C_4) \ll \Num_G(P_2, E)$, the power of $\kappa_4$ comes predominantly from $\kappa_4^{P_2}$ and $\kappa_4^{E}$; we believe that in this regime, the threshold in $d$ separating the $\TV(W(G,d), M(G)) \to 0$ and $\TV(W(G,d), M(G)) \to 1$ regimes does not coincide with the analogous threshold for masked $\RGG(n,p,d)$ and $\cG(n,p)$.
\end{remark}

\section{Sharp Phase Transitions for Random Masks}
\label{sec:er-masks}

In this section, we will specialize the results of Sections~\ref{sec:tv-upper-bounds} and \ref{sec:tv-lower-bounds} to the case where $G$ is a sample from an Erd\H{o}s-R\'enyi or bipartite Erd\H{o}s-R\'enyi graph.
In this setting, we will show sharp phase transitions between the $\TV(W(G,d), M(G)) \to 0$ and $\TV(W(G,d), M(G)) \to 1$ regimes for all Erd\H{o}s-R\'enyi and bipartite Erd\H{o}s-R\'enyi $G$ with more than $\polylog(n)$ expected edges.

Let $\cG(n,p)$ denote the Erd\H{o}s-R\'enyi graph with edge probability $p\in [0,1]$, which is implicitly a function of $n$.
The following result identifies the TV convergence and divergence regimes for Erd\H{o}s-R\'enyi $G$.

\begin{theorem}
    \label{thm:er-mask}
    Let $G = G_n \sim \cG(n,p)$.
    Over the randomness of the sample path $G_1, G_2, \ldots$, the following limits occur almost surely.
    \begin{enumerate}[label=(\alph*), ref=\alph*]
        \item \label{itm:er-mask-ub} If
        \begin{equation}
            \label{eq:er-mask-ub-hypothesis}
            d \gg n^3p^3 + n^{3/2} p + np^{1/2} + n^{1/2}p^{1/4}\log^2 n + \log^3 n,
        \end{equation}
        then $\TV(W(G,d), M(G)) \to 0$.
        \item \label{itm:er-mask-lb} Conversely, if $p \gg n^{-2} \log^3 n$ and
        \begin{equation}
            \label{eq:er-mask-lb-hypothesis}
            d \ll n^3p^3 + n^{3/2} p + np^{1/2},
        \end{equation}
        then $\TV(W(G,d), M(G)) \to 1$.
    \end{enumerate}
\end{theorem}
When $p=1$, we recover the $d\asymp n^3$ phase transition separating the regimes where the ordinary Wishart and GOE matrices (with diagonals removed) converge and diverge in total variation.
Moreover, for all $p \gtrsim n^{-2}\log^8 n$, (\ref{eq:er-mask-ub-hypothesis}) is equivalent to $d \gg n^3p^3 + n^{3/2}p + np^{1/2}$.
Thus, Theorem~\ref{thm:er-mask} implies that for all $p\gtrsim n^{-2}\log^8 n$, the sharp phase transition separating the regimes where $W(G,d)$ and $M(G)$ converge and diverge in total variation is $d \asymp n^3p^3 + n^{3/2} p + np^{1/2}$.

Figure~\ref{fig:er-mask-phase-diagram} summarizes the $(p,d)$ for which Theorem~\ref{thm:er-mask} gives that $W(G,d)$ and $M(G)$ converge and diverge in total variation.
Note the tradeoff between sparsity and combinatorial structure evident in this result.
As the mask $G$ becomes sparser, the combinatorial structure determining the threshold becomes more and more disordered:
3-cycles dominate at $p\gtrsim n^{-3/4}$, followed by 2-paths at $n^{-3/4}\gtrsim p \gtrsim n^{-1}$, followed by edges at $p\lesssim n^{-1}$.
At the phase transitions $p\asymp n^{-3/4}$ and $p\asymp n^{-1}$, one combinatorial structure replaces another as the structure determining the threshold.

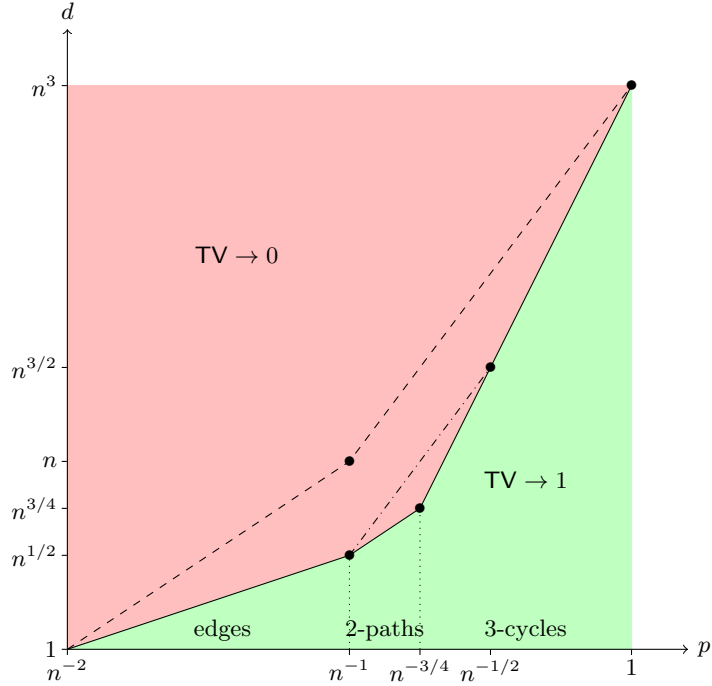
\begin{figure}[h!]
    \centering
    \begin{tikzpicture}[scale=0.75]
    \tikzstyle{every node}=[font=\footnotesize]
    \def\xmin{0}
    \def\xmax{11}
    \def\ymin{0}
    \def\ymax{11}

    \node at (0, 0) [below] {$n^{-2}$};
    \draw (0, -0.1) -- (0, 0);
    \node at (5, 0) [below] {$n^{-1}$};
    \draw (5, -0.1) -- (5, 0);
    \node at (6.25, 0) [below] {$n^{-3/4}$};
    \draw (6.25, -0.1) -- (6.25, 0);
    \node at (7.5, 0) [below] {$n^{-1/2}$};
    \draw (7.5, -0.1) -- (7.5, 0);
    \node at (10, 0) [below] {$1$};
    \draw (10, -0.1) -- (10, 0);
    \node at (0, 0) [left] {$1$};
    \draw (-0.1, 0) -- (0, 0);
    \node at (0, 1.667) [left] {$n^{1/2}$};
    \draw (-0.1, 1.667) -- (0, 1.667);
    \node at (0, 2.5) [left] {$n^{3/4}$};
    \draw (-0.1, 2.5) -- (0, 2.5);
    \node at (0, 3.333) [left] {$n$};
    \draw (-0.1, 3.333) -- (0, 3.333);
    \node at (0, 5) [left] {$n^{3/2}$};
    \draw (-0.1, 5) -- (0, 5);
    \node at (0, 10) [left] {$n^3$};
    \draw (-0.1, 10) -- (0, 10);

    \filldraw[fill=red!25, draw=red!25] (0,0) -- (5, 1.667) -- (6.25, 2.5) -- (10, 10) -- (0,10) -- (0,0);
    \filldraw[fill=green!25, draw=green!25] (0,0) -- (5, 1.667) -- (6.25, 2.5) -- (10,10) -- (10,0) -- (0,0);
    \draw (0,0) -- (5, 1.667) -- (6.25, 2.5) -- (10, 10);
    \draw[dashed] (0,0) -- (5, 3.333) -- (10, 10);
    \draw[dashdotted] (5, 1.667) -- (7.5, 5);

    \node at (3, 7) {$\TV \to 0$};
    \node at (8.125, 3) {$\TV \to 1$};
    \draw[dotted] (5, 0) -- (5, 1.667);
    \draw[dotted] (6.25, 0) -- (6.25, 2.5);
    \node at (8.125, 0.35) {3-cycles};
    \node at (5.625, 0.35) {2-paths};
    \node at (2.75, 0.35) {edges};
    \node at (5, 1.667) {\textbullet};
    \node at (6.25, 2.5) {\textbullet};
    \node at (10, 10) {\textbullet};
    \node at (5, 3.333) {\textbullet};
    \node at (7.5, 5) {\textbullet};

    \draw[->] (\xmin,\ymin) -- (\xmax,\ymin) node[right] {$p$};
    \draw[->] (\xmin,\ymin) -- (\xmin,\ymax) node[above] {$d$};
    \end{tikzpicture}
    \caption{
        Phase diagram of the TV convergence and divergence regimes for $G=\cG(n,p)$ in $(p, d)$ space given by Theorem~\ref{thm:er-mask}.
        We ignore $\polylog(n)$ factors along both axes.
        The red region is the TV convergence regime, and the green region is the TV divergence regime.
        The solid boundary indicates the sharp phase transition between these regimes.
        The dashed and dash-dotted lines in the TV convergence regime indicate, respectively, the thresholds for TV convergence obtained from the weaker TV upper bounds, Theorems~\ref{thm:general-naive-ub} and \ref{thm:general-weaker-ub}, that we will derive below.
        The dash-dotted line, where it is not drawn, coincides with the solid line (up to ignored polylogarithmic factors).
        In the TV divergence regime, the combinatorial structure witnessing the TV lower bound is written above the $p$-axis.
    }
    \label{fig:er-mask-phase-diagram}
\end{figure}

A similar phenomenon occurs for random bipartite masks.
Let $\cG(n,m,p)$ be the random graph on vertices $[n+m] = \{1,\ldots,n+m\}$ where each edge between $\{1,\ldots,n\}$ and $\{n+1,\ldots,n+m\}$ is present independently with probability $p\in [0,1]$.
Without loss of generality, we let $m\le n$.
As above, $m$ and $p$ are implicitly functions of $n$.
The following result identifies the asymptotic conditions under which $\TV(W(G,d), M(G)) \to 0$ or $1$ for $G\sim \cG(n,m,p)$.

\begin{theorem}
    \label{thm:bip-er-mask}
    Let $G = G_n \sim \cG(n,m,p)$.
    Over the randomness of the sample path $G_1, G_2, \ldots$, the following limits occur almost surely.
    \begin{enumerate}[label=(\alph*), ref=\alph*]
        \item \label{itm:bip-er-mask-ub} If
        \begin{equation}
            \label{eq:bip-er-mask-ub-hypothesis}
            d \gg nmp^2 + nm^{1/2}p + (nmp)^{1/2} + (nmp)^{1/3} \log n + (nmp)^{1/4} \log^{5/4} n + \log^{3/2} n,
        \end{equation}
        then $\TV(W(G,d), M(G)) \to 0$.
        \item \label{itm:bip-er-mask-lb} Conversely, if $p \gg (nm)^{-1} \log^3 n$ and
        \begin{equation}
             \label{eq:bip-er-mask-lb-hypothesis}
            d \ll nmp^2 + nm^{1/2} p + (nmp)^{1/2},
        \end{equation}
        then $\TV(W(G,d), M(G)) \to 1$.
    \end{enumerate}
\end{theorem}

This result holds for any dependence of $m,p$ on $n$, including when $m$ is much smaller than $n$.
When $p=1$, $G$ is the complete bipartite graph $K_{n,m}$, and we get the following family of sharp phase transitions.
\begin{corollary}
    The total variation between $W(K_{n,m}, d)$ and $M(K_{n,m})$ satisfies
    \[
        \TV(W(K_{n,m}, d), M(K_{n,m})) \to
        \begin{cases}
            0 & d \gg nm, \\
            1 & d \ll nm.
        \end{cases}
    \]
\end{corollary}
When we further set $m=n$, we have $G=K_{n,n}$.
In this case,
\[
    \TV(W(K_{n,n}, d), M(K_{n,n})) \to
    \begin{cases}
        0 & d \gg n^2, \\
        1 & d \ll n^2.
    \end{cases}
\]
This threshold, for the case $G = K_{n,n}$, was independently identified by Bubeck \cite{Bub20}.

Recall that $W(K_{n,n}, d)$ and $M(K_{n,n})$ are the laws of samples from the ordinary Wishart and GOE matrices $W(2n, d)$ and $M(2n)$, restricted to the top-right and bottom-left $n\times n$ blocks.
We see a remarkable contrast between this $d\asymp n^2$ threshold and the $d\asymp n^3$ threshold of \cite{BDER16} and \cite{JL15}: masking half the entries in the Wishart matrix caused the latent dimension at which the Wishart entries become approximately independent Gaussians to decrease by a polynomial factor, from $d\asymp n^3$ to $d\asymp n^2$.
Of course, this decrease is due to the specific choice of the deleted half of entries.
If we delete a random half of entries by setting $p=\f12$ in Theorem~\ref{thm:general-ub}, the threshold remains at $d\asymp n^3$.
Thus, the onset of approximate independence in the masked Wishart entries depends not only on the sparsity of the mask, but also crucially on the combinatorial structure.

In fact, for all $p\gtrsim (nm)^{-1}\log^6 n$, (\ref{eq:bip-er-mask-ub-hypothesis}) is equivalent to $d \gg nmp^2 + nm^{1/2}p + (nmp)^{1/2}$.
Thus, Theorem~\ref{thm:bip-er-mask} implies that for all $p\gtrsim (nm)^{-1}\log^6 n$, the sharp phase transition separating the regimes where $W(G,d)$ and $M(G)$ converge and diverge in total variation is $d \asymp nmp^2 + nm^{1/2}p + (nmp)^{1/2}$.
In the case $m=n$, this phase transition is at $d \asymp n^2p^2 + n^{3/2}p + np^{1/2}$.
Figure~\ref{fig:bip-er-mask-phase-diagram} summarizes the $(p,d)$ for which Theorem~\ref{thm:er-mask} gives that $W(G,d)$ and $M(G)$ converge and diverge in total variation.
Like above, we see a tradeoff between sparsity and combinatorial structure, and the combinatorial structure determining the threshold becomes progressively disordered as $G$ becomes sparser: 4-cycles dominate at $p\gtrsim n^{-1/2}$, followed by 2-paths at $n^{-1/2}\gtrsim p \gtrsim n^{-1}$, followed by edges at $p\lesssim n^{-1}$.
At the phase transitions $p\asymp n^{-1/2}$ and $p\asymp n^{-1}$, one combinatorial structure replaces another as the structure determining the threshold.

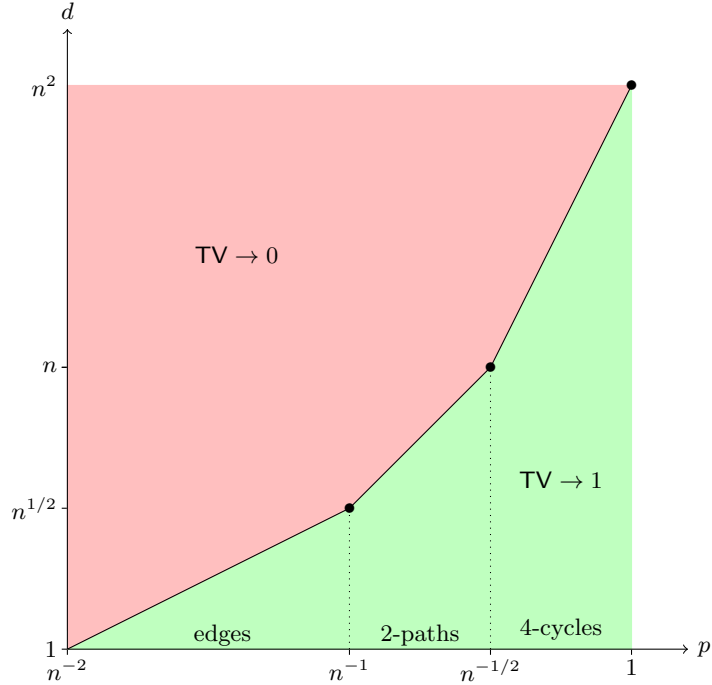
\begin{figure}[h!]
    \centering
    \begin{tikzpicture}[scale=0.75]
    \tikzstyle{every node}=[font=\footnotesize]
    \def\xmin{0}
    \def\xmax{11}
    \def\ymin{0}
    \def\ymax{11}

    \node at (0, 0) [below] {$n^{-2}$};
    \draw (0, -0.1) -- (0, 0);
    \node at (5, 0) [below] {$n^{-1}$};
    \draw (5, -0.1) -- (5, 0);
    \node at (7.5, 0) [below] {$n^{-1/2}$};
    \draw (7.5, -0.1) -- (7.5, 0);
    \node at (10, 0) [below] {$1$};
    \draw (10, -0.1) -- (10, 0);
    \node at (0, 0) [left] {$1$};
    \draw (-0.1, 0) -- (0, 0);
    \node at (0, 2.5) [left] {$n^{1/2}$};
    \draw (-0.1, 2.5) -- (0, 2.5);
    \node at (0, 5) [left] {$n$};
    \draw (-0.1, 5) -- (0, 5);
    \node at (0, 10) [left] {$n^2$};
    \draw (-0.1, 5) -- (0, 5);

    \filldraw[fill=green!25, draw=green!25] (0,0) -- (5, 2.5) -- (7.5, 5) -- (10,10) -- (10,0) -- (0,0);
    \filldraw[fill=red!25, draw=red!25] (0, 0) -- (5, 2.5) -- (7.5, 5) -- (10,10) -- (0,10) -- (0, 0.33);
    \draw (0, 0) -- (5, 2.5) -- (7. 5,5) -- (10, 10);
    % \draw[dashed] (0,0) -- (1, 0.5);

    \node at (3, 7) {$\TV \to 0$};
    \node at (8.75, 3) {$\TV \to 1$};
    \draw[dotted] (5, 0) -- (5, 2.5);
    \draw[dotted] (7.5, 0) -- (7.5, 5);
    \node at (8.75, 0.35) {4-cycles};
    \node at (6.25, 0.25) {2-paths};
    \node at (2.75, 0.25) {edges};
    \node at (5, 2.5) {\textbullet};
    \node at (7.5, 5) {\textbullet};
    \node at (10, 10) {\textbullet};

    \draw[->] (\xmin,\ymin) -- (\xmax,\ymin) node[right] {$p$};
    \draw[->] (\xmin,\ymin) -- (\xmin,\ymax) node[above] {$d$};s
    \end{tikzpicture}
    \caption{
        Phase diagram of the TV convergence and divergence regimes for $G=\cG(n,n,p)$ in $(p, d)$ space given by Theorem~\ref{thm:bip-er-mask}.
        We ignore $\polylog(n)$ factors along both axes.
        The red region is the TV convergence regime, and the green region is the TV divergence regime.
        The solid boundary indicates the sharp phase transition between these regimes.
        In the TV divergence regime, the combinatorial structure witnessing the TV lower bound is written above the $p$-axis.
    }
    \label{fig:bip-er-mask-phase-diagram}
\end{figure}

We will prove Theorems~\ref{thm:er-mask} and \ref{thm:bip-er-mask} in Appendix~\ref{appsec:er-masks-proofs}.
The proof is by a routine, though tedious, application of concentration inequalities.
Each subgraph count appearing in the theorems in Sections~\ref{sec:tv-upper-bounds} and \ref{sec:tv-lower-bounds} is a low-degree polynomial in the edge indicators of $G$, which are i.i.d. Bernoulli variables.
By theorems of Kim-Vu \cite{Vu02} and Janson \cite{JR02}, we will show these counts concentrate within a constant factor of their expectations with high probability.
We can make the failure probabilities summable, so that by the Borel-Cantelli lemma this concentration holds for all sufficiently large $n$ almost surely. 
Then, applying the appropriate theorems in Sections~\ref{sec:tv-upper-bounds} and \ref{sec:tv-lower-bounds} yields the desired results.

\section{Technical Overview}
\label{sec:technical-overview}

The main results of this paper are the TV upper bounds, Theorems~\ref{thm:general-ub} and \ref{thm:bipartite-ub}, and the main conceptual contributions of this paper are in the proofs of these theorems.
In this section, we give an overview of the ideas in these proofs.

\subsection{Information Distances}

We briefly review the properties of the $f$-divergences $\TV(\cdot,\cdot)$, $\KL(\cdot \parallel \cdot)$, and $\chisq(\cdot, \cdot)$ that we will use.
The choice of divergence at various points in our argument will be motivated by these properties. 
Given two measures $\mu, \nu$ on a measurable space $(\cX, \cB)$ where $\mu$ is absolutely continuous with respect to $\nu$, these $f$-divergences are defined by
\begin{align*}
    \TV(\mu,\nu) 
    &= 
    \f12 \E_{\xi \sim \nu} 
    \lt|
        \rn{\mu}{\nu}(\xi) - 1
    \rt|,
    \qquad
    \KL(\mu \parallel \nu) 
    = 
    \E_{\xi \sim \nu} \lt[
        \rn{\mu}{\nu}(\xi)
        \log
        \rn{\mu}{\nu}(\xi)
    \rt], \\
    \chisq(\mu,\nu)
    &= 
    \E_{\xi \sim \nu}
    \lt[
        \lt(\rn{\mu}{\nu}(\xi) - 1\rt)^2
    \rt].
\end{align*}
Here, $\rn{\mu}{\nu}: \cX \to [0, \infty)$ denotes the Radon-Nikodym derivative of $\mu$ with respect to $\nu$. 
These information distances all satisfy data-processing inequalities: if $K$ is a Markov transition from $(\cX, \cB)$ to another measurable space $(\cX', \cB')$, then $\TV(K\mu, K\nu) \le \TV(\mu, \nu)$, and analogously for $\KL$ and $\chisq$.
Moreover, they satisfy the inequalities
\begin{equation}
    \label{eq:technical-overview-information-distance-ineq-chain}
    2 \TV(\mu,\nu)^2 
    \le 
    \KL(\mu \parallel \nu) 
    \le 
    \chisq(\mu, \nu).
\end{equation}
Here, the left inequality is Pinsker's inequality and the right is standard.
By Cauchy-Schwarz, we also have $4\TV(\mu, \nu)^2 \le \chisq(\mu, \nu)$. 
Because our main objective is to upper bound a total variation distance, these inequalities allow us to pass from $\TV$ to $\KL$ and from $\KL$ to $\chisq$.
For more on the relationships between different information distances, see the survey \cite{GS02}.

\paragraph{Convexity of $f$-divergences.}
All $f$-divergences are convex with respect to mixtures.
Formally, let $\rho$ be a distribution on a set $\Phi$.
If $\{P_\phi : \phi \in \Phi\}$ is a collection of measures on $(\cX, \cB)$ indexed by $\phi\in \Phi$, let $\E_{\phi \sim \rho} P_\phi$ denote the measure that samples $\phi\sim \rho$ and then samples from $P_\phi$.
If $\{\mu_\phi: \phi \in \Phi\}$ and $\{\nu_\phi: \phi \in \Phi\}$ are two collections of such measures, and for every $\phi\in \Phi$, $\mu_\phi$ is absolutely continuous with respect to $\nu_\phi$, then for any $f$-divergence $D_f$,
\begin{equation}
    \label{eq:technical-overview-f-divergence-convexity}
    D_f \lt(
        \E_{\phi \sim \rho} \mu_\phi,
        \E_{\phi \sim \rho} \nu_\phi
    \rt)
    \le 
    \E_{\phi \sim \rho}
    D_f(\mu_\phi, \nu_\phi).
\end{equation}
We can interpret the right-hand side of this inequality as the expected $f$-divergence upon revealing the component $\phi$ of the mixture in which the data lies.
So, this inequality states that $f$-divergences only increase upon revealing latent information in a mixture.
This convexity bound will play a central role in bounding $f$-divergences of mixtures, as we will see below.

\paragraph{Mixtures and the Second Moment Method.}
An important application of information inequalities in this work is when $\mu$ is a mixture distribution.
Suppose that $\mu = \E_{\phi \sim \rho} \mu_\phi$ where $\rho$ is a distribution on a set $\Phi$ and $\{\mu_\phi : \phi \in \Phi\}$ is a collection of measures on $(\cX, \cB)$ that are each absolutely continuous with respect to $\nu$.
By (\ref{eq:technical-overview-f-divergence-convexity}), where we take all the $\nu_\phi$ to be $\nu$, we obtain the bound $\TV(\mu,\nu) \le \E_{\phi \sim \rho} \TV(\mu_\phi, \nu)$, and similarly for $\KL$ and $\chisq$. 
We emphasize that such convexity bounds are crude and usually not asymptotically optimal.
In the case of $\chisq$ divergence, we can improve on this bound by exploiting the following property, which makes $\chisq$ divergence particularly amenable to analyzing mixtures.
By Fubini's Theorem, we have that
\begin{align}
    \notag
    1 + \chisq(\mu, \nu)
    &= 
    \E_{\xi \sim \nu}
    \lt[\lt(\rn{\mu}{\nu}(\xi) \rt)^2\rt]
    = 
    \E_{\xi \sim \nu}
    \lt[\lt(
        \E_{\phi \sim \rho }
        \rn{\mu_\phi}{\nu}(\xi)
    \rt)^2\rt] \\
    \label{eq:technical-overview-2mm}
    &= 
    \E_{\phi^{(1)}, \phi^{(2)} \sim \rho \otimes \rho}
    \E_{\xi \sim \nu}
    \rn{\mu_{\phi^{(1)}}}{\nu}(\xi)
    \rn{\mu_{\phi^{(2)}}}{\nu}(\xi).
\end{align}
This expansion is the main idea of the so-called \textit{second moment method}, which we will use throughout our arguments.
When $\nu$ is a simple distribution and each of the $\mu_{\phi}$ is a simple distribution, the inner expectation over $\xi$ can often be evaluated explicitly, leaving an expectation over two independent replicas $\phi^{(1)}, \phi^{(2)}$ of the latent randomness.

Note that the crude convexity bound $\chisq(\mu,\nu) \le \E_{\phi \sim \rho} \chisq(\mu_\phi, \nu)$ yields the final expectation in (\ref{eq:technical-overview-2mm}), except with $\phi^{(1)}$ and  $\phi^{(2)}$ equal instead of independent.
The independence of $\phi^{(1)}$ and $\phi^{(2)}$ in (\ref{eq:technical-overview-2mm}) is crucial for getting the stronger bounds of the second moment method.

\paragraph{Tail events and Conditioning.}
An intuition to keep in mind is that $\TV$ is not sensitive to tail behavior, $\KL$ is slightly sensitive to tail behavior, and $\chisq$ is very sensitive to tail behavior.
Total variation satisfies the triangle inequality and the following conditioning property.
If $S\in \cB$ is an event and $\cL(\mu | S)$ denotes the law of a sample from $\mu$ conditioned to lie in $S$, then $\TV\lt(\mu, \cL(\mu | S)\rt) = \P(S^c)$.
Combined with the triangle inequality, this gives
\begin{equation}
    \label{eq:technical-overview-tv-conditioning}
    \TV(\mu, \nu) \le \P(S^c) + \TV\lt(\cL(\mu | S), \nu\rt).
\end{equation}
By data processing, if $\mu$ is a mixture distribution over latent randomness, this remains true if the conditioning is in the latent space. 
Due to the extra logarithmic factor, a similar inequality does not always hold for KL divergence, though for well behaved $\mu$ and $\nu$, we can often derive such an inequality by ad hoc techniques. 
Due to the square, $\chisq$ divergence behaves poorly when $\mu$ and $\nu$ have mismatched tails, and a similar inequality does not hold for $\chisq$ divergence. 

\paragraph{KL Tensorization.}
Finally, $\KL(\mu \parallel \nu)$ tensorizes when $\nu$ is a product measure.
Suppose $\cX$ is a product set $\cX = \cS_1 \times \cS_2 \times \cdots \times \cS_k$ and $\nu$ is a product measure $\nu = \otimes_{i=1}^k \nu_i$, where for $1\le i\le k$, $\nu_i$ is a measure on $\cS_i$. 
Let $\mu$ be the law of $(\xi_1,\ldots,\xi_k)$, where $\xi_i \in \cS_i$ for $1\le i\le k$. 
Let $\mu_{<i}$ denote the law of $\xi_{<i} = (\xi_1,\ldots,\xi_{i-1}) \in \cS_1 \times \cdots \times \cS_{i-1}$, and let $(\mu_i | \xi_{<i})$ denote the conditional law of $\xi_i$ given $\xi_{<i}$.
Then, we have that
\begin{equation}
    \label{eq:technical-overview-kl-tensorization}
    \KL(\mu \parallel \nu) 
    = 
    \sum_{i=1}^k
    \E_{\xi_{<i} \sim \mu_{<i}}
    \KL\lt((\mu_i | \xi_{<i}) \parallel \nu_i\rt).
\end{equation}

\subsection{Proof Outline of TV Upper Bound for General Masks}
\label{subsec:technical-overview-general-ub-outline}

In this and the next two subsections, we outline the proof of Theorem~\ref{thm:general-ub}. The three steps outlined in this subsection convert the task of bounding total variation to one of bounding a certain coupled exponential overlap. The subsequent two sections describe how we bound this quantity using carefully constructed couplings.  

Let $G$ be a graph on $[n]$, and let $\mu = W(G,d)$ and $\nu = M(G)$. 
We first pass from $\TV$ to $\KL$ and note that $\KL(\mu \parallel \nu)$ admits a natural tensorization of the form just above.
We can consider the sample space of $\mu$ and $\nu$ to be $\cS_1\times \cS_2 \times \cdots \times \cS_n$, where for each $v\in [n]$, $\cS_v$ consists of the entries in the $v$th column and first $v-1$ rows. (By symmetry, this contains the information in the $v$th row and first $v-1$ columns.)

\paragraph{Step 1: KL Tensorization.}
Let $X_1,\ldots,X_n \sim_{\mathrm{i.i.d.}} \cN(0, I_d)$ be the columns of the matrix $X$ generating $W(G,d)$ in Definition~\ref{defn:wishart-goe}.
For each $v\in [n]$, let $\Xvl = (X_1,\ldots,X_{v-1})$ denote the submatrix of $X$ consisting of the first $v-1$ columns.
Let
\begin{equation}
    \label{eq:general-ub-def-wv}
    W_v(\Xvl) = 
    A_{G[v-1]} \odot 
    d^{-1/2} \lt(
        \Xvlt \Xvl - I_{v-1}
    \rt),
\end{equation}
where $A_{G[v-1]}$ is the adjacency matrix of the induced subgraph of $G$ on $[v-1]$. 
Let $\muvl$ denote the measure of $W_v(\Xvl)$. 
Equivalently, $\muvl$ is the marginal measure of the upper left $(v-1)\times (v-1)$ submatrix of $W(G,d)$.
Let $\muv$ and $\nuv$ denote the marginal measures of the first $v-1$ entries of the $v$th columns of $W(G,d)$ and $M(G)$, respectively.
Recall that $\downNv = N(v)\cap [v-1]$ is the set of neighbors of $v$ in $[v-1]$. 
Let $\Xbn$ be the submatrix of $\Xvl$ whose columns are indexed by $\downNv$. 
Note that (ignoring entries deterministically set to $0$) $\muv$ is the measure of $d^{-1/2} \Xbnt X_v$.
By Pinsker's inequality (\ref{eq:technical-overview-information-distance-ineq-chain}) and the KL tensorization (\ref{eq:technical-overview-kl-tensorization}), we have that
\begin{equation}
    \label{eq:general-ub-starting-point}
    2\TV(W(G,d), M(G))^2
    \le
    \KL(\mu \parallel \nu)
    =
    \sum_{v=1}^n
    \E_{W\sim \muvl}
    \KL\lt((\muv | W) \parallel \nuv\rt),
\end{equation}
where $(\muv | W)$ denotes the measure of $d^{-1/2} \Xbnt X_v$, where $X_v\sim \cN(0, I_d)$ and independently $\Xvl \sim \cN(0, I_d)^{\otimes (v-1)}$ is conditioned on $W_v(\Xvl) = W$.
We will separately bound each summand in this sum.
In the following discussion, we outline how to bound the $v$th summand for arbitrary $v\in [n]$.
To reduce notational clutter, we write $W(\Xvl)$ for $W_v(\Xvl)$ when $v$ is clear.

\paragraph{Step 2: Controlling Tails and Second Moment Method.}
We will view $(\muv | W)$ as a mixture distribution, with the goal of applying the second moment method.
Note that $d^{-1/2} \Xbnt X_v$ conditioned on $\Xvl$ is Gaussian with covariance $d^{-1} \Xbnt \Xbn$. 
Thus, $(\muv | W)$ is a mixture of Gaussians parametrized by latent randomness $\Xvl$, where $\Xvl$ is sampled from $\cN(0, I_d)^{\otimes (v-1)}$ conditioned on the event $W(\Xvl) = W$.
This is the correct setup for the second moment method, and we would like to bound the $v$th summand of (\ref{eq:general-ub-starting-point}) by passing from $\KL$ to $\chisq$. 

However, due to a mismatch of tail behaviors between $(\muv | W)$ and $\nuv$, the resulting $\chisq$ divergence is infinite.
So, before we pass to $\chisq$, we condition the $v$th summand in (\ref{eq:general-ub-starting-point}) on a high probability event $S^v \in \sigma(\Xvl)$.
To reduce notational clutter, we will denote this event by $S$ when $v$ is clear.
We perform this conditioning in the latent space, leaving $X_v$ unchanged, so that the resulting distribution of $d^{-1/2} \Xbnt X_v$ is still a mixture of Gaussians parametrized by latent randomness $\Xvl$. 
In the conditioned distribution, $\Xvl \sim \cN(0, I_d)^{\otimes (v-1)}$ is now conditioned on both $\Xvl \in S$ and $W(\Xvl) = W$.
This conditioning changes the $v$th summand of (\ref{eq:general-ub-starting-point}) in two places: both $\muvl$ and $\muv$ become conditioned on $\Xvl \in S$.
We will show that, for $S$ with sufficiently high probability, this conditioning adds only a small error term to (\ref{eq:general-ub-starting-point}).
Because the conditioning happens in the latent space and $(\muv | W)$ does not have an explicit representation, showing this fact requires a technically subtle argument.
Then, we pass from $\KL$ to $\chisq$ and apply the second moment method to the conditioned distribution.

We remark that we must condition after passing from $\TV$ to $\KL$, not before. 
This is because the various events $S^v \in \sigma(\Xvl)$ for $v\in [n]$ overlap, so if we condition before passing to $\KL$, the resulting distributions after KL tensorization become intractable.
Thus we must bound the effect of conditioning on $\KL$, which is significantly more challenging than bounding its effect on $\TV$.

\paragraph{Step 3: Coupled Exponential Overlap.}
After the above steps, the inner expectation of the second moment method (\ref{eq:technical-overview-2mm}) evaluates explicitly.
The remaining expectation can be massaged into the coupled exponentiated overlap $\E \exp (\f12 Y_v)$, where
\begin{equation}
    \label{eq:general-ub-define-yv}
    Y_v 
    =
    \f{1}{d^2} 
    \sum_{i,j\in \downNv} 
    \lt(\la \Xonei, \Xonej \ra - d\delta_{i,j}\rt)
    \lt(\la \Xtwoi, \Xtwoj \ra - d\delta_{i,j}\rt).
\end{equation}
Here, recall that $\delta_{i,j} = \ind{i=j}$ is the Kronecker delta function.
The expectation in $\E \exp (\f12 Y_v)$ is over $W$ and two replicas $\Xonevl, \Xtwovl$ of $\Xvl$, from the following distribution.
Let $\cL\lt(\cN(0, I_d)^{\otimes (v-1)} | S \rt)$ denote the law of a sample from $\Xvl \sim \cN(0, I_d)^{\otimes (v-1)}$ conditioned on $\Xvl \in S$. 
Then, $W$ is the distribution of $W(\Xvl)$, where $\Xvl \sim \cL\lt(\cN(0, I_d)^{\otimes (v-1)} | S \rt)$.
For $r=1,2$, the replicas $\Xrvl$ are independently sampled from $\cL\lt(\cN(0, I_d)^{\otimes (v-1)} | S \rt)$ conditioned on $W(\Xrvl) = W$.
The resulting distribution over $(\Xonevl, \Xtwovl)$ can be defined equivalently as follows. 
\begin{definition}[Inner Product Coupling]
    \label{defn:general-ub-ip-coupling}
    Let $\cNvdips$ denote the distribution over pairs of matrices $(\Xonevl,\Xtwovl) \in \lt(\bR^{d\times (v-1)}\rt)^{2}$ generated as follows.
    \begin{enumerate}[label=(\arabic*), ref=\arabic*]
        \item Sample $\Xonevl$ from $\cL\lt(\cN(0, I_d)^{\otimes (v-1)} | S \rt)$.
        \item Independently sample $\Xtwovl$ from $\cL\lt(\cN(0, I_d)^{\otimes (v-1)} | S \rt)$ conditioned on
        \begin{equation}
            \label{eq:general-ub-ip-coupling-condition}
            \lt\{\text{
                $\la \Xonei, \Xonej \ra = \la \Xtwoi, \Xtwoj \ra$ for all $(i,j) \in E(G[v-1])$
            }\rt\}.
        \end{equation}
    \end{enumerate}
\end{definition}
The condition (\ref{eq:general-ub-ip-coupling-condition}) is equivalent to $W(\Xonevl) = W(\Xtwovl)$.
It is not difficult to see that $\Xonevl$ and $\Xtwovl$ are both marginally distributed as $\cL\lt(\cN(0, I_d)^{\otimes (v-1)} | S\rt)$, and that $(\Xonevl, \Xtwovl)$ and $(\Xtwovl, \Xonevl)$ are equidistributed.
This distribution will figure prominently in our arguments.
The remaining task, which is the core of our technique, is to bound $\E \exp\lt(\f12 Y_v\rt)$ over $(\Xonevl, \Xtwovl) \sim \cNvdips$.

We remark that $\E \exp(\f12 Y_v)$ arises in the approach of Section~\ref{sec:general-ub-main-argument}.
In the refined approach of Section~\ref{sec:general-ub-refined-higher-order-terms}, the coefficient $\f12$ becomes $1$ and we also need to bound a variant of the exponentiated overlap based on 4-cycles, but the technique to bound these expectations is similar.

\subsection{Handling Latent Information by Gram-Schmidt Orthogonalization}
\label{subsec:technical-overview-gs}

The conceptual challenge in bounding $\E \exp\lt(\f12 Y_v\rt)$ over $(\Xonevl, \Xtwovl) \sim \cNvdips$ lies in the following tradeoff between tractability and optimality.
The coupled distribution $\cNvdips$ contains many complex dependencies, which make this expectation intractable to evaluate.
By convexity arguments, replacing $\cNvdips$ with a stronger coupling (defined formally in the next paragraph) yields an upper bound on the original expectation.
By conditioning on more information, the stronger coupling also destroys some of the complexity in $\cNvdips$, yielding a more tractable expectation.
However, if the coupling is too strong, the resulting bound is too weak to produce Theorem~\ref{thm:general-ub}.

We will first see that a global convexity argument fails to navigate this tradeoff, which motivates our fine-grained convexity approach.
We can upper bound the $v$th summand of (\ref{eq:general-ub-starting-point}), after conditioning on $\Xvl \in S$, as follows.
Let $f$ be a function such that $W(\Xvl)$ is $f(\Xvl)$-measurable.
For a measure $\cD$ on $\bR^{d\times (v-1)}$, let $f_{\#} \cD$ denote the pushforward measure of $\cD$ under $f$.
By convexity of $\KL$, we have the upper bound
\begin{align*}
    &\E_{W\sim W_{\#} \cL\lt(\cN(0, I_d)^{\otimes (v-1)} | S \rt)} 
    \KL((\muv | W(\Xvl) = W) \parallel \nuv) \\
    &\qquad \le 
    \E_{a\sim f_{\#} \cL\lt(\cN(0, I_d)^{\otimes (v-1)} | S \rt)}
    \KL((\muv | f(\Xvl) = a) \parallel \nuv).
\end{align*}
We then bound the latter $\KL$ divergence by $\chisq$ and apply the second moment method as before.
We can think of this approach as revealing more information: instead of revealing $W(\Xvl)$, we reveal $f(\Xvl)$.
In the most extreme case, we can pick $f(\Xvl) = \Xvl$, thereby revealing $\Xvl$ instead of $W(\Xvl)$. 
By massaging the resulting expectation, we get a bound of $\E \exp\lt(\f12 Y_v\rt)$, where the expectation is now over a stronger coupling, given by Definition~\ref{defn:general-ub-ip-coupling} with (\ref{eq:general-ub-ip-coupling-condition}) replaced by the stronger condition $f(\Xonevl) = f(\Xtwovl)$.

Unfortunately, this approach is too weak to produce Theorem~\ref{thm:general-ub}. 
The problem is that the expectation $\E \exp\lt(\f12 Y_v\rt)$ over a stronger coupling remains intractable unless the stronger coupling is essentially $\Xonevl = \Xtwovl$.
This choice cannot produce Theorem~\ref{thm:general-ub}, for reasons discussed in Section~\ref{subsec:technical-overview-naive-comparison}.
Thus, a global convexity argument cannot prove the desired result.

Our key conceptual innovation is to use a \textit{term-by-term} convexity argument to navigate this tractability-optimality tradeoff.
We Taylor expand $\exp(\f12 Y_v)$ up to some finite order (1 in the approach of Section~\ref{sec:general-ub-main-argument}, and 3 in the approach of Section~\ref{sec:general-ub-refined-higher-order-terms}) and bound each low order term of this expansion with a tailored stronger coupling that is different for each term.
This technique is made possible by Lemma~\ref{lem:general-ub-stronger-couplings} below, which states that for each term in this expansion, replacing $\cNvdips$ with a stronger coupling can only increase the term's expectation.
The key technical task is to design a stronger coupling for each Taylor term that reveals enough latent information to make the expectation tractable, while maintaining near optimality by keeping hidden the information on which \textit{that specific term} most strongly depends.
Finally, we will control the Taylor error term deterministically by the event $S$.

\begin{lemma}[Upper Bounds through Stronger Couplings]
    \label{lem:general-ub-stronger-couplings}
    Let $\cD$ be a distribution over an arbitrary probability measure $(\cX, \cB)$.
    Given a measurable function $f$ on $\cX$, let $\cD(f)$ denote the coupling of $X,Y\sim \cD$ generated by first sampling $X\sim \cD$ and then independently sampling $Y\sim \cD$ conditioned on the event $\{f(X) = f(Y)\}$.
    Suppose that $f$ and $g$ are measurable functions such that $g(X) = g(Y)$ implies $f(X) = f(Y)$ almost surely.
    Then, for any measurable function $h: \cX \to \bR$, it holds that
    \[
        \E_{X,Y\sim \cD(f)}
        [h(X)h(Y)]
        \le
        \E_{X,Y\sim \cD(g)}
        [h(X)h(Y)].
    \]
\end{lemma}
\begin{proof}
    Note that $\cD(f)$ can be expressed as the mixture $\E_{a\sim f_{\#}\cD} \lt[\cL(\cD | f = a)^{\otimes 2}\rt]$.
    Furthermore, the given condition implies that $f(X) \to g(X) \to X$ is a Markov chain.
    Therefore we have that
    \begin{align*}
        \E_{X,Y\sim \cD(f)}
        [h(X)h(Y)]
        &=
        \E_{a\sim f_{\#}\cD}
        \lt[
            \E_{(X,Y) \sim \cL(\cD | f=a)^{\otimes 2}}
            h(X)h(Y)
        \rt] \\
        &=
        \E_{a\sim f_{\#}\cD}
        \lt(
            \E_{X \sim \cL(\cD | f=a)}
            h(X)
        \rt)^2 \\
        &=
        \E_{a\sim f_{\#}\cD}
        \lt(
            \E_{b\sim \cL(g | f=a)}
            \E_{X \sim \cL(\cD | g=b)}
            h(X)
        \rt)^2 \\
        &\le
        \E_{a\sim f_{\#}\cD}
        \E_{b\sim \cL(g | f=a)}
        \lt(
            \E_{X \sim \cL(\cD | g=b)}
            h(X)
        \rt)^2 \\
        &=
        \E_{b\sim g_{\#}\cD}
        \lt(
            \E_{X \sim \cL(\cD | g=b)}
            h(X)
        \rt)^2
        =
        \E_{X,Y\sim \cD(g)}
        [h(X)h(Y)],
    \end{align*}
    where the inequality is by Jensen's inequality.
\end{proof}

We now outline our technique for Taylor expanding $\exp(\f12 Y_v)$ and designing term-by-term stronger couplings.
We will first outline the approach of Section~\ref{sec:general-ub-main-argument}, and then sketch the improvements made in Section~\ref{sec:general-ub-refined-higher-order-terms}.
We begin with the expansion
\begin{align}
    \notag
    &\E \exp\lt(\f12 Y_v\rt) \\
    \label{eq:general-ub-main-exp-overlap-starting-point}
    &\qquad =
    1 +
    \f{1}{2d^2}
    \sum_{i,j\in \downNv}
    \E\lt[
        \lt(\la \Xonei, \Xonej\ra - d\delta_{i,j}\rt)
        \lt(\la \Xtwoi, \Xtwoj\ra - d\delta_{i,j}\rt)
    \rt] 
    + \E h\lt(\f12 Y_v\rt),
\end{align}
where $h(x) = \exp(x) -1 - x$ and the expectations are over $(\Xonevl, \Xtwovl) \sim \cNvdips$.
As discussed above, we will bound the final term deterministically on $S$: we will choose $S$ such that for all $\Xonevl, \Xtwovl \in S$, $Y_v$ is in a neighborhood of $0$.
Because $h(x)\le x^2$ for $x$ in a neighborhood of $0$, this bounds $\E h\lt(\f12 Y_v\rt)$.
The remaining task is to bound the first order terms in (\ref{eq:general-ub-main-exp-overlap-starting-point}).

To handle these terms, we first remove the conditioning on $S$, so now $\Xonevl$ is sampled from $\cN(0, I_d)^{\otimes (v-1)}$ unconditionally and $\Xtwovl$ is sampled from $\cN(0, I_d)^{\otimes (v-1)}$ conditioned on (\ref{eq:general-ub-ip-coupling-condition}).
Keeping with the notation of Definition~\ref{defn:general-ub-ip-coupling}, we denote this distribution $\cNvdipr$.
We show that removing this conditioning adds only a small error to (\ref{eq:general-ub-main-exp-overlap-starting-point}).
We then divide the first order terms of (\ref{eq:general-ub-main-exp-overlap-starting-point}) into three categories: 
\begin{enumerate}[label=(\roman*), ref=\roman*]
    \item \label{itm:technical-overview-main-overlap-term-ijedge} $(i,j)\in E(G[v-1])$; 
    \item \label{itm:technical-overview-main-overlap-term-ijnotedge} $i\neq j$ and $(i,j)\not\in E(G[v-1])$; and
    \item \label{itm:technical-overview-main-overlap-term-ii} $i=j$.
\end{enumerate}
These categories will be handled by Propositions~\ref{prop:general-ub-linear-term-ijedge}, \ref{prop:general-ub-linear-term-ijnotedge}, and \ref{prop:general-ub-linear-term-ii}, respectively.
The terms of (\ref{eq:general-ub-main-exp-overlap-starting-point}) in category (\ref{itm:technical-overview-main-overlap-term-ijedge}) are trivial: for such terms, (\ref{eq:general-ub-ip-coupling-condition}) gives that $\la \Xonei, \Xonej\ra = \la \Xtwoi, \Xtwoj\ra$, and the desired expectation reduces to $\E \la \Xonei, \Xonej\ra^2=d$.

The terms of (\ref{eq:general-ub-main-exp-overlap-starting-point}) in category (\ref{itm:technical-overview-main-overlap-term-ijnotedge}) are bounded by replacing $\cNvdipr$ with a stronger coupling based on a Gram-Schmidt orthogonalization.
The motivation of this coupling is twofold. 
\begin{enumerate}[label=(\arabic*)]
    \item In the Gram-Schmidt orthogonalization $X_i = \sum_{j=1}^i W_{i,j}U_j$ of i.i.d. isotropic Gaussians $X_1,\ldots,X_k \sim \cN(0, I_d)$, the Gram-Schmidt coefficients $W_{i,j}$ have a simple description: $W_{i,i} \sim \sqrt{\chisq(d+1-i)}$ and $W_{i,j} \sim \cN(0,1)$ for all $j<i$, and these coefficients are mutually independent. 
    So, we can control the strength of the coupling between two replicas $\Xone[1],\ldots,\Xonek$ and $\Xtwo[1],\ldots,\Xtwok$ by selecting which $W_{i,j}$ become coupled and which remain free.
    \item The inner product $\la X_i, X_j\ra$ (where $j<i$) is mostly controlled by a single Gram-Schmidt coefficient, namely $W_{i,j}$.
    As long as the two replicas of this coefficient remain independent, the correlations between $\la \Xonei, \Xonej\ra$ and $\la \Xtwoi, \Xtwoj\ra$ will remain small.
\end{enumerate}
We demonstrate this technique on two simple examples.

\begin{example}
    \label{ex:technical-overview-gs-3}
    Suppose $v-1=3$ and $G[v-1]$ is the graph with edges $(1,2)$, $(1,3)$.
    Explicitly, we sample $(\Xone[1], \Xone[2], \Xone[3]) \sim \cN(0, I_d)^{\otimes 3}$ and then sample $(\Xtwo[1], \Xtwo[2], \Xtwo[3]) \sim \cN(0, I_d)^{\otimes 3}$ conditioned on $\la \Xone[1], \Xone[2]\ra = \la \Xtwo[1], \Xtwo[2]\ra$ and $\la \Xone[1], \Xone[3]\ra = \la \Xtwo[1], \Xtwo[3]\ra$.
    We will upper bound $\E \la \Xone[2], \Xone[3]\ra \la \Xtwo[2], \Xtwo[3]\ra$.
    By Gram-Schmidt orthogonalization, we can write
    \begin{align*}
        \Xone[1] &= \Wone[1,1] \Uone[1], \\
        \Xone[2] &= \Wone[2,1] \Uone[1] + \Wone[2,2] \Uone[2], \\
        \Xone[3] &= \Wone[3,1] \Uone[1] + \Wone[3,2] \Uone[2] + \Wone[3,3]\Uone[3],
    \end{align*}
    where $\Uone[1], \Uone[2], \Uone[3]$ are orthogonal unit vectors and the Gram-Schmidt coefficients $\Wone[i,j]$ are mutually independent with the distributions described above.
    A similar expansion holds for $\Xtwo[1], \Xtwo[2], \Xtwo[3]$.
    The above inner product conditions are equivalent to $\Wone[1,1]\Wone[2,1] = \Wtwo[1,1]\Wtwo[2,1]$ and $\Wone[1,1]\Wone[3,1] = \Wtwo[1,1]\Wtwo[3,1]$.
    By Lemma~\ref{lem:general-ub-stronger-couplings}, we can upper bound $\E \la \Xone[2], \Xone[3]\ra \la \Xtwo[2], \Xtwo[3]\ra$ by replacing this coupling with the coupling where $\Wone[i,j]=\Wtwo[i,j]$ for all $(i,j)\neq (3,2)$.
    Note crucially that $\Wone[3,2]$ and $\Wtwo[3,2]$ remain independent.
    Thus,
    \begin{align*}
        &\E_{(\Xonevl, \Xtwovl) \sim \cNvdipr}
        \la \Xone[2], \Xone[3]\ra 
        \la \Xtwo[2], \Xtwo[3]\ra \\
        &\qquad \le
        \E \lt[
            \lt(\Wone[2,1]\Wone[3,1] + \Wone[2,2]\Wone[3,2]\rt)
            \lt(\Wone[2,1]\Wone[3,1] + \Wone[2,2]\Wtwo[3,2]\rt)
        \rt]
        =
        \E \lt[
            \lt(\Wone[2,1]\Wone[3,1]\rt)^2
        \rt]
        = 1,
    \end{align*}
    which is far stronger than the na\"ive bound of $d$ obtained from the coupling $\Xonevl = \Xtwovl$.
    This argument captures the fact that although the inner product coupling induces correlations between $\la \Xone[2], \Xone[3]\ra$ and $\la \Xtwo[2], \Xtwo[3]\ra$, these correlations are confined to the directions of $\Xone[1]$ and $\Xtwo[1]$, and in the vast majority of directions these replicas remain free.
\end{example}

\begin{example}
    \label{ex:technical-overview-gs-m}
    If we replace $3$ with any $m\le d$, Example~\ref{ex:technical-overview-gs-3} generalizes naturally.
    Suppose $v-1=m$ and $G[v-1]$ does not contain $(m-1, m)$.
    By coupling all the Gram-Schmidt coefficients except $W_{m,m-1}$, we get 
    \[
        \E_{(\Xonevl, \Xtwovl) \sim \cNvdipr} 
        \la \Xone[m-1], \Xone[m]\ra 
        \la \Xtwo[m-1],\Xtwo[m]\ra 
        \le m-2.
    \]
\end{example}
Because we can permute the labels $[v-1]$ so that $i$ and $j$ become $m-1$ and $m$, Example~\ref{ex:technical-overview-gs-m} gives a rudimentary bound on the terms of (\ref{eq:general-ub-main-exp-overlap-starting-point}) in category (\ref{itm:technical-overview-main-overlap-term-ijnotedge}).
We can optimize this bound by running Gram-Schmidt on only $X_k$ for $k\in \Nl(i) \cup \{i,j\}$ (or symmetrically, $\Nl(j)\cup \{i,j\}$; recall that $\Nl(i) = N(i) \cap [v-1]$). 
This is carried out in Proposition~\ref{prop:general-ub-linear-term-ijnotedge} and bounds these terms by $\min(\degvl(i), \degvl(j))$.
The terms in category (\ref{itm:technical-overview-main-overlap-term-ii}) can be handled with a similar Gram-Schmidt coupling, where the one coefficient left free is $W_{m,m}$ instead of $W_{m, m-1}$.
This is carried out in Proposition~\ref{prop:general-ub-linear-term-ii} and bounds these terms by $2\degvl(i)$.

In Section~\ref{sec:general-ub-refined-higher-order-terms}, we improve on the error term estimate in (\ref{eq:general-ub-main-exp-overlap-starting-point}) by expanding $\exp(Y_v)$ (as previously mentioned, in the approach of Section~\ref{sec:general-ub-refined-higher-order-terms} the $\f12$ coefficient on $Y_v$ becomes $1$) to Taylor order 3 instead of 1.
This produces order-2 terms of the form 
\begin{equation}
    \label{eq:general-ub-higher-order-overlaps}
    \sum_{i_1,j_1,i_2,j_2\in \downNv}
    \E_{(\Xonevl, \Xtwovl) \sim \cNvdipr} 
    \lt[
        \Delone[i_1,j_1]
        \Deltwo[i_1,j_1]
        \Delone[i_2,j_2]
        \Deltwo[i_2,j_2]
    \rt]
\end{equation}
and analogous order-3 terms, where $\Delr[i,j] = d^{-1} \lt(\la \Xri, \Xrj\ra - d\delta_{i,j}\rt)$.
We will bound such terms in Lemma~\ref{lem:general-ub-higher-order-gs}.
To bound each summand above, we detach $\Delone[i_1,j_1]\Deltwo[i_1,j_1]$ from the product by Cauchy-Schwarz and bound its contribution by a similar Gram-Schmidt technique, while the rest of the product is bounded by a crude application of AM-GM.

\subsection{Beyond Gram-Schmidt}
\label{subsec:technical-overview-beyond-gs}

The bounds on the coupled exponentiated overlap obtained by Gram-Schmidt couplings are nontrivial, but still suboptimal. 
For the terms (\ref{eq:general-ub-main-exp-overlap-starting-point}) in category (\ref{itm:technical-overview-main-overlap-term-ii}), and for the higher order overlaps in Section~\ref{sec:general-ub-refined-higher-order-terms} such as (\ref{eq:general-ub-higher-order-overlaps}), the bounds from Gram-Schmidt are good enough for our purposes.
However, to prove Theorem~\ref{thm:general-ub} in full generality, we will need to improve our upper bound on the terms of (\ref{eq:general-ub-main-exp-overlap-starting-point}) in category (\ref{itm:technical-overview-main-overlap-term-ijnotedge}).
Namely, we seek an improved upper bound on
\begin{equation}
    \label{eq:technical-overview-exp-overlap-linear-term}
    \E_{(\Xonevl, \Xtwovl) \sim \cNvdipr} 
    \la \Xonei, \Xonej\ra 
    \la \Xtwoi, \Xtwoj\ra
\end{equation}
where $i\neq j$ and $(i,j)\not\in E(G[v-1])$. 
The best upper bound the Gram-Schmidt method can attain on (\ref{eq:technical-overview-exp-overlap-linear-term}) is $\min\lt(\degvl(i), \degvl(j)\rt)$; the true order of this expectation is approximately $\degvl(i,j) = |\Nl(i) \cap \Nl(j)|$, which is much smaller for most $G$. 

We first give some geometric intuition for why $\degvl(i,j)$ is the right scale for this expectation.
The correlations between $\la \Xonei, \Xonej\ra$ and $\la \Xtwoi, \Xtwoj\ra$ arise from paths from $i$ to $j$ in $G[v-1]$. 
The strength of the correlations decays rapidly in the length of the path: as we saw in Section~\ref{subsec:technical-overview-gs}, an edge from $i$ to $j$ contributes a correlation of $d$, while a 2-path from $i$ to $j$ contributes a correlation of scale $1$, and the contributions of longer paths are even smaller.
We also expect these correlations to be approximately additive over multiple paths of the same length.
So, when $(i,j)\not\in E(G[v-1])$, the desired expectation should be dominated by the number of 2-paths from $i$ to $j$ in $G[v-1]$, which is $\degvl(i,j)$.

Let us now see why the Gram-Schmidt approach is suboptimal.
By definition of the inner product coupling (\ref{eq:general-ub-ip-coupling-condition}), we can identify conditioning on an additional inner product equality $\la \Xonek, \Xonel\ra = \la \Xtwok, \Xtwol\ra$ with adding an edge $(k,\ell)$ to $G[v-1]$.
By Lemma~\ref{lem:general-ub-stronger-couplings}, adding additional edges can only increase (\ref{eq:technical-overview-exp-overlap-linear-term}).
The basic implementation of the Gram-Schmidt technique in Example~\ref{ex:technical-overview-gs-m} upper bounds (\ref{eq:technical-overview-exp-overlap-linear-term}) by $v-3$.
We can see why this bound is not tight: the conditioning scheme of Example~\ref{ex:technical-overview-gs-m} effectively adds all edges to $G[v-1]$ except $(i,j)$, which creates many new 2-paths from $i$ to $j$.
The optimized implementation of Gram-Schmidt in Proposition~\ref{prop:general-ub-linear-term-ijnotedge}, which runs Gram-Schmidt on only $X_k$ for $k\in \Nl(i) \cup \{i,j\}$ (or $\Nl(j) \cup \{i,j\}$), upper bounds (\ref{eq:technical-overview-exp-overlap-linear-term}) by $\min(\degvl(i), \degvl(j))$.
Still, this conditioning scheme effectively draws edges from $j$ to $\Nl(i) \setminus \Nl(j)$, adding many new 2-paths from $i$ to $j$.
Using methods based on Gram-Schmidt, this suboptimality appears to be unavoidable.

To optimally bound (\ref{eq:technical-overview-exp-overlap-linear-term}), we will use a weaker coupling that adds no new 2-paths from $i$ to $j$.
This is carried out in Lemma~\ref{lem:general-ub-refined-linear-term}.
Our new coupling will condition on $\Xonek = \Xtwok$ for all $k\in [v-1]\setminus \{i,j\}$, but not on any additional information about $X_i$ or $X_j$.
Effectively, this approach draws all edges among $[v-1]\setminus \{i,j\}$ but no additional edges incident to $i$ or $j$, thereby adding many paths from $i$ to $j$ of length $3$ or more but no 2-paths.
With this coupling, we can bound (\ref{eq:technical-overview-exp-overlap-linear-term}) as follows.
Conditioned on $X_k$ for $k\in [v-1]\setminus \{i,j\}$ and the inner products involving $X_i, X_j$ corresponding to edges of $G[v-1]$, $X_i$ and $X_j$ are singular Gaussians whose means can be explicitly computed.
This reduces (\ref{eq:technical-overview-exp-overlap-linear-term}) to an explicit expectation over i.i.d. Gaussians $X_1,\ldots,X_{v-1} \sim \cN(0, I_d)$, albeit a complex one involving Wishart inverses.
To bound this expectation, we will expand these inverses into moments using the power series for $(I+A)^{-1}$ and then carefully control these moments.

\subsection{Comparison with Na\"ive Convexity Bound}
\label{subsec:technical-overview-naive-comparison}

We emphasize that the main challenge of the proof of Theorem~\ref{thm:general-ub} is to handle the latent information in the conditioned distribution $(\muv | W)$.
In the above proof outline, this difficulty translated into the difficulty of handling expectations over the coupled distribution $\cNvdips$. 
To illustrate the importance of tightly handling latent information, we sketch here the proof of a suboptimal upper bound, which is the na\"ive generalization of the convexity argument of Bubeck and Ganguly in \cite{BG16}.

Recall that $\downNv = N(v) \cap [v-1]$ and $X_{\downNv}$ denotes the submatrix of $\Xvl$ with columns indexed by $\downNv$.
Let $I_{\downNv}$ denote the identity matrix with rows and columns indexed by $\downNv$.
Because $W(\Xvl)$ is a $\Xvl$-measurable random variable, by convexity of KL divergence we have 
\begin{align}
    \notag
    &\E_{W\sim \muvl}
    \KL\lt((\muv | W) \parallel \nuv\rt)
    \le 
    \E_{\Xvl \sim \cN(0,I_d)^{\otimes v-1}}
    \KL\lt((\muv | \Xvl) \parallel \nuv\rt) \\
    \label{eq:general-naive-ub-starting-point}
    &\qquad = 
    \E_{\Xvl \sim \cN(0,I_d)^{\otimes v-1}}
    \KL\lt(
        \cN\lt(0, d^{-1}X_{\downNv}^\top X_{\downNv}\rt)
        \parallel
        \cN\lt(0, I_{\downNv}\rt)
    \rt).
\end{align}
This last KL divergence can be evaluated by the following two lemmas.
Lemma~\ref{lem:kl-gaussians} is well known.

\begin{lemma}[KL Divergence between Gaussians] 
    \label{lem:kl-gaussians}
    Let $\Sigma_1, \Sigma_2 \in \mathbb{R}^{k \times k}$ be positive definite matrices. 
    Then,
    \[
        \KL\lt( 
            \cN(0, \Sigma_1) 
            \parallel 
            \cN(0, \Sigma_2) 
        \rt) 
        = \f12 \lt[ 
            \log \f{\det \Sigma_2}{\det \Sigma_1} + 
            \Tr\lt( \Sigma_2^{-1} \Sigma_1 - I_k\rt) 
        \rt].
    \]
\end{lemma}

\begin{lemma}
    \label{lem:general-naive-ub-logdet}
    Let $d\ge Ck^2$ for some universal constant $C$, and let $X\in \bR^{d\times k}$ have i.i.d. Gaussian entries.
    Then, 
    \[
        \E \lt[-\log \det (d^{-1} X^\top X)\rt]
        \le 
        \f{Ck^2}{d}.
    \]
\end{lemma}
\begin{proof}
    The proof is essentially the same as \cite[Lemma 2]{BG16}, and we only outline the differences.
    Note that
    \[
        \E \lt[-\log \det (d^{-1} X^\top X)\rt]
        =
        \E \lt[
            -\log \det (d^{-1} X^\top X) + 
            \Tr(d^{-1} X^\top X - I_k) 
        \rt].
    \]
    We will bound the latter expectation.
    Let $\lambda_{\min}$ denote the smallest eigenvalue of $d^{-1} X^\top X$; we decompose this expectation on the events $\{\lambda_{\min}\ge \f12\}$ and its complement.
    Note that $-\log(x) + (x-1) \le 2(x-1)^2$ for $x\ge \f12$. 
    By the proof of \cite[Lemma 2]{BG16}, we have
    \[
        \E \lt[
            \lt(
                -\log \det (d^{-1} X^\top X) + 
                \Tr(d^{-1} X^\top X - I_k) 
            \rt)
            \ind{\lambda_{\min}\ge \f12}
        \rt]
        \le 
        \E \lt[
            \norm{d^{-1} X^\top X - I_k}_{\HS}^2
        \rt]
        \le
        \f{Ck^2}{d}.
    \]
    By the proof of \cite[Lemma 2]{BG16}, we also have, for $d \ge Ck^2$,
    \[
        \E \lt[
            -\log \det (d^{-1} X^\top X) 
            \ind{\lambda_{\min} < \f12}
        \rt]
        \le 
        k\exp(-d^{1/10}).
    \]
    By Cauchy-Schwarz, we have
    \[
        \E \lt[
            \Tr(d^{-1} X^\top X - I_k)
            \ind{\lambda_{\min} < \f12}
        \rt]
        \le 
        \E \lt[
            \Tr(d^{-1} X^\top X - I_k)^2
        \rt]^{1/2}
        \P \lt[
            \lambda_{\min} < \f12
        \rt]^{1/2}.
    \]
    By Lemma~\ref{lem:general-ub-computational}(\ref{itm:general-ub-computational-trsq-centered}), $\E \lt[\Tr(d^{-1} X^\top X - I_k)^2\rt] = \f{2k}{d}$.
    By \cite[Corollary 5.35]{Ver10}, we have $\P \lt[\lambda_{\min} < \f12\rt] \le \exp(-\Omega(d))$.
    Combining these bounds proves the lemma.
\end{proof}

This yields the following crude bound.
Recall that $\ddegv = |\downNv|$ is the number of edges from $v$ to $[v-1]$.
\begin{theorem}
    \label{thm:general-naive-ub}
    There exists a universal constant $C$ such that the following inequality holds. 
    If $d \ge C \max_{v\in G} \ddegv^2$, then
    \[
        \KL\lt( W(G,d) \parallel M(G) \rt) 
        \le 
        C \sum_{v\in G} \f{\ddegv^2}{d}.
    \]
\end{theorem}
\begin{proof}
    By (\ref{eq:general-naive-ub-starting-point}) and Lemmas~\ref{lem:kl-gaussians} and \ref{lem:general-naive-ub-logdet}, 
    \[
        \E_{W\sim \muvl}
        \KL\lt((\muv | W) \parallel \nuv\rt)
        \le 
        \E_{\Xvl \sim \cN(0, I_d)^{\otimes (v-1)}}
        \lt[
            -\f12 \log \det \lt(d^{-1}X_{\downNv}^\top X_{\downNv}\rt)
        \rt]
        \le \f{C\ddegv^2}{2d}.
    \]
    Substituting into the tensorization (\ref{eq:general-ub-starting-point}) yields the conclusion.
\end{proof}
By Pinsker's inequality, Theorem~\ref{thm:general-naive-ub} implies that if $d \gg \sum_{v\in G} \ddegv^2$, then we have that $\TV(W(G,d), M(G)) \to 0$.
In terms of graph statistics, this condition gives that if $d \gg \Num_G(P_2, E)$, then $\TV(W(G,d), M(G)) \to 0$.
This result recovers the $d \gg n^3$ threshold of \cite{BG16} for $G = K_n$, but for general $G$ it is considerably weaker than Theorem~\ref{thm:general-ub}.
One way to see this is to derive an analogue of Theorem~\ref{thm:er-mask}(\ref{itm:er-mask-ub}) by applying Theorem~\ref{thm:general-naive-ub} instead of Theorem~\ref{thm:general-ub} to $G\sim \cG(n,p)$.
We can show (analogously to Theorem~\ref{thm:er-mask}(\ref{itm:er-mask-ub})) that if 
\[
    d \gg n^3p^2 + n^2p + \log^2 n,
\]
then $\TV(W(G,d), M(G)) \to 0$ almost surely over the sample path of $G$.
This threshold is plotted in Figure~\ref{fig:er-mask-phase-diagram}, and we can see it is much weaker than the true threshold.

Lemma~\ref{lem:general-naive-ub-logdet} is tight up to constant factors.
So, this is the best asymptotic threshold we can attain with the convexity argument of (\ref{eq:general-naive-ub-starting-point}).
Thus, we cannot throw away the latent information in the distribution $(\muv | W)$; to improve on Theorem~\ref{thm:general-naive-ub}, we must handle this latent information more delicately. 
In the context of the convexity arguments in Section~\ref{subsec:technical-overview-gs}, the estimate in (\ref{eq:general-naive-ub-starting-point}), which reveals $\Xvl$ instead of $W(\Xvl)$, corresponds to estimating $\E_{(\Xonevl, \Xtwovl) \sim \cNvdips} \exp(\f12 Y_v)$ using the strongest coupling $\Xonevl = \Xtwovl$.
Thus, any approach to bounding this quantity that does not use the coupled distribution $\cNvdips$ in a nontrivial way can do at most as well as Theorem~\ref{thm:general-naive-ub}.
So, the conceptual innovations discussed in Sections~\ref{subsec:technical-overview-gs} and \ref{subsec:technical-overview-beyond-gs} are crucial to achieving the full power of Theorem~\ref{thm:general-ub}.

\subsection{Proof Outline of TV Upper Bound for Bipartite Masks}
\label{subsec:technical-overview-bipartite-ub-outline}

The proof of Theorem~\ref{thm:bipartite-ub} is considerably simpler than that of Theorem~\ref{thm:general-ub}, due to the following key observation.
Unlike in the general setting, where we use KL tensorization to break $W(G,d)$ into a sequence of mixture distributions, when $G$ is bipartite $W(G,d)$ is itself a mixture distribution.
This is because all the information in $W(G,d)$ is contained in the appropriate subset of $d^{-1/2}\XR^\top \XL$, where $\XL \in \bR^{d\times |V_L|}$ and $\XR \in \bR^{d\times |V_R|}$ are the submatrices of $X\in \bR^{d\times n}$ consisting of the columns in $V_L$ and $V_R$. 
This is a mixture of jointly Gaussian matrices parametrized by latent randomness~$\XR$.

The bipartite setting affords us two important simplifications over the argument for general masks. 
First, because $W(G,d)$ is a mixture distribution, we do not need the KL tensorization step. 
In fact, we can prove Theorem~\ref{thm:bipartite-ub} without reference to KL divergence at all, by passing directly from $\TV$ to $\chisq$ and applying the second moment method.
We still need to condition on a high probability event $S\in \sigma(\XR)$ before we pass to $\chisq$ divergence; however, not needing KL tensorization means we only need to bound the effect of conditioning (in the latent space) on $\TV$, which is trivial by (\ref{eq:technical-overview-tv-conditioning}) and data processing.
Second, because $V_R$ is an independent set, the two replicas $\XoneR, \XtwoR$ obtained from the second moment method are now fully independent, each distributed as a sample from $\cL\lt(\cN(0, I_d)^{\otimes |V_R|} | S\rt)$, so we no longer need to work with the coupled distribution.
This allows us to get extremely sharp bounds.

After evaluating the inner expectation of the second moment method (\ref{eq:technical-overview-2mm}), we are left with an expectation over this distribution whose integrand can be massaged into the exponential of a polynomial in $\XoneR$ and $\XtwoR$. 
This setup is a batched version of the overlap $Y_v$ defined in (\ref{eq:general-ub-define-yv}). 
For each fixed value of $\XoneR$, we can consider the integrand as the exponential of a polynomial in $\XtwoR$. 
We can integrate this quantity by tails, using Gaussian hypercontractivity to control tail probabilities.
The resulting estimate is an expression in $\XoneR$, which we can make deterministically small conditioned on $\XoneR \in S$.

\section{Main TV Upper Bound Argument for General Masks}
\label{sec:general-ub-main-argument}

In this section, we will develop our main techniques for showing total variation upper bounds for general masks $G$.
These techniques will yield the following intermediate result, which is a weaker variant of Theorem~\ref{thm:general-ub}.
In Sections~\ref{sec:general-ub-refined-linear-term} and \ref{sec:general-ub-refined-higher-order-terms}, we will refine the arguments in this section to arrive at Theorem~\ref{thm:general-ub}.

\begin{theorem}
    \label{thm:general-weaker-ub}
    Suppose the following asymptotic inequalities hold:
    \begin{eqnarray}
        \label{eq:general-weaker-ub-hypothesis-triangles}
        d &\gg& \Num_G(C_3), \\
        \label{eq:general-weaker-ub-hypothesis-4claws}
        d^2 &\gg& \Num_G(K_{1,4}) + \Num_G(P_2, E) \log^4 n.
    \end{eqnarray}
    Then, $\TV(W(G,d), M(G)) \to 0$ as $n\to \infty$.
\end{theorem}

\begin{remark}
    By applying Theorem~\ref{thm:general-weaker-ub} to $G \sim \cG(n,p)$, we can show (analogously to Theorem~\ref{thm:er-mask}(\ref{itm:er-mask-ub})) that if
    \[
        d \gg n^3p^3 + n^{5/2}p^2 + n^{3/2}p\log^2 n + np^{1/2} \log^2 n + \log^3 n,
    \]
    then $\TV(W(G,d), M(G)) \to 0$ almost surely over the sample path of $G$.
    This threshold is plotted in Figure~\ref{fig:er-mask-phase-diagram}.
    This threshold matches Theorem~\ref{thm:er-mask}(\ref{itm:er-mask-ub}) for $p\gtrsim n^{-1/2}$ and matches it up to a $\polylog(n)$ factor for $p\lesssim n^{-1}$, but these thresholds do not match for $n^{-1/2} \gtrsim p \gtrsim n^{-1}$.
    So, the full power of Theorem~\ref{thm:general-ub} is necessary to identify the sharp phase transition in Theorem~\ref{thm:er-mask}.
\end{remark}

Let $G$ be a graph on $[n]$, and let $\muvl, \muv, \nuv$ be the appropriate marginal measures of $W(G,d)$ and $M(G)$ defined in Section~\ref{subsec:technical-overview-general-ub-outline}.
The starting point of the proof of Theorem~\ref{thm:general-weaker-ub} is the KL tensorization (\ref{eq:general-ub-starting-point}).
Recall that $G[v-1]$ is the induced subgraph of $G$ on $[v-1]$, and let $X_1,X_2,\ldots, X_n \sim \cN(0, I_d)$ be i.i.d. latent Gaussians.
Throughout this section, we will generate $W\sim \muvl$ as $W_{k,\ell} = d^{-1/2}\la X_k, X_\ell\ra$ for each edge $(k,\ell) \in E(G[v-1])$.
The proof of Theorem~\ref{thm:general-weaker-ub} bounds each of the $n$ averaged KL divergences in (\ref{eq:general-ub-starting-point}) individually.
For each $v\in [n]$, the argument to bound the $v$th summand of (\ref{eq:general-ub-starting-point}) is divided into the following three steps.

\begin{enumerate}[label=(\arabic*), ref=\arabic*]
    \item We first will handle the contributions of tail events to the $v$th summand of (\ref{eq:general-ub-starting-point}). 
    More precisely, we show that it suffices to bound this term with $\muv$ and $\muvl$ replaced by the measures $\muvsv$ and $\muvlsv$, which are $\muv$ and $\muvl$ conditioned on the event $\Xvl \in S^v$ for a high probability set $S^v$ capturing the typical behavior of $\Xvl$.
    As discussed in Section~\ref{subsec:technical-overview-general-ub-outline}, we denote $S^v$ by $S$ when $v$ is clear from context.
    The set $S$ is defined in Section~\ref{subsec:general-ub-def-hp-set} and this step is carried out in Section~\ref{subsec:general-ub-kl-conditioning}.
    In Section~\ref{subsec:general-ub-s-hp-proof}, we prove that $S$ occurs with high probability.
    \item In Section~\ref{subsec:general-ub-main-2mm}, we upper bound this now-conditioned KL divergence by a $\chisq$ divergence and explicitly evaluate the resulting expression with the second moment method. 
    To carry out the second moment method computation, we represent $\cL(\muvs|W)$ as a mixture of the distributions $\cL(\muv | \Xvl)$, where $\Xvl \sim \cL\lt(\cN(0, I_d)^{\otimes (v-1)} | S\rt)$ conditioned on $W_v(\Xvl) = W$.
    We will then simplify the result to obtain an upper bound in terms of the coupled exponentiated overlap $\E \exp(\f12 Y_v)$ over $(\Xonevl, \Xtwovl) \sim \cNvdips$, where $\cNvdips$ is defined in Definition~\ref{defn:general-ub-ip-coupling} and $Y_v$ is defined in (\ref{eq:general-ub-define-yv}).
    \item To bound this overlap, we Taylor expand it into first and higher order terms as in (\ref{eq:general-ub-main-exp-overlap-starting-point}).
    We will condition each first order term on a stronger coupling, which by Lemma~\ref{lem:general-ub-stronger-couplings} yields an upper bound.
    The stronger coupling is different for each term and comes from selectively revealing the entries of a Gram-Schmidt orthogonalization of the latent Gaussians $X_1,\ldots,X_{v-1}$ sketched in Section~\ref{subsec:technical-overview-gs}.
    This coupling makes careful use of the specific entries revealed, and thus the dependencies of the original coupling.
    The error term from the Taylor expansion is handled through deterministic bounds following from the construction of $S$.
    This step is carried out in Section~\ref{subsec:general-ub-main-exp-overlap}, where we also put these steps together to complete the proof of Theorem~\ref{thm:general-weaker-ub}.
\end{enumerate}

\subsection{Identifying the High Probability Latent Sets $S^v$}
\label{subsec:general-ub-def-hp-set}

A key construction in this section will be the high probability sets $S^v \in \sigma(\Xvl)$ over the collection of vectors $\Xvl = (X_1,\ldots,X_{v-1})$, which we now formally introduce.
We will first introduce several quantities necessary to define $S^v$.
Given a subset $V\subseteq [n]$, let $X_V$ denote the $d\times |V|$ matrix with columns $X_i$ for $i\in V$, with rows indexed by $[d]$ and columns indexed by $V$.
Similarly, let $I_V$ denote the $|V|\times |V|$ identity matrix with rows and columns indexed by $V$.
Recall that $\Nl(i) = N(i) \cap [v-1]$ denotes the set of neighbors of $i$ in $[v-1]$ and $\dvli = |\Nl(i)|$.
Recall further that $\downNv = \Nl(v)$ and $\ddegv = \degvl(v)$.
Let $\Delbn$ be the symmetric $\ddegv\times \ddegv$ matrix given by
\[
    \Delbn = d^{-1} \Xbnt \Xbn - \Ibn
\]
for each $v\in [n]$, with rows and columns indexed by $\downNv$.
For $\Xvl\in \bR^{d\times (v-1)}$, define $W_v(\Xvl)$ as in (\ref{eq:general-ub-def-wv}); we write this as $W(\Xvl)$ when $v$ is clear from context.
For $W\in \bR^{(v-1)\times (v-1)}$, let $\gamma(W)$ denote the conditional distribution of $\Xvl \in \bR^{d\times (v-1)}$ with i.i.d. standard Gaussian entries, conditioned on $W(\Xvl) = W$.
For $\Xvl \in \bR^{d\times (v-1)}$, define
\[
    f_{\det}(\Xvl) =
    \E_{\Xpvl \sim \gamma(W(\Xvl))}
    \lt[ \det \lt(
        d^{-1} \Xpbnt \Xpbn
    \rt)^{-1/2} \rt].
\]
We can now define the sets $S^v \in \sigma(\Xvl)$ for each $v\in [n]$ as follows.
Let $\Ctr > 0$ be a sufficiently large constant to be determined later.
Then, define
\begin{equation}
    \label{eq:general-ub-def-sv}
    S^v = \Svop \cap \Svtr \cap \Svdet,
\end{equation}
where the constituent events $\Svop,\Svtr,\Svdet \in \sigma(\Xvl)$ are defined by
\begin{eqnarray}
    \label{eq:general-ub-def-svop}
    \Svop
    &=&
    \lt\{
        \Xvl \in \bR^{d\times (v-1)} :
        \norm{\Delta_{\downNv}}_{\op} \le 100 \sqrt{\f{\ddegv + \log n}{d}},
    \rt\}, \\
    \label{eq:general-ub-def-svtr}
    \Svtr
    &=&
    \lt\{
        \Xvl \in \bR^{d\times (v-1)} :
        \Tr\lt(\Delta_{\downNv}^2\rt)
        \le
        \f{2\ddegv^2}{d} +
        \Ctr \cdot \f{\ddegv \log^2 n}{d}
    \rt\}, \\
    \label{eq:general-ub-def-svdet}
    \Svdet
    &=&
    \lt\{
        \Xvl \in \bR^{d\times (v-1)} :
        f_{\det}(\Xvl) \le e^n
    \rt\}.
\end{eqnarray}
When $v$ is clear from context, we will refer to these sets as $S$, $\Sop$, $\Str$, and $\Sdet$, respectively.

The conditions in $\Sop$ and $\Str$ will be important in bounding $\chisq$ divergence by the coupled exponentiated overlap in Section~\ref{subsec:general-ub-main-2mm} and controlling the coupled exponentiated overlap in Section~\ref{subsec:general-ub-main-exp-overlap}.
The restriction in $\Sdet$ is very mild, as the typical value of $f_{\det}(\Xvl)$ is much smaller than exponential in $n$.
This condition will be crucial in the KL conditioning argument in Section~\ref{subsec:general-ub-kl-conditioning}.
In order to show that conditioning in the latent variables $\Xvl$ does not significantly affect the KL divergence of interest, we will need to use some property of the distributions $\gamma(W)$.
However, these conditional distributions do not have explicit representations, and finding a useful property of them that can be rigorously established is the main difficulty of the KL conditioning step.
A key idea in Section~\ref{subsec:general-ub-kl-conditioning} is to show KL conditioning is possible only given that the distributions $\gamma(W)$ satisfy the bound in $\Sdet$.
Showing $\Sdet$ occurs with high probability is then tractable using Markov's inequality, determinant bounds for Wishart matrices and the fact that the mixture $\E_{W}\gamma(W)$ is a matrix of i.i.d. standard Gaussians.
The next proposition asserts that $S^v$ is a high probability set and will be important throughout our proof of Theorem~\ref{thm:general-weaker-ub}.
The proof of this proposition is deferred to Section~\ref{subsec:general-ub-s-hp-proof}.
\begin{proposition}
    \label{prop:general-ub-s-hp}
    Suppose that $d\gg \max_{v\in [n]} \ddegv + \log n$.
    For $n$ sufficiently large and each $v\in [n]$, it holds that $\P(\Svop) \ge 1-n^{-20}$, $\P(\Svtr) \ge 1-n^{-20}$, and $\P(\Svdet) \ge 1-e^{-n/2}$.
\end{proposition}
We will see in Lemma~\ref{lem:general-weaker-ub-hypothesis-translation} that the condition $d\gg \max_{v\in [n]} \ddegv + \log n$, here and in future lemmas and propositions, is implied by the hypotheses of Theorem~\ref{thm:general-weaker-ub}.

\subsection{KL Conditioning in the Latent Space}
\label{subsec:general-ub-kl-conditioning}

Throughout this section, we will ignore entries deterministically set to zero and consider $\muv$ and $\nuv$ as measures on $\bR^{\downNv}$.
For any event $T\in \sigma(\Xvl)$ with positive probability, we let $\muvlt$ and $\muvt$ be $\muvl$ and $\muv$ conditioned on $\Xvl\in T$, respectively.
Formally, $\muvlt$ is the measure of $W_v(\Xvl)$ (recall the definition of $W_v$ in (\ref{eq:general-ub-def-wv})) where $\Xvl$ is sampled from $\cL \lt(\cN(0, I_d)^{\otimes (v-1)} | T \rt)$, and $\muvt$ is the measure of $d^{-1/2} \Xbnt X_v$ where $\Xvl \sim \cL \lt(\cN(0, I_d)^{\otimes (v-1)} | T \rt)$ and $X_v \sim \cN(0, I_d)$ independently of $\Xvl$.
The purpose of this section is to prove the following lemma, which constitutes our main KL conditioning step.
We remark that the $11n^{-9}$ can be replaced with any $n^{-\Theta(1)}$ term with minor modifications to the argument, assumption on $T$ and choice of $\Svdet$.
\begin{lemma}
    \label{lem:general-ub-kl-conditioning}
    Suppose that $d \gg \max_{v\in [n]} \ddegv + \log n$ and $n$ is sufficiently large.
    Let $v\in [n]$, and let $T\in \sigma(\Xvl)$ be an arbitrary measurable set such that $\P(T)\ge 1-2n^{-20}$.
    If $\Tdet = T \cap \Svdet$, it follows that
    \begin{equation}
        \label{eq:general-ub-kl-conditioning}
        \E_{W\sim \muvl}
        \KL\lt((\muv | W) \parallel \nuv\rt)
        \le
        \E_{W\sim \muvltd}
        \KL\lt((\muvtd | W) \parallel \nuv\rt)
        + 11n^{-9}.
    \end{equation}
\end{lemma}
For a fixed realization $\Xvl \in \bR^{d\times (v-1)}$, let $\muv(\Xvl)$ denote $\muv$ conditioned on the value of $\Xvl$.
Formally, $\muv(\Xvl)$ is the measure of $d^{-1/2}\Xbnt X_v$, where $\Xvl$ (and thus $\Xbn$) is fixed and $X_v \sim \cN(0, I_d)$.
Note that for an event $T\in \sigma(\Xvl)$ with positive probability, $\muvt = \E_{\Xvl \sim \cL} \muv(\Xvl)$, where $\cL = \cL \lt(\cN(0, I_{d})^{\otimes (v-1)} | T\rt)$.
We crucially have the property that $\muv(\Xvl) = \cN\lt(0, d^{-1} \Xbnt \Xbn\rt)$, which will be essential throughout the proof of Theorem~\ref{thm:general-weaker-ub}.

To prove Lemma~\ref{lem:general-ub-kl-conditioning}, we will introduce the following five KL-type quantities interpolating between KL divergences on the the left and right hand sides of (\ref{eq:general-ub-kl-conditioning}).
The proof of Lemma~\ref{lem:general-ub-kl-conditioning} consists of individually bounding the successive differences $\KL_A-\KL_B$, $\KL_B-\KL_C$, $\KL_C-\KL_D$, and $\KL_D-\KL_E$.
\begin{align*}
    \KL_A
    &=
    \E_{W\sim \muvl}
    \KL\lt((\muv | W) \parallel \nuv\rt) \\
    &=
    \E_{W\sim \muvl}
    \E_{\phi \sim \nuv}
    \lt(
        \E_{\Xvl \sim \gamma(W)}
        \rn{\muv(\Xvl)}{\nuv}(\phi)
    \rt)
    \log
    \lt(
        \E_{\Xvl \sim \gamma(W)}
        \rn{\muv(\Xvl)}{\nuv}(\phi)
    \rt), \\
    \KL_B
    &=
    \E_{W\sim \muvltd}
    \E_{\phi \sim \nuv}
    \lt(
        \E_{\Xvl \sim \gamma(W)}
        \rn{\muv(\Xvl)}{\nuv}(\phi)
    \rt)
    \log
    \lt(
        \E_{\Xvl \sim \gamma(W)}
        \rn{\muv(\Xvl)}{\nuv}(\phi)
    \rt), \\
    \KL_C
    &=
    \E_{W\sim \muvltd}
    \E_{\phi \sim \nuv}
    \lt(
        \E_{\Xvl \sim \gamma(W)}
        \ind{\Xvl \in \Tdet}
        \rn{\muv(\Xvl)}{\nuv}(\phi)
    \rt)
    \log
    \lt(
        \E_{\Xvl \sim \gamma(W)}
        \rn{\muv(\Xvl)}{\nuv}(\phi)
    \rt), \\
    \KL_D
    &=
    \E_{W\sim \muvltd}
    \E_{\phi \sim \nuv}
    \lt(
        \E_{\Xvl \sim \gamma(W)}
        \ind{\Xvl \in \Tdet}
        \rn{\muv(\Xvl)}{\nuv}(\phi)
    \rt)
    \log
    \lt(
        \E_{\substack{
            \Xvl \sim \gamma(W) \\
            \Xvl \in \Tdet
        }}
        \rn{\muv(\Xvl)}{\nuv}(\phi)
    \rt), \\
    \KL_E
    &=
    \E_{W\sim \muvltd}
    \KL\lt((\muvtd | W) \parallel \nuv\rt) \\
    &=
    \E_{W\sim \muvltd}
    \E_{\phi \sim \nuv}
    \lt(
        \E_{\substack{
            \Xvl \sim \gamma(W) \\
            \Xvl \in \Tdet
        }}
        \rn{\muv(\Xvl)}{\nuv}(\phi)
    \rt)
    \log
    \lt(
        \E_{\substack{
            \Xvl \sim \gamma(W) \\
            \Xvl \in \Tdet
        }}
        \rn{\muv(\Xvl)}{\nuv}(\phi)
    \rt).
\end{align*}

In the definitions of $\KL_D$ and $\KL_E$, $\Xvl\sim \gamma(W)$ with $\Xvl \in \Tdet$ denotes that $\Xvl$ is a sample from $\gamma(W)$ conditioned on $\Xvl \in \Tdet$.
The next proposition bounds the successive differences as described above.

\begin{proposition}
    \label{prop:general-ub-kl-successive-differences}
    Suppose that $d \gg \max_{v\in [n]} \ddegv + \log n$.
    Then, the following inequalities hold for all sufficiently large $n$. 
    \begin{enumerate}[label=(\alph*), ref=\alph*]
        \item \label{itm:kl-conditioning-kla-klb} $\KL_A \le \KL_B + 2\P(\Tdetc)^{1/2} n$. 
        \item \label{itm:kl-conditioning-klb-klc} $\KL_B \le \KL_C + 3\P(\Tdetc)^{1/2} n$.
        \item \label{itm:kl-conditioning-klc-kld} $\KL_C \le \KL_D + \P(\Tdetc)$. 
        \item \label{itm:kl-conditioning-kld-kle} $\KL_D \le \KL_E$.
    \end{enumerate}
\end{proposition}
We defer the proof of this proposition to Appendix~\ref{appsec:kl-conditioning}.
Lemma~\ref{lem:general-ub-kl-conditioning} readily follows from these bounds.
\begin{proof}[Proof of Lemma~\ref{lem:general-ub-kl-conditioning}]
    Because $d \gg \max_{v\in [n]} \ddegv + \log n$, Propositions~\ref{prop:general-ub-s-hp} and \ref{prop:general-ub-kl-successive-differences} hold.
    By Proposition~\ref{prop:general-ub-s-hp}, $\P(\Tdetc) \le \P(T^c) + \P((\Svdet)^c) \le 2n^{-20} + e^{-n/2} \le 3n^{-20}$ for sufficiently large $n$.
    Summing the bounds in Proposition~\ref{prop:general-ub-kl-successive-differences} gives
    \[
        \KL_A \le \KL_E + 6\P(\Tdetc)^{1/2} n \le \KL_E + 11n^{-9}
    \]
    because $6\sqrt{3} < 11$.
    This completes the proof of the lemma.
\end{proof}

\subsection{Bounding $\chisq$ Divergence with the Second Moment Method}
\label{subsec:general-ub-main-2mm}

In this section, we will evaluate and simplify an upper bound on right-hand side summands in Lemma~\ref{lem:general-ub-kl-conditioning} with $T=\Sop \cap \Str$ (and thus $\Tdet = S$).
We will derive an upper bound in terms of the coupled exponentiated overlap 
\[
    \E_{(\Xonevl, \Xtwovl) \sim \cNvdips} \exp\lt(\f12 Y_v\rt),
\]
where $\cNvdips$ is defined in Definition~\ref{defn:general-ub-ip-coupling} and $Y_v$ is defined in (\ref{eq:general-ub-define-yv}).
The next lemma is the main result of this section.

\begin{lemma}[Bounds from the Second Moment Method]
    \label{lem:general-ub-main-2mm}
    Suppose that $d \gg \max_{v\in [n]} \ddegv + \log n$.
    Let $v\in [n]$ and suppose that $n$ is sufficiently large.
    Define $Y_v$ as in (\ref{eq:general-ub-define-yv}).
    Then, we have that
    \begin{align*}
        &\E_{W\sim \muvls}
        \KL\lt((\muvs | W) \parallel \nuv\rt) \\
        &\qquad \le
        -1 + \exp\lt(
            \f{100^4}{d^2}
            (\ddegv^3 + \ddegv \log^2 n)
        \rt) 
        \E_{(\Xonevl, \Xtwovl) \sim \cNvdips}
        \exp\lt(\f12 Y_v\rt).
    \end{align*}
\end{lemma}

In order to prove this lemma, we will begin with the following standard fact on the expected value of an exponentiated quadratic form of i.i.d. standard Gaussians.

\begin{lemma}
    \label{lem:exp-quadratic-gaussian}
    If $\Sigma \in \bR^{k\times k}$ is symmetric and $I_k + \Sigma$ is positive definite, then
    \[
        \E_{\phi\sim \cN(0, I_k)}
        \exp\lt(-\f12 \phi^\top \Sigma \phi\rt)
        = \det(I_k + \Sigma)^{-1/2}.
    \]
\end{lemma}
\begin{proof}
    A direct computation yields that
    \begin{align*}
        \E_{\phi\sim \cN(0, I_k)}
        \exp\lt(-\f12 \phi^\top \Sigma \phi\rt)
        &=
        \int_{\bR^k}
        (2\pi)^{-k/2}
        \exp\lt(-\f12 \phi^\top \Sigma \phi\rt)
        \exp\lt(-\f12 \phi^\top \phi\rt)
        \diff{\phi} \\
        &=
        \det(I_k+\Sigma)^{-1/2}
        \int_{\bR^k}
        (2\pi)^{-k/2}
        \det(I_k+\Sigma)^{1/2}
        \exp\lt(-\f12 \phi^\top (I_k + \Sigma) \phi\rt)
        \diff{\phi} \\
        &=
        \det (I_k + \Sigma)^{-1/2}.
    \end{align*}
    The last equality holds because the last integrand is the probability density of a Gaussian vector with covariance matrix $(I_k + \Sigma)^{-1}$.
\end{proof}

This fact yields the following proposition, which evaluates the $\chisq$-type correlations that arise as the inner expectation of our application of the second moment method.

\begin{proposition}
    \label{prop:general-ub-2mm-inner-expectation}
    Let $\Xonevl, \Xtwovl \in \bR^{d\times (v-1)}$ be two fixed realizations of these random variables.
    For $r=1,2$, let $\Delrbn$ be the $\ddegv \times \ddegv$ real matrix with rows and columns indexed by $\downNv$ given by
    \[
        \Delrbn
        =
        d^{-1}
        \lt(\Xrbn\rt)^\top \Xrbn - \Ibn.
    \]
    If the matrix $\lt(\Delonebn + \Ibn\rt)^{-1} + \lt(\Deltwobn + \Ibn\rt)^{-1} - \Ibn$ is positive definite, then
    \[
        \E_{\phi\sim \nuv}
        \rn{\muv\lt(\Xonevl\rt)}{\nuv}(\phi)
        \rn{\muv\lt(\Xtwovl\rt)}{\nuv}(\phi)
        =
        \det\lt(\Ibn - \Delonebn \Deltwobn\rt)^{-1/2}.
    \]
\end{proposition}
\begin{proof}
    Conditioned on $\Xrvl$, the measure $\muv\lt(\Xrvl \rt)$ is a jointly Gaussian vector with covariance matrix $d^{-1} \lt(\Xrbn\rt)^{\top} \Xrbn$.
    Thus $\muv\lt(\Xrvl\rt)$ is the distribution $\cN\lt(0, \Ibn + \Delrbn\rt)$ for each $r=1,2$.
    Furthermore, $\nuv$ is the distribution $\cN\lt(0, \Ibn\rt)$.
    For positive definite $\Sigma \in \bR^{k\times k}$, let
    \[
        p_\Sigma(\phi)
        =
        (2\pi)^{-k/2}
        (\det \Sigma)^{-1/2}
        \exp\lt(-\f12 \phi^\top \Sigma^{-1} \phi\rt)
    \]
    be the probability density of $\cN(0, \Sigma)$ with respect to the Lebesgue measure on $\bR^k$.
    Directly expanding probability densities and applying Lemma~\ref{lem:exp-quadratic-gaussian} yields that
    \begin{align*}
        \E_{\phi\sim \nuv}
        &=
        \rn{\muv\lt(\Xonevl\rt)}{\nuv}(\phi)
        \rn{\muv\lt(\Xtwovl\rt)}{\nuv}(\phi)
        =
        \E_{\phi\sim \cN(0, \Ibn)}
        \f{
            p_{\Delonebn + \Ibn}(\phi)
            p_{\Deltwobn + \Ibn}(\phi)
        }{
            p_{\Ibn}(\phi)^2
        } \\
        &=
        \det\lt(\Delonebn + \Ibn\rt)^{-1/2}
        \det\lt(\Deltwobn + \Ibn\rt)^{-1/2} \\
        &\qquad \times
        \E_{\phi\sim \cN(0, \Ibn)}
        \exp\lt(
            -\f12 \phi^\top \lt[
                \lt(\Delonebn + \Ibn\rt)^{-1} +
                \lt(\Deltwobn + \Ibn\rt)^{-1} -
                2 \Ibn
            \rt] \phi
        \rt) \\
        &=
        \det\lt(\Delonebn + \Ibn\rt)^{-1/2}
        \det\lt(\Deltwobn + \Ibn\rt)^{-1/2} \\
        &\qquad \times
        \det\lt[
            \lt(\Delonebn + \Ibn\rt)^{-1} +
            \lt(\Deltwobn + \Ibn\rt)^{-1} -
            \Ibn
        \rt]^{-1/2} \\
        &=
        \det\lt(\Ibn - \Delonebn \Deltwobn\rt)^{-1/2}
    \end{align*}
    which completes the proof of the proposition.
\end{proof}

This proposition yields an integrand of the form $\det(I_k - \Pi)^{-1/2}$.
The following simple lemma gives an upper bound on this quantity that will be more convenient to work with in the rest of the proof.

\begin{lemma}
    \label{lem:det-to-exp-to-deg2}
    There exists an absolute constant $\eps > 0$ such that if $\Pi \in \bR^{k\times k}$ (and $\Pi$ is not necessarily symmetric) and $\norm{\Pi}_{\op} \le \eps$, then
    \[
        \det(I_k - \Pi)
        \ge
        \etr(-\Pi)
        \exp\lt(-\sum_{\lambda \in \spec(\Pi)} |\lambda|^2\rt).
    \]
\end{lemma}
\begin{proof}
    For real $\lambda$ in a sufficiently small neighborhood of $0$, we have that
    \[
        \log(1-\lambda)
        =
        - \lambda - \f12 \lambda^2 - O(|\lambda|^3)
        \ge
        -\lambda - |\lambda|^2.
    \]
    For complex conjugates $\lambda, \blambda$ in a neighborhood of $0$,
    \begin{align*}
        \log\lt((1-\lambda)(1-\blambda)\rt)
        &=
        \log \lt(1 - \lambda - \blambda + |\lambda|^2\rt) \\
        &=
        - (\lambda + \blambda - |\lambda|^2)
        - \f12 (\lambda + \blambda - |\lambda|^2)^2
        - O(|\lambda|^3) \\
        &\ge
        - \lt(\lambda + |\lambda|^2\rt)
        - \lt(\blambda + |\blambda|^2\rt).
    \end{align*}
    Set $\eps$ so the above bounds hold when $|\lambda| \le \eps$.
    If $\norm{\Pi}_{\op} \le \eps$, then $|\lambda|\le \eps$ for all $\lambda \in \spec(\Pi)$ and thus
    \[
        \det(I_k - \Pi)
        =
        \prod_{\lambda \in \spec(\Pi)}
        (1-\lambda)
        \ge
        \exp\lt(
            \sum_{\lambda \in \spec(\Pi)}
            (-\lambda - |\lambda|^2)
        \rt)
        =
        \etr(-\Pi)
        \exp\lt(
            -\sum_{\lambda \in \spec(\Pi)}
            |\lambda|^2
        \rt),
    \]
    which completes the proof of the lemma.
\end{proof}

We now combine these propositions and lemmas to prove Lemma~\ref{lem:general-ub-main-2mm}.

\begin{proof}[Proof of Lemma~\ref{lem:general-ub-main-2mm}]
    For now, fix some valid realization of the random matrix $W\in \bR^{(v-1)\times (v-1)}$
    with $W_{i,j} = 0$ for all $i,j\not\in E(G[v-1])$.
    Note that $\muvs$ given $W$ can be written as a mixture of Gaussians, as
    \[
        \cL(\muvs | W)
        =
        \E_{\substack{
            \Xvl \sim \gamma(W) \\
            \Xvl \in S
        }}
        \muv(\Xvl)
        =
        \E_{\substack{
            \Xvl \sim \gamma(W) \\
            \Xvl \in S
        }}
        \cN\lt(0, d^{-1} \Xbnt \Xbn\rt).
    \]
    Here, we recall that $\Xvl \sim \gamma(W)$ with $\Xvl \in S$ denotes that $\Xvl$ is sampled from $\gamma(W)$ conditioned on $\Xvl \in S$.
    Applying the $\chisq$ upper bound on KL divergence and the second moment method, we have that
    \begin{align}
        \notag
        1 + \KL\lt( (\muvs | W) \parallel \nuv\rt)
        &\le
        1 + \chisq\lt((\muvs | W), \nuv\rt) \\
        \notag
        &=
        \E_{\phi\sim \nuv}
        \lt(
            \E_{\substack{
                \Xvl \sim \gamma(W) \\
                \Xvl \in S
            }}
            \rn{\muv(\Xvl)}{\nuv}(\phi)
        \rt)^{2} \\
        \label{eq:general-ub-main-2mm-starting-point}
        &=
        \E_{\substack{
            \Xonevl \sim \gamma(W) \\
            \Xonevl \in S
        }}
        \E_{\substack{
            \Xtwovl \sim \gamma(W) \\
            \Xtwovl \in S
        }}
        \E_{\phi\sim \nuv}
        \rn{\muv\lt(\Xonevl\rt)}{\nuv}(\phi)
        \rn{\muv\lt(\Xtwovl\rt)}{\nuv}(\phi).
    \end{align}
    We now show that the positive definite condition in Proposition~\ref{prop:general-ub-2mm-inner-expectation} holds whenever $\Xonevl, \Xtwovl \in S$ and $n$ is sufficiently large.
    Recall that $S\subseteq \Sop$.
    The definition (\ref{eq:general-ub-def-sv}) of $\Sop$ implies that
    \[
        \norm{\Delrbn}_{\op}
        \le
        100
        \sqrt{\f{\ddegv + \log n}{d}}
    \]
    for each $r=1,2$.
    Since $d \gg \max_{v\in [n]} \ddegv + \log n$, we have that $\norm{\Delrbn}_{\op} \le \f12$ if $n$ is sufficiently large.
    When this holds, we have $\spec \lt(\Delrbn + \Ibn\rt)\le \f32$. Thus, $\spec\lt(\lt(\Delrbn + \Ibn\rt)^{-1}\rt) \ge \f23$ for each $r=1,2$.
    Therefore, we have that
    \[
        \spec\lt(
            \lt(\Delonebn + \Ibn\rt)^{-1} +
            \lt(\Deltwobn + \Ibn\rt)^{-1} -
            \Ibn
        \rt)
        \ge \f13
    \]
    and hence
    $\lt(\Delonebn + \Ibn\rt)^{-1} + \lt(\Deltwobn + \Ibn\rt)^{-1} - \Ibn$ is positive definite, verifying the condition in Proposition~\ref{prop:general-ub-2mm-inner-expectation}.
    Now, applying Proposition~\ref{prop:general-ub-2mm-inner-expectation} to (\ref{eq:general-ub-main-2mm-starting-point}), we have that
    \[
        1 + \KL\lt((\muvs | W) \parallel \nuv\rt)
        \le
        \E_{\substack{
            \Xonevl \sim \gamma(W) \\
            \Xonevl \in S
        }}
        \E_{\substack{
            \Xtwovl \sim \gamma(W) \\
            \Xtwovl \in S
        }}
        \det\lt(\Ibn - \Delonebn \Deltwobn\rt)^{-1/2}.
    \]
    Because $\Xonevl, \Xtwovl \in S\subseteq \Sop$, we have that
    \[
        \norm{\Delonebn \Deltwobn}_{\op}
        \le
        \norm{\Delonebn}_{\op}
        \norm{\Deltwobn}_{\op}
        \le
        \f{100^2}{d}
        \lt(\ddegv + \log n\rt).
    \]
    Since $d \gg \max_{v\in [n]} \ddegv + \log n$, for the $\eps$ in Lemma~\ref{lem:det-to-exp-to-deg2}, we have that $\norm{\Delonebn \Deltwobn}_{\op} \le \eps$ for all sufficiently large $n$.
    When this occurs, Lemma~\ref{lem:det-to-exp-to-deg2} implies that
    \begin{align*}
        &1 + \KL\lt((\muvs | W) \parallel \nuv\rt) \\
        &\qquad \le
        \E_{\substack{
            \Xonevl \sim \gamma(W) \\
            \Xonevl \in S
        }}
        \E_{\substack{
            \Xtwovl \sim \gamma(W) \\
            \Xtwovl \in S
        }}
        \etr\lt(\f12 \Delonebn \Deltwobn\rt)
        \exp\lt(
            \f12
            \sum_{\lambda\in \spec(\Delonebn \Deltwobn)}
            |\lambda|^2
        \rt).
    \end{align*}
    Because $\Xonevl, \Xtwovl \in S\subseteq \Sop$, we have that
    \begin{align*}
        \sum_{\lambda\in \spec(\Delonebn \Deltwobn)}
        |\lambda|^2
        &\le
        \ddegv
        \norm{\Delonebn \Deltwobn}_{\op}
        \le
        \f{100^4}{d^2} \ddegv
        \lt(\ddegv + \log n\rt)^2 \\
        &\le
        \f{2\cdot 100^4}{d^2}
        (\ddegv^3 + \ddegv \log^2 n)
    \end{align*}
    by AM-GM.
    Putting these inequalities together, we now have that
    \begin{align*}
        \E_{W\sim \muvls}
        \KL\lt((\muvs | W) \parallel \nuv\rt)
        &\le
        -1 +
        \exp\lt(
            \f{100^4}{d^2}
            (\ddegv^3 + \ddegv \log^2 n)
        \rt) \\
        &\qquad \times
        \E_{W\sim \muvls}
        \E_{\substack{
            \Xonevl \sim \gamma(W) \\
            \Xonevl \in S
        }}
        \E_{\substack{
            \Xtwovl \sim \gamma(W) \\
            \Xtwovl \in S
        }}
        \etr\lt(\f12 \Delonebn \Deltwobn\rt).
    \end{align*}
    As observed in the discussion before Definition~\ref{defn:general-ub-ip-coupling}, if $W\sim \muvls$ and $\Xrvl$ for $r=1,2$ are i.i.d. samples from $\gamma(W)$ conditioned on $\Xrvl \in S$, then $(\Xonevl, \Xtwovl)$ is distributed as $\cNvdips$.
    Expanding $\Tr\lt(\f12 \Delonebn \Deltwobn\rt)$ using the definition of $\Delrbn$ and recalling the definition (\ref{eq:general-ub-define-yv}) of $Y_v$ yields that
    \[
        \E_{W\sim \muvls}
        \E_{\substack{
            \Xonevl \sim \gamma(W) \\
            \Xonevl \in S
        }}
        \E_{\substack{
            \Xtwovl \sim \gamma(W) \\
            \Xtwovl \in S
        }}
        \etr\lt(\f12 \Delonebn \Deltwobn\rt) 
        =
        \E_{(\Xonevl, \Xtwovl)\sim \cNvdips}
        \exp\lt(\f12 Y_v\rt).
    \]
    This completes the proof of the lemma.
\end{proof}

\subsection{Bounding the Coupled Exponentiated Overlap}
\label{subsec:general-ub-main-exp-overlap}

In this section, we prove the following bound on the coupled exponentiated overlap obtained in the previous section and complete the proof of Theorem~\ref{thm:general-weaker-ub}.
Recall that $\downNv = N(v) \cap [v-1]$, $G[\downNv]$ denotes the induced subgraph of $G$ on $\downNv$, and $E(G[\downNv])$ denotes the edge set of this subgraph.

\begin{lemma}
    \label{lem:general-ub-main-exp-overlap}
    Suppose that $d \gg \max_{v\in [n]} \lt(\ddegv^2 + \ddegv \log^2 n\rt)$.
    Let $v\in [n]$ and suppose that $n$ is sufficiently large.
    Then, we have that
    \begin{align*}
        \E_{(\Xonevl, \Xtwovl) \sim \cNvdips}
        \exp\lt(\f12 Y_v\rt) 
        &\le
        1 +
        \f{|E(G[\downNv])|}{d}
        +
        \f{\ddegv}{d^2} \sum_{i\in \downNv} \dvli \\
        &\qquad +
        \f{(\Ctr + 2)^2}{2d^2}
        \lt(\ddegv^4 + \ddegv^2 \log^4 n\rt)
        +
        n^{-10}.
    \end{align*}
\end{lemma}

As discussed in Section~\ref{subsec:technical-overview-gs}, we will prove this bound by expanding the left-hand side using (\ref{eq:general-ub-main-exp-overlap-starting-point}), reproduced below for clarity.
The below expectations are over $(\Xonevl, \Xtwovl) \sim \cNvdips$.
\begin{align}
    \notag
    &\E \exp\lt(\f12 Y_v\rt) \\
    \tag{\ref{eq:general-ub-main-exp-overlap-starting-point}}
    &\qquad =
    1 +
    \f{1}{2d^2}
    \sum_{i,j\in \downNv}
    \E\lt[
        \lt(\la \Xonei, \Xonej\ra - d\delta_{i,j}\rt)
        \lt(\la \Xtwoi, \Xtwoj\ra - d\delta_{i,j}\rt)
    \rt] 
    + \E h\lt(\f12 Y_v\rt).
\end{align}
To bound the expectations of the linear terms, we will use Lemma~\ref{lem:general-ub-stronger-couplings} to replace $\cNvdips$ with a stronger coupling tailored to each term.
The Taylor error term $\E h(\f12 Y_v)$ can be immediately bounded using the definition (\ref{eq:general-ub-def-svtr}) of $\Str$, as in the next proposition.

\begin{proposition}
    \label{prop:general-ub-bound-hy}
    If $d \gg \max_{v\in [n]} \lt(\ddegv^2 + \ddegv \log^2 n\rt)$, then for sufficiently large $n$ we have
    \[
        \E_{(\Xonevl, \Xtwovl) \sim \cNvdips} 
        h\lt(\f12 Y_v\rt)
        \le
        \f{(\Ctr + 2)^2}{2d^2}
        \lt(\ddegv^4 + \ddegv^2 \log^4 n\rt).
    \]
\end{proposition}
\begin{proof}
    First note that since $h(x) = \sum_{t=2}^\infty \f{1}{t!} x^t$ for all $x\in \bR$, we have that $h(x) \le h(|x|)$ for all $x\in \bR$ and that $h$ is strictly increasing on $[0, \infty)$.
    Combining these properties with the triangle inequality and AM-GM, we have that
    \begin{align*}
        h\lt(\f12 Y_v\rt)
        &\le
        h\lt(\f12 |Y_v|\rt)
        \le
        h\lt(
            \f{1}{2d^2}
            \sum_{i,j\in \downNv}
            \lt|\la \Xonei, \Xonej\ra - d\delta_{i,j}\rt|
            \lt|\la \Xtwoi, \Xtwoj\ra - d\delta_{i,j}\rt|
        \rt)\\
        &\le
        h\lt(
            \f{1}{4d^2}
            \sum_{i,j\in \downNv}
            \lt(\la \Xonei, \Xonej\ra - d\delta_{i,j}\rt)^2
            +
            \f{1}{4d^2}
            \sum_{i,j\in \downNv}
            \lt(\la \Xtwoi, \Xtwoj\ra - d\delta_{i,j}\rt)^2
        \rt).
    \end{align*}
    Now note that if $\Xrvl \in S \subseteq \Str$, then the definition (\ref{eq:general-ub-def-svtr}) of $\Str$ implies that
    \[
        \f{1}{d^2}
        \sum_{i,j\in \downNv}
        \lt(\la \Xri, \Xrj\ra - d\delta_{i,j}\rt)^2
        =
        \Tr\lt(\lt(\Delrbn\rt)^2\rt)
        \le
        \f{2\ddegv^2}{d} +
        \Ctr\cdot \f{\ddegv \log^2 n}{d}
    \]
    almost surely for each of $r=1,2$.
    Since $(\Xonevl, \Xtwovl) \sim \cNvdips$ satisfies that $\Xonevl, \Xtwovl \in S \subseteq \Str$ almost surely, we have that
    \[
        \E h\lt(\f12 Y_v\rt)
        \le
        h\lt(
            \f{\ddegv^2}{d} +
            \Ctr \cdot \f{\ddegv \log^2 n}{2d}
        \rt)
        \le
        h\lt(
            \f{\Ctr+2}{2d}
            \lt(\ddegv^2 + \ddegv \log^2 n\rt)
        \rt).
    \]
    Because $d \gg \max_{v\in [n]} \lt(\ddegv^2 + \ddegv \log^2 n\rt)$, the argument of $h$ in the last bound is in a neighborhood of $0$.
    Because $h(x)\le x^2$ for all $|x|\le 1$, we have
    \[
        \E h\lt(\f12 Y_v\rt)
        \le
        \lt(
            \f{\Ctr+2}{2d}
            \lt(\ddegv^2 + \ddegv \log^2 n\rt)
        \rt)^2
        \le
        \f{(\Ctr+2)^2}{2d^2}
        (\ddegv^4 + \ddegv^2 \log^4 n)
    \]
    by AM-GM.
    This completes the proof of the proposition.
\end{proof}

To control the expectations of the linear terms in (\ref{eq:general-ub-main-exp-overlap-starting-point}), we first show that it suffices to bound their expectations without conditioning on $\Xonevl, \Xtwovl \in S$.
We will need the following general lemma bounding the effect of conditioning by the second moment.

\begin{lemma}
    \label{lem:unconditioning-by-second-moment}
    Let $Z\in \bR$ be a random variable with finite second moment and $A\subseteq \bR$ be a measurable event with respect to $Z$ such that $\P(A)> 0$.
    Then, it follows that
    \[
        |\E[Z] - \E[Z|A]|
        \le
        \f{2\P(A^c)^{1/2}\E[Z^2]^{1/2}}{\P(A)}.
    \]
\end{lemma}
\begin{proof}
    By the triangle inequality and Cauchy-Schwarz, we have that
    \begin{align*}
        |\E[Z] - \E[Z|A]|
        &\le
        \P(A)^{-1} |\E[Z] - E[Z\ind{A}]| + (\P(A)^{-1}-1) |\E[Z]| \\
        &=
        \P(A)^{-1} |\E[Z\ind{A^c}]| + (\P(A)^{-1}-1) |\E[Z]| \\
        &\le
        \P(A)^{-1} \E[Z^2]^{1/2} \E [\ind{A^c}]^{1/2} + (\P(A)^{-1}-1) \E[Z^2]^{1/2} \\
        &= \lt(\f{\P(A^c) + \P(A^c)^{1/2}}{\P(A)}\rt) \E[Z^2]^{1/2}.
    \end{align*}
    The lemma now follows from the fact that $\P(A^c)\le \P(A^c)^{1/2}$.
\end{proof}
Given this lemma, we can bound the effect of conditioning on $S$ on the linear terms of (\ref{eq:general-ub-main-exp-overlap-starting-point}).
Let $\cNvdipr$ denote the coupling $\cNvdips$ without conditioning on $\Xonevl, \Xtwovl \in S$, i.e. the coupling where $\Xonevl$ is sampled from $\cN(0, I_d)^{\otimes (v-1)}$ unconditionally and $\Xtwovl$ is sampled from $\cN(0, I_d)^{\otimes (v-1)}$ conditioned on (\ref{eq:general-ub-ip-coupling-condition}).

\begin{proposition}
    \label{prop:general-ub-linear-term-unconditioning}
    Suppose that $d \gg \max_{v\in [n]} \ddegv + \log n$ and $n$ is sufficiently large.
    For all $v\in [n]$ and $i,j\in \downNv$, it holds that
    \begin{align*}
        &\Bigg|
            \E_{(\Xonevl, \Xtwovl)\sim \cNvdips}\lt[
                \lt(\la \Xonei, \Xonej \ra - d\delta_{i,j}\rt)
                \lt(\la \Xtwoi, \Xtwoj \ra - d\delta_{i,j}\rt)
            \rt] \\
            &\qquad
            -
            \E_{(\Xonevl, \Xtwovl)\sim \cNvdipr}\lt[
                \lt(\la \Xonei, \Xonej \ra - d\delta_{i,j}\rt)
                \lt(\la \Xtwoi, \Xtwoj \ra - d\delta_{i,j}\rt)
            \rt]
        \Bigg|
        \le 48dn^{-10}.
    \end{align*}
\end{proposition}
\begin{proof}
    We aim to apply Lemma~\ref{lem:unconditioning-by-second-moment}.
    Throughout this proof, let $(\Xonevl, \Xtwovl) \sim \cNvdipr$.
    Note that $\Xonevl$ and $\Xtwovl$ are each marginally distributed as $\cN(0, I_d)^{\otimes (v-1)}$.
    When $i\neq j$, AM-GM and standard computations with Gaussian moments yield that
    \begin{align*}
        \E\lt[
            \la \Xonei, \Xonej\ra^2
            \la \Xtwoi, \Xtwoj\ra^2
        \rt]
        &\le
        \f12
        \E\lt[
            \la \Xonei, \Xonej\ra^4
        \rt]
        +
        \f12
        \E\lt[
            \la \Xtwoi, \Xtwoj\ra^4
        \rt] \\
        &=
        \E\lt[
            \la \Xonei, \Xonej\ra^4
        \rt]
        =
        3d^2 + 6d
        \le
        4d^2
    \end{align*}
    for sufficiently large $d$.
    Similarly, if $i=j$, thenca
    \[
        \E\lt[
            \lt(\la \Xonei, \Xonei\ra-1\rt)^2
            \lt(\la \Xtwoi, \Xtwoi\ra-1\rt)^2
        \rt] 
        \le
        \E\lt[
            \lt(\la \Xonei, \Xonei\ra-1\rt)^4
        \rt]
        =
        12d^2 + 48d
        \le
        16d^2
    \]
    for sufficiently large $d$.
    By Proposition~\ref{prop:general-ub-s-hp}, $\P(S^c) \le 2n^{-20} + e^{-n/2} \le 3n^{-20}$ for sufficiently large $n$.
    So, by a union bound,
    \[
        \P\lt[\lt(\Xonevl, \Xtwovl \in S\rt)^c\rt]
        \le
        2\P(S^c)
        \le
        6n^{-20}.
    \]
    Moreover, for sufficiently large $n$, $\P\lt[\Xonevl, \Xtwovl \in S\rt] \ge \f12$.
    Lemma~\ref{lem:unconditioning-by-second-moment} now upper bounds the expectation difference in this proposition by
    \[
        \f{2(6n^{-20})^{1/2}(16d^2)^{1/2}}{1/2}
        \le
        48dn^{-10}.
    \]
\end{proof}

We now handle the expectation of the linear terms in (\ref{eq:general-ub-main-exp-overlap-starting-point}) with respect to the coupling $\cNvdipr$.
As discussed in Section~\ref{subsec:technical-overview-gs}, we categorize these terms into three categories:
\begin{enumerate}[label=(\roman*), ref=\roman*]
    \item \label{itm:general-ub-main-overlap-term-ijedge} $(i,j)\in E(G[v-1])$; 
    \item \label{itm:general-ub-main-overlap-term-ijnotedge} $i\neq j$ and $(i,j)\not\in E(G[v-1])$; and 
    \item \label{itm:general-ub-main-overlap-term-ii} $i=j$.
\end{enumerate}
We begin by bounding the terms in category (\ref{itm:general-ub-main-overlap-term-ijedge}), the simplest case.

\begin{proposition}
    \label{prop:general-ub-linear-term-ijedge}
    If $(i,j)\in E(G[v-1])$, then
    \[
        \E_{(\Xonevl, \Xtwovl)\sim \cNvdipr}
        \lt[
            \la \Xonei, \Xonej\ra
            \la \Xtwoi, \Xtwoj\ra
        \rt]
        = d.
    \]
\end{proposition}
\begin{proof}
    Let $(\Xonevl, \Xtwovl)\sim \cNvdipr$.
    By the definition of $\cNvdipr$, $\la \Xonei, \Xonej\ra = \la \Xtwoi, \Xtwoj\ra$ and $\Xonei, \Xonej$ are i.i.d. samples from $\cN(0, I_d)$.
    Therefore,
    \[
        \E \lt[
            \la \Xonei, \Xonej\ra
            \la \Xtwoi, \Xtwoj\ra
        \rt]
        =
        \E \lt[
            \la \Xonei, \Xonej\ra^2
        \rt]
        = d.
    \]
\end{proof}

We will now handle the linear terms of (\ref{eq:general-ub-main-exp-overlap-starting-point}) in category (\ref{itm:general-ub-main-overlap-term-ijnotedge}).
To bound these terms, we will apply Lemma~\ref{lem:general-ub-stronger-couplings} to obtain an upper bound where $\cNvdipr$ is replaced with a coupling in terms of the Gram-Schmidt decompositions of $\Xonevl$ and $\Xtwovl$.
This coupling will be easier to work with and yield upper bounds that can be evaluated explicitly.
This is one of the key steps in our method and is the content of the next proposition.
Throughout this section and Sections~\ref{sec:general-ub-refined-linear-term} and \ref{sec:general-ub-refined-higher-order-terms}, we will require a number of couplings of a common form.
This general class of couplings is formalized in the following definition for notational convenience.

\begin{definition}[Modified Couplings of $\Xonevl$ and $\Xtwovl$]
    \label{defn:general-ub-modified-couplings}
    Let $\cC(\Xonevl, \Xtwovl)$ be a set of constraints on $\Xtwovl$ given $\Xonevl$ of the form $f(\Xonevl) = f(\Xtwovl)$ for some measurable function $f$.
    Let $\cMvd(\cC)$ denote the coupling of $\Xonevl, \Xtwovl \sim \cN(0, I_d)^{\otimes (v-1)}$ generated as follows.
    \begin{enumerate}[label=(\arabic*), ref=\arabic*]
        \item Sample $\Xonevl$ from $\cN(0, I_d)^{\otimes(v-1)}$.
        \item Sample $\Xtwovl$ from $\cN(0, I_d)^{\otimes(v-1)}$ conditioned on the event that $\cC(\Xonevl, \Xtwovl)$ holds.
    \end{enumerate}
    We will often refer to $\cMvd(\cC)$ as the coupling arising from $\cC$.
\end{definition}

Before proving our key proposition, we first will review the Gram-Schmidt orthogonalization procedure and establish some notation.
Given a sequence of vectors $Z_1,\ldots,Z_k\in \bR^d$ where $k\le d$, let
\[
    U_1 = \f{Z_1}{\norm{Z_1}_2}
    \quad
    \text{and}
    \quad
    W_{1,1} = \norm{Z_1}_2.
\]
For each $2\le i\le k$, recursively define
\begin{align*}
    U_i
    &=
    \f{
        Z_i - \sum_{j=1}^{i-1} W_{i,j}U_j
    }{\norm{
        Z_i - \sum_{j=1}^{i-1} W_{i,j}U_j
    }_2}, \\
    W_{i,j} &= \langle Z_i, U_j\rangle
    \quad
    \text{for all $1\le j < i$ and}
    \quad
    W_{i,i} =
    \norm{
        Z_i - \sum_{j=1}^{i-1} W_{i,j}U_j
    }_2
\end{align*}
and let $W_{i,j}=0$ for all $1\le i<j\le k$.
This yields an orthogonal collection of unit vectors $U_1,\ldots,U_k \in \bR^d$, which we collect as a matrix $U\in \bR^{d\times k}$, and a lower-triangular matrix $W\in \bR^{k\times k}$.
We will denote
\[
    (W,U) = \GS(Z_1,\ldots,Z_k).
\]
Observe that $Z = UW^\top$ where $Z\in \bR^{d\times k}$ is the matrix with columns given by the $Z_i$.
We will heavily rely on a standard fact about Gram-Schmidt orthogonalization applied to independent Gaussian vectors.
If the inputs $Z_1,\ldots,Z_k$ are i.i.d. samples from $\cN(0, I_d)$, then $W$ and $U$ are independent, with the following distributions:
$U$ is distributed according to the Haar measure on the Stiefel manifold of $d\times k$ matrices with columns that are orthogonal unit vectors and $W$ has mutually independent entries distributed as
\[
    W_{i,i} \sim \sqrt{\chisq(d+1-i)}
    \quad
    \text{for all $1\le i\le k$}
    \quad
    \text{and}
    \quad
    W_{i,j} \sim \cN(0,1)
    \quad
    \text{for all $1\le j<i\le k$.}
\]
This fact can be shown through a simple induction on $k$, can be deduced by directly performing a change of measure on the density function of $(Z_1,\ldots,Z_k)$, or can be derived from Bartlett's decomposition of Wishart matrices (see e.g. Theorem 3.2.14 of \cite{Mui09}).

We now will prove our key proposition bounding the terms of (\ref{eq:general-ub-main-exp-overlap-starting-point}) in category (\ref{itm:general-ub-main-overlap-term-ijnotedge}).
We remark that the bound $\min\lt(\dvli, \dvlj\rt)$ in this proposition can be improved to approximately $\dvlij$, the number of common neighbors of $i,j$ in $[v-1]$. 
This improvement will be necessary to refine Theorem~\ref{thm:general-weaker-ub} to Theorem~\ref{thm:general-ub}, and we will carry out this improvement in Lemma~\ref{lem:general-ub-refined-linear-term}. 

\begin{proposition}
    \label{prop:general-ub-linear-term-ijnotedge}
    Suppose that $i\neq j$, $(i,j)\not\in E(G[v-1])$, and $d \ge \min\lt(\dvli, \dvlj\rt) + 2$, where we recall that $\dvli = |\Nl(i)|$ denotes the number of vertices adjacent to $i$ in $[v-1]$.
    Then,
    \[
        \E_{(\Xonevl, \Xtwovl)\sim \cNvdipr}
        \lt[
            \la \Xonei, \Xonej \ra
            \la \Xtwoi, \Xtwoj \ra
        \rt]
        \le
        \min\lt(\dvli, \dvlj\rt).
    \]
\end{proposition}
\begin{proof}
    Without loss of generality, suppose that $\dvli \le \dvlj$.
    Let $V = \Nl(i) \cup \{i,j\}$, and let $m = |V| = \dvli + 2$.
    Fix some bijection $\pi: [m]\to V$ such that $\pi(m-1) = j$ and $\pi(m)=i$.
    Now consider the coupling $\cMvd(\cCgs[i,j])$ arising from the constraints $\cCgs[i,j]$, defined as follows:
    \begin{align*}
        \Xonek &= \Xtwok
        \quad
        \text{if $k\in [v-1] \setminus V$,} \\
        \Uone &= \Utwo, \\
        \Wone[k,\ell] &= \Wtwo[k,\ell]
        \quad
        \text{for all $1\le \ell \le k \le m$ such that $(k,\ell) \neq (m, m-1)$} \\
        &\text{where $\lt(\Wr, \Ur\rt) = \GS\lt(\Xr[\pi(1)], \ldots, \Xr[\pi(m)]\rt)$ for $r=1,2$.}
    \end{align*}
    The key property that we will exploit with this coupling is that $\Wone[m, m-1]$ and $\Wtwo[m, m-1]$, which are responsible for most of the components of $\Xonei$ and $\Xtwoi$ in the directions of $\Xonej$ and $\Xtwoj$, remain free.
    
    This Gram-Schmidt procedure is well-defined because $d\ge \dvli + 2 = m$.
    We now will verify that these conditions imply that $\la\Xonek, \Xonel\ra = \la\Xtwok, \Xtwol\ra$ for all $(k,\ell)\in E(G[v-1])$.
    First, note that for all $1\le k\le m$ and all $r=1,2$,
    \[
        \Xr[\pi(k)]
        =
        \sum_{a=1}^k
        \Wr[k,a] \Ur[a].
    \]
    Since $\Uone = \Utwo$, the equalities between the entries of $\Wone$ and $\Wtwo$ imply that $\Xonel = \Xtwol$ for all $\ell \in [v-1]\setminus \{i\}$.
    It suffices to verify that $\la \Xonel, \Xonei\ra = \la \Xtwol, \Xtwoi\ra$ for all $\ell \in \Nl(i)$.
    To see this, note that $k=\pi^{-1}(\ell)$ is well-defined and satisfies that $k\le m-2$.
    We have that
    \[
        \la \Xrl, \Xri \ra
        =
        \la \Xr[\pi(k)], \Xr[\pi(m)] \ra
        =
        \lt\la
            \sum_{a=1}^k
            \Wr[k,a] \Ur[a],
            \sum_{a=1}^{m}
            \Wr[m,a] \Ur[a]
        \rt\ra
        =
        \sum_{a=1}^k
        \Wr[k,a] \Wr[m, a],
    \]
    which are equal for $r=1,2$ because $k\le m-2$.
    Lemma~\ref{lem:general-ub-stronger-couplings} now implies that
    \begin{align*}
        &\E_{(\Xonevl, \Xtwovl)\sim \cNvdipr}
        \lt[
            \la \Xonei, \Xonej \ra
            \la \Xtwoi, \Xtwoj \ra
        \rt] \\
        &\qquad \le
        \E_{(\Xonevl, \Xtwovl)\sim \cMvd(\cCgs[i,j])}
        \lt[
            \la \Xonei, \Xonej \ra
            \la \Xtwoi, \Xtwoj \ra
        \rt].
    \end{align*}
    In the rest of this proof, let $\Xonevl, \Xtwovl$ and $\Wone, \Wtwo$ be as in the coupling $\cMvd(\cCgs[i,j])$.
    Prior to conditioning on the constraints $\cCgs[i,j]$, the following random variables are mutually independent for $r=1,2$.
    \begin{itemize}
        \item The vectors $\Xrk$ for $k\in [v-1]\setminus V$;
        \item The matrix $\Ur$; and
        \item The entries of $\Wr$.
    \end{itemize}
    Therefore after conditioning on $\cCgs[i,j]$, the following random variables are mutually independent.
    \begin{itemize}
        \item $\Wone[k,k] \sim \sqrt{\chisq(d+1-k)}$ for all $1\le k\le m$;
        \item $\Wone[k,\ell] \sim \cN(0,1)$ for all $1\le \ell < k\le m$; and
        \item $\Wtwo[m, m-1] \sim \cN(0,1)$.
    \end{itemize}
    From this, we now have that
    \begin{align*}
        &\E\lt[
            \la \Xonei, \Xonej \ra
            \la \Xtwoi, \Xtwoj \ra
        \rt] 
        =
        \E \lt[
            \lt(\sum_{a=1}^{m-1} \Wone[m-1, a] \Wone[m, a]\rt)
            \lt(\sum_{a=1}^{m-1} \Wtwo[m-1, a] \Wtwo[m, a]\rt)
        \rt] \\
        &\qquad =
        \E \lt[
            \lt(
                \Wone[m-1, m-1] \Wone[m, m-1] +
                \sum_{a=1}^{m-2} \Wone[m-1, a] \Wone[m, a]
            \rt)
            \lt(
                \Wone[m-1, m-1] \Wtwo[m, m-1] +
                \sum_{a=1}^{m-2} \Wone[m-1, a] \Wone[m, a]
            \rt)
        \rt] \\
        &\qquad =
        \sum_{a=1}^{m-2}
        \E \lt[
            \lt(\Wone[m-1, a] \Wone[m-2, a]\rt)^2
        \rt]
        = m-2 = \dvli.
    \end{align*}
    The second last equality eliminates all terms in the expansion with expectation zero.
    This completes the proof of the proposition.
\end{proof}
\begin{remark}
    If we had also coupled $\Wone[m, m-1]=\Wtwo[m, m-1]$, we recover the na\"ive coupling $\Xonevl = \Xtwovl$. 
    Then, in the last computation we get an extra summand of $\E \lt[\lt(\Wone[m-1, m-1]\Wtwo[m-2,m-1]\rt)^2\rt] = d-m+2$, which recovers the bound of $d$ from Proposition~\ref{prop:general-ub-linear-term-ijedge}, as expected.
    Thus, the independence of $\Wone[m, m-1]$ and $\Wtwo[m, m-1]$ is crucial to the improved bound of Proposition~\ref{prop:general-ub-linear-term-ijnotedge}.
\end{remark}

The linear terms of (\ref{eq:general-ub-main-exp-overlap-starting-point}) in category (\ref{itm:general-ub-main-overlap-term-ii}) are handled similarly using a variant on the coupling $\cCgs[i,j]$ in the next proposition.
This time, instead of leaving the $(m, m-1)$ entry of the Gram-Schmidt orthogonalization free, we will leave the $(m,m)$ entry free.

\begin{proposition}
    \label{prop:general-ub-linear-term-ii}
    If $i\in [v-1]$ and $d \ge \dvli+1$, then we have that
    \[
        \E_{(\Xonevl, \Xtwovl)\sim \cNvdipr}
        \lt[
            \lt(\la \Xonei, \Xonei \ra - d\rt)
            \lt(\la \Xtwoi, \Xtwoi \ra - d\rt)
        \rt]
        \le
        2\dvli.
    \]
\end{proposition}
\begin{proof}
    Let $V = \Nl(i) \cup \{i\} $ and let $m = |V| = \dvli+1$.
    Fix some bijection $\pi : [m] \to V$ such that $\pi(m)=i$.
    Now consider the coupling $\cMvd(\cCgs[i])$ arising from the collection of constraints $\cCgs[i]$ given as follows:
    \begin{align*}
        \Xonek &= \Xtwok
        \quad
        \text{if $k\in [v-1] \setminus V$,} \\
        \Uone &= \Utwo, \\
        \Wone[k,\ell] &= \Wtwo[k,\ell]
        \quad
        \text{for all $1\le \ell \le k \le m$ such that $(k,\ell) \neq (m, m)$} \\
        &\text{where $\lt(\Wr, \Ur\rt) = \GS\lt(\Xr[\pi(1)], \ldots, \Xr[\pi(m)]\rt)$ for $r=1,2$.}
    \end{align*}
    Note that this Gram-Schmidt procedure is well-defined because $d\ge \dvli + 1 = m$.
    The same argument as in the previous proposition shows that these conditions ensure $\la \Xonek, \Xonel\ra = \la \Xtwok, \Xtwol \ra$ for all $(k,\ell)\in E(G[v-1])$.
    Let $\Xonevl, \Xtwovl$ and $\Wone, \Wtwo$ be as in the coupling $\cMvd(\cCgs[i])$.
    By the same argument as in the previous proposition, the following random variables are mutually independent.
    \begin{itemize}
        \item $\Wone[k,k]\sim \sqrt{\chisq(d+1-k)}$ for all $1\le k \le m$;
        \item $\Wone[k,\ell] \sim \cN(0,1)$ for all $1\le \ell < k \le m$; and
        \item $\Wtwo[m,m] \sim \sqrt{\chisq(d+1-m)}$.
    \end{itemize}
    We now have that
    \begin{align*}
        &\E_{(\Xonevl, \Xtwovl) \sim \cMvd(\cCgs[i])}
        \lt[
            \lt(\la \Xonei, \Xonei \ra - d\rt)
            \lt(\la \Xtwoi, \Xtwoi \ra - d\rt)
        \rt] \\
        &\qquad =
        \E \lt[
            \lt(\sum_{a=1}^{m} \lt(\Wone[m, a]\rt)^2 - d\rt)
            \lt(\sum_{a=1}^{m} \lt(\Wtwo[m, a]\rt)^2 - d\rt)
        \rt] \\
        &\qquad =
        \E \Bigg[
            \lt(
                \lt(\lt(\Wone[m, m]\rt)^2 - (d+1-m)\rt) +
                \sum_{a=1}^{m-1} \lt(\lt(\Wone[m, a]\rt)^2-1\rt)
            \rt) \\
        &\qquad \qquad \times
            \lt(
                \lt(\lt(\Wtwo[m, m]\rt)^2 - (d+1-m)\rt) +
                \sum_{a=1}^{m-1} \lt(\lt(\Wone[m, a]\rt)^2-1\rt)
            \rt)
        \Bigg] \\
        &\qquad= \sum_{a=1}^{m-1}
        \E \lt[
            \lt(\lt(\Wone[m, a]\rt)^2-1\rt)^2
        \rt]
        = 2(m-1)
    \end{align*}
    where the second last equality eliminates all terms in the expansion with expectation zero.
    The last equality follows from the fact that $\E_{\xi \sim \cN(0,1)}[(\xi^2-1)^2] = 2$.
    Applying Lemma~\ref{lem:general-ub-stronger-couplings} completes the proof of this proposition.
\end{proof}

Combining all of these propositions now completes the proof of Lemma~\ref{lem:general-ub-main-exp-overlap}.

\begin{proof}[Proof of Lemma~\ref{lem:general-ub-main-exp-overlap}]
    Because $d \gg \max_{v\in [n]} (\ddegv^2 + \ddegv \log^2 n)$, Propositions~\ref{prop:general-ub-bound-hy}, \ref{prop:general-ub-linear-term-unconditioning}, \ref{prop:general-ub-linear-term-ijedge}, \ref{prop:general-ub-linear-term-ijnotedge}, and \ref{prop:general-ub-linear-term-ii} all hold for sufficiently large $n$.
    Substituting the bounds from these propositions into (\ref{eq:general-ub-main-exp-overlap-starting-point}) yields that
    \begin{align*}
        &\E_{(\Xonevl, \Xtwovl) \sim \cNvdips}
        \exp\lt(
            \f12 Y_v
        \rt) \\
        &\qquad \le
        1 +
        \f{1}{2d^2}
        \lt[
            \sum_{\substack{
                i,j\in \downNv, i\neq j \\
                (i,j)\in E(G[v-1])
            }}
            d
            +
            \sum_{\substack{
                i,j\in \downNv, i\neq j \\
                (i,j)\not\in E(G[v-1])
            }}
            \min\lt(\dvli, \dvlj\rt)
            +
            \sum_{i\in \downNv}
            2 \dvli
        \rt] \\
        &\qquad \qquad +
        \f{(\Ctr + 2)^2}{2d^2}
        \lt(\ddegv^4 + \ddegv^2 \log^4 n\rt)
        +
        24\ddegv^2 d^{-1} n^{-10}\,.
    \end{align*}
    Note that
    \[
        \sum_{\substack{
            i,j\in \downNv, i\neq j \\
            (i,j)\in E(G[v-1])
        }}
        d
        =
        2d \cdot |E(G[\downNv])|\,,
    \]
    and that
    \begin{align*}
        &\sum_{\substack{
            i,j\in \downNv, i\neq j \\
            (i,j)\not\in E(G[v-1])
        }}
        \min\lt(\dvli, \dvlj\rt) +
        \sum_{i\in \downNv}
        2\dvli \\
        &\qquad \le
        \sum_{i,j\in \downNv, i\neq j}
        (\dvli + \dvlj)
        +
        \sum_{i\in \downNv}
        2\dvli \\
        &\qquad =
        \sum_{i,j\in \downNv}
        (\dvli + \dvlj)
        = 2\ddegv \sum_{i\in \downNv} \dvli\,.
    \end{align*}
    Moreover, because $d \gg \max_{v\in [n]} (\ddegv^2 + \ddegv \log^2 n)$, for sufficiently large $n$ we have $d \ge 24\degv^2 \ge 24\ddegv^2$ for all $v\in [n]$.
    When this occurs, we have $24\ddegv^2 d^{-1} n^{-10} \le n^{-10}$.
    Combining these bounds proves the lemma.
\end{proof}

The final ingredient in the proof of Theorem~\ref{thm:general-weaker-ub} is the following lemma, which parses the hypothesis (\ref{eq:general-weaker-ub-hypothesis-4claws}), stated in terms of the subgraph counts of $G$, into one involving the degrees of $G$ that is compatible with the rest of the proof's parts.
Recall that $\degv$ is the degree of a vertex $v$, not restricted to neighbors in $[v-1]$.
\begin{lemma}
    \label{lem:general-weaker-ub-hypothesis-translation}
    Suppose the hypothesis (\ref{eq:general-weaker-ub-hypothesis-4claws}) of Theorem~\ref{thm:general-weaker-ub} holds.
    Then,
    \begin{equation}
        \label{eq:general-weaker-ub-hypothesis-4claws-translated}
        d^2 \gg
        \sum_{v\in G}
        \lt(
            \degv^4 +
            \degv^2 \log^4 n
        \rt).
    \end{equation}
\end{lemma}
\begin{proof}
    For all nonnegative integers $a$, we have $a^2 = 2\binom{a}{2} + a$ and $a^4 \le 4^4 \binom{a}{4} + 3^3a$.
    So,
    \begin{align*}
        \sum_{v\in G}
        \lt(
            \degv^4 +
            \degv^2 \log^4 n
        \rt)
        &\le
        \sum_{v\in G} \lt(
            4^4 \binom{\degv}{4} +
            3^3\degv +
            \lt(
                2\binom{\degv}{2} +
                \degv
            \rt)
            \log^4 n
        \rt) \\
        &=
        4^4 \Num_G(K_{1,4}) +
        2 \Num_G(P_2) \log^4 n +
        2 \Num_G(E) (\log^4 n + 3^3).
    \end{align*}
    So, (\ref{eq:general-weaker-ub-hypothesis-4claws}) implies the result.
\end{proof}

We are now ready to prove Theorem~\ref{thm:general-weaker-ub}.
The final outstanding task is to verify that $S^v$ occurs with high probability, as in Proposition~\ref{prop:general-ub-s-hp}.
This is carried out in Section~\ref{subsec:general-ub-s-hp-proof}.

\begin{proof}[Proof of Theorem~\ref{thm:general-weaker-ub}]
    For now, fix a vertex $v\in [n]$.
    By Proposition~\ref{prop:general-ub-s-hp}, $\P(\Svop), \P(\Svtr) \ge 1 - n^{-20}$, so $\P(\Svop \cap \Svtr) \ge 1 - 2n^{-20}$.
    By Lemma~\ref{lem:general-ub-kl-conditioning} with $T = \Svop \cap \Svtr$, we have that
    \[
        \E_{W \sim \muvl}
        \KL\lt((\muv | W) \parallel \nuv\rt)
        \le
        \E_{W \sim \muvlsv}
        \KL\lt((\muv | W) \parallel \nuv\rt)
        + 11n^{-9}
    \]
    for sufficiently large $n$.

    By Lemma~\ref{lem:general-weaker-ub-hypothesis-translation}, (\ref{eq:general-weaker-ub-hypothesis-4claws-translated}) holds.
    Therefore, $d \gg \max_{v\in [n]} \lt(\degv^2 + \degv \log^2 n\rt)$.
    Since $\degv \ge \ddegv$ for all $v$, the hypotheses of Lemmas~\ref{lem:general-ub-main-2mm} and \ref{lem:general-ub-main-exp-overlap} both hold.
    By these lemmas, we have that for sufficiently large $n$,
    \begin{align*}
        \E_{W \sim \muvl}
        \KL\lt((\muv | W) \parallel \nuv\rt)
        &\le
        11n^{-9}-1 + \exp\lt(\f{100^4}{d^2} (\ddegv^3 + \ddegv \log^2 n)\rt) \\
        &\qquad \times
        \Bigg[
            1 + \f{|E(G[\downNv])|}{d} +
            \f{\ddegv}{d^2}
            \sum_{i\in \downNv}
            \dvli \\
            &\qquad \qquad +
            \f{(\Ctr + 2)^2}{2d^2}
            \lt(\ddegv^4 + \ddegv^2 \log^4 n\rt) +
            n^{-10}
        \Bigg].
    \end{align*}
    Because $d \gg \max_{v\in [n]} \lt(\degv^2 + \degv \log^2 n\rt)$, the argument of the exponential in this bound is $o(1)$ and the quantity inside square brackets is $1+o(1)$ (note that $|E(G[\downNv])| \le \ddegv^2$).
    For $x,y \in [0, \f12]$, we have that $\exp(x)(1+y) \le 1+2x+2y$.
    So, for all sufficiently large $n$ we have
    \begin{align*}
        \E_{W \sim \muvl}
        \KL\lt((\muv | W) \parallel \nuv\rt)
        &\le
        \f{2|E(G[\downNv])}{d} +
        \f{2\ddegv}{d^2}
        \sum_{i\in \downNv} \dvli \\
        &\qquad +
        \f{(\Ctr + 2)^2 + 2\cdot 100^4}{d^2}
        \lt(\ddegv^4 + \ddegv^2 \log^4 n\rt) +
        13n^{-9},
    \end{align*}
    where we used that $2n^{-10} \le 2n^{-9}$.
    Substituting this into (\ref{eq:general-ub-starting-point}), we now have that
    \begin{align*}
        2\TV(W(G,d), M(G))^2
        &\le
        \sum_{v=1}^n
        \E_{W\sim \muvl}
        \KL\lt((\muv|W) \parallel \nuv\rt) \\
        &\le
        \f{2}{d}
        \sum_{v=1}^d
        |E(G[\downNv])| +
        \f{2}{d^2}
        \sum_{v=1}^d
        \ddegv
        \sum_{i\in \downNv}
        \dvli \\
        &\qquad +
        \f{(\Ctr+2)^2 + 2\cdot 100^4}{d^2}
        \sum_{v=1}^n
        \lt(\ddegv^4 + \ddegv^2 \log^4 n\rt) +
        13n^{-8}.
    \end{align*}
    Each edge in some $G[\downNv]$ creates a unique 3-cycle with largest vertex $v$, so we have
    \[
        \sum_{v=1}^d |E(G[\downNv])| = \Num_G(C_3).
    \]
    Recall that $N(v)$ and $\degv$ are the neighbor set and degree of $v$, not restricted to $[v-1]$.
    By AM-GM,
    \begin{align*}
        \sum_{v=1}^d
        \ddegv
        \sum_{i\in \downNv}
        \deg(i)
        &\le
        \sum_{v=1}^d
        \sum_{u\in N(v)}
        \degu\degv
        =
        \sum_{(u,v)\in E(G)}
        2\degu\degv \\
        &\le
        \sum_{(u,v)\in E(G)}
        \lt(\degu^2 + \degv^2\rt)
        \le
        2
        \sum_{v\in G}
        \degv^3
        \le
        2
        \sum_{v\in G}
        \degv^4.
    \end{align*}
    Finally, we have $\ddegv \le \degv$ for all $v$.
    So,
    \[
        2\TV(W(G,d), M(G))^2 \le
        \f{2}{d} \Num_G(C_3) +
        \f{(\Ctr + 2)^2 + 2 \cdot 100^4 + 4}{d^2}
        \lt(\degv^4 + \degv^2 \log n\rt)
        + 13n^{-8}.
    \]
    The bounds (\ref{eq:general-weaker-ub-hypothesis-triangles}) and (\ref{eq:general-weaker-ub-hypothesis-4claws-translated}) imply that this upper bound is $o(1)$.
    This completes the proof of the theorem.
\end{proof}

\subsection{High Probability Bounds on $S$ and Determinant Bounds}
\label{subsec:general-ub-s-hp-proof}

In this section, we show that $S^v$, defined in (\ref{eq:general-ub-def-sv}), holds with high probability, proving Proposition~\ref{prop:general-ub-s-hp}.
We first will establish that $\Svop$, defined in (\ref{eq:general-ub-def-svop}), occurs with high probability.
This is implied by the following standard bound on the singular values of a Gaussian matrix.
A proof of this bound can be found in \cite{Ver10} and follows from Gordon's Theorem and Gaussian concentration.

\begin{lemma}{\cite[Corollary 5.35]{Ver10}}
    \label{lem:gaussian-singular-values-concentration}
    Suppose $X\in \bR^{d\times k}$ has i.i.d. standard Gaussian entries.
    Let $s_{\max}(X)$ and $s_{\min}(X)$ be the largest and smallest singular values of $X$.
    For all $t\ge 0$, we have that
    \[
        \sqrt{d} - \sqrt{k} - t
        \le
        s_{\min}(X)
        \le
        s_{\max}(X)
        \le
        \sqrt{d} + \sqrt{k} + t
    \]
    with probability at least $1 - 2\exp(-t^2/2)$.
\end{lemma}

With this bound, we can deduce the following proposition controlling the probability of $\Svop$.

\begin{proposition}
    \label{prop:general-ub-svop-high-prob}
    Suppose that $d \gg \max_{v\in [n]} \ddegv + \log n$.
    For all $v\in [n]$ and all sufficiently large $n$, we have that $\P(\Svop) \ge 1 - n^{-20}$.
\end{proposition}
\begin{proof}
    Let $n$ be large enough that $d > 100^2 \max_{v\in [n]} \lt(\ddegv + \log n\rt)$.
    Applying Lemma~\ref{lem:gaussian-singular-values-concentration} to $\Xbn$ with $t = 32\sqrt{\log n}$ implies that with probability at least $1 - 2n^{-512} \ge 1 - n^{-20}$,
    \begin{align*}
        \spec\lt(\Xbnt \Xbn\rt)
        &\subseteq
        \lt[
            \lt(\sqrt{d} - \sqrt{\ddegv} - 32\sqrt{\log n}\rt)^2,
            \lt(\sqrt{d} + \sqrt{\ddegv} + 32\sqrt{\log n}\rt)^2
        \rt] \\
        &\subseteq
        \lt[
            \lt(\sqrt{d} - 33 \sqrt{\ddegv + \log n}\rt)^2,
            \lt(\sqrt{d} + 33 \sqrt{\ddegv + \log n}\rt)^2
        \rt] \\
        &=
        [d(1-\eps)^2, d(1+\eps)^2]
    \end{align*}
    for $\eps = 33\sqrt{\f{\ddegv + \log n}{d}} < 1$.
    Suppose this event occurs.
    Since $\Delbn = d^{-1} \Xbnt \Xbn - \Ibn$, we have
    \begin{align*}
        \spec(\Delbn)
        &\subseteq
        \lt[(1-\eps)^2-1, (1+\eps)^2-1\rt]
        \subseteq
        [-3\eps, 3\eps] \\
        &\subseteq
        \lt[
            -100 \sqrt{\f{\ddegv + \log n}{d}},
            100 \sqrt{\f{\ddegv + \log n}{d}}
        \rt].
    \end{align*}
    This implies that $\norm{\Delbn}_{\op} \le 100 \sqrt{\f{\ddegv + \log n}{d}}$.
\end{proof}

We will now show that $\Svtr$, as defined in (\ref{eq:general-ub-def-svtr}), occurs with high probability.
For this, we will need the following standard lemma on Gaussian hypercontractivity.
The lemma follows from, for example, Theorem 9.23 in \cite{Odo14}, after modeling each Gaussian input as a normalized sum of i.i.d. Rademacher variables and applying the Central Limit Theorem.

\begin{lemma}
    \label{lem:gaussian-hypercontractivity-tails}
    Let $f$ be a polynomial of degree at most $k$ in i.i.d. standard Gaussian inputs, and let $\sigma^2 = \E f^2$.
    There exist positive constants $c_k$ and $C_k$, dependent only on $k$, such that
    \[
        \P[|f| > t\sigma]
        \le
        C_k \exp\lt(-c_k t^{2/k}\rt).
    \]
\end{lemma}

With this lemma, we can deduce the following proposition controlling the probability of $\Svtr$.

\begin{proposition}
    \label{prop:general-ub-svtr-high-prob}
    Suppose that $d \gg \max_{v\in [n]} \ddegv$ and $\Ctr > 0$ is a sufficiently large constant.
    For all $v\in [n]$ and sufficiently large $n$, we have that $\P(\Svtr) \ge 1 - n^{-20}$.
\end{proposition}
\begin{proof}
    Since $d \gg \max_{v\in [n]} \ddegv$, for sufficiently large $n$ we have $d \ge \max_{v\in [n]} \ddegv$.
    When this occurs, Lemma~\ref{lem:general-ub-computational}(\ref{itm:general-ub-computational-vartrdelsq}) applies.
    By Lemma~\ref{lem:general-ub-computational}(\ref{itm:general-ub-computational-trdelsq},\ref{itm:general-ub-computational-vartrdelsq}), we have
    \[
        \E \Tr\lt(\Delbnsq\rt)
        = 
        \f{\ddegv^2 + \ddegv}{d}
        \le 
        \f{2\ddegv^2}{d},
        \qquad
        \Var \Tr\lt(\Delbnsq\rt)
        \le 
        \f{56\ddegv^2}{d^2}.
    \]
    Consider applying Lemma~\ref{lem:gaussian-hypercontractivity-tails} to $f = \Tr\lt(\Delbnsq\rt) - \E \Tr\lt(\Delbnsq\rt)$, which is a degree 4 polynomial in the entries of the matrix $\Xbn$.
    Let $C_4, c_4$ be the constants in Lemma~\ref{lem:gaussian-hypercontractivity-tails}.
    It follows that
    \begin{align*}
        \P\lt[
            \Tr\lt(\Delbnsq\rt) >
            \f{2\ddegv^2}{d} +
            \Ctr \cdot \f{\ddegv \log^2 n}{d}
        \rt]
        &\le
        \P \lt[
            |f| > \Ctr \cdot \f{\ddegv \log^2 n}{d}
        \rt] \\
        &\le C_4 \exp\lt(-\f{c_4 \Ctr}{\sqrt{56}} (\log^2 n)^{1/2}\rt)
        \le n^{-20}
    \end{align*}
    for a sufficiently large choice of the constant $\Ctr > 0$.
    This completes the proof of the proposition.
\end{proof}

Finally, we show that $\Svdet$ occurs with high probability.
Recall that $\Svdet$ is defined in (\ref{eq:general-ub-def-svdet}).
\begin{proposition}
    \label{prop:general-ub-svdet-high-prob}
    If $d \gg \max_{v\in [n]} \ddegv$, then for all  sufficiently large $n$ and all $v\in [n]$ we have that $\P(\Svdet) \ge 1 - e^{-n/2}$.
\end{proposition}
\begin{proof}
    Take $n$ large enough that $d \ge 2 \max_{v\in [n]} \ddegv + 2$.
    By Lemma~\ref{lem:general-ub-computational}(\ref{itm:general-ub-computational-det}) and Cauchy-Schwarz, we have that
    \[
        \E_{\Xvl}
        \det\lt(d^{-1} \Xbnt \Xbn\rt)^{-1/2}
        \le
        \E_{\Xvl}
        \lt[\det\lt(d^{-1} \Xbnt \Xbn\rt)^{-1}\rt]^{1/2}
        \le
        e^{\ddegv/2}.
    \]
    So, by the definition of $\gamma(W)$,
    \begin{align*}
        \E_{\Xvl}
        f_{\det}(\Xvl)
        &=
        \E_{\Xvl}
        \E_{\Xpvl \sim \gamma(W(\Xvl))}
        \lt[\det\lt(d^{-1} \Xpbnt \Xpbn\rt)^{-1/2}\rt] \\
        &=
        \E_{\Xvl}
        \lt[\det\lt(d^{-1} \Xbnt \Xbn\rt)^{-1/2}\rt]
        \le
        e^{\ddegv/2}
        \le
        e^{n/2}.
    \end{align*}
    By Markov's inequality, $f_{\det}(\Xvl) > e^n$ occurs with probability at most $e^{-n/2}$.
\end{proof}

Proposition~\ref{prop:general-ub-s-hp} follows readily from these propositions.

\begin{proof}[Proof of Proposition~\ref{prop:general-ub-s-hp}]
    The desired bounds follow from Propositions~\ref{prop:general-ub-svop-high-prob}, \ref{prop:general-ub-svtr-high-prob}, and \ref{prop:general-ub-svdet-high-prob}.
\end{proof}

\section{Sharp Bounds on the Linear Terms of the Coupled Exponential Overlap}
\label{sec:general-ub-refined-linear-term}

In this and the next section, we strengthen the argument used to prove Theorem~\ref{thm:general-weaker-ub} in order to show Theorem~\ref{thm:general-ub}.
There are three parts of the proof of Theorem~\ref{thm:general-weaker-ub} that we will need to improve.
\begin{enumerate}[label=(\arabic*), ref=\arabic*]
    \item The upper bound on the first order terms of (\ref{eq:general-ub-main-exp-overlap-starting-point}) in category (\ref{itm:general-ub-main-overlap-term-ijnotedge}) shown in Proposition~\ref{prop:general-ub-linear-term-ijnotedge} currently depends on the size of the smaller of the neighborhoods of $i$ and $j$.
    This upper bound can be strengthened to depend only on the size of the intersection of these neighborhoods, which naturally yields the tradeoff between 3-cycles and 4-cycles in Theorem~\ref{thm:general-ub}.
    \item The multiplicative factor of $\exp\lt(O\lt(d^{-2} (\ddegv^3 + \ddegv \log^2 n)\rt)\rt)$ in Lemma~\ref{lem:general-ub-main-2mm} is too large to yield Theorem~\ref{thm:general-ub}.
    This can be tightened by using a stronger variant of Lemma~\ref{lem:det-to-exp-to-deg2} replacing the first order approximation of the determinant with a second order approximation.
    With this improvement, the third order and higher terms of the determinant are bounded deterministically by the set $\Sop$.
    We no longer bound the second order terms deterministically, and they instead give rise to a 4-cycles variant of the coupled exponentiated overlap that will need to be bounded.
    \item The bound on the higher order terms of the coupled exponentiated overlap in Proposition~\ref{prop:general-ub-bound-hy} is too crude to yield Theorem~\ref{thm:general-ub}.
    We will need to handle the second and third order terms in $h(Y_v)$ with more care to obtain tight enough bounds to prove this theorem.
\end{enumerate}

Among these three improvements, the first is the most challenging and will require new ideas.
The second and third improvements follow from optimizations of the ideas in Lemma~\ref{lem:general-ub-main-2mm}, Proposition~\ref{prop:general-ub-bound-hy} and Proposition~\ref{prop:general-ub-linear-term-ijnotedge}.
These improvements and the proof of Theorem~\ref{thm:general-ub} will be carried out in Section~\ref{sec:general-ub-refined-higher-order-terms}.
The goal of this section will be to prove the following lemma, which carries out the first improvement and strengthens the estimate from Proposition~\ref{prop:general-ub-linear-term-ijnotedge}.

\begin{lemma}[Sharp First Order Term Bounds]
    \label{lem:general-ub-refined-linear-term}
    Let $v\in [n]$ and let $k,\ell \in [v-1]$ such that $k\neq \ell$ and $(k,\ell)\not\in E(G[v-1])$.
    Let $\dvlkl$ denote the number of common neighbors of $k$ and $\ell$ in $G[v-1]$.
    Fix a constant $\eps \in (0,1)$ and suppose that $d^{1-\eps} \ge \max\lt(\dvlk, \dvll\rt)$.
    Then,
    \[
        \E_{(\Xonevl, \Xtwovl)\sim \cNvdipr}
        \lt[
            \la \Xonek, \Xonel \ra
            \la \Xtwok, \Xtwol \ra
        \rt]
        \le
        C_\eps\lt[
            \dvlkl + \f{\dvlk\dvll}{d}
        \rt]
        + d^{-10}
    \]
    for sufficiently large $d$, where $C_\eps > 0$ is a constant depending only on $\eps$.
\end{lemma}
\begin{remark}
    Recall the discussion from Section~\ref{subsec:technical-overview-beyond-gs}: the coupling used to prove Lemma~\ref{lem:general-ub-refined-linear-term} can be understood as drawing all edges among $[v-1]\setminus \{k,\ell\}$.
    After drawing these edges, the numbers of 2-paths and 3-paths from $k$ to $\ell$ are, respectively, $\dvlkl$ and $\dvlk\dvll$.
    So, this lemma can be understood to mean that each 2-path contributes an expectation of order $1$, each 3-path contributes an expectation of order $1/d$, and longer paths are dominated by 3-paths.
\end{remark}

\subsection{Reducing to an Expectation over Independent Gaussians}

Throughout this section, fix $k,\ell \in [v-1]$ such that $k\neq \ell$ and $(k,\ell)\not\in E(G[v-1])$.
Like in the proof of Proposition~\ref{prop:general-ub-linear-term-ijnotedge}, our first step is to apply Lemma~\ref{lem:general-ub-stronger-couplings} and replace $\cNvdipr$ with a coupling that is easier to work with.
While the Gram-Schmidt coupling $\cMvd(\cCgs[k,\ell])$ led to an upper bound with minimal computation, it is too strong of a coupling to capture the actual dependence of the desired upper bound on $\dvlkl$ instead of $\min\lt(\dvlk, \dvll\rt)$.

In this section, we will instead work with a weaker coupling that leads to a tighter but more computationally involved upper bound.
This coupling of $\Xonevl$ and $\Xtwovl$ will be denoted by $\cMvd(\cCtckl)$.
It similarly is generated as in Definition~\ref{defn:general-ub-modified-couplings} with the collection of constraints $\cCtckl$ given as follows:
\begin{align*}
    &\Xonei = \Xtwoi
    \quad
    \text{for all $i\in [v-1]\setminus \{k,\ell\}$,} \\
    &\la \Xonei, \Xonej \ra
    =
    \la \Xtwoi, \Xtwoj \ra
    \quad
    \text{for $i\in \{k,\ell\}$ and $j\in \Nl(i)$.}
\end{align*}
Since these conditions imply that $\la \Xonei, \Xonej \ra = \la \Xtwoi, \Xtwoj\ra$ for all $i,j\in E(G[v-1])$, Lemma~\ref{lem:general-ub-stronger-couplings} yields that
\begin{align*}
    &\E_{(\Xonevl, \Xtwovl) \sim \cNvdipr}
    \lt[
        \la \Xonek, \Xonel \ra
        \la \Xtwok, \Xtwol \ra
    \rt] \\
    &\qquad \le
    \E_{(\Xonevl, \Xtwovl) \sim \cMvd(\cCtckl)}
    \lt[
        \la \Xonek, \Xonel \ra
        \la \Xtwok, \Xtwol \ra
    \rt].
\end{align*}
To reduce notation, we will write this upper bound as $\E \lt[\la X_k, X_\ell \ra \la W_k, W_\ell \ra\rt]$ where $X_1,\ldots,X_{v-1}$ are i.i.d. samples from $\cN(0, I_d)$ and $W_k, W_\ell$ are i.i.d. samples from $\cN(0, I_d)$ conditioned on
\[
    \la W_i, X_j \ra = \la X_i, X_j\ra
    \quad
    \text{for all $i\in \{k,\ell\}$ and $j\in \Nl(i)$.}
\]
Note that if $\dvlk=0$, then $W_k$ is independent of $X_1,\ldots,X_{v-1}$ and $W_\ell$. Then, we have $\E \lt[\la X_k, X_\ell \ra \la W_k, W_\ell \ra\rt] = 0$.
The same is true if $\dvll = 0$.
So, we assume throughout the rest of this section that $\dvlk, \dvll > 0$.
Now consider the conditional distribution of $(W_k,W_\ell)$ given $X_1,\ldots,X_{v-1}$.
Define the matrices $\Mk, \Ml$ as follows.
\begin{itemize}
    \item $\Mk \in \bR^{d\times \dvlk}$ has columns $X_i$ for each $i\in \Nl(k)$, ordered and indexed by the vertex $i$.
    \item $\Ml \in \bR^{d\times \dvll}$ is defined analogously for the vertex $\ell$.
\end{itemize}
Note that $\rank(\Mk) = \dvlk$ and $\rank(\Ml) = \dvll$ almost surely.
Conditioned on the vectors $X_1,\ldots,X_{v-1}$, the two random vectors $W_k,W_\ell$ are distributed as independent samples from $\cN(0, I_d)$ conditioned on the events $\Mkt W_k = \Mkt X_k$ and $\Mlt W_\ell = \Mlt X_\ell$, respectively.
Standard conditioning properties of Gaussians imply that conditioned on $X_1,\ldots,X_{v-1}$, $W_k$ and $W_\ell$ are independent and are distributed as singular Gaussians given by
\[
    W_k \sim \cN\lt(\Mkproj X_k, R_k R_k^\top\rt)
    \quad
    \text{and}
    \quad
    W_\ell \sim \cN\lt(\Mlproj X_l, R_\ell R_\ell^\top\rt).
\]
Here, the columns of $R_k \in \bR^{d\times (d - \dvlk)}$ and $R_\ell \in \bR^{d\times (d - \dvll)}$ are chosen to form orthogonal bases of the nullspaces of $\Mk$ and $\Ml$, respectively.
The matrix inverses $\Mkinv$ and $\Mlinv$ are well-defined almost surely because $\Mk$ and $\Ml$ are full rank almost surely.
Therefore,
\[
    \E \lt[
        \la W_k, W_\ell\ra
        \Big|
        X_1,\ldots,X_{v-1}
    \rt]
    =
    X_k^{\top} \Mkproj \Mlproj X_\ell.
\]
Note that, because $(k,\ell)\not\in E(G[v-1])$, $\Mk$ and $\Ml$ are each independent of both $X_k$ and $X_\ell$, which are i.i.d. samples from $\cN(0, I_d)$.
So,
\begin{align*}
    \E \lt[
        \la X_k,X_\ell\ra
        \la W_k,W_\ell\ra
    \rt]
    &=
    \E \lt[
        X_\ell^\top X_k \cdot
        X_k^{\top} \Mkproj \Mlproj X_\ell
    \rt] \\
    &=
    \E \Tr \lt(
            X_k X_k^{\top} \Mkproj \Mlproj X_\ell X_\ell^\top
    \rt) \\
    &= \E \Tr \lt(\Mkproj \Mlproj\rt)
\end{align*}
by taking expectations over $X_k$ and $X_\ell$.
Summarizing the results in this section, we now have that
\begin{align}
    \notag
    &\E_{(\Xonevl,\Xtwovl)\sim \cNvdipr}
    \lt[
        \la \Xonek, \Xonel \ra
        \la \Xtwok, \Xtwol \ra
    \rt] \\
    \label{eq:general-ub-refined-linear-term-starting-point}
    &\qquad \le
    \E \Tr \lt(\Mkproj \Mlproj\rt).
\end{align}
Crucially, the last expectation is only over a collection of i.i.d. Gaussians.
The next two sections are devoted to estimating this last expectation.

\subsection{Moment Approximations to Wishart Inverses}

The key idea in this and the next section will be to use approximations of the Wishart inverses $\Mkinv$ and $\Mlinv$ amenable to approximating the desired expectation (\ref{eq:general-ub-refined-linear-term-starting-point}).
Our argument will be motivated by the following simple fact.
Suppose $A$ is a real symmetric matrix that is a contraction in the sense that all of its eigenvalues are in $(-1, 1)$.
Then,
\[
    (I+A)^{-1} = \sum_{t=0}^\infty (-1)^t A^t,
\]
and furthermore this series converges absolutely in spectral norm.
Observe that $\Mkt\Mk$ is distributed according to the law of an $m\times m$ Wishart matrix with $d$ degrees of freedom, where $m = \dvlk$.
Standard results on the spectra of Wishart matrices imply that the spectral norm of $d^{-1} \Mkt \Mk - I_{m}$ is approximately $\sqrt{m/d}$ and therefore that this matrix is a contraction with overwhelming probability.
This motivates us to define the finite series
\[
    \MkT =
    d^{-1}
    \sum_{t=0}^T
    (-1)^t
    \lt(d^{-1} \Mkt\Mk - I_{\Nl(k)}\rt)^t
\]
for each positive integer $T$, which approximates $\Mkinv$.
Define $\MlT$ analogously.
In this section, we will estimate (\ref{eq:general-ub-refined-linear-term-starting-point}) with $\Mkinv$ and $\Mlinv$ replaced by $\MkT$ and $\MlT$, respectively.
In the next section, we will make the above heuristic argument formal, showing that both the event that $d^{-1} \Mkt\Mk - I_{\Nl(k)}$ is not a contraction and the error term from a $T$th order approximation have a negligible effect on this estimate.
The main result of this subsection is the following lemma.

\begin{lemma}
    \label{lem:general-ub-refined-linear-term-orderT}
    Suppose that $d\ge \max\lt(\dvlk, \dvll\rt)$.
    Then, for all positive integers $T$, there is a constant $C_T > 0$, dependent only on $T$, such that
    \[
        \E \Tr \lt(\MkprojT \MlprojT\rt)
        \le
        C_T \lt[
            \dvlkl + \f{\dvlk\dvll}{d}
        \rt].
    \]
\end{lemma}
Before proving this lemma, we will first need a computational proposition bounding the expected values of the terms that arise from expanding this trace.
For a sequence $(i_1,\ldots,i_r)$ of elements of $[v-1]$, possibly with repeated elements, define the cyclic product
\begin{equation}
    \label{eq:general-ub-refined-linear-term-cyclic-product}
    \Pi\lt(\Xvl, (i_1,\ldots,i_r)\rt)
    =
    \prod_{t=1}^r
    \la X_{i_t}, X_{i_{t+1}} \ra,
\end{equation}
where we cyclicly define $i_{r+1} = i_1$.

\begin{proposition}
    \label{prop:general-ub-refined-linear-term-computational}
    Let $(i_1,\ldots,i_r)$ be a sequence of elements of $[v-1]$.
    There is a constant $C^{(1)}_r > 0$, dependent only on $r$, such that
    \[
        0 \le
        \E \Pi\lt(\Xvl, (i_1,\ldots,i_r)\rt)
        \le
        C^{(1)}_r d^{r+1 - |I|},
    \]
    where $I = \{i_1,\ldots,i_r\}$ is the set of distinct elements in $(i_1,\ldots,i_r)$.
\end{proposition}
\begin{proof}
    Let $\sL$ be the set of all functions $L: [r] \to [d]$.
    The desired expectation can be expanded as
    \[
        \E \Pi\lt(\Xvl, (i_1,\ldots,i_r)\rt)
        =
        \sum_{L\in \sL}
        \E \lt[\prod_{t=1}^r X_{i_t, L(t)} X_{i_{t+1}, L(t)}\rt],
    \]
    where for each $t\in [r]$, $L(t)$ encodes the selection of the term $X_{i_t, L(t)} X_{i_{t+1}, L(t)}$ from the inner product $\la X_{i_t}, X_{i_{t+1}}\ra$.
    The random variables $X_{i,j}$ are i.i.d. samples from $\cN(0,1)$.
    So, an expectation in this sum is nonzero if and only if it contains each $X_{i,j}$ an even number of times.
    Note that each nonzero term is the product of moments of i.i.d. standard Gaussians, where the order of each moment is at most $2r$ and the total number of moments in the product is at most $r$.
    Therefore there is a constant $C^{(2)}_r > 0$ such that each nonzero term is bounded above by $C^{(2)}_r$.
    Since each term is the product of moments of independent standard Gaussians, each term in the sum is nonnegative.

    It suffices to bound the number of $L\in \sL$ corresponding to nonzero terms.
    For each vertex $u\in I$, let $m_u$ be the number of times $u$ appears in $(i_1,\ldots,i_r)$.
    Consider one $u\in I$, and let those times be $i_t$ for $t \in \{t_1, \ldots, t_{m_u}\}$, where $t_1 < \cdots < t_{m_u}$.
    In each term $\prod_{t=1}^r X_{i_t, L(t)}X_{i_{t+1}, L(t)}$, there are exactly $2m_u$ factors containing $X_{u,s}$ for some $s$: $m_u$ each of the forms $X_{i_t, L(t)}$ and $X_{i_t, L(t-1)}$ for $t \in \{t_1, \ldots, t_{m_u}\}$.
    Since each $X_{u,s}$ that appears must appear at least twice for the expectation to be nonzero, the set
    \begin{equation}
        \label{eq:general-ub-refined-linear-term-computational-labels-of-u}
        \{L(t_j) : 1\le j\le m_u\}
        \cup
        \{L(t_j-1) : 1\le j\le m_u\}
    \end{equation}
    has at most $m_u$ distinct values.
    In particular, if $i_t$ appears only once in $(i_1,\ldots,i_r)$, then $L(t-1) = L(t)$.
    Now consider the following procedure for generating a function $L$.
    \begin{enumerate}[label=(\arabic*)]
        \item Choose an initial value $L_1\in [d]$.
        \item For each $u\in I$, do the following.
        \begin{enumerate}[label=(\alph*)]
            \item Choose a subset $A_u\subseteq [d]$ of size $m_u-1$.
            \item Choose an $m_u$-tuple $B_u = \lt(b^{(1)}_u,\ldots,b^{(m_u)}_u\rt)$ of values from $A_u\cup \{*\}$ where $*$ is a special element.
        \end{enumerate}
        \item Set $L(1) = L_1$.
        \item For $t=2,3,\ldots,r$ in that order, do the following.
        \begin{enumerate}[label=(\alph*)]
            \item Let $u = i_t$, and let $i_t$ be the $j$th occurrence of $u$ in $(i_1,\ldots,i_r)$.
            \item If $b^{(j)}_u \neq *$, set $L(t) = b^{(j)}_u$.
            If $b^{(j)}_u = *$, set $L(t) = L(t^*(u)-1)$ where $t^*(u)$ is the smallest $t^* \ge 2$ (possibly equal to $t$) such that $u = i_{t^*}$.
        \end{enumerate}
    \end{enumerate}
    Note that the choices of $L_1$, $A_u$ and $B_u$, for $u\in I$, encode any $L\in \cL$ corresponding to a nonzero term.
    This is because, in addition to the value $L(t^*(u)-1)$, the set (\ref{eq:general-ub-refined-linear-term-computational-labels-of-u}) can contain at most $m_u-1$ additional values.

    Therefore, the number of $L$ corresponding to nonzero terms is upper bounded by the number of valid choices for $L_1$, $A_u$ and $B_u$, for $u\in I$.
    There are $d$ options for $L_1$.
    The total number of choices for all of the subsets $A_u$ is at most $d^a$ where $a = \sum_{u\in I} (m_u-1) = r - |I|$.
    Given the subset $A_u$, the number of ways to choose $B_u$ is bounded by $m_u^{m_u}\le r^r$.
    Therefore given all of the $A_u$, the number of ways to choose all of the $B_u$ is at most $C^{(3)}_r = r^{r^2}$.
    Thus the number of nonzero terms is bounded above by $C^{(3)}_r d^{r + 1 - |I|}$, which completes the proof of the lemma on setting $C^{(1)}_r = C^{(2)}_r C^{(3)}_r$.
\end{proof}

We remark that the procedure in the proof above can generate $L$ that correspond to terms with expectation zero and generate the same $L$ multiple times, and hence overcounts the number of nonzero terms.
With this lemma, we now can complete the proof of Lemma~\ref{lem:general-ub-refined-linear-term-orderT}.
\begin{proof}[Proof of Lemma~\ref{lem:general-ub-refined-linear-term-orderT}]
    First note that
    \begin{align}
        \notag
        \MkT
        &=
        d^{-1} \sum_{t=0}^T (-1)^t \lt(d^{-1}\Mkt \Mk - I_{\Nl(k)}\rt)^t \\
        \notag
        &=
        d^{-1} \sum_{t=0}^T \sum_{r=0}^t (-1)^r d^{-r} \binom{t}{r} \lt(\Mkt\Mk\rt)^r \\
        \notag
        &=
        d^{-1} \sum_{r=0}^T (-1)^r \binom{T+1}{r+1} \cdot d^{-r} \lt(\Mkt\Mk\rt)^r \\
        \label{eq:general-ub-refined-linear-term-mkt-expansion}
        &=
        \sum_{r=1}^{T+1} (-1)^{r-1} \binom{T+1}{r} \cdot d^{-r} \lt(\Mkt\Mk\rt)^{r-1}
    \end{align}
    since $\sum_{t=r}^T \binom{t}{r} = \binom{T+1}{r+1}$.
    The quantity $\MlT$ admits a similar expansion.
    Therefore,
    \begin{align*}
        &\E \Tr \lt(\MkprojT \MlprojT\rt) \\
        &\qquad =
        \sum_{r=1}^{T+1} \sum_{s=1}^{T+1}
        (-1)^{r+s} d^{-r-s}
        \binom{T+1}{r} \binom{T+1}{s}
        \E \Tr\lt(
            \lt(\Mk\Mkt\rt)^r
            \lt(\Ml\Mlt\rt)^s
        \rt) \\
        &\qquad \le
        2^{2T+2}
        \sum_{r=1}^{T+1} \sum_{s=1}^{T+1}
        d^{-r-s}
        \lt|
            \E \Tr\lt(
                \lt(\Mk\Mkt\rt)^r
                \lt(\Ml\Mlt\rt)^s
            \rt)
        \rt|
    \end{align*}
    by the triangle inequality, since $\binom{T+1}{r} \le 2^{T+1}$ for all $r\le T+1$.
    Directly expanding the trace yields that
    \begin{align*}
        \E \Tr\lt(
            \lt(\Mk\Mkt\rt)^r
            \lt(\Ml\Mlt\rt)^s
        \rt)
        &=
        \sum_{\substack{
            (i_1,\ldots,i_r) \in \Nl(k)^r \\
            (j_1,\ldots,j_s) \in \Nl(\ell)^s
        }}
        \E \Pi\lt(\Xvl, (i_1,\ldots,i_r, j_1,\ldots,j_s) \rt) \\
        &\le
        C^{(4)}_T
        \sum_{\substack{
            (i_1,\ldots,i_r) \in \Nl(k)^r \\
            (j_1,\ldots,j_s) \in \Nl(\ell)^s
        }}
        d^{r+s+1 - |\{i_1,\ldots,i_r,j_1,\ldots,j_s\}|}
    \end{align*}
    by Proposition~\ref{prop:general-ub-refined-linear-term-computational}, where $C^{(4)}_T = \max_{1\le r\le 2T+2} C^{(1)}_r$.
    Furthermore, it follows that this expectation is nonnegative for each pair $r,s$.
    Combining these inequalities yields that
    \begin{align*}
        &\E \Tr \lt(\MkprojT \MlprojT\rt) \\
        &\qquad \le
        2^{2T+2} C^{(4)}_T d
        \sum_{r=1}^{T+1} \sum_{s=1}^{T+1}
        \sum_{\substack{
            (i_1,\ldots,i_r) \in \Nl(k)^r \\
            (j_1,\ldots,j_s) \in \Nl(\ell)^s
        }}
        d^{-|\{i_1,\ldots,i_r,j_1,\ldots,j_s\}|}.
    \end{align*}
    Now, define the sets
    \[
        V_1 = \Nl(k) \cap \Nl(\ell),
        \quad
        V_2 = \Nl(k) \setminus V_1,
        \quad
        \text{and}
        \quad
        V_3 = \Nl(\ell) \setminus V_1,
    \]
    and set $\Lambda = \{i_1,\ldots,i_r,j_1,\ldots,j_s\}$.
    For each triple $(a,b,c)$ of integers with $0\le a,b,c\le T+1$, we will upper bound the number of terms in the above sum with $a = |V_1 \cap \Lambda|$, $b = |V_2 \cap \Lambda|$, and $c = |V_3 \cap \Lambda|$.
    First note that given the intersections $V_1 \cap \Lambda$, $V_2\cap \Lambda$, and $V_3 \cap \Lambda$, the number of ways to form the sequences $(i_1,\ldots,i_r)$ and $(j_1,\ldots,j_s)$ for some pairs of lengths $r,s\le T+1$ is upper bounded by a constant $C^{(5)}_T > 0$.
    Furthermore, the number of choices for these three intersections is at most $|V_1|^a |V_2|^b |V_3|^c$.
    Thus the number of terms corresponding to the triple $(a,b,c)$ is at most $C^{(5)}_T |V_1|^a |V_2|^b |V_3|^c \le C^{(5)}_T\dvlkl^a \dvlk^b \dvll^c$.

    Since $V_1$, $V_2$, and $V_3$ partition $\Nl(k) \cup \Nl(\ell)$, which contains $\Lambda$, it follows that any term corresponding to the triple $(a,b,c)$ satisfies that $|\Lambda| = a+b+c$.
    Since $r,s \ge 1$, the set $\Lambda$ must have nonempty intersection with both $\Nl(k)$ and $\Nl(\ell)$.
    Therefore, either $a\ge 1$ or $b,c\ge 1$.
    Combining all of these observations yields that
    \begin{align*}
        & \E \Tr \lt(\MkprojT \MlprojT\rt) \\
        &\qquad \le
        2^{2T+2} C^{(4)}_T C^{(5)}_T
        d
        \sum_{\substack{
            0\le a,b,c\le T+1 \\
            \text{$a\ge 1$ or $b,c\ge 1$}
        }}
        \lt(\f{\dvlkl}{d}\rt)^a
        \lt(\f{\dvlk}{d}\rt)^b
        \lt(\f{\dvll}{d}\rt)^c \\
        &\qquad \le
        2^{2T+2} (T+1)^3 C^{(4)}_T C^{(5)}_T
        d
        \lt[
            \lt(\f{\dvlkl}{d}\rt) +
            \lt(\f{\dvlk}{d}\rt)
            \lt(\f{\dvll}{d}\rt)
        \rt] \\
        &\qquad =
        C_T \lt[
            \dvlkl +
            \f{\dvlk\dvll}{d}
        \rt],
    \end{align*}
    where $C_T = 2^{2T+2} (T+1)^3 C^{(4)}_T C^{(5)}_T$.
    Here, the final inequality follows from the condition $d\ge \max\lt(\dvlk, \dvll\rt)$ and the fact that either $a\ge 1$ or $b,c\ge 1$ for each term in the sum.
\end{proof}

\subsection{Bounding the Error Term in the Moment Approximation}

In this section, we will bound $\E Z$, where $Z$ is the error term in the moment approximation given in the previous section.
More precisely, let
\begin{align*}
    Z
    &=
    \Tr\lt(\Mkproj \Mlproj\rt)
    - \Tr\lt(\MkprojT \MlprojT\rt) \\
    &=
    \Tr\lt(\Mk \lt[\Mkinv - \MkT\rt] \Mkt \Mlproj \rt) \\
    &\qquad +
    \Tr\lt(\MkprojT \Ml \lt[\Mlinv - \MlT\rt] \Mlt \rt)
\end{align*}
As in the statement of Lemma~\ref{lem:general-ub-refined-linear-term}, let $\eps \in (0,1)$ be such that $d^{1-\eps} \ge \max\lt(\dvlk, \dvll\rt)$.
Let $E_Z$ denote the event that
\[
    \sqrt{d} - 2d^{(1-\eps)/2} \le s_{\min}(\Mk) \le s_{\max}(\Mk) \le \sqrt{d} + 2d^{(1-\eps)/2}
\]
and
\[
    \sqrt{d} - 2d^{(1-\eps)/2} \le s_{\min}(\Ml) \le s_{\max}(\Ml) \le \sqrt{d} + 2d^{(1-\eps)/2}.
\]
Here, $s_{\max}(M) = \norm{M}_{\op}$ and $s_{\min}(M)$ denote the maximum and minimum singular values of a matrix $M$.
We now will prove three simple propositions that will be used to complete the proof of Lemma~\ref{lem:general-ub-refined-linear-term}.
The first bounds $Z$ on the event $E_Z$.

\begin{proposition}
    \label{prop:general-ub-refined-linear-term-z-clamp-bound}
    If $d$ is at least a sufficiently large constant $C^{(6)}_{\eps} > 0$ dependent only on $\eps$, then on the event $E_Z$, it holds that $|Z|\le 2^{3T+10} d^{1 - \f{(T+1)\eps}{2}}$ almost surely.
\end{proposition}
\begin{proof}
    First note that the largest and smallest eigenvalues of $\Mkt\Mk$ are by definition $s_{\max}(\Mk)^2$ and $s_{\min}(\Mk)^2$, respectively.
    On the event $E_Z$, the eigenvalues of $d^{-1} \Mkt\Mk - I_{\Nl(k)}$ all lie in the interval
    \begin{align*}
        \lt[d^{-1}s_{\min}(\Mk)^2-1, d^{-1}s_{\max}(\Mk)^2-1\rt]
        &\subseteq
        \lt[(1 - 2d^{-\eps/2})^2-1, (1 + 2d^{-\eps/2})^2-1\rt] \\
        &\subseteq
        \lt[-8d^{-\eps/2}, 8d^{-\eps/2}\rt].
    \end{align*}
    Since $\Mkt\Mk$, $\Mkinv$, and all powers of $d^{-1}\Mkt\Mk - I_{\Nl(k)}$ have a common eigenbasis, the eigenvalues of $d^{-1}\Mkt\Mk - I_{\Nl(k)}$ and $\Mkinv-\MkT$ are in one-to-one correspondence, with each eigenvalue $\lambda$ of $d^{-1}\Mkt\Mk - I_{\Nl(k)}$ corresponding to an eigenvalue
    \[
        \f{1}{d} \lt(\f{1}{1 + \lambda} - \sum_{t=0}^T (-1)^t \lambda^t\rt) = \f{(-\lambda)^{T+1}}{d(1+\lambda)}
    \]
    of $\Mkinv-\MkT$.
    If $d$ is larger than a constant depending only on $\eps$, then $|\lambda| \le 8d^{-\eps/2}$ implies that
    \[
        \lt|\f{(-\lambda)^{T+1}}{d(1+\lambda)}\rt|
        \le
        \f{2}{d} |\lambda|^{T+1}
        \le
        2\cdot 8^{T+1} d^{-1 - \f{(T+1)\eps}{2}},
    \]
    and therefore
    \[
        \norm{\Mkinv-\MkT}_{\op} \le 2\cdot 8^{T+1} d^{-1 - \f{(T+1)\eps}{2}}.
    \]
    Similarly, the eigenvalues of $\MkT$ are given by $\f{1 - (-\lambda)^{T+1}}{d(1+\lambda)}$ and it follows that $\norm{\MkT}_{\op} \le \f{2}{d}$ for sufficiently large $d$.
    Moreover, we have $\norm{\Mk}_{\op} = s_{\max}(\Mk) \le 2\sqrt{d}$ and $\norm{\Mkinv}_{\op} = s_{\min}(\Mk)^{-2} \le \f{2}{d}$ for sufficiently large $d$.
    All of these inequalities hold symmetrically for $\Ml$.
    On $E_Z$, we now have that
    \begin{align*}
        |Z|
        &\le
        d
        \norm{\Mk \lt[\Mkinv - \MkT\rt] \Mkt \Mlproj}_{\op} \\
        &\qquad +
        d
        \norm{\MkprojT \Ml \lt[\Mlinv - \MlT\rt] \Mlt}_{\op} \\
        &\le
        d
        \norm{\Mk}_{\op} \norm{\Mkinv - \MkT}_{\op} \norm{\Mkt}_{\op} \norm{\Ml}_{\op} \norm{\Mlinv}_{\op} \norm{\Mlt}_{\op} \\
        &\qquad +
        d
        \norm{\Mk}_{\op} \norm{\MkT}_{\op} \norm{\Mkt}_{\op} \norm{\Ml}_{\op} \norm{\Mlinv - \MlT}_{\op} \norm{\Mlt}_{\op} \\
        &\le 2^{3T+10} d^{1 - \f{(T+1)\eps}{2}},
    \end{align*}
    which proves the proposition.
\end{proof}

The next proposition bounds the probability of the complement of $E_Z$ and follows from standard concentration bounds on the singular values of a Gaussian matrix.

\begin{proposition}
    \label{prop:general-ub-refined-linear-term-z-error-prob}
    If $d^{1-\eps} \ge \max\lt(\dvlk, \dvll\rt)$, it holds that $\P(E_Z^c) \le 4\exp\lt(-\f12 d^{1-\eps}\rt)$.
\end{proposition}
\begin{proof}
    Since $d^{1-\eps} \ge \dvlk$, we have $d^{(1-\eps)/2} + \sqrt{\dvlk} \le 2d^{(1-\eps)/2}$.
    So, Lemma~\ref{lem:gaussian-singular-values-concentration} with $t = d^{(1-\eps)/2}$ yields that
    \[
        \sqrt{d} - 2d^{(1-\eps)/2} \le s_{\min}(\Mk) \le s_{\max}(\Mk) \le \sqrt{d} + 2d^{(1-\eps)/2}
    \]
    holds with probability at least $1 - 2\exp\lt(-\f12 d^{1-\eps}\rt)$.
    The same bound holds for $\Ml$.
    A union bound over these two events completes the proof of the lemma.
\end{proof}

The last proposition establishes a crude upper bound on the second moment of $Z$.

\begin{proposition}
    \label{prop:general-ub-refined-linear-term-z-second-moment}
    If $d \ge \max\lt(\dvlk, \dvll\rt)$, then there is a constant $C^{(7)}_T >0$, dependent only on $T$, such that $\E[Z^2] \le C^{(7)}_T d^{8T+8}$.
\end{proposition}
\begin{proof}
    By AM-GM,
    \begin{align*}
        Z^2 
        &\le
        2 \Tr\lt(\Mkproj \Mlproj\rt)^2 \\
        &\qquad + 2 \Tr\lt(\MkprojT \MlprojT\rt)^2.
    \end{align*}
    Since $\lt(\Mkproj\rt)^2 = \Mkproj$, it follows that $\Mkproj$ is a projection matrix, and thus $\norm{\Mkproj}_{\op}\le 1$.
    Similarly, we have $\norm{\Mlproj}_{\op}\le 1$.
    Therefore,
    \begin{align*}
        \lt|\Tr \lt(\Mkproj \Mlproj\rt)\rt|
        &\le
        d \norm{\Mkproj \Mlproj}_{\op} \\
        &\le d
    \end{align*}
    almost surely.
    It follows that
    \[
        \E 2 \Tr\lt(\Mkproj \Mlproj\rt)^2 \le 2d^2.
    \]
    To bound the second term in the upper bound for $Z^2$, we will again use the expansion (\ref{eq:general-ub-refined-linear-term-mkt-expansion}).
    This yields
    \begin{align*}
        &\Tr\lt(\MkprojT \MlprojT\rt) \\
        &\qquad =
        \sum_{r=1}^{T+1}
        \sum_{s=1}^{T+1}
        (-1)^{r+s}
        \binom{T+1}{r} \binom{T+1}{s}
        \Tr\lt(\lt(\Mk\Mkt\rt)^r \lt(\Ml\Mlt\rt)^s\rt) \\
        &\qquad =
        \sum_{r=1}^{T+1}
        \sum_{s=1}^{T+1}
        \sum_{\substack{
            (i_1,\ldots,i_r) \in \Nl(k)^r \\
            (j_1,\ldots,j_s) \in \Nl(\ell)^s
        }}
        (-1)^{r+s}
        \binom{T+1}{r} \binom{T+1}{s}
        \Pi\lt(\Xvl, (i_1,\ldots,i_r,j_1,\ldots,j_s)\rt),
    \end{align*}
    where we recall the definition of the cyclic product $\Pi$ in (\ref{eq:general-ub-refined-linear-term-cyclic-product}).
    There are at most
    \[
        \lt(\sum_{r=1}^{T+1} \dvlk^r\rt)
        \lt(\sum_{s=1}^{T+1} \dvll^s\rt)
        \le
        (T+1)^2 \dvlk^{T+1} \dvll^{T+1}
        \le
        (T+1)^2 d^{2T+2}
    \]
    summands in this sum, where we use that $d \ge \max\lt(\dvlk, \dvll\rt)$.
    By the inequality
    \[
        (x_1+\cdots+x_N)^2 \le N(x_1^2+\cdots+x_N^2)
    \]
    for all $x_1,\ldots,x_N\in \bR$, which follows from Cauchy-Schwarz, we have
    \begin{align*}
        &\E \Tr\lt(\MkprojT \MlprojT\rt)^2 \\
        &\qquad \le
        (T+1)^2 d^{2T+2}
        \sum_{r=1}^{T+1}
        \sum_{s=1}^{T+1}
        \sum_{\substack{
            (i_1,\ldots,i_r) \in \Nl(k)^r \\
            (j_1,\ldots,j_s) \in \Nl(\ell)^s
        }}
        \binom{T+1}{r}^2 \binom{T+1}{s}^2
        \E \Pi\lt(\Xvl, (i_1,\ldots,i_r,j_1,\ldots,j_s)\rt)^2 \\
        &\qquad \le
        2^{4T+4} (T+1)^2 d^{2T+2}
        \sum_{r=1}^{T+1}
        \sum_{s=1}^{T+1}
        \sum_{\substack{
            (i_1,\ldots,i_r) \in \Nl(k)^r \\
            (j_1,\ldots,j_s) \in \Nl(\ell)^s
        }}
        \E \Pi\lt(\Xvl, (i_1,\ldots,i_r,j_1,\ldots,j_s)\rt)^2
    \end{align*}
    Finally, note that
    \[
        \Pi\lt(\Xvl, (i_1,\ldots,i_r,j_1,\ldots,j_s)\rt)^2
        =
        \Pi\lt(\Xvl, (i_1,\ldots,i_r,j_1,\ldots,j_s,i_1,\ldots,i_r,j_1,\ldots,j_s)\rt),
    \]
    so Proposition~\ref{prop:general-ub-refined-linear-term-computational} implies that
    \begin{align*}
        \E \Pi\lt(\Xvl, (i_1,\ldots,i_r,j_1,\ldots,j_s)\rt)^2
        &\le
        C^{(1)}_{2r+2s} d^{2r+2s+1 - |\{i_1,\ldots,i_r,j_1,\ldots,j_s\}|} \\
        &\le
        C^{(1)}_{2r+2s} d^{2r+2s}
        \le
        C^{(8)}_{T} d^{4T+4},
    \end{align*}
    where $C^{(8)}_T = \max_{1\le r\le 4T+4} C^{(1)}_r$.
    Putting this all together, we have
    \begin{align*}
        &\E 2 \Tr\lt(\MkprojT \MlprojT\rt)^2 \\
        &\qquad \le
        2 \cdot  2^{4T+4} (T+1)^2 d^{2T+2}
        \cdot
        (T+1)^2 d^{2T+2} \cdot C^{(8)}_{T} d^{4T+4} 
        =
        2^{4T+5} (T+1)^4 C^{(8)}_{T} d^{8T+8}.
    \end{align*}
    So,
    \[
        \E [Z^2]
        \le 
        2^{4T+5} (T+1)^4 C^{(8)}_{T} d^{8T+8} + 2d^2
        \le 
        C^{(7)}_T
        d^{8T+8}
    \]
    for $C^{(7)}_T = 2^{4T+5} (T+1)^4 C^{(8)}_{T} + 2$.
\end{proof}

With these propositions, we can now complete the proof of Lemma~\ref{lem:general-ub-refined-linear-term}.
\begin{proof}[Proof of Lemma~\ref{lem:general-ub-refined-linear-term}]
    Let $T$ be a positive integer to be determined later.
    By Propositions~\ref{prop:general-ub-refined-linear-term-z-clamp-bound}, \ref{prop:general-ub-refined-linear-term-z-error-prob}, and \ref{prop:general-ub-refined-linear-term-z-second-moment}, the triangle inequality and Cauchy-Schwarz, we have that for all $d$ larger than a constant dependent only on $\eps$,
    \begin{align*}
        \E Z
        &\le
        \E[|Z| \ind{E_Z}] +
        \E[|Z| \ind{E_Z^c}] 
        \le
        \E[|Z| \ind{E_Z}] +
        \E[\ind{E_Z^c}]^{1/2} \E[Z^2]^{1/2} \\
        &\le
        2^{3T+10} d^{1 - \f{(T+1)\eps}{2}} +
        2\sqrt{C^{(7)}_T} \exp\lt(-\f14 d^{1-\eps}\rt)d^{4T+4}.
    \end{align*}
    By (\ref{eq:general-ub-refined-linear-term-starting-point}) and Lemma~\ref{lem:general-ub-refined-linear-term-orderT}, we have that
    \begin{align*}
        &\E_{(\Xonevl, \Xtwovl) \sim \cNvdipr}
        \lt[
            \la \Xonek, \Xonel \ra
            \la \Xtwok, \Xtwol \ra
        \rt] \\
        &\qquad \le
        \E \Tr \lt(\Mkproj \Mlproj\rt) \\
        &\qquad =
        \E \Tr \lt(\MkprojT \MlprojT\rt) + \E Z \\
        &\qquad \le
        C_T \lt[
            \dvlkl + \f{\dvlk\dvll}{d}
        \rt] +
        2^{3T+10} d^{1 - \f{(T+1)\eps}{2}} +
        2\sqrt{C^{(7)}_T} \exp\lt(-\f14 d^{1-\eps}\rt)d^{4T+4}.
    \end{align*}
    Now, take $T = \lceil 23/\eps\rceil$, which is constant because $\eps$ is a fixed constant.
    For sufficiently large $d$, the last two terms sum to at most $d^{-10}$.
    This completes the proof of the lemma.
\end{proof}

\section{Refined TV Upper Bound Argument and Bounds on Higher Order Terms}
\label{sec:general-ub-refined-higher-order-terms}

In this section, we strengthen the argument used to prove Theorem~\ref{thm:general-weaker-ub} and complete the proof of Theorem~\ref{thm:general-ub}.
Recall that for $r=1,2$, we defined
\[
    \Delrbn = d^{-1} \lt(\Xrbn\rt)^\top \Xrbn - \Ibn.
\]
In the first part of this section, we tighten the analysis in Lemma~\ref{lem:general-ub-main-2mm} by expanding the determinant $\det\lt(\Ibn - \Delonebn \Deltwobn\rt)^{-1/2}$ to Taylor order two instead of one.
The second order term of this determinant, which we no longer bound deterministically with the event $\Sop$, gives rise to a 4-cycles variant of the coupled exponentiated overlap.
The remaining task is to bound two coupled exponentiated overlaps: the original overlap $\E \exp(Y_v)$ (now without the $\f12$) and the new 4-cycles variant, both over $(\Xonevl, \Xtwovl) \sim \cNvdips$.
In the second part of this section, we show, like in Proposition~\ref{prop:general-ub-bound-hy}, that it suffices to consider low order terms of these two exponentials.
To obtain the sharper bounds of Theorem~\ref{thm:general-ub}, we will now need tailored arguments for terms of the coupled exponentiated overlap up to third order, and the 4-cycles coupled exponentiated overlap up to first order.
In the third part of this section, we give Gram-Schmidt couplings similar to those in Proposition~\ref{prop:general-ub-linear-term-ijnotedge} to bound these low order terms.
In the final part, we combine these results with Lemma~\ref{lem:general-ub-refined-linear-term}, which controls the first order terms, to complete the proof of Theorem~\ref{thm:general-ub}.

\subsection{A Tighter Determinant Expansion}

The main result of this section is a sharper variant of Lemma~\ref{lem:general-ub-main-2mm} that bounds the KL divergence of interest in terms of the coupled exponentiated overlap and a 4-cycles variant of it.
Throughout this section, define $\Delr[i,j] = d^{-1}\la \Xri, \Xrj \ra - \delta_{i,j}$ for all $i,j\in [v-1]$ and $r=1,2$.
Recall the coupled overlap $Y_v$ defined in (\ref{eq:general-ub-define-yv}); we can equivalently write
\begin{equation}
    \label{eq:general-ub-redefine-yv}
    Y_v =
    \Tr\lt(\Delonebn \Deltwobn\rt)
    =
    \sum_{i,j\in \downNv}
    \Delone[i,j]
    \Deltwo[i,j].
\end{equation}
We further define the 4-cycles variant of the coupled overlap $Z_v$ by
\begin{equation}
    \label{eq:general-ub-define-zv}
    Z_v =
    \Tr\lt(\lt(\Delonebn \Deltwobn\rt)^2\rt)
    =
    \sum_{i,j,k,\ell \in \downNv}
    \Delone[i,j]
    \Deltwo[j,k]
    \Delone[k,\ell]
    \Deltwo[\ell,i].
\end{equation}
The following lemma is our refinement of Lemma~\ref{lem:general-ub-main-2mm}.
\begin{lemma}[Refined Second Moment Method]
    \label{lem:general-ub-refined-2mm}
    Suppose that $d \gg \max_{v\in [n]} \ddegv + \log n$.
    Fix some $v\in [n]$ and suppose that $n$ is sufficiently large.
    Then, we have that
    \begin{align*}
        \E_{W\sim \muvls}
        \KL\lt((\muvs | W) \parallel \nuv\rt)
        &\le
        -1 + \exp\lt(
            \f{2 \cdot 100^6}{d^2}
            (\ddegv^4 + \ddegv \log^3 n)
        \rt) \\
        & \quad \times
        \lt[
            \E_{(\Xonevl, \Xtwovl) \sim \cNvdips}
            \exp\lt(Y_v\rt)
        \rt]^{1/2}
        \lt[
            \E_{(\Xonevl, \Xtwovl) \sim \cNvdips}
            \exp\lt(\f12 Z_v\rt)
        \rt]^{1/2}.
    \end{align*}
\end{lemma}

To prove this, we will need the following variant of Lemma~\ref{lem:det-to-exp-to-deg2}, where we expand the determinant to Taylor order two.

\begin{lemma}
    \label{lem:det-to-exp-to-deg3}
    There exists an absolute constant $\eps > 0$ such that if $\Pi \in \bR^{k\times k}$ (and $\Pi$ is not necessarily symmetric) and $\norm{\Pi}_{\op} \le \eps$, then
    \[
        \det(I_k - \Pi)
        \ge
        \etr(-\Pi)
        \etr\lt(-\f12 \Pi^2\rt)
        \exp\lt(-\sum_{\lambda \in \spec(\Pi)} |\lambda|^3\rt).
    \]
\end{lemma}
\begin{proof}
    This lemma follows from the same argument as in Lemma~\ref{lem:det-to-exp-to-deg2}, with the inequalities
    \begin{align*}
        \log(1-\lambda)
        &=
        -\lambda - \f12 \lambda^2 - \f13 \lambda^3 - O(|\lambda|^4)
        \ge
        -\lambda - \f12 \lambda^2 - |\lambda|^3, \\
        \log\lt((1-\lambda)(1-\blambda)\rt)
        &=
        \log \lt(1 - \lambda - \blambda + |\lambda|^2\rt) \\
        &=
        - (\lambda + \blambda - |\lambda|^2)
        - \f12 (\lambda + \blambda - |\lambda|^2)^2
        - \f13 (\lambda + \blambda - |\lambda|^2)^3
        - O(|\lambda|^4) \\
        &\ge
        - \lt(\lambda + \f12 \lambda^2 + |\lambda|^3\rt)
        - \lt(\blambda + \f12 \blambda^2 + |\blambda|^3\rt),
    \end{align*}
    which hold for all $|\lambda|\le \eps$ for a sufficiently small choice of $\eps > 0$.
\end{proof}

We now prove Lemma~\ref{lem:general-ub-refined-2mm}.

\begin{proof}[Proof of Lemma~\ref{lem:general-ub-refined-2mm}]
    By the same argument as in Lemma~\ref{lem:general-ub-main-2mm}, we have that
    \begin{align*}
        \E_{W\sim \muvls}
        \KL\lt(\lt(\muvs|W\rt) \parallel \nuv\rt)
        &\le
        -1 +
        \E_{W\sim \muvls}
        \E_{\substack{
            \Xonevl \sim \gamma(W) \\
            \Xonevl \in S
        }}
        \E_{\substack{
            \Xtwovl \sim \gamma(W) \\
            \Xtwovl \in S
        }}
        \det\lt(I_{\Nl} - \Delonebn \Deltwobn\rt)^{-1/2} \\
        &=
        -1 +
        \E_{(\Xonevl, \Xtwovl) \sim \cNvdips}
        \det\lt(I_{\Nl} - \Delonebn \Deltwobn\rt)^{-1/2}.
    \end{align*}
    As in Lemma~\ref{lem:general-ub-main-2mm}, if $\Xonevl, \Xtwovl \in S \subseteq \Sop$ it holds that $\norm{\Delonebn \Deltwobn}_{\op} \le \f{100^2}{d} \lt(\ddegv + \log n\rt)$, where $\Delrbn = d^{-1} \lt(\Xrbn\rt)^\top \Xrbn - \Ibn$ for $r=1,2$.
    So, when $\Xonevl, \Xtwovl \in S$ we have that
    \begin{align*}
        \sum_{\lambda \in \spec\lt(\Delonebn \Deltwobn\rt)} |\lambda|^3
        &\le
        \ddegv \norm{\Delonebn \Deltwobn}_{\op}^3
        \le
        \f{100^6}{d^3} \ddegv \lt(\ddegv + \log n\rt)^3 \\
        &\le
        \f{4\cdot 100^6}{d^3} \lt(\ddegv^4 + \ddegv \log^3 n\rt).
    \end{align*}
    The last estimate uses Jensen's inequality, in the form $\lt(\f{x+y}{2}\rt)^3 \le \f{x^3 + y^3}{2}$ for $x,y \ge 0$.
    Since $d \gg \max_{v\in [n]} \ddegv + \log n$, for the $\eps$ in Lemma~\ref{lem:det-to-exp-to-deg3} we have $\norm{\Delonebn \Deltwobn}_{\op} \le \eps$ for all sufficiently large $n$.
    When this occurs, by Lemma~\ref{lem:det-to-exp-to-deg3} we have that
    \begin{align*}
        &
        \E_{(\Xonevl, \Xtwovl) \sim \cNvdips}
        \det\lt(I_{\Nl} - \Delonebn \Deltwobn\rt)^{-1/2} \\
        &\qquad \le
        \exp\lt(
            \f{2\cdot 100^6}{d^3}
            \lt(\ddegv^4 + \ddegv \log^3 n\rt)
        \rt) \\
        &\qquad \qquad \times
        \E_{(\Xonevl, \Xtwovl) \sim \cNvdips}
        \etr\lt(\f12 \Delonebn \Deltwobn\rt)
        \etr\lt(\f14 \lt(\Delonebn \Deltwobn\rt)^2 \rt) \\
        &\qquad \le
        \exp\lt(
            \f{2\cdot 100^6}{d^3}
            \lt(\ddegv^4 + \ddegv \log^3 n\rt)
        \rt)
        \lt[
            \E_{(\Xonevl, \Xtwovl) \sim \cNvdips}
            \etr\lt(\Delonebn \Deltwobn\rt)
        \rt]^{1/2} \\
        &\qquad \qquad \times
        \lt[
            \E_{(\Xonevl, \Xtwovl) \sim \cNvdips}
            \etr\lt(\f12 \lt(\Delonebn \Deltwobn\rt)^2\rt)
        \rt]^{1/2},
    \end{align*}
    where the last inequality is by Cauchy-Schwarz.
    The lemma now follows from the definitions (\ref{eq:general-ub-redefine-yv}) and (\ref{eq:general-ub-define-zv}) of $Y_v$ and $Z_v$.
\end{proof}

\subsection{Refined Bounds on the Coupled Exponentiated Overlaps}

The main result in this and the next section is the following lemma, which is a sharper analogue of Lemma~\ref{lem:general-ub-main-exp-overlap} that bounds the expectations $\E \exp(Y_v)$ and $\E \exp(\f12 Z_v)$ over $(\Xonevl, \Xtwovl) \sim \cNvdips$.

\begin{lemma}
    \label{lem:general-ub-refined-coupled-exp-overlap-ubs}
    Suppose that $d \gg \max_{v\in [n]} \lt(\ddegv^2 + \ddegv \log^2 n\rt)$.
    There exists a constant $C>0$ such that for all sufficiently large $n$,
    \begin{align*}
        &\max\lt(
            \E_{(\Xonevl, \Xtwovl) \sim \cNvdips} \exp(Y_v),
            \E_{(\Xonevl, \Xtwovl) \sim \cNvdips} \exp\lt(\f12 Z_v\rt)
        \rt) \\
        &\qquad \le
        1 +
        C \Bigg[
            d^{-1}|E(G[\downNv])| +
            d^{-2} \sum_{i,j\in \downNv, i\neq j} \dvlij +
            d^{-2} \sum_{i\in \downNv} \dvli +
            d^{-2} \ddegv^2 
            \\
            &\qquad \qquad +
            d^{-3} \ddegv^4 +
            d^{-4} \ddegv^8 +
            d^{-4} \ddegv^4 \log^8 n +
            d^{-3} \lt(\sum_{i\in \downNv} \dvli\rt)^2 \\
            &\qquad \qquad +
            d^{-3} \ddegv^3 \sum_{i\in \downNv} \dvli +
            d^{-4} \ddegv^5 \sum_{i\in \downNv} \dvli +
            n^{-10}
        \Bigg].
    \end{align*}
    for all $v\in [n]$.
\end{lemma}

Analogously to the expansion (\ref{eq:general-ub-main-exp-overlap-starting-point}) in the proof of Lemma~\ref{lem:general-ub-main-exp-overlap}, we will prove Lemma~\ref{lem:general-ub-refined-coupled-exp-overlap-ubs} by Taylor expanding the two exponentials and bounding terms of different orders individually.
Let the below expectations be over $(\Xonevl, \Xtwovl) \sim \cNvdips$.
We have that
\begin{align}
    \notag
    \E \exp(Y_v)
    &=
    1 +
    \sum_{i,j \in \downNv}
    \E\lt[\Delone[i,j] \Deltwo[i,j]\rt] +
    \f12
    \sum_{i_1,j_1,i_2,j_2 \in \downNv}
    \E \lt[
        \Delone[i_1,j_1] \Deltwo[i_1,j_1]
        \Delone[i_2,j_2] \Deltwo[i_2,j_2]
    \rt] \\
    \label{eq:general-ub-refined-exp-yv-expansion}
    &\qquad +
    \f{1}{6}
    \sum_{i_1,j_1,i_2,j_2,i_3,j_3 \in \downNv}
    \E \lt[
        \Delone[i_1,j_1] \Deltwo[i_1,j_1]
        \Delone[i_2,j_2] \Deltwo[i_2,j_2]
        \Delone[i_3,j_3] \Deltwo[i_3,j_3]
    \rt] +
    \E g(Y_v), \\
    \label{eq:general-ub-refined-exp-zv-expansion}
    \E \exp\lt(\f12 Z_v\rt)
    &=
    1 +
    \f12
    \sum_{i,j,k,\ell \in \downNv}
    \E \lt[
        \Delone[i,j]
        \Deltwo[j,k]
        \Delone[k,\ell]
        \Deltwo[\ell,i]
    \rt]
    +
    \E h\lt(\f12 Z_v\rt),
\end{align}
where $h(x) = \exp(x)-1-x$ as in Section~\ref{subsec:general-ub-main-exp-overlap} and $g(x) = \exp(x) -1 - x - \f12 x^2 - \f16 x^3$.
Note that each summand in the first order term of (\ref{eq:general-ub-refined-exp-yv-expansion}) has the form
\[
    \E
    \lt[
        \Delone[i,j] \Deltwo[i,j]
    \rt] 
    =
    d^{-2}
    \E
    \lt[
        \lt(\la \Xonei, \Xonej \ra - d\delta_{i,j}\rt)
        \lt(\la \Xtwoi, \Xtwoj \ra - d\delta_{i,j}\rt)
    \rt],
\]
which has already been bounded by combining Propositions~\ref{prop:general-ub-linear-term-unconditioning}, \ref{prop:general-ub-linear-term-ijedge}, and \ref{prop:general-ub-linear-term-ii} and Lemma~\ref{lem:general-ub-refined-linear-term}.
In the next two propositions, we will deterministically bound the contributions of the high degree terms $g(Y_v)$ and $h(\f12 Z_v)$ of (\ref{eq:general-ub-refined-exp-yv-expansion}) and (\ref{eq:general-ub-refined-exp-zv-expansion}) using the event $\Str$.
In the next section, we will complete the proof of Lemma~\ref{lem:general-ub-refined-coupled-exp-overlap-ubs} by using Gram-Schmidt couplings to bound the second and third order terms of (\ref{eq:general-ub-refined-exp-yv-expansion}) and the first order terms of (\ref{eq:general-ub-refined-exp-zv-expansion}).

\begin{proposition}
    \label{prop:general-ub-bound-gy}
    If $d \gg \max_{v\in [n]}\lt(\ddegv^2 + \ddegv \log^2 n\rt)$, then for sufficiently large $n$ it follows that
    \[
        \E_{(\Xonevl, \Xtwovl) \sim \cNvdips} g(Y_v)
        \le
        \f{8(\Ctr + 2)^4}{d^4}
        \lt(\ddegv^8 + \ddegv^4 \log^8 n\rt).
    \]
\end{proposition}
\begin{proof}
    Similarly to Proposition~\ref{prop:general-ub-bound-hy}, since $g(x) = \sum_{t=4}^\infty \f{1}{t!} x^t$ for all $x\in \bR$, it holds that $g(x) \le g(|x|)$ and $g$ is increasing on $[0, \infty)$.
    Let the below expectations be over $(\Xonevl, \Xtwovl) \sim \cNvdips$.
    The same applications of the triangle inequality, AM-GM and the conditioning of $\Str$ as in Proposition~\ref{prop:general-ub-bound-hy} imply that
    \begin{align*}
        \E g(Y_v)
        &\le
        \E g(|Y_v|)
        \le
        \E g\lt(
            \sum_{i,j\in \downNv}
            \lt|\Delone[i,j]\rt|
            \lt|\Deltwo[i,j]\rt|
        \rt) \\
        &\le
        \E g\lt(
            \f12 \sum_{i,j\in \downNv} \lt(\Delone[i,j]\rt)^2 +
            \f12 \sum_{i,j\in \downNv} \lt(\Deltwo[i,j]\rt)^2
        \rt) \\
        &=
        \E g\lt(
            \f12 \Tr\lt(\lt(\Delone[\downNv]\rt)^2\rt) +
            \f12 \Tr\lt(\lt(\Deltwo[\downNv]\rt)^2\rt)
        \rt) \\
        &\le
        g\lt(
            \f{(\Ctr + 2)}{d}
            \lt(\ddegv^2 + \ddegv \log^2 n\rt)
        \rt) \\
        &\le
        \lt(
            \f{(\Ctr + 2)}{d}
            \lt(\ddegv^2 + \ddegv \log^2 n\rt)
        \rt)^4 \\
        &\le
        \f{8(\Ctr + 2)^4}{d^4}
        \lt(\ddegv^8 + \ddegv^4 \log^8 n\rt).
    \end{align*}
    The second last inequality follows from the condition $d \gg \max_{v\in [n]}\lt(\ddegv^2 + \ddegv \log^2 n\rt)$ and the fact that $g(x) \le x^4$ for all $|x|\le 1$.
    The last inequality follows from Jensen's inequality, in the form $\lt(\f{x+y}{2}\rt)^4 \le \f{x^4+y^4}{2}$.
    This proves the proposition.
\end{proof}

\begin{proposition}
    \label{prop:general-ub-bound-hz}
    If $d \gg \max_{v\in [n]}\lt(\ddegv^2 + \ddegv \log^2 n\rt)$, then for sufficiently large $n$ it follows that
    \[
        \E_{(\Xonevl, \Xtwovl) \sim \cNvdips} h\lt(\f12 Z_v\rt)
        \le
        \f{2(\Ctr + 2)^4}{d^4}
        \lt(\ddegv^8 + \ddegv^4 \log^8 n\rt).
    \]
\end{proposition}
\begin{proof}
    Recall from the proof of Proposition~\ref{prop:general-ub-bound-hy} that $h(x) \le h(|x|)$ and $h$ is increasing on $[0,\infty)$.
    Let the below expectations be over $(\Xonevl, \Xtwovl) \sim \cNvdips$.
    The same applications of the triangle inequality, AM-GM and conditioning on $\Str$ imply that
    \begin{align*}
        \E h\lt(\f12 Z_v\rt)
        &\le
        \E h\lt(\f12 |Z_v|\rt)
        \le
        \E h\lt(
            \f12
            \sum_{i,j,k,\ell \in \downNv}
            \lt|\Delone[i,j]\rt|
            \lt|\Deltwo[j,k]\rt|
            \lt|\Delone[k,\ell]\rt|
            \lt|\Deltwo[\ell,i]\rt|
        \rt) \\
        &\le
        \E h\lt(
            \f14
            \sum_{i,j,k,\ell \in \downNv}
            \lt(\Delone[i,j]\rt)^2
            \lt(\Delone[k,\ell]\rt)^2
            +
            \f14
            \sum_{i,j,k,\ell \in \downNv}
            \lt(\Deltwo[j,k]\rt)^2
            \lt(\Deltwo[\ell,i]\rt)^2
        \rt) \\
        &=
        \E h\lt(
            \f14 \Tr \lt(\lt(\Delonebn\rt)^2 \rt)^2 +
            \f14 \Tr \lt(\lt(\Deltwobn\rt)^2 \rt)^2
        \rt) \\
        &\le
        h \lt(
            \f{(\Ctr + 2)^2}{2d^2}
            \lt(\ddegv^2 + \ddegv \log^2 n\rt)^2
        \rt) \\
        &\le
        \f{(\Ctr + 2)^4}{4d^4}
        \lt(\ddegv^2 + \ddegv \log^2 n\rt)^4 \\
        &\le
        \f{2(\Ctr + 2)^4}{d^4}
        \lt(\ddegv^8 + \ddegv^4 \log^8 n\rt).
    \end{align*}
    The second last inequality follows from the condition $d \gg \max_{v\in [n]}\lt(\ddegv^2 + \ddegv \log^2 n\rt)$ and the fact that $h(x) \le x^2$ for all $|x|\le 1$.
    The last inequality follows from Jensen's inequality, as in Proposition~\ref{prop:general-ub-bound-gy}.
    This proves the proposition.
\end{proof}

\subsection{Gram-Schmidt Coupling Bounds for Higher Order Terms}

It remains to bound the second and third order terms of (\ref{eq:general-ub-refined-exp-yv-expansion}) and the first order terms of (\ref{eq:general-ub-refined-exp-zv-expansion}).
We will bound the first order terms of (\ref{eq:general-ub-refined-exp-zv-expansion}) in terms of the second order terms of (\ref{eq:general-ub-refined-exp-yv-expansion}).
Then, we will bound he second and third order terms of (\ref{eq:general-ub-refined-exp-yv-expansion}) by a Gram-Schmidt argument, using couplings similar to $\cMvd(\cCgs[i,j])$ and $\cMvd(\cCgs[i])$ from Section~\ref{subsec:general-ub-main-exp-overlap}.

We begin with the following proposition, which controls the first order terms of (\ref{eq:general-ub-refined-exp-zv-expansion}) by the second order terms of (\ref{eq:general-ub-refined-exp-yv-expansion}).

\begin{proposition}
    \label{prop:general-ub-4cyc-exp-overlap-first-order-ub}
    The following inequality holds, over $(\Xonevl, \Xtwovl) \sim \cNvdips$.
    \[
        \sum_{i,j,k,\ell \in \downNv}
        \E
        \lt[
            \Delone[i,j]
            \Deltwo[j,k]
            \Delone[k,\ell]
            \Deltwo[\ell,i]
        \rt]
        \le
        \sum_{i,j,k,\ell \in \downNv}
        \E
        \lt[
            \Delone[i,j]
            \Deltwo[i,j]
            \Delone[k,\ell]
            \Deltwo[k,\ell]
        \rt].
    \]
\end{proposition}
\begin{proof}
    Recall from the discussion before Definition~\ref{defn:general-ub-ip-coupling} that if $(\Xonevl, \Xtwovl) \sim \cNvdips$, then we can sample $(\Xonevl, \Xtwovl)$ by sampling $W\sim \muvls$ and then sampling $\Xrvl$, for $r=1,2$, i.i.d. from $\gamma(W)$ conditioned on $\Xrvl \in S$.
    By AM-GM, each term in this sum can be bounded by
    \begin{align*}
        &\E_{(\Xonevl, \Xtwovl) \sim \cNvdips}
        \lt[
            \Delone[i,j]
            \Deltwo[j,k]
            \Delone[k,\ell]
            \Deltwo[\ell,i]
        \rt] 
        =
        \E_{W\sim \muvls}
        \lt(
            \E_{\substack{
                \Xonevl \sim \gamma(W) \\
                \Xonevl \in S
            }}
            \Delone[i,j]
            \Delone[k,\ell]
        \rt)
        \lt(
            \E_{\substack{
                \Xtwovl \sim \gamma(W) \\
                \Xtwovl \in S
            }}
            \Deltwo[j,k]
            \Deltwo[\ell,i]
        \rt) \\
        &\qquad \le
        \f12
        \E_{W\sim \muvls}
        \lt(
            \E_{\substack{
                \Xonevl \sim \gamma(W) \\
                \Xonevl \in S
            }}
            \Delone[i,j]
            \Delone[k,\ell]
        \rt)^2 +
        \f12
        \E_{W\sim \muvls}
        \lt(
            \E_{\substack{
                \Xtwovl \sim \gamma(W) \\
                \Xtwovl \in S
            }}
            \Deltwo[j,k]
            \Deltwo[\ell,i]
        \rt)^2 \\
        &\qquad =
        \f12
        \E_{(\Xonevl, \Xtwovl) \sim \cNvdips}
        \lt[
            \Delone[i,j]
            \Deltwo[i,j]
            \Delone[k,\ell]
            \Deltwo[k,\ell]
        \rt] +
        \f12
        \E_{(\Xonevl, \Xtwovl) \sim \cNvdips}
        \lt[
            \Delone[j,k]
            \Deltwo[j,k]
            \Delone[\ell,i]
            \Deltwo[\ell,i]
        \rt].
    \end{align*}
    Summing over $i,j,k,\ell \in \downNv$ proves the proposition.
\end{proof}

It now suffices to bound the second and third order terms of (\ref{eq:general-ub-refined-exp-yv-expansion}).
Recall that $(\Xonevl, \Xtwovl) \sim \cNvdipr$ is sampled by sampling $\Xonevl \sim \cN(0, I_d)^{\otimes (v-1)}$ and independently sampling $\Xtwovl \sim \cN(0, I_d)^{\otimes (v-1)}$ conditioned on (\ref{eq:general-ub-ip-coupling-condition}).
We begin with a technical lemma controlling constant moments of the the $\Delr[i,j]$.

\begin{lemma}
    \label{lem:general-ub-delr-large-moments}
    Let $T\ge 1$ be a positive integer.
    Let $\cM$ be any coupling of $(\Xonevl, \Xtwovl)$ wherein $\Xonevl, \Xtwovl$ are each marginally distribued as $\cN(0, I_d)^{\otimes (v-1)}$.
    For any $r=1,2$ and $i,j\in [v-1]$, we have that
    \[
        \E_{(\Xonevl, \Xtwovl) \sim \cM} \lt[\lt(\Delr[i,j]\rt)^{2T}\rt]
        \le 
        2^T (2T)^{2T} d^{-T}.
    \]
\end{lemma}
\begin{proof}
    Throughout this proof, let $(\Xonevl, \Xtwovl) \sim \cM$, so each of $\Xonevl, \Xtwovl$ is marginally distributed as $\cN(0, I_d)^{\otimes (v-1)}$.
    So, $\Delr[i,j]$ is a degree 2 polynomnial of i.i.d. standard Gaussians.
    By standard Gaussian moment computations, we have that
    \[
        \E \lt[\lt(\Delr[i,j]\rt)^{2}\rt]
        = 
        \begin{cases}
            d^{-1} & i\neq j, \\
            2d^{-1} & i=j.
        \end{cases}
    \]
    By Gaussian hypercontractivity,
    \[
        \E \lt[\lt(\Delr[i,j]\rt)^{2T}\rt]
        \le 
        (2T-1)^{2T}
        \lt(\E \lt[\lt(\Delr[i,j]\rt)^{2}\rt]\rt)^T
        \le 
        2^T (2T)^{2T} d^{-T}.
    \]
    
\end{proof}

Next we show, analogously to Proposition~\ref{prop:general-ub-linear-term-unconditioning}, that it suffices to bound the second and third order terms of (\ref{eq:general-ub-refined-exp-yv-expansion}) without conditioning on $\Xonevl, \Xtwovl \in S$.

\begin{proposition}
    \label{prop:general-ub-higher-term-unconditioning}
    Let $T \ge 2$ be a positive integer.
    Suppose that $d \gg \max_{v\in [n]} \ddegv + \log n$ and $n$ is sufficiently large.
    For all $v\in [n]$ and $i_1,j_1,i_2,j_2,\ldots, i_T,j_T\in \downNv$, we have that
    \begin{align*}
        &\Bigg|
            \E_{(\Xonevl, \Xtwovl) \sim \cNvdips}
            \lt[
                \prod_{t=1}^T
                \Delone[i_t,j_t]
                \Deltwo[i_t,j_t]
            \rt] \\
            &\qquad
            -
            \E_{(\Xonevl, \Xtwovl) \sim \cNvdipr}
            \lt[
                \prod_{t=1}^T
                \Delone[i_t,j_t]
                \Deltwo[i_t,j_t]
            \rt]
        \Bigg|
        \le
        10\cdot 2^T (4T)^{2T} d^{-T} n^{-10}.
    \end{align*}
\end{proposition}
\begin{proof}
    Throughout this proof, let $(\Xonevl, \Xtwovl) \sim \cNvdipr$.
    By AM-GM and Lemma~\ref{lem:general-ub-delr-large-moments},
    \[
        \E \lt[\lt(
            \prod_{t=1}^T
            \Delone[i_t,j_t]
            \Deltwo[i_t,j_t]
        \rt)^2\rt]
        \le
        \f{1}{2T}
        \sum_{r=1}^r
        \sum_{t=1}^T
        \E \lt[\lt(\Delr[i_t,j_t]\rt)^{4T}\rt]
        \le
        2^{2T} (4T)^{4T} d^{-2T}.
    \]
    By Proposition~\ref{prop:general-ub-s-hp}, $\P\lt((S^v)^c\rt) \le 2n^{-20} + e^{-n/2} \le 3n^{-20}$ for all sufficiently large $n$.
    When this holds, by a union bound
    \[
        \P \lt[
            \lt(\Xonevl, \Xtwovl \in S^v\rt)^c
        \rt]
        \le 6n^{-20}.
    \]
    Moreover, for sufficiently large $n$, $\P\lt[\Xonevl, \Xtwovl \in S\rt] \ge \f12$.
    By Lemma~\ref{lem:unconditioning-by-second-moment}, the expectation difference in this proposition is bounded by
    \[
        \f{2(6n^{-20})^{1/2}\lt(2^{2T} (4T)^{4T} d^{-2T}\rt)^{1/2}}{1/2}
        \le
        10 \cdot 2^T (4T)^{2T} d^{-T} n^{-10}.
    \]
\end{proof}

We now will use Gram-Schmidt couplings to bound the expectations of the second and third order terms of (\ref{eq:general-ub-refined-exp-yv-expansion}) over $(\Xonevl, \Xtwovl) \sim \cNvdipr$.
All couplings in this section will be of the form $\cMvd(\cC)$ for some collection of constraints $\cC$ as in Definition~\ref{defn:general-ub-modified-couplings}.
The core of our argument is the following lemma.
In the proof of Lemma~\ref{lem:general-ub-refined-coupled-exp-overlap-ubs}, we will apply this lemma with $T=2,3$.

\begin{lemma}[Bounds on Higher Order Terms]
    \label{lem:general-ub-higher-order-gs}
    Let $T\ge 2$ and $i_1,j_1,i_2,j_2,\ldots,i_T,j_T \in \downNv$ be such that $d \ge \degvl(i_1) + 2T$ (and $i_1,j_1,i_2,j_2,\ldots,i_T,j_T$ are not necessarily distinct).
    Then, it holds that
    \begin{align*}
        &\E_{(\Xonevl, \Xtwovl) \sim \cNvdipr}
        \lt[
            \prod_{t=1}^T
            \Delone[i_t,j_t]\Deltwo[i_t,j_t]
        \rt] \\
        &\qquad \le
        2^T (4T)^{2T}
        \cdot
        \begin{cases}
            d^{-T-1} (\degvl(i_1) + 2T) & \text{if $i_1\not\in \{i_2,j_2,\ldots,i_T,j_T\}$ and $(i_1,j_1)\not\in E(G[v-1])$}, \\
            d^{-T} & \text{otherwise}.
        \end{cases}
    \end{align*}
\end{lemma}
\begin{proof}
    This proof argues over many couplings of $(\Xonevl, \Xtwovl)$ where $\Xonevl, \Xtwovl$ are each marginally distributed as $\cN(0, I_d)^{\otimes (v-1)}$.
    To keep track of the couplings, we adopt the following notational convention when we write expectations.
    Every expectation over a coupling $\cM$ of $(\Xonevl, \Xtwovl)$ will be written explicitly as $\E_{(\Xonevl, \Xtwovl) \sim \cM}$, with the coupling $\cM$ indicated.
    We will write $\E$ without subscript only when the argument of the expectation is a function of only one of $\Xonevl$ and $\Xtwovl$.

    By AM-GM and Lemma~\ref{lem:general-ub-delr-large-moments}, for all sequences $i$ and $j$ we have
    \[
        \E_{(\Xonevl, \Xtwovl) \sim \cNvdipr}
        \lt[
            \prod_{t=1}^T
            \Delone[i_t,j_t]
            \Deltwo[i_t,j_t]
        \rt]
        \le
        \f{1}{2T}
        \sum_{r=1}^2
        \sum_{t=1}^T
        \E \lt[\lt(
            \Delr[i_t,j_t]
        \rt)^{2T}\rt]
        \le 
        2^T (4T)^{2T} d^{-T}.
    \]

    It remains to show the first bound in the lemma.
    Suppose that $i_1 \not \in \{i_2,j_2,\ldots,i_T,j_T\}$ and $(i_1,j_1)\not \in E(G[v-1])$.
    We now divide into two main cases.

    \noindent {\em Case 1. $i_1 \neq j_1$.}

    Let $V = \Nl(i_1) \cup \{i_1,j_1,i_2,j_2,\ldots,i_T,j_T\}$ and $m=|V|$.
    Note that $m \le \degvl(i_1) + 2T$.
    Fix a bijection $\pi : [m] \to V$ such that $\pi(m-1) = j_1$ and $\pi(m) = i_1$.
    Now consider the coupling $\cMvd(\cC)$ arising from the collection of constraints $\cC$ given as follows:
    \begin{align*}
        \Xonek &= \Xtwok
        \quad
        \text{if $k\in [v-1] \setminus V$,} \\
        \Uone &= \Utwo, \\
        \Wone[k,\ell] &= \Wtwo[k,\ell]
        \quad
        \text{for all $1\le \ell\le k\le m$ such that $(k,\ell)\neq (m, m-1)$} \\
        &\text{where $\lt(\Wr, \Ur\rt) = \GS\lt(\Xr[\pi(1)], \ldots, \Xr[\pi(m)]\rt)$ for $r=1,2$.}
    \end{align*}
    The condition $d \ge \degvl(i_1) + 2T$ ensures that the Gram-Schmidt procedure above is well-defined.
    The same argument as in Proposition~\ref{prop:general-ub-linear-term-ijnotedge} shows that $\cC$ implies that $\la \Xonek, \Xonel \ra = \la \Xtwok, \Xtwol \ra$ for all $k,\ell \in E(G[v-1])$.
    We have by Lemma~\ref{lem:general-ub-stronger-couplings} that
    \begin{equation}
        \label{eq:general-ub-higher-order-coupling-ineq}
        \E_{(\Xonevl, \Xtwovl) \sim \cNvdipr}
        \lt[
            \prod_{t=1}^T
            \Delone[i_t,j_t]
            \Deltwo[i_t,j_t]
        \rt]
        \le
        \E_{(\Xonevl, \Xtwovl) \sim \cMvd(\cC)}
        \lt[
            \prod_{t=1}^T
            \Delone[i_t,j_t]
            \Deltwo[i_t,j_t]
        \rt].
    \end{equation}
    Throughout the rest of this case, let $\Wone, \Wtwo$ be as in the coupling $\cMvd(\cC)$.
    As in Proposition~\ref{prop:general-ub-linear-term-ijnotedge}, we have that the following random variables are independent.
    \begin{itemize}
        \item $\Wone[k,k] \sim \sqrt{\chisq(d+1-k)}$ for $1\le k\le m$;
        \item $\Wone[k,\ell] \sim \cN(0,1)$ for $1\le \ell < k\le m$; and
        \item $\Wtwo[m,m-1] \sim \cN(0,1)$.
    \end{itemize}
    Note that $\prod_{t=2}^T \Delone[i_t,j_t] \Deltwo[i_t,j_t]$ is independent of $\Wone[m,m-1]$ and $\Wtwo[m,m-1]$.
    By Cauchy-Schwarz and AM-GM,
    \begin{align}
        \notag
        &\E_{(\Xonevl, \Xtwovl) \sim \cMvd(\cC)}
        \lt[
            \prod_{t=1}^T
            \Delone[i_t,j_t]
            \Deltwo[i_t,j_t]
        \rt] \\
        \notag
        &\qquad =
        \E_{(\Xonevl, \Xtwovl) \sim \cMvd(\cC)}
        \lt[
            \lt(
                \E_{\Wone[m,m-1], \Wtwo[m, m-1]}
                \Delone[i_1,j_1]
                \Deltwo[i_1,j_1]
            \rt)
            \prod_{t=2}^T
            \Delone[i_t,j_t]
            \Deltwo[i_t,j_t]
        \rt] \\
        \notag
        &\qquad \le
        \lt[
            \E_{(\Xonevl, \Xtwovl) \sim \cMvd(\cC)}
            \lt(
                \E_{\Wone[m,m-1], \Wtwo[m, m-1]}
                \Delone[i_1,j_1]
                \Deltwo[i_1,j_1]
            \rt)^2
        \rt]^{1/2} \\
        \label{eq:general-ub-higher-order-isolate-nonedge-by-cs}
        &\qquad \qquad \times
        \lt[
            \E_{(\Xonevl, \Xtwovl) \sim \cMvd(\cC)}
            \lt(
                \prod_{t=2}^T
                \Delone[i_t,j_t]
                \Deltwo[i_t,j_t]
            \rt)^2
        \rt]^{1/2}.
    \end{align}
    We will bound these last two expectations.
    By AM-GM and Lemma~\ref{lem:general-ub-delr-large-moments},
    \begin{align}
        \notag
        \E_{(\Xonevl, \Xtwovl) \sim \cMvd(\cC)}
        \lt(
            \prod_{t=2}^T
            \Delone[i_t,j_t]
            \Deltwo[i_t,j_t]
        \rt)^2
        &\le 
        \f{1}{2T-2}
        \sum_{r=1}^2
        \sum_{t=2}^{T}
        \E \lt[\lt(\Delr[i_t,j_t]\rt)^{4T-4}\rt] \\
        &\le 
        \label{eq:general-ub-higher-order-gaussian-hypercontractivity}
        2^{2T-2}
        (4T)^{4T-4}
        d^{-(2T-2)}.
    \end{align}
    By expanding $\Xr[i_1]$ and $\Xr[j_1]$, for $r=1,2$ in terms of the columns of $\Ur$ and entries of $\Wr$, we get
    \begin{align*}
        \Delone[i_1,j_1]
        \Deltwo[i_1,j_1]
        &=
        d^{-2}
        \la \Xone[i_1], \Xone[j_1] \ra
        \la \Xtwo[i_1], \Xtwo[j_1] \ra \\
        &=
        d^{-2}
        \lt\la
            \sum_{a=1}^{m}   \Wone[m,a]  \Uone[a],
            \sum_{a=1}^{m-1} \Wone[m-1,a]\Uone[a]
        \rt\ra
        \lt\la
            \sum_{a=1}^{m}   \Wtwo[m,a]  \Utwo[a],
            \sum_{a=1}^{m-1} \Wtwo[m-1,a]\Utwo[a]
        \rt\ra \\
        &=
        d^{-2}
        \lt(\sum_{a=1}^{m-1} \Wone[m-1,a] \Wone[m,a] \rt)
        \lt(\sum_{a=1}^{m-1} \Wtwo[m-1,a] \Wtwo[m,a] \rt) \\
        &=
        d^{-2}
        \lt(
            \Wone[m-1, m-1] \Wone[m, m-1] +
            \sum_{a=1}^{m-2} \Wone[m-1,a] \Wone[m,a]
        \rt) \\
        &\qquad
        \times \lt(
            \Wone[m-1, m-1] \Wtwo[m, m-1] +
            \sum_{a=1}^{m-2} \Wone[m-1,a] \Wone[m,a]
        \rt),
    \end{align*}
    using the constraints in $\cC$.
    Therefore,
    \[
        \E_{\Wone[m,m-1], \Wtwo[m, m-1]}
        \Delone[i_1,j_1]
        \Deltwo[i_1,j_1]
        =
        d^{-2}
        \lt(
            \sum_{a=1}^{m-2} \Wone[m-1,a] \Wone[m,a]
        \rt)^2.
    \]
    By standard computations with Gaussian moments,
    \begin{align*}
        \E_{(\Xonevl, \Xtwovl) \sim \cMvd(\cC)} \lt[\lt(
            \E_{\Wone[m,m-1], \Wtwo[m, m-1]}
            \Delone[i_1,j_1]
            \Deltwo[i_1,j_1]
        \rt)^2\rt]
        &=
        d^{-4}
        \E \lt[\lt(
            \sum_{a=1}^{m-2} \Wone[m-1,a] \Wone[m,a]
        \rt)^4\rt] \\
        &=
        3m(m-2)d^{-4}
        \le
        4m^2 d^{-4}.
    \end{align*}
    Therefore,
    \begin{align*}
        \E_{(\Xonevl, \Xtwovl) \sim \cMvd(\cC)}
        \lt[
            \prod_{t=1}^T
            \Delone[i_t,j_t]
            \Deltwo[i_t,j_t]
        \rt]
        &\le
        \lt(4m^2 d^{-4}\rt)^{1/2}
        \lt(2^{2T-2} (4T)^{4T-4} d^{-(2T-2)}\rt)^{1/2} \\
        &=
        2^T (4T)^{2T-2} md^{-T-1}
        \le
        2^T (4T)^{2T} d^{-T-1} (\degvl(i_1) + 2T).
    \end{align*}
    In the final inequality, we use the bound $m \le \degvl(i_1) + 2T$.

    \noindent {\em Case 2. $i_1 = j_1$.}

    This case will be handled through a similar argument as that used in Case 1, the main difference being that we will leave the $(m,m)$ Gram-Schmidt entry free, instead of the $(m,m-1)$ entry.
    In other words, this case is to the previous case as Proposition~\ref{prop:general-ub-linear-term-ii} is to Proposition~\ref{prop:general-ub-linear-term-ijnotedge}.
    
    Let $V = \Nl(i_1) \cup \{i_1,j_1,i_2,j_2, \ldots, i_T, j_T\}$ let $m=|V|$.
    Note that $i_1=j_1$ implies $m \le \degvl(i_1) + 2T-1$.
    Fix a bijection $\pi : [m] \to V$ such that $\pi(m) = i_1$.
    Now consider the coupling $\cMvd(\cC)$ arising from the collection of constraints $\cC$ given as follows:
    \begin{align*}
        \Xonek &= \Xtwok
        \quad
        \text{if $k\in [v-1] \setminus V$,} \\
        \Uone &= \Utwo, \\
        \Wone[k,\ell] &= \Wtwo[k,\ell]
        \quad
        \text{for all $1\le \ell\le k\le m$ such that $(k,\ell)\neq (m, m)$} \\
        &\text{where $\lt(\Wr, \Ur\rt) = \GS\lt(\Xr[\pi(1)], \ldots, \Xr[\pi(m)]\rt)$ for $r=1,2$.}
    \end{align*}
    Similarly to the previous case, the Gram-Schmidt procedure is well-defined, and $\cC$ implies that $\la \Xonek, \Xonel \ra = \la \Xtwok, \Xtwol \ra$ for all $k,\ell \in E(G[v-1])$.
    By Lemma~\ref{lem:general-ub-stronger-couplings}, (\ref{eq:general-ub-higher-order-coupling-ineq}) also holds with this new $\cC$.
    Throughout the rest of this case, let $\Wone, \Wtwo$ be as in the coupling $\cMvd(\cC)$.
    We have that the following random variables are independent.
    \begin{itemize}
        \item $\Wone[k,k] \sim \sqrt{\chisq(d+1-k)}$ for $1\le k\le m$;
        \item $\Wone[k,\ell] \sim \cN(0,1)$ for $1\le \ell < k\le m$; and
        \item $\Wtwo[m,m] \sim \sqrt{\chisq(d+1-m)}$.
    \end{itemize}
    Note that $\prod_{t=2}^T \Delone[i_t,j_t] \Deltwo[i_t,j_t]$ is independent of $\Wone[m,m]$ and $\Wtwo[m,m]$.
    Analogously to the inequality chain (\ref{eq:general-ub-higher-order-isolate-nonedge-by-cs}), we can show that
    \begin{align*}
        \E_{(\Xonevl, \Xtwovl) \sim \cMvd(\cC)}
        \lt[
            \prod_{t=1}^T
            \Delone[i_t,j_t]
            \Deltwo[i_t,j_t]
        \rt] 
        &\le
        \lt[
            \E_{(\Xonevl, \Xtwovl) \sim \cMvd(\cC)}
            \lt(
                \E_{\Wone[m,m], \Wtwo[m,m]}
                \Delone[i_1,i_1]
                \Deltwo[i_1,i_1]
            \rt)^2
        \rt]^{1/2} \\
        &\qquad \times \lt[
            \E_{(\Xonevl, \Xtwovl) \sim \cMvd(\cC)}
            \lt(
                \prod_{t=2}^T
                \Delone[i_t,j_t]
                \Deltwo[i_t,j_t]
            \rt)^2
        \rt]^{1/2},
    \end{align*}
    and (\ref{eq:general-ub-higher-order-gaussian-hypercontractivity}) bounds the last expectation.
    Expanding $\Xr[i_1]$ for $r=1,2$ in terms of the columns of $\Ur$ and entries of $\Wr$, we get
    \begin{align*}
        \Delone[i_1,i_1]
        \Deltwo[i_1,i_1]
        &=
        d^{-2}
        \lt(\norm {\Xone[i_1]}_2^2 - d\rt)
        \lt(\norm {\Xtwo[i_1]}_2^2 - d\rt) \\
        &=
        d^{-2}
        \lt(
            \norm{\sum_{a=1}^m \Wone[m,a] \Uone[a]}_2^2 - d
        \rt)
        \lt(
            \norm{\sum_{a=1}^m \Wtwo[m,a] \Utwo[a]}_2^2 - d
        \rt) \\
        &=
        d^{-2}
        \lt(
            \sum_{a=1}^m \lt(\Wone[m,a]\rt)^2 - d
        \rt)
        \lt(
            \sum_{a=1}^m \lt(\Wtwo[m,a]\rt)^2 - d
        \rt) \\
        &=
        d^{-2}
        \lt(
            \lt(\lt(\Wone[m,m]\rt)^2 - (d + 1 - m)\rt) +
            \sum_{a=1}^{m-1} \lt(\lt(\Wone[m,a]\rt)^2-1\rt)
        \rt) \\
        &\qquad \times
        \lt(
            \lt(\lt(\Wtwo[m,m]\rt)^2 - (d + 1 - m)\rt) +
            \sum_{a=1}^{m-1} \lt(\lt(\Wone[m,a]\rt)^2-1\rt)
        \rt).
    \end{align*}
    Therefore,
    \[
        \E_{\Wone[m,m], \Wtwo[m,m]}
        \Delone[i_1,i_1]
        \Deltwo[i_1,i_1]
        =
        d^{-2}
        \lt(
            \sum_{a=1}^{m-1} \lt(\lt(\Wone[m,a]\rt)^2-1\rt)
        \rt)^2.
    \]
    By standard computations with Gaussian moments,
    \begin{align*}
        \E_{(\Xonevl, \Xtwovl) \sim \cMvd(\cC)}
        \lt[\lt(
            \E_{\Wone[m,m], \Wtwo[m,m]}
            \Delone[i_1,i_1]
            \Deltwo[i_1,i_1]
        \rt)^2\rt] 
        &=
        d^{-4} \E \lt[\lt(
            \sum_{a=1}^{m-1} \lt(\lt(\Wone[m,a]\rt)^2-1\rt)
        \rt)^4\rt] \\
        &= 12(m-1)(m+3) d^{-4}
        \le
        16(m+1)^2 d^{-4}.
    \end{align*}
    Therefore,
    \begin{align*}
        \E_{(\Xonevl, \Xtwovl) \sim \cMvd(\cC)}
        \lt[
            \prod_{t=1}^T
            \Delone[i_t,j_t]
            \Deltwo[i_t,j_t]
        \rt]
        &\le
        \lt(16(m+1)^2 d^{-4}\rt)^{1/2}
        \lt(2^{2T-2} (4T)^{4T-4} d^{-(2T-2)}\rt)^{1/2} \\
        &=
        2^{T+1} (4T)^{2T-2} (m+1)d^{-T-1} \\
        &\le
        2^T (4T)^{2T} d^{-T-1} (\degvl(i_1) + 2T).
    \end{align*}
    In the last inequality, we use the bound $m \le \degvl(i_1) + 2T-1$.
\end{proof}

In the next proposition, we apply Proposition~\ref{prop:general-ub-higher-term-unconditioning} and Lemma~\ref{lem:general-ub-higher-order-gs} to bound the second and third order terms of (\ref{eq:general-ub-refined-exp-yv-expansion}).

\begin{proposition}
    \label{prop:general-ub-refined-deg2-deg3-term}
    Let $T\ge 2$ be a constant.
    Suppose that $d \gg \max_{v\in [n]} \ddegv + \log n$ and $n$ is sufficiently large.
    For all $v\in [n]$, it holds that
    \begin{align*}
        &\sum_{i_1,j_1,\ldots,i_T,j_T\in \downNv}
        \E_{(\Xonevl, \Xtwovl) \sim \cNvdips}
        \lt[
            \prod_{t=1}^T \Delone[i_t,j_t] \Deltwo[i_t,j_t]
        \rt] \\
        &\qquad  \le
        (4T)^{4T}
        \Bigg[
            d^{-T} \lt(\ddegv^{2T-2} + |E(G[\downNv])|^T\rt) \\
            &\qquad \qquad +
            d^{-T-1} \lt(\ddegv^{2T} + \ddegv^{2T-1} \sum_{i\in \downNv} \dvli\rt) +
            d^{-T} n^{-10} \ddegv^{2T}
        \Bigg].
    \end{align*}
\end{proposition}
\begin{proof}
    We first bound this sum where the expectation is over $\cNvdipr$ instead of $\cNvdips$.
    Let $R\subseteq \downNv^{2T}$ be the set of all vertex tuples $(i_1,j_1,\ldots,i_T,j_T)$ such that either $|\{i_1,j_1,\ldots,i_T,j_T\}| \le 2T-2$ or $(i_t,j_t) \in E(G[v-1])$ for all $1\le t\le T$.
    The number of tuples in $R$ can be crudely upper bounded by
    \[
        |R| \le
        (2T-2)^{2T} \ddegv^{2T-2} +
        2^T |E(G[\downNv])|^T.
    \]
    The first term bounds the tuples with $|\{i_1,j_1,\ldots,i_T,j_T\}| \le 2T-2$: there are at most $\ddegv^{2T-2}$ ways to pick the set of distinct elements among $i_1,j_1,\ldots,i_t,j_t$, and given this set, at most $(2T-2)^{2T}$ ways to pick $(i_1,j_1,\ldots,i_T,j_T)$.
    The second term bounds the tuples with $(i_t,j_t) \in E(G[v-1])$ for all $1\le t\le T$ because there are $2|E(G[\downNv])|$ possible choices for each $(i_t,j_t)$.
    
    We bound the sub-sum corresponding to tuples $(i_1,j_1,\ldots,i_T,j_T) \in R$ by the second bound in Lemma~\ref{lem:general-ub-higher-order-gs}.
    This yields that
    \begin{align*}
        &\sum_{(i_1,j_1,\ldots,i_T,j_T)\in R}
        \E_{(\Xonevl, \Xtwovl) \sim \cNvdipr}
        \lt[
            \prod_{t=1}^T
            \Delone[i_t,j_t]
            \Deltwo[i_t,j_t]
        \rt] \\
        &\qquad
        \le
        2^T (4T)^{2T} d^{-T}
        \lt(
            (2T-2)^{2T} \ddegv^{2T-2} +
            2^T |E(G[\downNv])|^T.
        \rt) \\
        &\qquad \le
        (4T)^{4T}
        d^{-T} \lt(\ddegv^{2T-2} + |E(G[\downNv])|^T\rt).
    \end{align*}
    For each tuple $(i_1,j_1,\ldots,i_T,j_T) \in \downNv^{2T} \setminus R$, we have that $|\{i_1,j_1,\ldots,i_T,j_T\}| \ge 2T-1$, and that there exists $t$ such that $(i_t,j_t)\not\in E(G[v-1])$.
    Because $|\{i_1,j_1,\ldots,i_T,j_T\}| \ge 2T-1$, at least one of $i_t, j_t$ is not in
    \[
        A_t =
        \{i_{t'}: 1\le t'\le T, t'\neq t\}
        \cup
        \{j_{t'}: 1\le t'\le T, t'\neq t\}.
    \]
    Suppose $i_t \not \in A_t$.
    Then, by reordering the sequences $i$ and $j$, we may apply the first bound in Lemma~\ref{lem:general-ub-higher-order-gs} to deduce
    \begin{align*}
        \E_{(\Xonevl, \Xtwovl) \sim \cNvdipr}
        \lt[
            \prod_{t=1}^T
            \Delone[i_t,j_t]
            \Deltwo[i_t,j_t]
        \rt]
        &\le
        2^T (4T)^{2T} d^{-T-1} \lt(\degvl(i_t) + 2T\rt) \\
        &\le
        2^T (4T)^{2T} d^{-T-1}
        \sum_{t=1}^T
        \lt(
            \degvl(i_t) +
            \degvl(j_t) +
            2
        \rt).
    \end{align*}
    If $j_t\not\in A_t$, we get the same upper bound by swapping the roles of $i$ and $j$.
    So,
    \begin{align*}
        &\sum_{(i_1,j_1,\ldots,i_T,j_T)\in \downNv^{2T} \setminus R}
        \E_{(\Xonevl, \Xtwovl) \sim \cNvdipr}
        \lt[
            \prod_{t=1}^T
            \Delone[i_t,j_t]
            \Deltwo[i_t,j_t]
        \rt] \\
        &\qquad \le
        2^T (4T)^{2T} d^{-T-1}
        \sum_{(i_1,j_1,\ldots,i_T,j_T)\in \downNv^{2T}}
        \sum_{t=1}^T
        \lt(
            \degvl(i_t) +
            \degvl(j_t) +
            2
        \rt) \\
        &\qquad =
        2T \cdot 2^T (4T)^{2T} d^{-T-1}
        \lt(
            \ddegv^{2T} + \ddegv^{2T-1} \sum_{i\in \downNv} \dvli
        \rt) \\
        &\qquad \le
        (4T)^{4T} d^{-T-1}
        \lt(
            \ddegv^{2T} + \ddegv^{2T-1} \sum_{i\in \downNv} \dvli
        \rt).
    \end{align*}
    Combining these bounds with the bound in Proposition~\ref{prop:general-ub-higher-term-unconditioning} (note that $10\cdot 2^T (4T)^{2T} \le (4T)^{4T}$ for all $T\ge 2$) completes the proof of the proposition.
\end{proof}

We now have the tools to prove Lemma~\ref{lem:general-ub-refined-coupled-exp-overlap-ubs}.

\begin{proof}[Proof of Lemma~\ref{lem:general-ub-refined-coupled-exp-overlap-ubs}]
    The condition $d \gg \max_{v\in [n]} \lt(\ddegv^2 + \ddegv \log^2 n\rt)$ implies that for all sufficiently large $n$, $d \ge \max_{v\in [n]} \ddegv^2$.
    So, for all $k,\ell \in [v-1]$ with $(k,\ell) \not\in E(G[v-1])$, Lemma~\ref{lem:general-ub-refined-linear-term} holds with $\eps = \f12$.
    Combined with Propositions~\ref{prop:general-ub-linear-term-unconditioning}, \ref{prop:general-ub-linear-term-ijedge}, and \ref{prop:general-ub-linear-term-ii}, this implies that
    \begin{align*}
        &\sum_{i,j\in \downNv}
        \E_{(\Xonevl, \Xtwovl) \sim \cNvdips}
        \lt[
            \Delone[i,j]
            \Deltwo[i,j]
        \rt]
        \le
        2d^{-1}|E(G[\downNv])| +
        C_{1/2} d^{-2} \sum_{i,j\in \downNv, i\neq j} \dvlij \\
        &\qquad +
        2 d^{-2} \sum_{i\in \downNv} \dvli +
        C_{1/2} d^{-3} \lt(\sum_{i\in \downNv} \dvli\rt)^2 +
        48 d^{-1} n^{-10} \ddegv^2 + d^{-10} \ddegv^2,
    \end{align*}
    where $C_{1/2}$ is defined in Lemma~\ref{lem:general-ub-refined-linear-term}.
    Combined with (\ref{eq:general-ub-refined-exp-yv-expansion}), Proposition~\ref{prop:general-ub-bound-gy} and Proposition~\ref{prop:general-ub-refined-deg2-deg3-term} with $L=2,3$, this implies that for some constant $C' > 0$, we have that
    \begin{align}
        \notag
        &\E_{(\Xonevl, \Xtwovl) \sim \cNvdips}
        \exp(Y_v) \\
        \notag
        &\qquad  \le
        1 + C' \Bigg[
            \bigg(
            d^{-1} |E(G[\downNv])| +
            d^{-2} \sum_{i,j\in \downNv, i\neq j} \dvlij +
            d^{-2} \sum_{i\in \downNv} \dvli \\
            \notag
            &\qquad \qquad \qquad +
            d^{-3} \lt(\sum_{i\in \downNv} \dvli\rt)^2 +
            d^{-1} n^{-10} \ddegv^2 +
            d^{-10} \ddegv^2
            \bigg) \\
            \notag
            &\qquad \qquad +
            \bigg(
                d^{-2} \ddegv^2 +
                d^{-2} |E(G[\downNv])|^2 +
                d^{-3} \ddegv^4 +
                d^{-3} \ddegv^3\sum_{i\in \downNv} \dvli \\
            \notag
            &\qquad \qquad \qquad +
                d^{-2} n^{-10} \ddegv^4
            \bigg) \\
            \notag
            &\qquad \qquad +
            \bigg(
                d^{-3} \ddegv^4 +
                d^{-3} |E(G[\downNv])|^3 +
                d^{-4} \ddegv^6 +
                d^{-4} \ddegv^5\sum_{i\in \downNv} \dvli \\
            \notag
            &\qquad \qquad \qquad +
                d^{-3} n^{-10} \ddegv^6
            \bigg) \\
            \label{eq:general-ub-exp-yv-zv-unsimplified-ub}
            &\qquad \qquad +
            \bigg(
                d^{-4} \ddegv^8 +
                d^{-4} \ddegv^4 \log^8 n
            \bigg)
        \Bigg].
    \end{align}
    The inequality $d \ge \max_{v\in [n]} \ddegv^2$ implies that for sufficiently large $n$,
    \[
        d^{-1} n^{-10} \ddegv^2,
        d^{-2} n^{-10} \ddegv^4,
        d^{-3} n^{-10} \ddegv^6
        \le n^{-10}.
    \]
    Because $|E(G[\downNv])| \le \ddegv^2$, the inequality $d \ge \max_{v\in [n]} \ddegv^2$ also implies that $d \ge |E(G[\downNv])|$.
    So, for sufficiently large $n$,
    \[
        d^{-2} |E(G[\downNv])|^2,
        d^{-3} |E(G[\downNv])|^3
        \le
        d^{-1}|E(G[\downNv])|.
    \]
    Moreover, we have $d^{-4} \ddegv^6 \le d^{-4} \ddegv^8$ and $d^{-10} \ddegv^2 \le d^{-2} \ddegv^2$.
    Combining these bounds proves the desired bound on $\E \exp(Y_v)$.

    By (\ref{eq:general-ub-refined-exp-zv-expansion}) and Propositions~\ref{prop:general-ub-bound-hz}, \ref{prop:general-ub-4cyc-exp-overlap-first-order-ub}, and \ref{prop:general-ub-refined-deg2-deg3-term} with $T=2$, we can show (\ref{eq:general-ub-exp-yv-zv-unsimplified-ub}) is also an upper bound for $\E \exp(\f12 Z_v)$.
    The desired inequality for $\E \exp(\f12 Z_v)$ follows similarly.
\end{proof}

\subsection{Proof of Theorem~\ref{thm:general-ub}}

Similarly to Lemma~\ref{lem:general-weaker-ub-hypothesis-translation}, the following lemma parses the hypotheses of Theorem~\ref{thm:general-ub} into a form compatible with the above results.

\begin{lemma}
    \label{lem:general-ub-hypothesis-translation}
    Suppose the hypotheses (\ref{eq:general-ub-hypothesis-4cycles}) and (\ref{eq:general-ub-hypothesis-k18}) of Theorem~\ref{thm:general-ub} hold.
    Then,
    \begin{eqnarray}
        \label{eq:general-ub-p2-translated}
        d^2 &\gg& \sum_{v\in G} \degv^2, \\
        \label{eq:general-ub-k18-translated}
        d^4 &\gg& \sum_{v\in G} \lt(\degv^8 + \degv^4 \log^8 n\rt).
    \end{eqnarray}
\end{lemma}
\begin{proof}
    For all nonnegative integers $a$, we have $a^2 = 2\binom{a}{2} + a$.
    So,
    \[
        \sum_{v\in G} \degv^2
        =
        \sum_{v\in G}
        \lt[
            2\binom{\degv}{2} +
            \degv
        \rt]
        =
        2\Num_G(P_2, E).
    \]
    Similarly, for all nonnegative $a$, we have $a^4 \le 4^4\binom{a}{4} + 3^3a$ and $a^8 \le 8^8 \binom{a}{8} + 7^7 a$.
    So,
    \begin{align*}
        &\sum_{v\in G}
        \lt(
            \degv^8 +
            \degv^4 \log^8 n
        \rt) \\
        &\qquad \le
        \sum_{v\in G} \lt(
            8^8 \binom{\degv}{8} +
            7^7 \degv +
            \lt(
                4^4 \binom{\degv}{4} +
                3^3 \degv
            \rt)
            \log^8 n
        \rt) \\
        &\qquad =
        8^8 \Num_G(K_{1,8}) +
        4^4 \Num_G(K_{1,4}) \log^8 n +
        2 \Num_G(E) (3^3 \log^8 n + 7^7).
    \end{align*}
    The hypotheses (\ref{eq:general-ub-hypothesis-4cycles}) and (\ref{eq:general-ub-hypothesis-k18}) now imply the result.
\end{proof}

We will now complete the proof of our main result for general masks.

\begin{proof}[Proof of Theorem~\ref{thm:general-ub}]
    For now, fix a vertex $v\in [n]$.
    Like in the proof of Theorem~\ref{thm:general-weaker-ub}, we have
    \[
        \E_{W\sim \muvl}
        \KL\lt((\muv | W) \parallel \nuv\rt)
        \le
        \E_{W\sim \muvlsv}
        \KL\lt((\muv | W) \parallel \nuv\rt)
        + 11n^{-9}
    \]
    for sufficiently large $n$.
    By Lemma~\ref{lem:general-ub-hypothesis-translation}, both (\ref{eq:general-ub-p2-translated}) and (\ref{eq:general-ub-k18-translated}) hold.
    By (\ref{eq:general-ub-k18-translated}), we have that $d \gg \max_{v\in [n]} \lt(\degv^2 + \degv \log^2 n\rt)$.
    So, the hypotheses of Lemmas~\ref{lem:general-ub-refined-2mm} and \ref{lem:general-ub-refined-coupled-exp-overlap-ubs} both hold.
    By these lemmas, there exists a constant $C$ such that
    \begin{align*}
        &\E_{W\sim \muvl}
        \KL\lt((\muv | W) \parallel \nuv\rt)
        \le
        11n^{-9}
        -1 +
        \exp\lt(Cd^{-3} (\ddegv^4 + \ddegv \log^3 n)\rt) \\
        &\qquad \times
        \Bigg[
            1 + C \bigg(
                d^{-1}|E(G[\downNv])| +
                d^{-2} \sum_{i\in \downNv} \dvli +
                d^{-2} \sum_{i,j\in \downNv, i\neq j} \dvlij \\
                &\qquad \qquad +
                d^{-2} \ddegv^2 +
                d^{-3} \ddegv^4 +
                d^{-4} \ddegv^8 +
                d^{-4} \ddegv^4 \log^8 n \\
                &\qquad \qquad +
                d^{-3} \lt(\sum_{i\in \downNv} \dvli\rt)^2 +
                d^{-3} \ddegv^3 \sum_{i\in \downNv} \dvli +
                d^{-4} \ddegv^5 \sum_{i\in \downNv} \dvli \\
                &\qquad \qquad +
                n^{-10}
            \bigg)
        \Bigg].
    \end{align*}
    The condition $d \gg \max_{v\in [n]} \lt(\degv^2 + \degv \log^2 n\rt)$ implies that the argument of the exponential is $o(1)$ and the quantity within square brackets is $1 + o(1)$ (recall again that $|E(G[\downNv])| \le \ddegv^2$).
    We can simplify this bound by noting that $\exp(x)(1+y) \le 1 + 2x + 2y$ for all $x,y\in [0,\f12]$.
    So, for sufficiently large $n$, there exists another constant $C^{(1)}$ such that
    \begin{align*}
        &\E_{W\sim \muvl}
        \KL\lt((\muv | W) \parallel \nuv\rt)
        \le
        C^{(1)}
        \Bigg[
            d^{-1}|E(G[\downNv])| +
            d^{-2} \sum_{i,j\in \downNv, i\neq j} \dvlij \\
            & \qquad +
            d^{-2} \sum_{i\in \downNv} \dvli +
            d^{-2} \ddegv^2 +
            d^{-3} \ddegv^4 +
            d^{-4} \ddegv^8 +
            d^{-3} \ddegv \log^3 n \\
            & \qquad +
            d^{-4} \ddegv^4 \log^8 n +
            d^{-3} \lt(\sum_{i\in \downNv} \dvli\rt)^2 +
            d^{-3} \ddegv^3 \sum_{i\in \downNv} \dvli \\
            &\qquad +
            d^{-4} \ddegv^5 \sum_{i\in \downNv} \dvli +
            n^{-9}
        \Bigg].
    \end{align*}
    Substituting this into (\ref{eq:general-ub-starting-point}) yields the upper bound
    \begin{align*}
        &2 \TV\lt(W(G,d), M(G)\rt)^2 
        \le
        C^{(1)}
        \Bigg[
            d^{-1} \sum_{v=1}^n |E(G[\downNv])| +
            d^{-2} \sum_{v=1}^n \sum_{i,j\in \downNv, i\neq j} \dvlij \\
            &\qquad +
            d^{-2} \sum_{v=1}^n \sum_{i\in \downNv} \dvli +
            d^{-2} \sum_{v=1}^n \ddegv^2 +
            d^{-3} \sum_{v=1}^n \ddegv^4 +
            d^{-4} \sum_{v=1}^n \ddegv^8 \\
            &\qquad +
            d^{-3} \sum_{v=1}^n \ddegv \log^3 n +
            d^{-4} \sum_{v=1}^n \ddegv^4 \log^8 n +
            d^{-3} \sum_{v=1}^n \lt(\sum_{i\in \downNv} \dvli\rt)^2 + \\
            &\qquad +
            d^{-3} \sum_{v=1}^n \ddegv^3 \sum_{i\in \downNv} \dvli +
            d^{-4} \sum_{v=1}^n \ddegv^5 \sum_{i\in \downNv} \dvli +
            n^{-8}
        \Bigg].
    \end{align*}
    As $|E(G[\downNv])|$ counts the number of 3-cycles in $G$ with largest vertex $v$, we have that
    \[
        \sum_{v=1}^n |E(G[\downNv])| = \Num_G(C_3).
    \]
    Similarly, because $\sum_{i,j\in \downNv, i\neq j} \dvlij$ counts twice the number of 4-cycles in $G$ with largest vertex $v$, we have that
    \[
        \sum_{v=1}^n \sum_{i,j\in \downNv, i\neq j} \dvlij = 2\Num_G(C_4).
    \]
    The remaining quantities in this bound can be estimated as follows.
    Recall that $N(v)$ denotes the set of neighbors of $v$ (not restricted to $[v-1]$), and $\degv = |N(v)|$ is the degree of $v$.
    We have that
    \[
        \sum_{v=1}^n \sum_{i\in \downNv} \dvli
        \le
        \sum_{v=1}^n \sum_{i\in N(v)} \deg(i)
        =
        \sum_{i=1}^n \deg(i)^2.
    \]
    By AM-GM,
    \[
        d^{-3}
        \sum_{v=1}^n
        \ddegv \log^3 n
        \le
        d^{-2}
        \sum_{v=1}^n
        \ddegv^2 +
        d^{-4}
        \sum_{v=1}^n
        \ddegv^4 \log^8 n.
    \]
    By Cauchy-Schwarz and AM-GM,
    \begin{align*}
        \sum_{v=1}^n
        \lt(\sum_{i\in \downNv} \dvli\rt)^2
        &\le
        \sum_{v=1}^n
        \lt(\sum_{u\in N(v)} \degu\rt)^2
        \le
        \sum_{v=1}^n \degv
        \sum_{u\in N(v)} \degu^2 \\
        &=
        \sum_{(u,v)\in E(G)} \lt(
            \degu^2 \degv +
            \degu\degv^2
        \rt) \\
        &\le
        \sum_{(u,v)\in E(G)}
        \lt(\degu^3 + \degv^3\rt)
        =
        2\sum_{v\in G} \degv^4.
    \end{align*}
    By AM-GM,
    \begin{align*}
        \sum_{v=1}^n
        \ddegv^3
        \sum_{i\in \downNv}
        \dvli
        &\le
        \sum_{v=1}^n
        \sum_{u\in N(v)}
        \degv^3 \degu \\
        &=
        \sum_{(u,v)\in E(G)}
        \lt(
            \degu^3 \degv +
            \degu \degv^3
        \rt)\\
        & \le
        \sum_{(u,v)\in E(G)}
        \lt(\degu^4 + \degv^4\rt)
        \le
        2\sum_{v\in G} \degv^5.
    \end{align*}
    Similarly
    \begin{align*}
        \sum_{v=1}^n
        \ddegv^5
        \sum_{i\in \downNv}
        \dvli
        &\le
        \sum_{v=1}^n
        \sum_{u\in N(v)}
        \degv^5 \degu \\
        &=
        \sum_{(u,v)\in E(G)}
        \lt(
            \degu^5 \degv +
            \degu \degv^5
        \rt)\\
        & \le
        \sum_{(u,v)\in E(G)}
        \lt(\degu^6 + \degv^6\rt)
        \le
        2\sum_{v\in G} \degv^7.
    \end{align*}
    Combining these bounds, we get that for another constant $C^{(2)}$, it holds that
    \begin{align*}
        2 \TV\lt(W(G,d), M(G)\rt)^2
        &\le
        C^{(2)}
        \Bigg[
            d^{-1} \Num_G(C_3) +
            d^{-2} \Num_G(C_4) +
            d^{-2} \sum_{v=1}^n \degv^2 \\
        &\qquad +
            d^{-3} \sum_{v=1}^n \degv^5 +
            d^{-4} \sum_{v=1}^n \lt(\degv^8 + \degv^4 \log^8 n\rt) +
            n^{-8}
        \Bigg].
    \end{align*}
    Finally, by AM-GM,
    \[
        d^{-3} \sum_{v=1}^n \degv^5
        \le
        d^{-2} \sum_{v=1}^n \degv^2 +
        d^{-4} \sum_{v=1}^n \degv^8.
    \]
    The result now follows from the bounds (\ref{eq:general-ub-hypothesis-triangles}), (\ref{eq:general-ub-hypothesis-4cycles}), (\ref{eq:general-ub-p2-translated}), and (\ref{eq:general-ub-k18-translated}).
\end{proof}

\section{Proof of TV Upper Bound for Bipartite Masks}
\label{sec:bipartite-ub-proof}

In this section, we will prove Theorem~\ref{thm:bipartite-ub}, which gives conditions for bipartite masks $G$ under which $\TV \lt(W(G,d), M(G)\rt) \to 0$. As discussed in Section~\ref{subsec:technical-overview-bipartite-ub-outline} the proof follows the same overall outline as the argument for general masks, but with several simplifications that end up resulting in sharper bounds.
Throughout this section, let $\XL \in \bR^{d\times |V_L|}$ and $\XR \in \bR^{d\times |V_R|}$ be the submatrices of $X\in \bR^{d\times n}$ consisting of the columns in $V_L$ and $V_R$.
We keep the indexing of $X$, so the columns of $\XL$ and $\XR$ are indexed by $V_L$ and $V_R$, respectively.
Let $A'_G \in \bR^{|V_R| \times |V_L|}$ denote the submatrix of the adjacency matrix $A_G$ indexed by $V_R\times V_L$.
Let 
\[
    \mu = \cL \lt(A'_G \odot d^{-1/2} \XRt \XL\rt),
\]
where $X\sim \cN(0, I_d)^{\otimes n}$.
Similarly, let 
\[
    \nu = \cL \lt(A'_G \odot M\rt),
\]
where $M\in \bR^{|V_R| \times |V_L|}$ has i.i.d. standard Gaussian entries whose coordinates are indexed by $V_R \times V_L$.
Note that, by symmetry, all information in the random matrices $W(G,d)$ and $M(G)$ is contained in the entries indexed by $V_R \times V_L$.
So, $\TV(W(G,d), M(G)) = \TV(\mu, \nu)$.
In the rest of this section, we will show $\TV(\mu, \nu) \to 0$ under the hypotheses of Theorem~\ref{thm:bipartite-ub}.

For an event $S\in \sigma(\XR)$ with positive probability, let $\mu^S$ denote $\mu$ conditioned on $\XR \in S$.
Formally, $\mu^S$ is the law of $A'_G \odot d^{-1/2} \XRt \XL$, where $\XL \sim \cN(0, I_d)^{\otimes |V_L|}$ and independently $\XR \sim \cN(0, I_d)^{\otimes |V_R|}$ conditioned on $\XR \in S$.
Let $\cL \lt(\cN(0, I_d)^{\otimes |V_R|} | S\rt)$ denote the law of a sample from $\cN(0, I_d)^{\otimes |V_R|}$ conditioned on being in the set $S$.
Note that $\XR \mapsto A'_G \odot d^{-1/2} \XRt \XL$ is a Markov transition.
By (\ref{eq:technical-overview-tv-conditioning}), data processing, and Cauchy-Schwarz, we derive the following inequality, which is the starting point of the proof of Theorem~\ref{thm:bipartite-ub}.
\begin{equation}
    \label{eq:bipartite-ub-starting-point}
    \TV(\mu, \nu)
    \le
    \TV\lt(\mu, \mu^S\rt) +
    \TV\lt(\mu^S, \nu\rt)
    \le
    \P(S^c) + \f12 \sqrt{\chisq\lt(\mu^S, \nu\rt)}.
\end{equation}
The main task of the proof of Theorem~\ref{thm:bipartite-ub} is to show this $\chisq$ divergence is $o(1)$ for an appropriate high probability event $S$.
We will carry this out in the following three steps.
\begin{enumerate}[label=(\arabic*)]
    \item We will choose a high probability event $S \in \sigma(\XR)$, on which $\chisq\lt(\mu^S, \nu\rt)$ can be bounded.
    This set $S$ is defined in Section~\ref{subsec:bipartite-ub-def-hp-set}.
    In Section~\ref{subsec:bipartite-ub-s-hp-proof} we will prove that $S$ occurs with $1-o(1)$ probability.
    \item In Section~\ref{subsec:bipartite-ub-2mm}, we simplify the quantity $\chisq\lt(\mu^S, \nu\rt)$ using a second moment method computation, in a manner analogous to Lemmas~\ref{lem:general-ub-main-2mm} and \ref{lem:general-ub-refined-2mm}.
    We will upper bound this $\chisq$ divergence in terms of two exponentiated overlaps.
    These are analogous to the coupled exponentiated overlap and 4-cycles variant coupled exponentiated overlap of Section~\ref{sec:general-ub-refined-higher-order-terms}, with two differences.
    First, as mentioned in Section~\ref{subsec:technical-overview-bipartite-ub-outline}, the coupling is now trivial, and the two replicas $\XoneR, \XtwoR$ defining the overlap are independent and each distributed according to $\cL \lt(\cN(0, I_d)^{\otimes |V_R|} | S\rt)$.
    Second, this second moment method calculation can be considered a batched version of the second moment method calculation for general $G$, in which we process all of $v\in V_L$ separately instead of one vertex at a time.
    Consequently, the overlaps in this setting are \textit{sums} of terms of the form $Y_v$ and $Z_v$ from Section~\ref{sec:general-ub-refined-higher-order-terms}, instead of single terms $Y_v$ and $Z_v$.
    \item In Section~\ref{subsec:bipartite-ub-exp-overlap}, we will bound the exponentiated overlaps with integration by tails, as described in Section~\ref{subsec:technical-overview-bipartite-ub-outline}.
    Finally, we will combine these results to prove Theorem~\ref{thm:bipartite-ub}.
\end{enumerate}

\subsection{Identifying the High Probability Latent Set $S$}
\label{subsec:bipartite-ub-def-hp-set}

In this section, we will define the high probability set $S \in \sigma(\XR)$ that we will use in our proof.
We will first introduce the quantities necessary to define $S$.
Recall that $N(v)$ denotes the set of neighbors of vertex $v$, and $\degv = |N(v)|$ denotes the degree of $v$.
We introduce a notion of shared degrees: for vertices $v_1,\ldots,v_k\in [n]$, let $\deg(v_1,\ldots,v_k)$ be the number of vertices adjacent to all of $v_1,\ldots,v_k$.
This definition holds even when some of $v_1,\ldots,v_k$ are the same, so for example $\deg(v_1,v_1,v_2) = \deg(v_1,v_2)$.
Recall that for a subset $V\subseteq [n]$, $X_V$ denotes the $d\times |V|$ matrix with columns $\{X_i : i\in V\}$, with rows indexed by $[d]$ and columns indexed by $V$, and that $I_V$ denotes the $|V| \times |V|$ identity matrix with rows and columns indexed by $V$.
For each $v\in V_L$, let $\Deln$ be the $\degv \times \degv$ matrix given by
\[
    \Deln = d^{-1} \Xnt \Xn - \In.
\]
So, each $\Deln$ is a function of $\XR$.
For two matrices $\XoneR, \XtwoR \in \bR^{d\times |V_R|}$ with the same indexing as $\XR$, we similarly define
\[
    \Delrn = d^{-1} \lt(\Xrn\rt)^\top \Xrn - \In
\]
for $v\in V_L$ and $r=1,2$.
We further define the overlaps
\begin{eqnarray*}
    Y \lt(\XoneR, \XtwoR\rt)
    &=&
    \sum_{v\in V_L}
    \Tr\lt(\Delonen \Deltwon\rt), \\
    Z \lt(\XoneR, \XtwoR\rt)
    &=&
    \sum_{v\in V_L}
    \Tr\lt(\lt(\Delonen \Deltwon\rt)^2\rt).
\end{eqnarray*}
These are analogous to the coupled overlaps $Y_v, Z_v$ defined in (\ref{eq:general-ub-redefine-yv}) and (\ref{eq:general-ub-define-zv}).
Each summand in $Y$ and $Z$ is a term of the form $Y_v$ and $Z_v$, respectively, with $\downNv$ replaced by $\Nv$; because we are bounding the information contribution of multiple vertices at once, $Y$ and $Z$ consist of multiple such terms.
Next, we define the following quantities, which are moments of $Y \lt(\XoneR, \XtwoR\rt)$ and $Z \lt(\XoneR, \XtwoR\rt)$ with respect to $\XtwoR$ with i.i.d. standard Gaussian entries, conditioned on $\XoneR$.
\begin{eqnarray*}
    \sigY{\XoneR}
    &=&
    \Var_{\XtwoR \sim \cN(0, I_d)^{\otimes |V_R|}}
    Y \lt(\XoneR, \XtwoR\rt) \\
    \eZ{\XoneR}
    &=&
    \E_{\XtwoR \sim \cN(0, I_d)^{\otimes |V_R|}}
    Z \lt(\XoneR, \XtwoR\rt) \\
    \sigZ{\XoneR}
    &=&
    \Var_{\XtwoR \sim \cN(0, I_d)^{\otimes |V_R|}}
    Z \lt(\XoneR, \XtwoR\rt) \\
\end{eqnarray*}
The final ingredient needed to define $S$ is the following lemma, which is the analogue of Lemmas~\ref{lem:general-weaker-ub-hypothesis-translation} and Lemmas~\ref{lem:general-ub-hypothesis-translation} for the bipartite setting.
\begin{lemma}
    \label{lem:bipartite-ub-hypothesis-translation}
    Suppose the hypotheses (\ref{eq:bipartite-ub-hypothesis-4cycles}), (\ref{eq:bipartite-ub-hypothesis-k14}), (\ref{eq:bipartite-ub-hypothesis-d8}), and (\ref{eq:bipartite-ub-hypothesis-d9}) of Theorem~\ref{thm:bipartite-ub} hold.
    Then, we have that
    \begin{align}
        \label{eq:bipartite-ub-hypothesis-k14-translated}
        1
        &\gg
        d^{-3} \lt[
            \sum_{v\in V_L} \lt(
                \deg(v)^4 + \deg(v) \log^3 n
            \rt)
        \rt], \\
        \notag
        1
        &\gg
        d^{-2}
        \sum_{i,j \in V_R} \deg(i,j)^2, \\
        \notag
        1
        &\gg
        d^{-4} \sum_{i,j,k,\ell \in V_R} \deg(i,j,k,\ell)^2 +
        d^{-5} \sum_{i,j,k\in V_R} \deg(i,j) \deg(i,k) +
        d^{-6} \sum_{i,j,k,\ell \in V_R} \deg(i,j,k) \deg(i,j,\ell), \\
        \notag
        1
        &\gg
        d^{-4}
        \lt[
            \sum_{v\in V_L} \lt(
                \deg(v)^3 + \deg(v) \log^2 n
            \rt)
        \rt]^2
        \Bigg[
            d^{-4} \sum_{i,j,k,\ell \in V_R} \deg(i,j,k,\ell)^2 \\
            \notag
            &\qquad +
            d^{-5} \sum_{i,j,k\in V_R} \deg(i,j) \deg(i,k) +
            d^{-6} \sum_{i,j,k,\ell \in V_R} \deg(i,j,k) \deg(i,j,\ell)
        \Bigg].
    \end{align}
\end{lemma}
We defer the proof of this lemma to Appendix~\ref{appsec:subgraph-statistics}.
By Lemma~\ref{lem:bipartite-ub-hypothesis-translation}, we can choose functions $\alpha(n), \beta(n), \gamma(n) \gg 1$ such that
\begin{align}
    \label{eq:bipartite-ub-alpha-ub}
    1
    &\gg
    \alpha(n) d^{-2}
    \sum_{i,j \in V_R} \deg(i,j)^2, \\
    \label{eq:bipartite-ub-beta-ub}
    1
    &\gg
    \beta(n) d^{-2}
    \sum_{i,j \in V_R} \deg(i,j), \\
    \notag
    1
    &\gg
    \gamma(n) \Bigg[
        d^{-4} \sum_{i,j,k,\ell \in V_R} \deg(i,j,k,\ell)^2 +
        d^{-5} \sum_{i,j,k\in V_R} \deg(i,j) \deg(i,k) \\
        \label{eq:bipartite-ub-gamma-ub1}
        &\qquad +
        d^{-6} \sum_{i,j,k,\ell \in V_R} \deg(i,j,k) \deg(i,j,\ell)
    \Bigg], \\
    \notag
    1
    &\gg
    \gamma(n) d^{-4}
    \lt[
        \sum_{v\in V_L}
        \lt(
            \deg(v)^3 + \deg(v) \log^2 n
        \rt)
    \rt]^2 
    \Bigg[
        d^{-4} \sum_{i,j,k,\ell \in V_R} \deg(i,j,k,\ell)^2 \\
        \label{eq:bipartite-ub-gamma-ub2}
        &\qquad +
        d^{-5} \sum_{i,j,k\in V_R} \deg(i,j) \deg(i,k) +
        d^{-6} \sum_{i,j,k,\ell \in V_R} \deg(i,j,k) \deg(i,j,\ell)
    \Bigg].
\end{align}
We can now define the set $S\in \sigma(\XR)$ by
\begin{equation}
    \label{eq:bipartite-ub-def-s}
    S = \Sop \cap \Syv \cap \Sze \cap \Szv,
\end{equation}
where the constituent events $\Sop, \Syv, \Sze, \Szv \in \sigma(\XR)$ are defined by
\begin{eqnarray}
    \label{eq:bipartite-ub-def-sop}
    \Sop
    &=&
    \lt\{
        \XR \in \bR^{d\times |V_R|}:
        \norm{\Delta_{N(v)}}_{\op}
        \le
        100 \sqrt{\f{\deg(v) + \log n}{d}}
        ~\text{for all}~
        v\in V_L
    \rt\}, \\
    \label{eq:bipartite-ub-def-syv}
    \Syv
    &=&
    \lt\{
        \XR \in \bR^{d\times |V_R|}:
        \sigY{\XR}
        \le
        \alpha(n) d^{-2}
        \sum_{i,j \in V_R} \deg(i,j)^2
    \rt\}, \\
    \label{eq:bipartite-ub-def-sze}
    \Sze
    &=&
    \lt\{
        \XR \in \bR^{d\times |V_R|}:
        \eZ{\XR}
        \le
        \beta(n) d^{-2}
        \sum_{i,j \in V_R} \deg(i,j)
    \rt\}, \\
    \label{eq:bipartite-ub-def-szv}
    \Szv
    &=&
    \lt\{
        \begin{array}{l}
            \XR \in \bR^{d\times |V_R|}:
            \sigZ{\XR}
            \le
            \displaystyle
            \gamma(n) \Bigg[
                d^{-4} \sum_{i,j,k,\ell \in V_R} \deg(i,j,k,\ell)^2 \\
                \displaystyle
                \qquad
                + d^{-5} \sum_{i,j,k\in V_R} \deg(i,j) \deg(i,k) 
                d^{-6} \sum_{i,j,k,\ell \in V_R} \deg(i,j,k) \deg(i,j,\ell)
            \Bigg]
        \end{array}
    \rt\}.
\end{eqnarray}
The event $\Sop$ provides a spectral bound on the $\Deln$, which will be important for bounding $\chisq$ divergence by the exponentiated overlap in Section~\ref{subsec:bipartite-ub-2mm}.
The events $\Syv$, $\Sze$, $\Szv$ control the first and second moments of the overlaps $Y$ and $Z$ with respect to the second input, when the first input is held fixed.
(Note that for any fixed $\XoneR$, $\E Y\lt(\XoneR, \XtwoR\rt) = 0$ over $\XtwoR \sim \cN(0, I_d)^{\otimes |V_R|}$, so the first moment of $Y$ does not need to be explicitly controlled.)
By Gaussian hypercontractivity, this controls the tails of the overlaps $Y$ and $Z$ with respect to $\XtwoR$ for fixed $\XoneR$. 
This, in turn, controls their exponential integrals by integration by tails.

The next proposition states that $S$ is a high probability event over $\XR \sim \cN(0, I_d)^{\otimes n}$.
The proof of this proposition is deferred to Section~\ref{subsec:bipartite-ub-s-hp-proof}.

\begin{proposition}
    \label{prop:bipartite-ub-s-hp}
    Suppose that $d \gg \max_{v\in V_L} \degv + \log n$.
    Then, there exists a constant $C>0$ such that for all sufficiently large $n$, $\P(\Sopc) \le Cn^{-1}$, $\P(\Syvc) \le C\alpha(n)^{-1}$, $\P(\Szec) \le C\beta(n)^{-1}$, and $\P(\Szvc) \le C\gamma(n)^{-1}$.
\end{proposition}

For technical reasons, the nonnegativity of $\eZ{\XR}$ will be useful in the proof of Proposition~\ref{prop:bipartite-ub-s-hp} and elsewhere in our argument.
We record this property in the following proposition, whose proof is also deferred to Section~\ref{subsec:bipartite-ub-s-hp-proof}.
\begin{proposition}
    \label{prop:bipartite-ub-ez-nonnegative}
    For all $\XR\in \bR^{d\times |V_R|}$, we have that $\eZ{\XR} \ge 0$.
\end{proposition}

\subsection{Bounding $\chisq$ Divergence with the Second Moment Method}
\label{subsec:bipartite-ub-2mm}

In this section, we will upper bound the quantity $\chisq(\mu^S, \nu)$ for the $S$ defined in (\ref{eq:bipartite-ub-def-s}).
Throughout this section, we adopt the notational convention that $\E_{\XR}$ denotes expectation with respect to $\XR \sim \cN(0, I_d)^{\otimes |V_R|}$, and $\E_{\XR \in S}$ denotes expectation with respect to $\XR \sim \cL \lt(\cN(0, I_d)^{\otimes |V_R|} | S\rt)$.
Similarly, let $\E_{\XoneR, \XtwoR}$ and $\E_{\XoneR, \XtwoR \in S}$ denote expectation with respect to $\XoneR, \XtwoR$ sampled i.i.d. from $\cN(0, I_d)^{\otimes |V_R|}$ and $\cL \lt(\cN(0, I_d)^{\otimes |V_R|} | S\rt)$, respectively.
The main result of this section is the following lemma.

\begin{lemma}[Bounds from the Second Moment Method, Bipartite Setting]
    \label{lem:bipartite-ub-2mm}
    Suppose that $d \gg \max_{v\in V_L} \deg(v) + \log n$.
    For all sufficiently large $n$, we have that
    \begin{align*}
        \label{lem:bipartite-ub-2mm}
        &\chisq\lt(\mu^S, \nu\rt)
        \le
        -1 + \exp\lt(
            4\cdot 10^{12}d^{-3}
            \sum_{v\in V_L} \lt(
                \deg(v)^4 + \deg(v) \log^3 n
            \rt)
        \rt) \\
        &\qquad \times
        \lt[
            \E_{\XoneR, \XtwoR \in S}
            \exp\lt(Y\lt(\XoneR, \XtwoR\rt)\rt)
        \rt]^{1/2}
        \lt[
            \E_{\XoneR, \XtwoR \in S}
            \exp\lt(\f12 Z\lt(\XoneR, \XtwoR\rt)\rt)
        \rt]^{1/2}.
    \end{align*}
\end{lemma}

We will prove this result by techniques analogous to the proof of Lemma~\ref{lem:general-ub-refined-2mm}, using Lemma~\ref{lem:exp-quadratic-gaussian} to bound the inner expectation arising from the second moment method and Lemma~\ref{lem:det-to-exp-to-deg3} to bound the resulting determinant by an exponentiated trace.

For a fixed realization $\XR\in \bR^{d\times |V_R|}$, let $\mu(\XR)$ denote the measure $\mu$ conditioned on $\XR$.
Formally, $\mu(\XR)$ is the law of $A'_G \odot d^{-1/2} \XRt \XL$, where $\XL \sim \cN(0, I_d)^{\otimes |V_L|}$ and $\XR$ is fixed.
Note that $\mu^S = \E_{\XR \in S} \mu(\XR)$.
Crucially, the measure $\mu\lt(\XR\rt)$ is a jointly Gaussian matrix with independent columns indexed by $V_L$, where the nonzero entries of the column corresponding to $v\in V_L$ form a jointly Gaussian vector with covariance matrix $d^{-1} \Xnt \Xn = \Deln + \In$.
Therefore,
\[
     \mu(\XR)
     =
     \bigotimes_{v\in V_L}
     \cN\lt(0, \Deln + \In\rt).
\]
Since $\nu$ is a jointly Gaussian matrix with the same nonzero entries as $\mu\lt(\XR\rt)$, where every column has identity covariance, we similarly have
\[
     \nu
     =
     \bigotimes_{v\in V_L}
     \cN\lt(0, \In\rt).
\]
The next proposition bounds the inner expectation in our application of the second moment method, analogously to Proposition~\ref{prop:general-ub-2mm-inner-expectation}.
\begin{proposition}
    \label{prop:bipartite-ub-2mm-inner-expectation}
    Let $\XoneR, \XtwoR \in \bR^{d\times |V_R|}$ be two fixed realizations of these random variables such that for all $v\in V_L$, the matrix $\lt(\Delonen + \In\rt)^{-1} + \lt(\Deltwon + \In\rt)^{-1} - \In$ is positive definite.
    Then, it follows that
    \[
        \E_{M\sim \nu}
        \rn{\mu\lt(\XoneR\rt)}{\nu}(M)
        \rn{\mu\lt(\XtwoR\rt)}{\nu}(M)
        =
        \prod_{v\in V_L}
        \det\lt(\In - \Delonen \Deltwon\rt)^{-1/2}.
    \]
\end{proposition}
\begin{proof}
    Because $\mu(\XR)$ and $\nu$ are both product distributions over $v\in V_L$, we have that
    \begin{align*}
        &\E_{M\sim \nu}
        \rn{\mu\lt(\XoneR\rt)}{\nu}(M)
        \rn{\mu\lt(\XtwoR\rt)}{\nu}(M) \\
        &\qquad =
        \prod_{v\in V_L}
        \lt(
            \E_{\phi \sim \cN(0, \In)}
            \rn{\cN\lt(0, \Delonen + \In\rt)}{\cN\lt(0, \In\rt)}(\phi)
            \rn{\cN\lt(0, \Deltwon + \In\rt)}{\cN\lt(0, \In\rt)}(\phi)
        \rt).
    \end{align*}
    By a computation identical to that of Proposition~\ref{prop:general-ub-2mm-inner-expectation}, where we use Lemma~\ref{lem:exp-quadratic-gaussian} to evaluate the exponential integral of a Gaussian quadratic form, the inner expectation evaluates as
    \[
        \E_{\phi \sim \cN(0, \In)}
        \rn{\cN\lt(0, \Delonen + \In\rt)}{\cN\lt(0, \In\rt)}(\phi)
        \rn{\cN\lt(0, \Deltwon + \In\rt)}{\cN\lt(0, \In\rt)}(\phi)
        \le
        \det\lt(\In - \Delonen \Deltwon\rt)^{-1/2}.
    \]
\end{proof}

The following bound on the spectral error term arising from Lemma~\ref{lem:det-to-exp-to-deg3} will be useful in the proof of Lemma~\ref{lem:bipartite-ub-2mm} and in the analysis to upper bound the exponentiated overlap $\E \exp(\f12 Z)$.
We will use this proposition with $T=3$ in the proof of Lemma~\ref{lem:bipartite-ub-2mm} and with $T=2$ in the argument to control $\E \exp(\f12 Z)$.

\begin{proposition}
    \label{prop:bipartite-ub-spectral-error-term}
    If $\XoneR, \XtwoR \in S$, then for all $v\in V_L$,
    \[
        \norm{\Delonen \Deltwon}_{\op}
        \le
        \f{10000}{d} \lt(\degv + \log n\rt).
    \]
    Furthermore, if $T$ is a positive integer, then
    \[
        \sum_{v\in V_L}
        \sum_{\lambda \in \spec\lt(\Delonen \Deltwon\rt)}
        |\lambda|^T
        \le
        20000^Td^{-T}
        \sum_{v\in V_L}
        \lt(\degv^{T+1} + \degv \log^T n \rt).
    \]
\end{proposition}
\begin{proof}
    Recall that $S\subseteq \Sop$.
    The definition (\ref{eq:bipartite-ub-def-sop}) of $\Sop$ implies that
    \[
        \norm{\Delrn}_{\op} \le
        100 \sqrt{\f{\degv + \log n}{d}}
    \]
    for all $v\in V_L$ and each $r=1,2$.
    Therefore, we have that
    \[
        \norm{\Delonen \Deltwon}_{\op}
        \le
        \norm{\Delonen}_{\op}
        \norm{\Deltwon}_{\op}
        \le
        \f{10000}{d} \lt(\degv + \log n\rt),
    \]
    proving the first part of the proposition.
    Moreover, for each $v\in V_L$ we have that
    \begin{align*}
        \sum_{\lambda \in \spec\lt(\Delonen \Deltwon\rt)}
        |\lambda|^T 
        &\le 
        \degv \lt(\f{10000}{d} \lt(\degv + \log n\rt)\rt)^T \\
        &\le
        20000^Td^{-T}
        \lt(\degv^{T+1} + \degv \log^T n\rt),
    \end{align*}
    where the last inequality follows from Jensen's inequality in the form $\lt(\f{x+y}{2}\rt)^T \le \f{x^T + y^T}{2}$ for $x,y > 0$.
    Summing this inequality over $v\in V_L$ completes the proof of the proposition.
\end{proof}

We now have the tools to prove Lemma~\ref{lem:bipartite-ub-2mm}.

\begin{proof}[Proof of Lemma~\ref{lem:bipartite-ub-2mm}]
    By the second moment method, we have that
    \[
        1 + \chisq\lt(\mu^S, \nu\rt)
        =
        \E_{M \sim \nu}
        \lt(
            \E_{\XR \in S}
            \rn{\mu(\XR)}{\nu}
        \rt)
        =
        \E_{\XoneR, \XtwoR \in S}
        \E_{M \sim \nu}
        \rn{\mu\lt(\XoneR\rt)}{\nu}(M)
        \rn{\mu\lt(\XtwoR\rt)}{\nu}(M).
    \]
    By the same argument as in Lemma~\ref{lem:general-ub-main-2mm}, the definition (\ref{eq:bipartite-ub-def-sop}) of $\Sop$ implies that the positive definite condition in Proposition~\ref{prop:bipartite-ub-2mm-inner-expectation} holds when $\XoneR, \XtwoR\in S$ and $n$ is sufficiently large.
    By Proposition~\ref{prop:bipartite-ub-2mm-inner-expectation}, we have that
    \[
        1 + \chisq\lt(\mu^S, \nu\rt)
        =
        \E_{\XoneR, \XtwoR \in S}
        \prod_{v\in V_L}
        \det\lt(\In - \Delonen \Deltwon\rt)^{-1/2}.
    \]
    By Proposition~\ref{prop:bipartite-ub-spectral-error-term}, we have that $\norm{\Delonen \Deltwon}_{\op} \le \f{10000}{d} \lt(\degv + \log n\rt)$ for all $v\in V_L$.
    Since $d \gg \max_{v\in V_L} \degv + \log n$, for the $\eps$ in Lemma~\ref{lem:det-to-exp-to-deg3} we have $\norm{\Delonen \Deltwon}_{\op} \le \eps$ for all $v\in V_L$ for all sufficiently large $n$.
    When this occurs, Lemma~\ref{lem:det-to-exp-to-deg3} implies that
    \begin{align*}
        1 + \chisq\lt(\mu^S, \nu\rt)
        &\le
        \E_{\XoneR, \XtwoR \in S}
        \prod_{v\in V_L}
        \lt[
            \begin{array}{l}
            \displaystyle \etr\lt(\f12 \Delonen \Deltwon\rt)
            \etr\lt(\f14 \lt(\Delonen \Deltwon\rt)^2\rt)
            \\
            \displaystyle
            \qquad \times
            \exp\lt(
                \f12
                \sum_{\lambda \in \spec\lt(\Delonen \Deltwon\rt)}
                |\lambda|^3
            \rt)
            \end{array}
        \rt] \\
        &\le
        \exp\lt(
            4\cdot 10^{12}d^{-3}
            \sum_{v\in V_L} \lt(
                \deg(v)^4 + \deg(v) \log^3 n
            \rt)
        \rt)
        \\
        &\qquad
        \times
        \E_{\XoneR, \XtwoR \in S}
        \exp\lt(
            \f12 Y\lt(\XoneR, \XtwoR\rt) +
            \f14 Z\lt(\XoneR, \XtwoR\rt)
        \rt).
    \end{align*}
    where the last inequality follows from Proposition~\ref{prop:bipartite-ub-spectral-error-term} with $T=3$ and the definitions of $Y$ and $Z$.
    The lemma now follows from Cauchy-Schwarz.
\end{proof}

\subsection{Bounding the Exponentiated Overlap of Two Independent Replicas}
\label{subsec:bipartite-ub-exp-overlap}

Lemma~\ref{lem:bipartite-ub-2mm} leaves the task of bounding the two exponentiated overlaps $\E \exp\lt(Y\lt(\XoneR, \XtwoR\rt)\rt)$ and $\E \exp\lt(\f12 Z\lt(\XoneR, \XtwoR\rt)\rt)$, where both expectations are over independent $\XoneR, \XtwoR \in S$.
In this section, we will bound these quantities by proving the following two lemmas.

\begin{lemma}
    \label{lem:bipartite-ub-exp-y-ub}
    There exists a constant $C$ such that for all sufficiently large $n$,
    \[
        \E_{\XoneR, \XtwoR \in S}
        \exp\lt(
            Y\lt(\XoneR, \XtwoR\rt)
        \rt)
        \le
        \P(S)^{-1}
        \lt(
            1 +
            C\alpha(n)d^{-2}
            \sum_{i,j\in V_R}
            \deg(i,j)^2
        \rt).
    \]
\end{lemma}
\begin{lemma}
    \label{lem:bipartite-ub-exp-z-ub}
    There exists a constant $C$ such that for all sufficiently large $n$,
    \begin{align*}
        &\E_{\XoneR, \XtwoR \in S}
        \exp\lt(
            \f12 Z\lt(\XoneR, \XtwoR\rt)
        \rt)
        \le
        \P(S)^{-1}
        \exp\lt(
            C\beta(n)d^{-2}
            \sum_{i,j\in V_R}
            \deg(i,j)
        \rt) \\
        &\qquad \times
        \Bigg[
            1 + C\gamma(n)^{1/2}
            \Bigg(
                d^{-4} \sum_{i,j,k,\ell \in V_R} \deg(i,j,k,\ell)^2 +
                d^{-5} \sum_{i,j,k\in V_R} \deg(i,j) \deg(i,k) \\
                &\qquad \qquad +
                d^{-6} \sum_{i,j,k,\ell \in V_R} \deg(i,j,k) \deg(i,j,\ell)
            \Bigg)^{1/2}
        \Bigg].
    \end{align*}
\end{lemma}

Lemma~\ref{lem:bipartite-ub-exp-y-ub} is the simpler of these two results, and we will prove it first.
Our main tool is the following lemma, which translates the tail bounds of a degree 2 polynomial $f$ of i.i.d. Gaussians obtained from hypercontractivity to a bound on $\E \exp(f)$.

\begin{lemma}
    \label{lem:dheeraj-method-deg2}
    Let $f$ be a degree 2 polynomial of i.i.d. Gaussian inputs, such that $\E f = 0$ and $\E f^2 = \sigma^2$.
    There exists a small enough constant $\delta > 0$ and a large enough constant $\eta > 0$ such that if $\sigma \le \delta$, then $\E \exp(f) \le 1 + \eta \sigma^2$.
\end{lemma}

Intuitively, this lemma states that if $f$ is centered with sufficiently small variance, then the amount by which $\E \exp(f)$ exceeds $1$ is dominated by its second moment, as one would expect from Taylor expanding $\exp(f)$.

\begin{proof}
    By Lemma~\ref{lem:gaussian-hypercontractivity-tails},
    there exist constants $c_2, C_2 > 0$ such that
    \[
        \P\lt[|f| > t\rt] \le C_2\exp(-c_2 t/\sigma).
    \]
    Let $f'$ be an independent copy of $f$.
    Then,
    \[
        \P\lt[|f-f'| > t\rt]
        \le \P\lt[|f| > t/2\rt] + \P\lt[|f'| > t/2\rt]
        \le 2C_2\exp(-c_2 t/2\sigma).
    \]
    Since $\E f' = 0$, by Jensen's inequality we have $\E \exp(-f') \ge 1$.
    We integrate by tails, exploiting the symmetry of $f-f'$, to deduce
    \begin{align*}
        \E \exp(f)
        &\le \E \exp(f-f')
        = \E \cosh(|f-f'|)
        = 1 + \int_1^\infty \P[\cosh(|f-f'|) > t] \diff{t} \\
        &= 1 + \int_0^\infty \P[|f-f'| > s] \sinh(s) \diff{s}
        \le 1 + 2C_2 \int_0^\infty \exp\lt(-\f{c_2s}{2\sigma}\rt) \sinh(s) \diff{s}.
    \end{align*}
    Set $\delta = \f{c_2}{4}$, so $\sigma \le \delta$ implies $\f{c_2}{2\sigma} \ge 2$.
    Then, this bound implies
    \begin{align*}
        \E \exp(f)
        &\le
        1 + 2C_2 \cdot \f12 \lt[\f{1}{(c_2/2\sigma)-1} - \f{1}{(c_2/2\sigma) + 1}\rt]
        = 1 + \f{2C_2}{(c_2/2\sigma)^2-1} \\
        &\le 1 + \f{2C_2}{\f34 (c_2/2\sigma)^2}
        = 1 + \f{32C_2}{3c_2^2} \sigma^2.
    \end{align*}
    So, the lemma holds with $\eta = \f{32c_2}{C_2^2}$.
\end{proof}

This lemma yields a short proof of Lemma~\ref{lem:bipartite-ub-exp-y-ub}.

\begin{proof}[Proof of Lemma~\ref{lem:bipartite-ub-exp-y-ub}]
    For now, fix a realization of $\XoneR \in S$.
    We will first bound the exponentiated overlap integrated only over $\XtwoR$.
    First, by the nonnegativity of the integrand,
    \[
        \E_{\XtwoR\in S}
        \exp\lt(Y\lt(\XoneR, \XtwoR\rt)\rt)
        \le
        \P(S)^{-1}
        \E_{\XtwoR}
        \exp\lt(Y\lt(\XoneR, \XtwoR\rt)\rt),
    \]
    where we recall that $\E_{\XtwoR}$ denotes expectation over $\XtwoR \sim \cN(0, I_d)^{\otimes |V_R|}$.
    Conditioned on $\XoneR$, $Y\lt(\XoneR, \XtwoR\rt)$ is a degree 2 polynomial in the i.i.d. standard Gaussian entries of $\XtwoR$.
    We can easily check that its conditional expectation is $0$, and its conditional variance is, by definition, $\sigY{\XoneR}$.
    As $\XoneR \in S \subseteq \Syv$, the bounds (\ref{eq:bipartite-ub-alpha-ub}) and (\ref{eq:bipartite-ub-def-syv}) imply that $\sigY{\XoneR}\ll 1$.
    Therefore, for sufficiently large $n$, we have $\lt(\sigY{\XoneR}\rt)^{1/2} \le \delta$ for the $\delta$ in Lemma~\ref{lem:dheeraj-method-deg2}.
    By Lemma~\ref{lem:dheeraj-method-deg2} and the definition (\ref{eq:bipartite-ub-def-syv}) of $\Syv$,
    \begin{align*}
        \E_{\XtwoR\in S}
        \exp\lt(Y\lt(\XoneR, \XtwoR\rt)\rt)
        &\le
        \P(S)^{-1}
        \lt(
            1 + \eta \sigY{\XoneR}
        \rt) \\
        &\le
        \P(S)^{-1}
        \lt(
            1 +
            \eta \alpha(n) d^{-2}
            \sum_{i,j\in V_R}
            \deg(i,j)^2
        \rt).
    \end{align*}
    Because this bound holds for all $\XoneR \in S$, it also holds in expectation over $\XoneR \in S$.
    This proves the lemma with $C = \eta$.
\end{proof}

Next, we will prove Lemma~\ref{lem:bipartite-ub-exp-z-ub} by generalizing the technique used to prove Lemma~\ref{lem:bipartite-ub-exp-y-ub}.
There are two important differences between these two lemmas that will need to be overcome.
The first difference is that, unlike $Y\lt(\XoneR, \XtwoR\rt)$, the polynomial $Z\lt(\XoneR, \XtwoR\rt)$ does not have expectation $0$ over $\XtwoR$ for every fixed $\XoneR$.
So, instead of applying hypercontractivity to $\f12 Z$, we will write
\[
    \E_{\XtwoR}
    \lt(\f12 Z \rt)
    =
    \lt(\f12 \E_{\XtwoR} Z \rt)
    \E_{\XtwoR}
    \lt(\f12 \lt(Z - \E_{\XtwoR} Z\rt)\rt)
\]
and apply hypercontractivity to $\f12 \lt(Z - \E_{\XtwoR} Z\rt)$.
The second and more challenging difference is that $Z$ is a degree 4 polynomial in $\XtwoR$, whereas $Y$ is a degree 2 polynomial, and degree 4 polynomials are generally not exponentially integrable.
If we try to integrate by tails as in Lemma~\ref{lem:dheeraj-method-deg2}, we will reach an integral of the form $\int_0^\infty \exp\lt(-O\lt(\sqrt{\f{s}{\sigma}}\rt) + s\rt) \diff{s}$, which diverges.

We overcome this difficulty with the following observation.
As $s$ increases starting from $0$, the value of $\exp\lt(-O\lt(\sqrt{\f{s}{\sigma}}\rt) + s\rt)$ exponentially decays until $s$ is of scale $1/\sigma$, and then diverges to infinity.
So, if we can truncate this integral at scale $1/\sigma$, we get a well-behaved integral.
Moreover, we do have a way to truncate this integral, because conditioning on the event $\Sop$ gives a finite, albeit weak, upper bound on $Z\lt(\XoneR, \XtwoR\rt)$.
Thus, our hypercontractivity and integration by tails technique allows us to bootstrap from this weak upper bound to the stronger bound of Lemma~\ref{lem:bipartite-ub-exp-z-ub}.

Note that, in this method, we use the high probability event $S$ in two distinct ways.
\begin{enumerate}[label=(\arabic*)]
    \item Similarly to the proof of Lemma~\ref{lem:bipartite-ub-exp-y-ub}, we argue that conditioned on $\XoneR\in S$, the mean and variance of $Z\lt(\XoneR, \XtwoR\rt)$ over the randomness of $\XtwoR$ is not unusually large.
    This uses the events $\Sze$ and $\Szv$.
    \item Using Proposition~\ref{prop:bipartite-ub-spectral-error-term}, we show a deterministic upper bound on $Z\lt(\XoneR, \XtwoR\rt)$ for all pairs $\XoneR, \XtwoR\in S$, which allows us to truncate the integral we obtain when we integrate by tails.
    This uses the event $\Sop$.
\end{enumerate}

We will now present this proof formally.
The following proposition bounds the integral of $\exp\lt(-\sqrt{\f{s}{\sigma}} + s\rt)$ when it is truncated below the scale at which it diverges.
\begin{proposition}
    \label{prop:bipartite-ub-truncated-exp}
    Let $\sigma, U > 0$ with $\sigma U \le \f{1}{4}$. Then,
    \[
        \int_0^U
        \exp\lt(
            -\sqrt{\f{s}{\sigma}} + s
        \rt)
        \diff{s}
        \le
        8\sigma.
    \]
\end{proposition}
\begin{proof}
    For each $s\le U$, we have $s\le \f{1}{4\sigma}$, so $s\le \f12 \sqrt{\f{s}{\sigma}}$.
    By a routine calculation,
    \[
        \int_0^U
        \exp\lt(-\sqrt{\f{s}{\sigma}} + s\rt)
        \diff{s}
        \le
        \int_0^U
        \exp\lt(-\f12 \sqrt{\f{s}{\sigma}}\rt)
        \diff{s}
        =
        8\sigma -
        \lt(4\sqrt{\sigma U} + 8a\rt)
        \exp\lt(-\f12 \sqrt{\f{U}{\sigma}}\rt)
        \le 8\sigma.
    \]
\end{proof}

The following lemma is the desired analogue of Lemma~\ref{lem:dheeraj-method-deg2}.
As discussed above, we truncate on an event $X\in T$ that provides a deterministic bound on values of $f(X)$.
\begin{lemma}
    \label{lem:dheeraj-method-deg4}
    Let $f$ be a degree 4 polynomial of i.i.d. Gaussian inputs $X$, such that $\E f = \mu \ge 0$ and $\Var f = \sigma^2$.
    Let $U>0$, and let $T$ be an event such that $|f(X)|\le U$ for all $X\in T$.
    There exists a small enough constant $\delta > 0$ and a large enough constant $\eta > 0$ such that if $U\sigma \le \delta$, then
    \[
        \E \lt[\exp(f(X)) \ind{X\in T}\rt]
        \le
        \exp(\mu)(1 + \eta \sigma).
    \]
\end{lemma}
\begin{proof}
    By Lemma~\ref{lem:gaussian-hypercontractivity-tails} on $f-\mu$, there exist constants $c_4, C_4 > 0$ such that
    \[
        \P\lt[|f|>t+\mu\rt]
        \le \P\lt[|f-\mu|>t\rt]
        \le C_4 \exp\lt(-c_4\sqrt{\f{t}{\sigma}}\rt).
    \]
    So,
    \begin{align*}
        &\E
        \lt[
            \exp(f(X))
            \ind{X\in T}
        \rt]
        \le
        \E \exp(|f|)
        \ind{|f|\le U}
        \le
        \int_0^{\exp(U)}
        \P\lt[\exp(|f|) > t\rt]
        \diff{t} \\
        &\qquad \le
        \exp(\mu) +
        \int_{\exp(\mu)}^{\exp(U)}
        \P\lt[\exp(|f|) > t\rt]
        \diff{t}
        =
        \exp(\mu) +
        \int_{\mu}^U
        \P\lt[|f| > s\rt]
        \exp(s)
        \diff{s} \\
        &\qquad \le
        \exp(\mu) \lt[
            1 +
            \int_{0}^U
            \P\lt[|f| > s+\mu\rt]
            \exp(s)
            \diff{s}
        \rt]
        \le
        \exp(\mu) \lt[
            1 +
            C_4 \int_{0}^U \exp\lt(
                -c_4 \sqrt{\f{s}{\sigma}}
                + s
            \rt) \diff{s}
        \rt] \\
        &\qquad =
        \exp(\mu) \lt[
            1 +
            C_4 \int_{0}^U \exp\lt(
                -\sqrt{\f{s}{\sigma/c_4^2}}
                + s
            \rt) \diff{s}
        \rt].
    \end{align*}
    Pick $\delta = \f{c_4^2}{4}$, so that when $U\sigma \le \delta$, we have $U\cdot \f{\sigma}{c_4^2} \le \f14$.
    Then, Proposition~\ref{prop:bipartite-ub-truncated-exp} implies that
    \[
        \int_{0}^U
        \exp\lt(-\sqrt{\f{s}{\sigma/c_4^2}} + s\rt)
        \diff{s}
        \le
        \f{8\sigma}{c_4^2}.
    \]
    This proves the lemma with $\eta = \f{8C_4}{c_4^2}$.
\end{proof}

We can now prove Lemma~\ref{lem:bipartite-ub-exp-z-ub}.

\begin{proof}[Proof of Lemma~\ref{lem:bipartite-ub-exp-z-ub}]
    For now, fix a realization $\XoneR \in S$.
    We will first bound the exponentiated overlap integrated only over $\XtwoR$.
    First, we have that
    \[
        \E_{\XtwoR \in S}
        \exp\lt(\f12 Z\lt(\XoneR, \XtwoR\rt)\rt)
        =
        \P(S)^{-1}
        \E_{\XtwoR}
        \lt[
            \exp\lt(\f12 Z\lt(\XoneR, \XtwoR\rt)\rt)
            \ind{\XtwoR\in S}
        \rt].
    \]
    By spectrally expanding $Z$, we have that for $\XtwoR\in S$,
    \begin{align*}
        \f12
        Z\lt(\XoneR, \XtwoR\rt)
        &=
        \f12
        \sum_{v\in V_L}
        \sum_{\lambda \in \spec\lt(\Delonen \Deltwon\rt)}
        \lambda^2
        \le
        \f12
        \sum_{v\in V_L}
        \sum_{\lambda \in \spec\lt(\Delonen \Deltwon\rt)}
        |\lambda|^2 \\
        &\le
        2\cdot 10^8 d^{-2}
        \sum_{v\in V_L}
        \lt(
            \degv^3 + \degv \log^2 n
        \rt),
    \end{align*}
    where the last inequality is by Proposition~\ref{prop:bipartite-ub-spectral-error-term} with $T=2$.
    Let $U$ denote this upper bound.
    Conditioned on $\XoneR$, $\f12 Z\lt(\XoneR, \XtwoR\rt)$ is a degree 4 polynomial in the i.i.d. standard Gaussian entries of $\XtwoR$ with mean $\f12 \eZ{\XoneR}$ and variance $\f14 \sigZ{\XoneR}$.
    By Proposition~\ref{prop:bipartite-ub-ez-nonnegative}, we have $\f12 \eZ{\XoneR} \ge 0$.
    As $\XoneR \in S \subseteq \Szv$, the bounds (\ref{eq:bipartite-ub-gamma-ub2}) and (\ref{eq:bipartite-ub-def-szv}) imply that $U^2 \sigZ{\XoneR} \ll 1$.
    Therefore, for sufficiently large $n$, we have $U \lt(\f14 \sigZ{\XoneR}\rt)^{1/2} \le \delta$ for the $\delta$ in Lemma~\ref{lem:dheeraj-method-deg4}.
    By Lemma~\ref{lem:dheeraj-method-deg4} and the definitions (\ref{eq:bipartite-ub-def-sze}), (\ref{eq:bipartite-ub-def-szv}) of $\Sze$ and $\Szv$, we have that
    \begin{align*}
        &\E_{\XtwoR \in S}
        \exp\lt(
            \f12 Z\lt(\XoneR, \XtwoR\rt)
        \rt)
        \le
        \P(S)^{-1}
        \exp\lt(\f12 \eZ{\XoneR}\rt)
        \lt[1 + \f{\eta}{2} \lt(\sigZ{\XoneR}\rt)^{1/2}\rt] \\
        &\qquad \le
        \P(S)^{-1}
        \exp\lt(\f12 \beta(n) d^{-2} \sum_{i,j\in V_R} \deg(i,j)\rt)
        \Bigg[
            1 +
            \f{\eta}{2} \gamma(n)^{1/2}
            \Bigg(
                d^{-4} \sum_{i,j,k,\ell \in V_R}
                \deg(i,j,k,\ell)^2 \\
                &\qquad \qquad +
                d^{-5} \sum_{i,j,k\in V_R}
                \deg(i,j) \deg(i,k) +
                d^{-6} \sum_{i,j,k,\ell \in V_R}
                \deg(i,j,k) \deg(i,j,\ell)
            \Bigg)^{1/2}
        \Bigg].
    \end{align*}
    This proves the lemma for $C = \max\lt(\f12, \f{\eta}{2}\rt)$.
\end{proof}

We are now ready to prove Theorem~\ref{thm:bipartite-ub}.
The remaining task is to verify that $S$ occurs with high probability, as in Proposition~\ref{prop:bipartite-ub-s-hp}.
We carry out this task in Section~\ref{subsec:bipartite-ub-s-hp-proof}.

\begin{proof}[Proof of Theorem~\ref{thm:bipartite-ub}]
    By a union bound and Proposition~\ref{prop:bipartite-ub-s-hp}, we have that
    \[
        \P(S^c)
        \le
        \P(\Sopc) + \P(\Syvc) + \P(\Szec) + \P(\Szvc)
        \le
        C \lt[n^{-1} + \alpha(n)^{-1} + \beta(n)^{-1} + \gamma(n)^{-1}\rt]
    \]
    for some constant $C>0$.
    Because $\alpha(n), \beta(n), \gamma(n) \gg 1$, we have $\P(S^c) = o(1)$.

    By Lemma~\ref{lem:bipartite-ub-hypothesis-translation}, (\ref{eq:bipartite-ub-hypothesis-k14-translated}) holds.
    This implies that $d \gg \max_{v\in V_L} \degv + \log n$, so Lemma~\ref{lem:bipartite-ub-2mm} holds.
    Combining Lemmas~\ref{lem:bipartite-ub-2mm}, \ref{lem:bipartite-ub-exp-y-ub}, and \ref{lem:bipartite-ub-exp-z-ub} gives that for sufficiently large $n$, there exists a constant $C>0$ such that
    \begin{align*}
        1 + \chisq(\mu^S, \nu)
        &\le
        \P(S)^{-1}
        \exp\lt(
            Cd^{-3}
            \sum_{v\in V_L} \lt(
                \deg(v)^4 + \deg(v) \log^3 n
            \rt)
        \rt) \\
        &\qquad \times
        \lt(
            1 +
            C\alpha(n)d^{-2}
            \sum_{i,j\in V_R}
            \deg(i,j)^2
        \rt)^{1/2}
        \exp\lt(
            C\beta(n)d^{-2}
            \sum_{i,j\in V_R}
            \deg(i,j)
        \rt)^{1/2} \\
        &\qquad \times
        \Bigg[
            1 + C\gamma(n)^{1/2}
            \Bigg(
                d^{-4} \sum_{i,j,k,\ell \in V_R} \deg(i,j,k,\ell)^2 +
                d^{-5} \sum_{i,j,k\in V_R} \deg(i,j) \deg(i,k) \\
                &\qquad \qquad +
                d^{-6} \sum_{i,j,k,\ell \in V_R} \deg(i,j,k) \deg(i,j,\ell)
            \Bigg)^{1/2}
        \Bigg]^{1/2}.
    \end{align*}
    Since $\P(S^c) = o(1)$, we have $\P(S)^{-1} = 1 + o(1)$.
    The bounds (\ref{eq:bipartite-ub-alpha-ub}), (\ref{eq:bipartite-ub-beta-ub}), and (\ref{eq:bipartite-ub-gamma-ub1}) imply that the remaining terms in this upper bound are all $1+o(1)$.
    Therefore, $\chisq(\mu^S, \nu) = o(1)$.
    Substituting this into (\ref{eq:bipartite-ub-starting-point}) yields the result.
\end{proof}

\subsection{High Probability Bounds on $S$ and Nonnegativity of $\eZ{\XR}$}
\label{subsec:bipartite-ub-s-hp-proof}

In this section, we show that $S$, defined in (\ref{eq:bipartite-ub-def-s}), holds with high probability, proving Proposition~\ref{prop:bipartite-ub-s-hp}.
We will show that the constituent events $\Sop$, $\Syv$, $\Sze$, and $\Szv$ of $S$ all occur with high probability.
We begin with $\Sop$, defined in (\ref{eq:bipartite-ub-def-sop}).

\begin{proposition}
    \label{prop:bipartite-ub-sop-high-prob}
    Suppose that $d \gg \max_{v\in V_L} \degv + \log n$.
    Then, for all sufficiently large $n$, we have that $\P(\Sopc) \le n^{-1}$.
\end{proposition}
\begin{proof}
    By an argument identical to Proposition~\ref{prop:general-ub-svop-high-prob}, for sufficiently large $n$ we have
    \[
        \norm{\Deln}_{\op}
        \le
        100\sqrt{\f{\degv + \log n}{d}}
    \]
    with probability at least $1 - n^{-2}$ for each $v\in V_L$.
    The result follows from a union bound because $|V_L| \le n$.
\end{proof}

We will show the events $\Syv$, $\Sze$, and $\Szv$, defined in (\ref{eq:bipartite-ub-def-syv}), (\ref{eq:bipartite-ub-def-sze}), and (\ref{eq:bipartite-ub-def-szv}), are high probability by the following technique.
Each of their complements is the event that one of the random variables $\sigY{\XR}$, $\eZ{\XR}$, and $\sigZ{\XR}$ does not exceed $\alpha(n)$, $\beta(n)$, and $\gamma(n)$ times its typical scale, respectively.
These random variables are nonnegative: $\sigY{\XR}$ and $\sigZ{\XR}$ are variances, and we will prove $\eZ{\XR}$ is nonnegative as in Proposition~\ref{prop:bipartite-ub-ez-nonnegative}.
The desired probability bounds then follow from Markov's inequality.

In the following proofs, we let
\[
    \Delr[V_R]
    =
    d^{-1} \lt(\XrR\rt)^\top \XrR - I_{V_R}
\]
for $r=1,2$.
For $i,j \in V_R$, let $\Delr[i,j] = d^{-1} \la \Xri, \Xrj\ra - \delta_{i,j}$ be the $(i,j)$ entry of $\Delr[V_R]$.
We first prove Proposition~\ref{prop:bipartite-ub-ez-nonnegative}, that $\eZ{\XR}$ is nonnegative.

\begin{proof}[Proof of Proposition~\ref{prop:bipartite-ub-ez-nonnegative}]
    By expanding into coordinates, we have that
    \[
        Z\lt(\XoneR, \XtwoR\rt) =
        \sum_{v\in V_L}
        \sum_{i,j,k,\ell \in N(v)}
        \Delone[i,j]
        \Deltwo[j,k]
        \Delone[k,\ell]
        \Deltwo[\ell,i].
    \]
    By standard computations with Gaussian moments, we have
    \begin{equation}
        \label{eq:delta-second-moments}
        \E_{\XrR}
        \lt[
            \Delr[i,j]
            \Delr[k,\ell]
        \rt]
        =
        \begin{cases}
            2/d & i=j=k=\ell, \\
            1/d & \{i,j\} = \{k,\ell\} ~\text{and}~ i\neq j, \\
            0 & \text{otherwise}.
        \end{cases}
    \end{equation}
    So,
    \begin{align}
        \notag
        \eZ{\XoneR}
        &=
        \sum_{v\in V_L}
        \sum_{i,j,k,\ell \in N(v)}
        \Delone[i,j]
        \Delone[k,\ell]
        \E_{\XtwoR}
        \lt[
        \Deltwo[j,k]
        \Deltwo[\ell,i]
        \rt] \\
        \notag
        &=
        \f{1}{d}
        \sum_{v\in V_L}
        \lt[
            \sum_{i\in N(v)}
            2\lt(\Delone[i,i]\rt)^2
            +
            \sum_{i,j\in N(v), i\neq j}
            \lt(\Delone[i,j]\rt)^2
            +
            \sum_{i,j\in N(v), i\neq j}
            \Delone[i,i]\Delone[j,j]
        \rt] \\
        \label{eq:bipartite-ub-ez-expansion}
        &=
        \f{1}{d}
        \sum_{v\in V_L}
        \lt[
            \sum_{i,j\in N(v)}
            \lt(\Delone[i,j]\rt)^2
            +
            \lt(\sum_{i\in N(v)} \Delone[i,i]\rt)^2.
        \rt].
    \end{align}
    Therefore $\eZ{\XoneR} \ge 0$, as desired.
\end{proof}

We now proceed to bounding the probabilities of $\Syvc$, $\Szec$, and $\Szvc$.

\begin{proposition}
    \label{prop:bipartite-ub-syv-high-prob}
    There exists a constant $C$ such that $\P(\Syvc) \le C\alpha(n)^{-1}$.
\end{proposition}
\begin{proof}
    We can expand $Y$ into coordinates, by
    \[
        Y\lt(\XoneR, \XtwoR\rt)
        =
        \sum_{i,j\in V_R}
        \deg(i,j)
        \Delone[i,j]
        \Deltwo[i,j].
    \]
    As $\E_{\XtwoR} \Deltwo[i,j] = 0$ for all $i,j\in V_R$, we have $\E_{\XtwoR} Y\lt(\XoneR, \XtwoR\rt) = 0$.
    Therefore, we have that $\sigY{\XoneR} = \E_{\XtwoR} Y\lt(\XoneR, \XtwoR\rt)^2$.
    Because $\XoneR$ and $\XtwoR$ are independent, we have
    \begin{align*}
        \E_{\XoneR}
        \sigY{\XoneR}
        &=
        \E_{\XoneR, \XtwoR}
        Y\lt(\XoneR, \XtwoR\rt)^2 \\
        &=
        \sum_{i,j,i',j'\in V_R}
        \deg(i,j)\deg(i',j')
        \E_{\XoneR} \lt[
            \Delone[i,j]
            \Delone[i',j']
        \rt]
        \E_{\XtwoR} \lt[
            \Deltwo[i,j]
            \Deltwo[i',j']
        \rt].
    \end{align*}
    By (\ref{eq:delta-second-moments}), this implies that
    \[
        \E_{\XoneR}
        \sigY{\XoneR}
        =
        d^{-2} \lt(
            \sum_{i\in V_R}
            4
            \deg(i)^2 +
            \sum_{i,j\in V_R, i\neq j}
            2
            \deg(i,j)^2
        \rt)
        \le
        4d^{-2}
        \sum_{i,j\in V_R}
        \deg(i,j)^2.
    \]
    Because $\sigY{\XoneR}$ is a variance, we have $\sigY{\XoneR} \ge 0$ almost surely.
    The proposition now follows from Markov's inequality.
\end{proof}

\begin{proposition}
    \label{prop:bipartite-ub-sze-high-prob}
    There exists a constant $C$ such that $\P(\Szec) \le C\beta(n)^{-1}$.
\end{proposition}
\begin{proof}
    By expanding (\ref{eq:bipartite-ub-ez-expansion}) into coordinates, we get that
    \[
        \eZ{\XoneR}
        =
        \f{1}{d}
        \sum_{i,j\in V_R}
        \deg(i,j)
        \lt[
            \lt(\Delone[i,j]\rt)^2 +
            \Delone[i,i]
            \Delone[j,j]
        \rt].
    \]
    By (\ref{eq:delta-second-moments}), this implies that
    \[
        \E_{\XoneR}
        \eZ{\XoneR}
        =
        d^{-2} \lt(
            \sum_{i\in V_R}
            4
            \deg(i) +
            \sum_{i,j\in V_R, i\neq j}
            \deg(i,j)
        \rt)
        \le
        4d^{-2}
        \sum_{i,j\in V_R}
        \deg(i,j).
    \]
    By Proposition~\ref{prop:bipartite-ub-ez-nonnegative}, $\eZ{\XoneR} \ge 0$ almost surely.
    The proposition now follows from Markov's inequality.
\end{proof}

\begin{proposition}
    \label{prop:bipartite-ub-szv-high-prob}
    There exists a constant $C$ such that $\P(\Szvc) \le C\gamma(n)^{-1}$.
\end{proposition}
\begin{proof}
    This proposition follows from the same technique as the previous two, though the computation bounding $\sigZ{\XoneR}$ is more involved.
    We can expand $Z$ into coordinates, by
    \[
        Z\lt(\XoneR, \XtwoR\rt)
        =
        \sum_{i,j,k,\ell \in V_R}
        \deg(i,j,k,\ell)
        \Delone[i,j]
        \Delone[k,\ell]
        \Deltwo[j,k]
        \Deltwo[\ell,i].
    \]
    So, we have that
    \begin{equation}
        \label{eq:bipartite-ub-sigz-expansion}
        \E_{\XoneR}
        \sigZ{\XoneR}
        =
        \E_{\XoneR, \XtwoR}
        \lt[\lt(
            \sum_{i,j,k,\ell \in V_R}
            \deg(i,j,k,\ell)
            \Delone[i,j]
            \Delone[k,\ell]
            \lt(
                \Deltwo[j,k]
                \Deltwo[\ell,i]
                -
                \E_{\XtwoR}
                \lt[
                    \Deltwo[j,k]
                    \Deltwo[\ell,i]
                \rt]
            \rt)
        \rt)^2\rt].
        \end{equation}
    The inner expectation can be computed by (\ref{eq:delta-second-moments}).
    We separate this sum into fifteen sub-sums, by partitioning the quadruples of indices $(i,j,k,\ell)\in V_R^4$ into the fifteen sets described in Table~\ref{tab:bipartite-ub-sigz-expansion-index-partition}.

    \begin{table}[h!]
    \centering
    \begin{tabular}{|c|c||c|c||c|c|}
        \hline
        Set & Condition & Set & Condition & Set & Condition \\
        \hline
        $A_{1}$ & $i=j=k=\ell$ & $A_{3}$ & $i=j$, $k=\ell$ & $A_{6c}$ & $k=\ell$ \\
        \hline
        $A_{2a}$ & $i=j=k$ & $A_{4}$ & $i=k$, $j=\ell$ & $A_{6d}$ & $i=\ell$ \\
        \hline
        $A_{2b}$ & $i=j=\ell$ & $A_{5}$ & $i=\ell$, $j=k$ & $A_{7a}$ & $i=k$ \\
        \hline
        $A_{2c}$ & $i=k=\ell$ & $A_{6a}$ & $i=j$ & $A_{7b}$ & $j=\ell$ \\
        \hline
        $A_{2d}$ & $j=k=\ell$ & $A_{6b}$ & $j=k$ & $A_{8}$ & $i,j,k,\ell$ all distinct \\
        \hline
    \end{tabular}
    \caption{A partition of the quadruples $(i,j,K,L)$ summed over in (\ref{eq:bipartite-ub-sigz-expansion}). In each set, all indices not indicated as equal in the given condition are distinct. For example, $A_3$ is the set of indices $(i,j,k,\ell)$ for $i,j,k,\ell\in V_R$ where $i=j$, $k=\ell$, and $i\neq k$.}
    \label{tab:bipartite-ub-sigz-expansion-index-partition}
\end{table}

    Define the sub-sum
    \[
        f_1 =
        \sum_{(i,j,k,\ell)\in A_1}
        \deg(i,j,k,\ell)
        \Delone[i,j]
        \Delone[k,\ell]
        \lt(
            \Deltwo[j,k]
            \Deltwo[\ell,i]
            -
            \E_{\XtwoR}
            \lt[
                \Deltwo[j,k]
                \Deltwo[\ell,i]
            \rt]
        \rt),
    \]
    and likewise define $f_{2a},f_{2b},\ldots,f_8$.
    By the inequality $(x_1+\cdots+x_n)^2 \le n\lt(x_1^2+\cdots+x_n^2\rt)$, we have that
    \[
        \E_{\XoneR}
        \sigZ{\XoneR}
        =
        \E_{\XoneR, \XtwoR}
        \lt[\lt(
            f_1 + f_{2a} + \cdots + f_8
        \rt)^2\rt]
        \le
        15\E_{\XoneR, \XtwoR}
        \lt[
            f_1^2 + f_{2a}^2 + \cdots + f_{8}^2
        \rt].
    \]
    We define a symmetric sum notation: $\sum_{(i,j,k,\ell) \in \Sym(V_R, 4)}$ denotes the sum over all $|V_R|(|V_R|-1)(|V_R|-2)(|V_R|-3)$ tuples of distinct $i,j,k,\ell \in V_R$.
    Similarly, $\sum_{(i,j,k) \in \Sym(V_R, 3)}$ denotes the sum over all $|V_R|(|V_R|-1)(|V_R|-2)$ tuples of distinct $i,j,k\in V_R$, and so on.
    By standard computations with Gaussian moments, we can compute the following identities.
    \begin{align*}
        \E_{\XoneR, \XtwoR}
        f_1^2
        &=
        \lt(
            \f{96}{d^4} +
            \f{960}{d^5} +
            \f{2304}{d^6}
        \rt)
        \sum_{i\in \Sym(V_R, 1)}
        \deg(i)^2,
        \\
        \E_{\XoneR, \XtwoR}
        f_{2a}^2
        &=
        \lt(
            \f{4}{d^4} +
            \f{32}{d^5} +
            \f{80}{d^6}
        \rt)
        \sum_{(i,j) \in \Sym(V_R, 2)}
        \deg(i,j)^2,
        \\
        \E_{\XoneR, \XtwoR}
        f_{3}^2
        &=
        \lt(
            \f{16}{d^4} +
            \f{48}{d^5}
        \rt)
        \sum_{(i,j) \in \Sym(V_R, 2)}
        \deg(i,j)^2,
        \\
        \E_{\XoneR, \XtwoR}
        f_{4}^2
        &=
        \lt(
            \f{12}{d^4} +
            \f{60}{d^5} +
            \f{72}{d^6}
        \rt)
        \sum_{(i,j) \in \Sym(V_R, 2)}
        \deg(i,j)^2 \\
        &\qquad +
        \lt(
            \f{8}{d^5} +
            \f{16}{d^6}
        \rt)
        \sum_{(i,j,k) \in \Sym(V_R, 3)}
        \deg(i,j) \deg(i,k),
        \\
        \E_{\XoneR, \XtwoR}
        f_{5}^2
        &=
        \lt(
            \f{24}{d^4} +
            \f{72}{d^5}
        \rt)
        \sum_{(i,j) \in \Sym(V_R, 2)}
        \deg(i,j)^2,
        \\
        \E_{\XoneR, \XtwoR}
        f_{6a}^2
        &=
        \lt(
            \f{4}{d^4} +
            \f{8}{d^5}
        \rt)
        \sum_{(i,j,k) \in \Sym(V_R, 3)}
        \deg(i,j,k)^2,
        \\
        \E_{\XoneR, \XtwoR}
        f_{7a}^2
        &=
        \lt(
            \f{2}{d^4} +
            \f{4}{d^5}
        \rt)
        \sum_{(i,j,k) \in \Sym(V_R, 3)}
        \deg(i,j,k)^2 \\
        &\qquad +
        \f{2}{d^6}
        \sum_{(i,j,k,\ell) \in \Sym(V_R, 4)}
        \deg(i,j,k) \deg(i,j,\ell),
        \\
        \E_{\XoneR, \XtwoR}
        f_{8}^2
        &=
        \lt(
            \f{4}{d^4} +
            \f{8}{d^5} +
            \f{12}{d^6}
        \rt)
        \sum_{(i,j,k,\ell) \in \Sym(V_R, 4)}
        \deg(i,j,k,\ell)^2.
    \end{align*}
    By symmetry $\E_{\XoneR, \XtwoR} f_{7a}^2 = \E_{X,X^{(2)}} f_{7b}^2$, and likewise for $f_{2a}$ and $f_{6a}$.
    Combining these bounds, we get that for some constant $C > 0$,
    \begin{align*}
        \E_{\XoneR}
        \sigZ{\XoneR}
        &\le
        C\bigg[
            d^{-4}
            \sum_{i,j,k,\ell \in V_R}
            \deg(i,j,k,\ell)^2
            +
            d^{-5}
            \sum_{i,j,k\in V_R}
            \deg(i,j) \deg(i,k) \\
            &\qquad + d^{-6}
            \sum_{i,j,k,\ell \in V_R}
            \deg(i,j,k) \deg(i,j,\ell)
        \bigg].
    \end{align*}
    The proposition now follows from Markov's inequality.
\end{proof}

Finally, we can prove Proposition~\ref{prop:bipartite-ub-s-hp}.
\begin{proof}[Proof of Proposition~\ref{prop:bipartite-ub-s-hp}]
The proposition follows from Propositions~\ref{prop:bipartite-ub-sop-high-prob}, \ref{prop:bipartite-ub-syv-high-prob}, \ref{prop:bipartite-ub-sze-high-prob}, and \ref{prop:bipartite-ub-szv-high-prob}.
\end{proof}

\section{Proofs of TV Lower Bounds}
\label{sec:tv-lower-bounds-proofs}

In this section, we will prove Theorems~\ref{thm:deg3-lb}, \ref{thm:deg4-lb}, and \ref{thm:maxdeg-lb}, which provide conditions under which $\TV\lt(W(G,d), M(G)\rt) \to 1$.
These bounds are witnessed by the degree 3 statistic $\kappa_3$, the degree 4 statistic $\kappa_4$, and longest row statistic $\kappa_r$, respectively.
The proofs of these theorems are a natural generalization of the method of \cite{BDER16}: we consider the hypothesis testing problem with hypotheses $M\sim M(G)$ and $M\sim W(G,d)$ and devise tests using these statistics that distinguish these two hypotheses with $o(1)$ type I+II error.
We carry out this task for Theorem~\ref{thm:deg3-lb} by bounding the mean and variance of the statistic $\kappa_3(M)$ under these hypotheses.
Then, by Chebyshev's inequality, the test that thresholds this statistic halfway between the two means distinguishes the hypotheses.
The proof of Theorem~\ref{thm:deg4-lb} proceeds similarly, albeit with a more involved computation.
The proof of Theorem~\ref{thm:maxdeg-lb} characterizes the distribution of $\kappa_r(M)$ under these hypotheses as, respectively, a $\chisq$ random variable scaled to mean $1$ and the product of two independent $\chisq$ random variables scaled to mean $1$.
We will show the latter distribution has larger fluctuations.
Then, a test that appropriately thresholds $|\kappa_r(M)-1|$ distinguishes the two distributions.

\subsection{Analysis of the Degree 3 Statistic}

In this section, we will prove Theorem~\ref{thm:deg3-lb}.
As discussed above, we will show the distributions $\kappa_3(M(G))$ and $\kappa_3(W(G,d))$ separate by computing their means and variances and applying Chebyshev's inequality.

\begin{lemma}
    \label{lem:deg3-lb-goe-stats}
    If $M \sim M(G)$, then $\E \kappa_3(M)=0$ and $\Var \kappa_3(M) = \Num_G(C_3)$.
\end{lemma}
\begin{proof}
    If $M\sim M(G)$, the entries $M_{i,j}$ for $(i,j)\in E(G)$ are mutually independent and centered.
    By linearity of expectation, $\E \kappa_3(M) = 0$.
    This implies that
    \[
        \Var \kappa_3(M) =
        \sum_{\substack{(i,j,k)\in C_3(G) \\ (i',j',k')\in C_3(G)}}
        \E\lt[
            M_{i,j} M_{j,k} M_{k,i}
            M_{i',j'} M_{j',k'} M_{k',i'}
        \rt].
    \]
    The last expectation is $1$ if $(i,j,k)=(i',j',k')$ and $0$ otherwise.
    Therefore, $\Var \kappa_3(M) = \Num_G(C_3)$.
\end{proof}

\begin{lemma}
    \label{lem:deg3-lb-wishart-stats}
    If $M\sim W(G,d)$, then $\E \kappa_3(M) = d^{-1/2} \Num_G(C_3)$ and there exists a constant $C$ such that
    \[
        \Var \kappa_3(M)
        \le
        C \lt(
            \Num_G(C_3) +
            d^{-1} \Num_G(C_3^{2,e}) +
            d^{-2} \Num_G(C_3^{2,v})
        \rt).
    \]
\end{lemma}
\begin{proof}
    Recall that for $i\neq j$, $M_{i,j} = d^{-1/2} \la X_i,X_j\ra$, where $X_1,\ldots,X_n$ are sampled i.i.d. from $\cN(0, I_d)$.
    For each $(i,j,k)\in C_3(G)$,
    \[
        \E \lt[M_{i,j} M_{j,k} M_{k,i}\rt]
        = d^{-3/2}
        \E \lt[
            \la X_i, X_j\ra
            \la X_j, X_k\ra
            \la X_k, X_i\ra
        \rt]
        = d^{-1/2}.
    \]
    By linearity of expectation, $\E \kappa_3(M) = d^{-1/2} \Num_G(C_3)$.
    We can compute the variance by expanding
    \begin{equation}
        \label{eq:deg3-lb-wishart-var-as-double-sum}
        \Var \kappa_3(M)
        =
        \sum_{\substack{(i,j,k)\in C_3(G) \\ (i',j',k')\in C_3(G)}}
        \lt[
            \E\lt[
                M_{i,j} M_{j,k} M_{k,i}
                M_{i',j'} M_{j',k'} M_{k',i'}
            \rt]
            - d^{-1}
        \rt].
    \end{equation}
    The expectation inside the sum depends on the shape of the graph formed by the six edges $(i,j)$, $(j,k)$, $(k,i)$, $(i',j')$, $(j',k')$, and $(k',i')$.
    If the sets $\{i,j,k\}$ and $\{i',j',k'\}$ do not intersect, then $M_{i,j} M_{j,k} M_{k,i}$ and $M_{i',j'} M_{j',k'} M_{k',i'}$ are independent, and so the summand in (\ref{eq:deg3-lb-wishart-var-as-double-sum}) is $0$.
    Otherwise, these six edges can form the graphs shown in Figure~\ref{fig:deg3-statistic-wishart-graphs}.
    For each graph these edges can form, we can compute the expectation $\E\lt[M_{i,j} M_{j,k} M_{k,i} M_{i',j'} M_{j',k'} M_{k',i'}\rt]$ by a standard Gaussian moment computation.
    We can also count the number of times each graph arises in the sum (\ref{eq:deg3-lb-wishart-var-as-double-sum}).
    Table~\ref{tab:deg3-lb-wishart-variance-analysis} summarizes this computation.
    Therefore, $\Var \kappa_3(M) \lesssim \Num_G(C_3) + d^{-1}\Num_G(C_3^{2,e}) + d^{-2} \Num_G(C_3^{2,v})$, as desired.

    \begin{figure}[h!]
    \centering
    \includegraphics{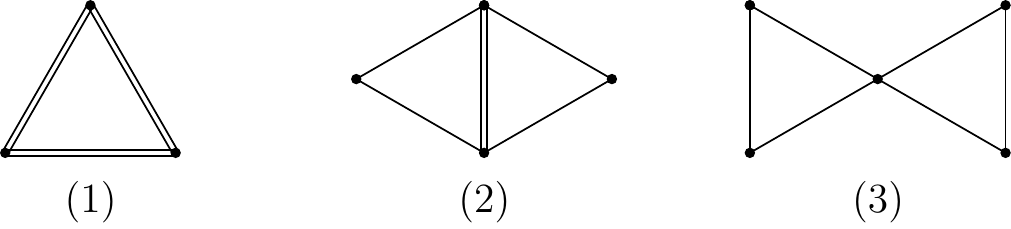}
    \caption{Graphs formed by the edges $(i,j)$, $(j,k)$, $(k,i)$, $(i',j')$, $(j',k')$, and $(k',i')$ yielding nonzero terms in (\ref{eq:deg3-lb-wishart-var-as-double-sum}).}
    \label{fig:deg3-statistic-wishart-graphs}
\end{figure}

    \begin{table}[h!]
    \centering
    \begin{tabular}{|c|c|c|}
    \hline
    Graph & $\E\lt[M_{i,j} M_{j,k} M_{k,i} M_{i',j'} M_{j',k'} M_{k',i'}\rt]$ & Number of occurrences in (\ref{eq:deg3-lb-wishart-var-as-double-sum}) \\
    \hline
    (1) & $1 + 10d^{-1} + 16d^{-2}$ & $O(\Num_G(C_3))$ \\ \hline
    (2) & $3d^{-1} + 6d^{-2}$ & $O(\Num_G(C_3^{2,e}))$ \\ \hline
    (3) & $d^{-1} + 2d^{-2}$ & $O(\Num_G(C_3^{2,v}))$ \\ \hline
    \end{tabular}
    \caption{Analysis of terms in variance of degree-3 statistic under Wishart distribution.}
    \label{tab:deg3-lb-wishart-variance-analysis}
\end{table}

\end{proof}

\begin{proof}[Proof of Theorem~\ref{thm:deg3-lb}]
    We will distinguish the hypotheses $M\sim M(G)$ and $M\sim W(G,d)$ with the following test.
    If $\kappa_3(M)\ge \f{1}{2} d^{-1/2}\Num_G(C_3)$, predict $W(G,d)$, and otherwise predict $M(G)$.

    By Lemma~\ref{lem:deg3-lb-goe-stats} and Chebyshev's inequality,
    the probability of predicting $W(G,d)$ when $M\sim M(G)$ is at most
    \[
        \P_{M\sim M(G)}
        \lt(\kappa_3(M) \ge \f{1}{2} d^{-1/2} \Num_G(C_3)\rt)
        \lesssim \f{\Num_G(C_3)}{\lt(\f{1}{2} d^{-1/2} \Num_G(C_3)\rt)^2}
        \lesssim \f{d}{\Num_G(C_3)}.
    \]
    By hypothesis (\ref{eq:deg3-lb-hypothesis-triangles}), this probability is $o(1)$.
    Similarly, by Lemma~\ref{lem:deg3-lb-wishart-stats} and Chebyshev's inequality, the probability of predicting $M(G)$ when $M\sim W(G,d)$ is bounded by
    \begin{align*}
        \P_{M\sim W(G,d)}
        \lt(\kappa_3(M) < \f{1}{2} d^{-1/2} \Num_G(C_3)\rt)
        &\lesssim
        \f{
            \Num_G(C_3) +
            d^{-1} \Num_G(C_3^{2,e}) +
            d^{-2} \Num_G(C_3^{2,v})
        }{
            \lt(
                \f{1}{2} d^{-1/2}
                \Num_G(C_3)
            \rt)^2
        } \\
        &\lesssim
        \f{d}{\Num_G(C_3)}
        + \f{
            \Num_G(C_3^{2,e}, C_3^{2,v})
        }{
            \Num_G(C_3)^2
        }.
    \end{align*}
    By the hypotheses (\ref{eq:deg3-lb-hypothesis-triangles}) and (\ref{eq:deg3-lb-hypothesis-regularity}), this probability is $o(1)$.
    Since this test separates $\kappa_3(M(G))$ and $\kappa_3(W(G,d))$ with $o(1)$ Type I+II error, we have $\TV\lt(\kappa_3(W(G,d)), \kappa_3(M(G))\rt) \to 1$.
    By the data processing inequality, $\TV\lt(W(G,d), M(G)\rt) \to 1$.
\end{proof}

\subsection{Analysis of the Degree 4 Statistic}

In this section, we will prove Theorem~\ref{thm:deg4-lb}.
Recall that the degree 4 statistic is defined by $\kappa_4(M) = \kappa_4^{C_4}(M) + \kappa_4^{P_2}(M) + \kappa_4^{E}(M)$, where the three constituent statistics are defined by
\begin{eqnarray*}
    \kappa_4^{C_4}(M)
    &=&
    \sum_{(i,j,k,\ell)\in C_4(G)} M_{i,j} M_{j,k} M_{k,\ell} M_{\ell,i}, \\
    \kappa_4^{P_2}(M)
    &=&
    \sum_{(i,j,k)\in P_2(G)} (M_{i,j}^2-1) (M_{j,k}^2-1), \\
    \kappa_4^{E}(M)
    &=&
    \sum_{(i,j)\in E(G)} (M_{i,j}^4-6M_{i,j}^2+3).
\end{eqnarray*}
Like for the degree 3 statistic, we will show the distributions $\kappa_4(M(G))$ and $\kappa_4(W(G,d))$ separate by computing their means and variances and applying Chebyshev's inequality.

\begin{lemma}
    \label{lem:deg4-lb-goe-stats}
    If $M \sim M(G)$, then $\E \kappa_4(M) = 0$. 
    Moreover, there exists a constant $C$ such that $\Var \kappa_4(M) \le C\Num_G(C_4, P_2, E)$.
\end{lemma}
\begin{proof}
    For i.i.d. standard Gaussians $g_1,g_2,g_3,g_4$, each of $g_1g_2g_3g_4$, $(g_1^2-1)(g_2^2-1)$, and $g_1^4-6g_1^2+3$ has mean $0$.
    So, by linearity of expectation, $\E \kappa_4(M) = 0$.
    Therefore, $\Var \kappa_4(M) = \E \kappa_4(M)^2$.
    Moreover, we compute that
    \begin{align*}
        \E \kappa_4(M)^2
        &=
        \sum_{(i,j,k,\ell)\in C_4(G)}
        \E\lt[
            M_{i,j}^2
            M_{j,k}^2
            M_{k,\ell}^2
            M_{\ell,i}^2
        \rt]
        +
        \sum_{(i,j,k)\in P_2(G)}
        \E\lt[
            \lt(M_{i,j}^2-1\rt)^2
            \lt(M_{j,k}^2-1\rt)^2
        \rt] \\
        &\qquad +
        \sum_{(i,j)\in E(G)}
        \E\lt[
            \lt(M_{i,j}^4-6M_{i,j}^2+3\rt)^2
        \rt] \\
        &= \Num_G(C_4) + 4\Num_G(P_2) + 6\Num_G(E)
        \le 6\Num_G(C_4, P_2, E)
    \end{align*}
    because the cross terms all vanish.
\end{proof}

\begin{lemma}
    \label{lem:deg4-lb-wishart-stats}
    If $M \sim W(G,d)$,
    then $\E \kappa_4(M) \ge d^{-1} \Num_G(C_4, P_2, E)$ and there exists a constant $C$ such that
    \begin{align*}
        \Var \kappa_4(M)
        &\le
        C \bigg(
            \Num_G(C_4, P_2, E) +
            d^{-1} \Num_G(C_4, P_2)^{3/2} \\
        &\qquad
            + d^{-2} \Num_G(K_{1,4}, K_{2,4}, C_4^{2,e}) +
            d^{-3} \Num_G(C_4^{2,v})
        \bigg).
    \end{align*}
\end{lemma}
We will prove this lemma by analyzing the three constituent statistics of $\kappa_4$ in the propositions below.

\begin{proposition}
    \label{prop:deg4-lb-4cycs-wishart-stats}
    If $M\sim W(G,d)$, then $\E \kappa_4^{C_4}(M) = d^{-1}\Num_G(C_4)$ and there exists a constant $C$ such that
    \[
        \Var \kappa_4^{C_4}(M) \le C \lt(
            \Num_G(C_4) +
            d^{-1} \Num_G(K_{2,3}) +
            d^{-2} \Num_G(K_{2,4}, C_4^{2,e}, C_4^{2,ev}) +
            d^{-3} \Num_G(C_4^{2, v})
        \rt).
    \]
\end{proposition}
\begin{proof}
    For $i\neq j$, we have $M_{i,j} = d^{-1/2} \la X_i,X_j\ra$, where $X_1,\ldots,X_n$ are sampled i.i.d. from $\cN(0, I_d)$.
    By a Gaussian moment computation, for each $(i,j,k,\ell)\in C_4(G)$ we have
    \[
        \E \lt[ M_{i,j} M_{j,k} M_{k,\ell} M_{\ell,i} \rt]
        = d^{-2} \E \lt[
            \la X_i, X_j \ra
            \la X_j, X_k \ra
            \la X_k, X_\ell \ra
            \la X_\ell, X_i \ra
        \rt]
        = d^{-1}.
    \]
    By linearity of expectation, we have $\E \kappa_4^{C_4}(M) = d^{-1} \Num_G(C_4)$.
    We can control the variance by expanding
    \begin{equation}
        \label{eq:deg4-lb-4cycles-wishart-var-as-double-sum}
        \Var \kappa_4^{C_4}(M)
        =
        \sum_{\substack{
            (i,j,k,\ell)\in C_4(G) \\
            (i',j',k',\ell')\in C_4(G)
        }}
        \lt(
            \E\lt[
                M_{i,j} M_{j,k} M_{k,\ell} M_{\ell,i}
                M_{i',j'} M_{j',k'} M_{k',\ell'} M_{\ell',i'}
            \rt] - d^{-2}
        \rt).
    \end{equation}
    The expectation inside this sum depends on the shape of the graph formed by the eight edges $(i,j)$, $(j,k)$, $(k,\ell)$, $(\ell, i)$, $(i',j')$, $(j',k')$, $(k',\ell')$, and $(\ell', i')$.
    If the sets $\{i,j,k,\ell\}$ and $\{i',j',k',\ell'\}$ do not intersect, then $M_{i,j} M_{j,k} M_{k,\ell} M_{\ell, i}$ and $M_{i',j'} M_{j',k'} M_{k',\ell'} M_{\ell',i'}$ are independent, and so the summand in (\ref{eq:deg4-lb-4cycles-wishart-var-as-double-sum}) is $0$. Otherwise, these eight edges can form the graphs shown in Figure~\ref{fig:deg4-4cycles-statistic-wishart-graphs}.
    We can compute the expectation $\E\lt[M_{i,j} M_{j,k} M_{k,\ell} M_{\ell,i} M_{i',j'} M_{j',k'} M_{k',\ell'} M_{\ell',i'}\rt]$ for each such graph by a Gaussian moment computation and count the number of times each graph occurs in the sum (\ref{eq:deg4-lb-4cycles-wishart-var-as-double-sum}).
    Table~\ref{tab:deg4-lb-4cycles-wishart-variance-analysis} summarizes this computation.
    This completes the proof of the proposition.

    \begin{figure}[h!]
    \centering
    \includegraphics{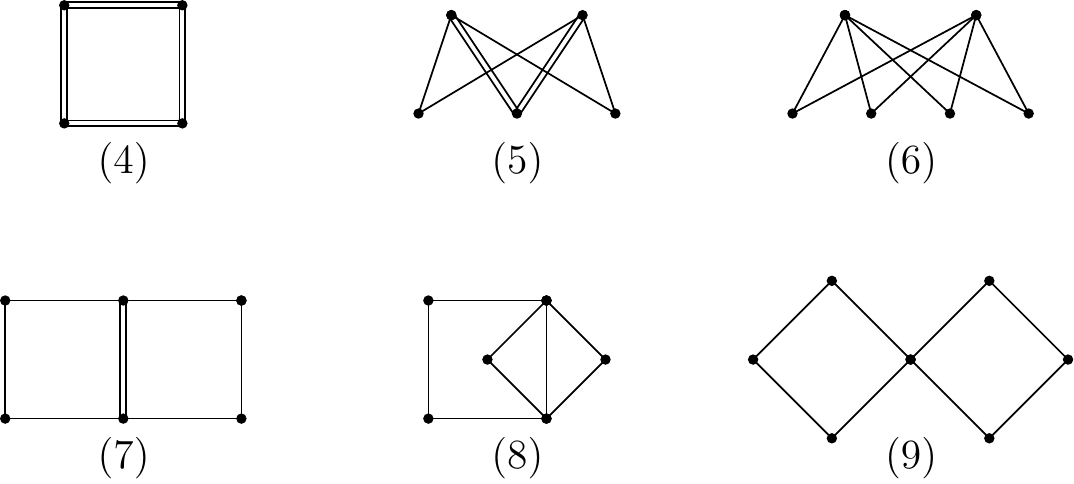}
    \caption{Graphs formed by the edges $(i,j)$, $(j,k)$, $(k,\ell)$, $(\ell, i)$, $(i',j')$, $(j',k')$, $(k',\ell')$, and $(\ell',i')$ yielding nonzero terms in (\ref{eq:deg4-lb-4cycles-wishart-var-as-double-sum}).}
    \label{fig:deg4-4cycles-statistic-wishart-graphs}
\end{figure}

    \begin{table}[h!]
    \centering
    \begin{tabular}{|c|c|c|}
        \hline
        Graph &
        $\E\lt[
             M_{i,j} M_{j,k} M_{k,\ell} M_{\ell,i}
             M_{i',j'} M_{j',k'} M_{k',\ell'} M_{\ell',i'}
        \rt]$ &
        Number of occurrences in (\ref{eq:deg4-lb-4cycles-wishart-var-as-double-sum})\\
        \hline
        (4) & $1 + 8d^{-1} + 32d^{-2} + 40d^{-3}$ & $O(\Num_G(C_4))$ \\ \hline
        (5) & $d^{-1} + 10d^{-2} + 16d^{-3}$ & $O(\Num_G(K_{2,3}))$ \\ \hline
        (6) & $3d^{-2} + 6d^{-3}$ & $O(\Num_G(K_{2,4}))$ \\ \hline
        (7) & $3d^{-2} + 6d^{-3}$ & $O(\Num_G(C_4^{2,e}))$ \\ \hline
        (8) & $3d^{-2} + 6d^{-3}$ & $O(\Num_G(C_4^{2,ev}))$ \\ \hline
        (9) &  $d^{-2} + 2d^{-3}$ & $O(\Num_G(C_4^{2,v}))$ \\ \hline
    \end{tabular}
    \caption{Analysis of terms in variance of 4-cycles statistic under Wishart distribution.}
    \label{tab:deg4-lb-4cycles-wishart-variance-analysis}
\end{table}

\end{proof}

\begin{proposition}
    \label{prop:deg4-lb-2paths-wishart-stats}
    If $M\sim W(G,d)$, then $\E \kappa_4^{P_2}(M) = 2d^{-1}\Num_G(P_2)$ and there exists a constant $C$ such that
    \begin{align*}
        \Var \kappa_4^{P_2}(M) 
        &\le C \bigg(
            \Num_G(P_2) +
            d^{-1} \Num_G(K_{1,3}, C_3) +
            d^{-2} \Num_G(K_{1,4}, C_4,C_3^{+}, P_3) \\
            &\qquad +
            d^{-3} \Num_G(K_{1,3}^{+}), P_4)
        \bigg).
    \end{align*}
\end{proposition}
\begin{proof}[Proof of Proposition~\ref{prop:deg4-lb-2paths-wishart-stats}]
    For each $(i,j,k)\in P_2(G)$, we can compute that
    \[
        \E \lt[
            \lt(M_{i,j}^2-1\rt)
            \lt(M_{j,k}^2-1\rt)
        \rt]
        =
        \E \lt[
            \lt(d^{-1} \la X_i,X_j\ra^2-1\rt)
            \lt(d^{-1} \la X_j,X_k\ra^2-1\rt)
        \rt]
        = 2d^{-1}.
    \]
    By linearity of expectation, $\E \kappa_4^{P_2}(M) = 2d^{-1} \Num_G(P_2)$.
    We can bound the variance by expanding
    \begin{equation}
        \label{eq:deg4-lb-2paths-wishart-var-as-double-sum}
        \Var \kappa_4^{P_2}(M)
        =
        \sum_{\substack{(i,j,k)\in P_2(G) \\ (i',j',k')\in P_2(G)}}
        \lt(
            \E\lt[
                \lt(M_{i,j}^2-1\rt)
                \lt(M_{j,k}^2-1\rt)
                \lt(M_{i',j'}^2-1\rt)
                \lt(M_{j',k'}^2-1\rt)
            \rt]
            - 4d^{-2}
        \rt).
    \end{equation}
    The expectation inside this sum depends on the shape of the graph formed by the four edges $(i,j)$, $(j,k)$, $(i',j')$, $(j',k')$.
    If $\{i,j,k\}$ and $\{i',j',k'\}$ do not intersect, the summand in (\ref{eq:deg4-lb-2paths-wishart-var-as-double-sum}) is $0$.
    Otherwise, these edges can form the graphs shown in Figure~\ref{fig:deg4-2paths-statistic-wishart-graphs}.
    For each graph, we can compute the value of $\E\lt[ \lt(M_{i,j}^2-1\rt) \lt(M_{j,k}^2-1\rt) \lt(M_{i',j'}^2-1\rt) \lt(M_{j',k'}^2-1\rt)\rt]$ and count the number of times it appears in the sum (\ref{eq:deg4-lb-2paths-wishart-var-as-double-sum}).
    Table~\ref{tab:deg4-lb-2paths-wishart-variance-analysis} summarizes this computation.
    This completes the proof of the proposition.

    \begin{figure}[h!]
    \centering
    \includegraphics{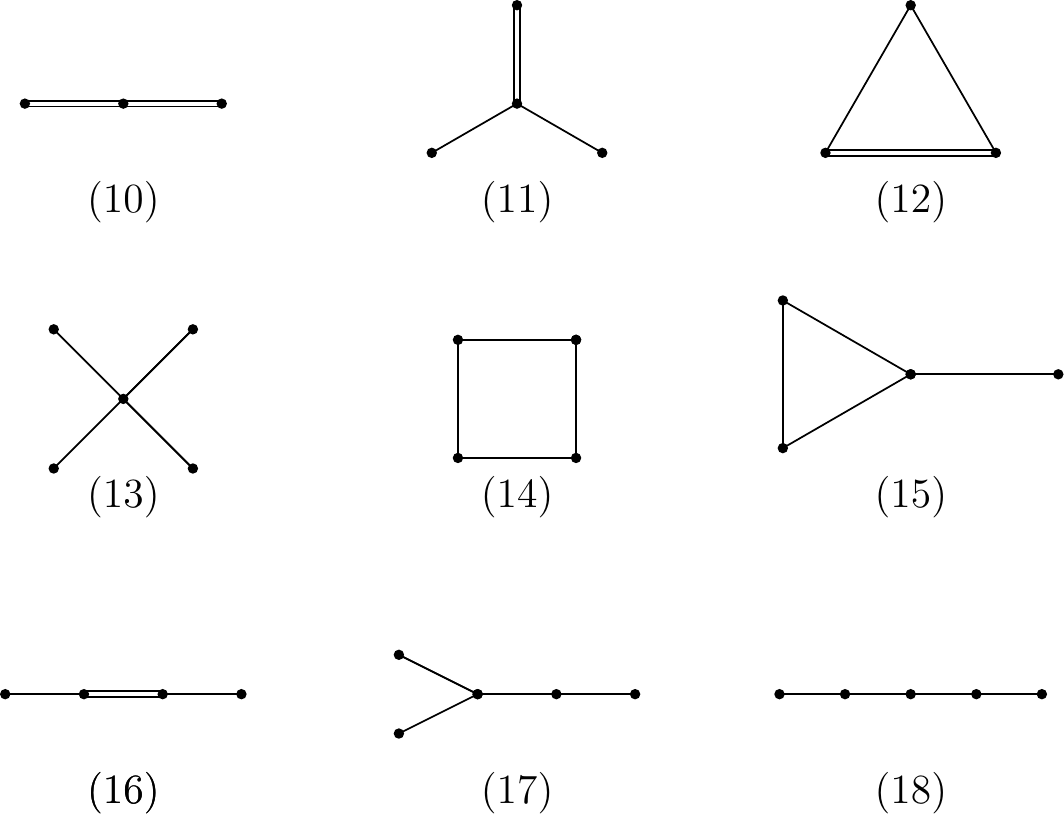}
    \caption{Graphs formed by the edges $(i,j)$, $(j,k)$, $(i',j')$, and $(j',k')$ yielding nonzero terms in (\ref{eq:deg4-lb-2paths-wishart-var-as-double-sum}).}
    \label{fig:deg4-2paths-statistic-wishart-graphs}
\end{figure}

    \begin{table}[h!]
    \centering
    \begin{tabular}{|c|c|c|}
        \hline
        Graph &
        $\E\lt[
            \lt(M_{i,j}^2-1\rt)
            \lt(M_{j,k}^2-1\rt)
            \lt(M_{i',j'}^2-1\rt)
            \lt(M_{j',k'}^2-1\rt)
        \rt]$ &
        Number of occurrences in (\ref{eq:deg4-lb-2paths-wishart-var-as-double-sum})\\
        \hline
        (10) & $4 + 56d^{-1} + 300d^{-2} + 432d^{-3}$ & $O(\Num_G(P_2))$ \\ \hline
        (11) & $4d^{-1} + 68d^{-2} + 144d^{-3}$ & $O(\Num_G(K_{1,3}))$ \\ \hline
        (12) & $20d^{-1} + 196d^{-2} + 336d^{-3}$ & $O(\Num_G(C_3))$ \\ \hline
        (13) & $12d^{-2} + 48d^{-3}$ & $O(\Num_G(K_{1,4}))$ \\ \hline
        (14) & $16d^{-2} + 40d^{-3}$ & $O(\Num_G(C_4))$ \\ \hline
        (15) & $24d^{-2} + 64d^{-3}$ & $O(\Num_G(C_3^+))$ \\ \hline
        (16) & $40d^{-2} + 96d^{-3}$ & $O(\Num_G(P_3))$ \\ \hline
        (17) & $4d^{-2} + 16d^{-3}$ & $O(\Num_G(K_{1,3}^+))$ \\ \hline
        (18) & $4d^{-2} + 8d^{-3}$ & $O(\Num_G(P_4))$ \\ \hline
    \end{tabular}
    \caption{Analysis of terms in variance of 2-paths statistic under Wishart distribution.}
    \label{tab:deg4-lb-2paths-wishart-variance-analysis}
\end{table}

\end{proof}

\begin{proposition}
    \label{prop:deg4-lb-edges-wishart-stats}
    If $M\sim W(G,d)$, then $\E \kappa_4^{E}(M) = 6d^{-1}\Num_G(E)$ and there exists a constant $C$ such that
    \[
        \Var \kappa_4^{E}(M) \le C \lt(
            \Num_G(E) + d^{-2} \Num_G(P_2)
        \rt).
    \]
\end{proposition}
\begin{proof}[Proof of Proposition~\ref{prop:deg4-lb-edges-wishart-stats}]
    For each $(i,j)\in E(G)$, we can compute that
    \[
        \E \lt[M_{i,j}^4-6M_{i,j}^2+3\rt]
        =
        \E \lt[
            d^{-2} \la X_i,X_j\ra ^4 -
            6d^{-1} \la X_i,X_j\ra^2 +
            3
        \rt]
        = 6d^{-1}.
    \]
    By linearity of expectation, $\E \kappa_4^{E} = 6d^{-1} \Num_G(E)$.
    We can bound the variance by expanding
    \begin{equation}
        \label{eq:deg4-lb-edges-wishart-var-as-double-sum}
        \Var \kappa_4^{E}(M)
        =
        \sum_{\substack{(i,j)\in E(G) \\ (i',j')\in E(G)}}
        \lt(
            \E\lt[
                \lt(M_{i,j}^4-6M_{i,j}^2+3\rt)
                \lt(M_{i',j'}^4-6M_{i',j'}^2+3\rt)
            \rt]
            - 36d^{-2}
        \rt).
    \end{equation}
    The expectation inside this sum depends on the shape of the graph formed by the edges $(i,j)$ and $(i',j')$.
    If $\{i,j\}$ and $\{i',j'\}$ do not intersect, the summand in (\ref{eq:deg4-lb-edges-wishart-var-as-double-sum}) is $0$.
    Otherwise, these edges can form the graphs shown in Figure~\ref{fig:deg4-edges-statistic-wishart-graphs}.
    We can compute $\E\lt[ \lt(M_{i,j}^4 - 6M_{i,j}^2 + 3\rt) \lt(M_{i',j'}^4 - 6M_{i',j'}^2 + 3\rt)\rt]$ for each such graph and count the number of times it appears in the sum (\ref{eq:deg4-lb-edges-wishart-var-as-double-sum}).
    Table~\ref{tab:deg4-lb-edges-wishart-variance-analysis} summarizes this computation.
    This completes the proof of the proposition.

    \begin{figure}[h!]
    \centering
    % \begin{asy}
    %     unitsize(1.0cm);
    %     pair LABEL_SHIFT = (2, 0.5);
    %     pair CENTERING_SHIFT = (2, 1.5);
    %     void dot_shift(pair p, pair dif) {
    %         dot(p + dif + CENTERING_SHIFT);
    %     }
    %     void draw_shift(path p, pair dif) {
    %         draw(shift(dif + CENTERING_SHIFT) * (p));
    %     }
    %     void label_shift(string s, pair dif) {
    %         label(s, dif + LABEL_SHIFT);
    %     }
    %     void dot_path(string s, pair[] p, pair dif) {
    %         int n = p.length;
    %         label_shift(s, dif);
    %         for (int i=0; i < n-1; ++i) {
    %             dot_shift(p[i], dif);
    %             draw_shift(p[i]--p[i+1], dif);
    %         }
    %         dot_shift(p[n-1], dif);
    %     }
    %     real EPS = 0.03;
    %     pair vec090 = (0, 1);
    %     void double_draw_shift(path p, pair normal, pair dif) {
    %         draw_shift(p, dif + EPS * normal);
    %         draw_shift(p, dif - EPS * normal);
    %     }
    %     // E
    %     pair E_SHIFT = (0, 0);
    %     dot_shift((-0.5, 0), E_SHIFT);
    %     dot_shift((0.5, 0), E_SHIFT);
    %     double_draw_shift((-0.5, 0)--(0.5, 0), vec090, E_SHIFT);
    %     label_shift("(19)", E_SHIFT);
    %     // P2
    %     pair P2_SHIFT = (4, 0);
    %     pair[] P2_VERTICES = {(-1, 0), (0, 0), (1, 0)};
    %     dot_path("(20)", P2_VERTICES, P2_SHIFT);
    % \end{asy}
    \includegraphics{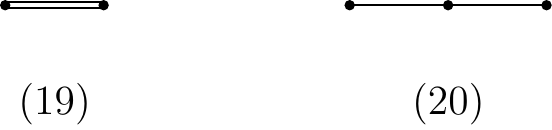}
    \caption{Graphs formed by the edges $(i,j)$ and $(i',j')$ yielding nonzero terms in (\ref{eq:deg4-lb-edges-wishart-var-as-double-sum}).}
    \label{fig:deg4-edges-statistic-wishart-graphs}
\end{figure}

    \begin{table}[h!]
    \centering
    \begin{tabular}{|c|c|c|}
        \hline
        Graph &
        $\E\lt[
            \lt(M_{i,j}^4-6M_{i,j}^2+3\rt)
            \lt(M_{i',j'}^4-6M_{i',j'}^2+3\rt)
        \rt]$ &
        Number of occurrences in (\ref{eq:deg4-lb-edges-wishart-var-as-double-sum})\\
        \hline
        (19) & $24 + 432d^{-1} + 3180d^{-2} + 5040d^{-3}$ & $O(\Num_G(E))$ \\ \hline
        (20) & $108d^{-2} + 432d^{-3}$ & $O(\Num_G(P_2))$ \\ \hline
    \end{tabular}
    \caption{Analysis of terms in variance of edges statistic under Wishart distribution.}
    \label{tab:deg4-lb-edges-wishart-variance-analysis}
\end{table}

\end{proof}

We can now prove Lemma~\ref{lem:deg4-lb-wishart-stats} and Theorem~\ref{thm:deg4-lb}, which follow readily from these propositions.
\begin{proof}[Proof of Lemma~\ref{lem:deg4-lb-wishart-stats}]
    By linearity of expectation and Propositions~\ref{prop:deg4-lb-4cycs-wishart-stats}, \ref{prop:deg4-lb-2paths-wishart-stats}, and \ref{prop:deg4-lb-edges-wishart-stats},
    \[
        \E \kappa_4(M) =
        \E \kappa_4^{C_4}(M) +
        \E \kappa_4^{P_2}(M) +
        \E \kappa_4^{E}(M)
        \ge
        d^{-1} \Num_G(C_4, P_2, E),
    \]
    as desired.
    By the inequality $(x+y+z)^2 \le 3(x^2+y^2+z^2)$ applied pointwise,
    \begin{align*}
        \Var \kappa_4(M)
        &=
        \E \lt[\lt(
            \lt(\kappa_4^{C_4}(M) - \E \kappa_4^{C_4}(M)\rt) +
            \lt(\kappa_4^{P_2}(M) - \E \kappa_4^{P_2}(M)\rt) +
            \lt(\kappa_4^{E}(M) - \E \kappa_4^{E}(M)\rt)
        \rt)^2\rt] \\
        &\le
        3 \E \lt[
            \lt(\kappa_4^{C_4}(M) - \E \kappa_4^{C_4}(M)\rt)^2 +
            \lt(\kappa_4^{P_2}(M) - \E \kappa_4^{P_2}(M)\rt)^2 +
            \lt(\kappa_4^{E}(M) - \E \kappa_4^{E}(M)\rt)^2
        \rt] \\
        &=
        3 \Var \kappa_4^{C_4}(M) +
        3 \Var \kappa_4^{P_2}(M) +
        3 \Var \kappa_4^{E}(M) \\
        &\lesssim
        \Num_G(C_4, P_2, E) +
        d^{-1} \Num_G(K_{1,3}, K_{2,3}, C_3) \\
        &\qquad +
        d^{-2} \Num_G(K_{1,4}, K_{2,4}, C_4^{2,e}, C_4^{2,ev}, C_3^+, P_3) +
        d^{-3} \Num_G(C_4^{2,v}, K_{1,3}^{+}, P_4).
    \end{align*}
    By Lemma~\ref{lem:subgraph-statistics}(\ref{itm:ss-c3},\ref{itm:ss-k13-p3-c3+},\ref{itm:ss-k13+-p4},\ref{itm:ss-k23},\ref{itm:ss-c42ev}), this simplifies to the desired bound on $\Var \kappa_4(M)$.
\end{proof}

\begin{proof}[Proof of Theorem~\ref{thm:deg4-lb}]
    We will distinguish the hypotheses $M\sim M(G)$ and $M\sim W(G,d)$ with the following test.
    If $\kappa_4(M) \ge \f{1}{2} d^{-1} \Num_G(C_4, P_2, E)$, predict $W(G,d)$, and otherwise predict $M(G)$.

    By Lemma~\ref{lem:deg4-lb-goe-stats} and Chebyshev's inequality, the probability of predicting $W(G,d)$ when $M\sim M(G)$ is at most
    \[
        \P_{M\sim M(G)}
        \lt(
            \kappa_4(M) \ge
            \f{1}{2} d^{-1} \Num_G(C_4, P_2, E)
        \rt)
        \lesssim
        \f{\Num_G(C_4, P_2, E)}{\lt(\f{1}{2} d^{-1} \Num_G(C_4, P_2, E)\rt)^2}
        \lesssim
        \f{d^2}{\Num_G(C_4, P_2, E)}.
    \]
    By hypothesis (\ref{eq:deg4-lb-hypothesis-4cycs}), this probability is $o(1)$.
    By Lemma~\ref{lem:deg4-lb-wishart-stats} and Chebyshev's inequality, the probability of predicting $M(G)$ when $M\sim W(G,d)$ is bounded by
    \begin{align*}
    &\P_{M\sim W(G,d)} \lt(\kappa_4(M) < \f{1}{2d}\Num_G(C_4, P_2, E)\rt) \\
    &\qquad \lesssim
    \f{
        \Num_G(C_4, P_2, E) +
        d^{-1} \Num_G(C_4, P_2)^{3/2} +
        d^{-2} \Num_G(K_{1,4},K_{2,4},C_4^{2,e}) +
        d^{-3} \Num_G(C_4^{2,v})
    }{\lt(\f{1}{2d} \Num_G(C_4, P_2, E)\rt)^2} \\
    &\qquad \lesssim
    \f{d^2}{\Num_G(C_4, P_2, E)} +
    \f{d}{\Num_G(C_4, P_2, E)^{1/2}} +
    \f{\Num_G(K_{1,4},K_{2,4},C_4^{2,e},C_4^{2,v})}{\Num_G(C_4, P_2, E)^2}.
    \end{align*}
    By hypotheses (\ref{eq:deg4-lb-hypothesis-4cycs}) and (\ref{eq:deg4-lb-hypothesis-regularity}), this probability is $o(1)$.
    Since this test separates $\kappa_4(M(G))$ and $\kappa_4(W(G,d))$ with $o(1)$ Type I+II error, we have $\TV\lt(\kappa_4(W(G,d)), \kappa_4(M(G))\rt) \to 1$.
    By the data processing inequality, $\TV\lt(W(G,d), M(G)\rt) \to 1$.
\end{proof}

\subsection{Analysis of the Longest Row Statistic}

In this section, we will prove Theorem~\ref{thm:maxdeg-lb}.
Throughout this subsection, let $\vst$ be the vertex of maximal degree in $G$, and let $D = \deg(\vst) = \maxdeg(G)$.
Thus $d \ll D$.
We will characterize the distributions of $\kappa_r(M(G))$ and $\kappa_r(W(G,d))$, which both concentrate around $1$ but have different fluctuations.
Throughout this section, let $\chisq(k)$ denote the $\chisq$ random variable with $k$ degrees of freedom.

\begin{lemma}
    \label{lem:longest-row-lb-goe-stats}
    If $M\sim M(G)$, then $\kappa_r(M) =_d \f{1}{D} \chisq(D)$.
\end{lemma}
\begin{proof}
    If $M\sim M(G)$, then the entries $M_{\vst,i}$ are i.i.d. and standard Gaussian.
    Therefore,
    \[
        \kappa_r(M)
        = \f{1}{D} \sum_{i\in N(\vst)} M_{\vst,i}^2
        =_d \f{1}{D} \chisq(D).
    \]
\end{proof}

\begin{lemma}
    \label{lem:longest-row-lb-wishart-stats}
    If $M\sim W(G,d)$, then
    $\kappa_r(M) =_d \f{1}{D} \chisq(D) \cdot \f{1}{d}\chisq(d)$,
    where the two $\chisq$ variables are independent.
\end{lemma}
\begin{proof}
    Recall that for $i\neq j$, $M_{i,j} = d^{-1/2} \la X_i,X_j\ra$, where $X_1,\ldots,X_n$ are sampled i.i.d. from $\cN(0, I_d)$.
    Let $\pi(X_i; X_{\vst})$ be the signed length of the projection of $X_i$ onto $X_{\vst}$, so that $\la X_i, X_{\vst}\ra = \norm{X_{\vst}}_2 \pi(X_i; X_{\vst})$.
    Then, we have
    \[
        \kappa_r(M)
        = \f{1}{dD} \sum_{i\in N(\vst)} \la X_i, X_{\vst}\ra^2
        = \f{1}{d} \norm{X_{\vst}}_2^2 \cdot \f{1}{D} \sum_{i\in N(\vst)} \pi(X_i; X_{\vst})^2.
    \]
    Since the $X_i$ are i.i.d. and Gaussian distributions are rotationally invariant, conditioned on any $X_{\vst}$ the projections $\pi(X_i; X_{\vst})$ for $i\in N(\vst)$ are distributed as i.i.d. standard Gaussians.
    So, $\kappa_r(M)$ has the claimed distribution.
\end{proof}

\begin{proof}[Proof of Theorem~\ref{thm:maxdeg-lb}]
    We will distinguish the hypotheses $M\sim M(G)$ and $M\sim W(G,d)$ with the following test.
    Let $\eps = \lt(d D\rt)^{-1/4}$.
    If $|\kappa_r(M)-1| \le \eps$, predict $M(G)$, and otherwise predict $W(G,d)$.
    The motivation for this test is that under $M\sim M(G)$, $\kappa_r(M)$ has mean $1$ and fluctuations of scale $D^{-1/2}$, while under $M\sim W(G,d)$, $\kappa_r(M)$ has mean $1$ and fluctuations of scale $d^{-1/2}$, which is much larger because $d \ll D$.

    The distribution $\f{1}{D} \chisq(D)$ has mean $1$ and variance $2/D$.
    By Chebyshev's inequality,
    \begin{equation}
        \label{eq:longest-row-lb-chisq-tail-ub}
        \f{1}{D} \chisq(D)
        \in
        \lt[1 - \eps, 1 + \eps\rt]
    \end{equation}
    except with probability $\f{2/D}{\eps^2} = 2(d/D)^{1/2}$.
    Because $d \ll D$, this probability is $o(1)$.

    We will show that
    \begin{equation}
        \label{eq:longest-row-lb-chisq-anticoncentration}
        \f{1}{d}\chisq(d) \not\in [1-3\eps, 1+3\eps]
    \end{equation}
    except with probability $o(1)$.
    If $d = \omega(1)$, then $\f{\chisq(d) - d}{\sqrt{2d}} \rightarrow_d \cN(0,1)$ by the Central Limit Theorem.
    Since $d \ll D$, we have $\eps \ll d^{-1/2}$.
    This implies that
    \[
        \P\lt(
            \f{1}{d}\chisq(d)
            \in
            [1-3\eps, 1+3\eps]
        \rt)
        =
        \P\lt(
            \f{\chisq(d) - d}{\sqrt{2d}}
            \in
            \lt[
                -\f{3\eps d^{1/2}}{\sqrt{2}},
                \f{3\eps d^{1/2}}{\sqrt{2}}
            \rt]
        \rt)
        = o(1),
    \]
    because the interval $\lt[ -\f{3\eps d^{1/2}}{\sqrt{2}}, \f{3\eps d^{1/2}}{\sqrt{2}} \rt]$ shrinks to width $0$ as $n \to \infty$.
    Otherwise, we have $d=O(1)$ and the interval $[1-3\eps, 1+3\eps]$ shrinks to width $0$ as $n\to\infty$ (because $D \gg d$), while $\f{1}{d} \chisq(d)$ is one of a finite number of random variables.
    So, (\ref{eq:longest-row-lb-chisq-anticoncentration}) still holds except with probability $o(1)$.

    If $M\sim M(G)$, then $\kappa_r(M) =_d \f{1}{D} \chisq(D)$ by Lemma~\ref{lem:longest-row-lb-goe-stats}.
    Because (\ref{eq:longest-row-lb-chisq-tail-ub}) holds with probability $1-o(1)$, the probability of predicting $W(G,d)$ is $o(1)$.

    If $M\sim W(G,d)$, then $\kappa_r(M) =_d \f{1}{D} \chisq(D) \cdot \f{1}{d} \chisq(d)$ by Lemma~\ref{lem:longest-row-lb-wishart-stats}.
    By a union bound, (\ref{eq:longest-row-lb-chisq-tail-ub}) and (\ref{eq:longest-row-lb-chisq-anticoncentration}) simultaneously hold with probability $1-o(1)$.
    On this event,
    \[
        \kappa_r(M)
        \not\in \lt[
            \lt(1 - 3\eps\rt)\lt(1 +  \eps\rt),
            \lt(1 + 3\eps\rt)\lt(1 -  \eps\rt)
        \rt]
        \supset
        \lt[1 - \eps, 1 + \eps\rt]
    \]
    for all sufficiently large $n$.
    Therefore, the probability of predicting $M(G)$ is $o(1)$.
\end{proof}

\bibliographystyle{alpha}

\appendix

\section{Deferred Proofs from Section~\ref{sec:general-ub-main-argument}}
\label{appsec:kl-conditioning}

In this section, we will prove Lemma~\ref{lem:general-ub-computational} below, which gives several bounds on expectations with respect to the Wishart distribution that we use in the proof of Theorem~\ref{thm:general-weaker-ub}.
These bounds are used in Section~\ref{subsec:general-ub-s-hp-proof} to show that the set $S$ we choose has high probability. 
They are also used in the proof of Proposition~\ref{prop:general-ub-kl-successive-differences}, which bounds the successive differences $\KL_A-\KL_B$, $\KL_B-\KL_C$, $\KL_C-\KL_D$, and $\KL_D-\KL_E$ in the proof of Lemma~\ref{lem:general-ub-kl-conditioning}.
In this section, we also present the deferred proof of Proposition~\ref{prop:general-ub-kl-successive-differences}.

\begin{lemma}
    \label{lem:general-ub-computational}
    Let $M = X^\top X$, where $X\in \bR^{d\times k}$ is a matrix of i.i.d. standard Gaussians.
    Then the following bounds hold.
    \begin{enumerate}[label=(\alph*), ref=\alph*]
      \item \label{itm:general-ub-computational-trsq-centered} For all $d,k$, we have that $\E \Tr(d^{-1} M - I_k)^2 = \f{2k}{d}$.
      \item \label{itm:general-ub-computational-trsq} For all $d,k$, we have that $\E \Tr(d^{-1} M)^2 = k^2 + \f{2k}{d}$.
      \item \label{itm:general-ub-computational-trdelsq} For all $d,k$, we have that $\E \Tr\lt((d^{-1}M - I_k)^2\rt) = \f{k^2+k}{d}$.
      \item \label{itm:general-ub-computational-vartrdelsq} If $d\ge k$, then $\Var \Tr\lt((d^{-1}M - I_k)^2\rt) \le \f{56k^2}{d^2}$.
      \item \label{itm:general-ub-computational-det} If $d\ge 2k+2$, then $\E \det(d^{-1} M)^{-1} \le e^k$.
      \item \label{itm:general-ub-computational-log2det} If $d\ge 2k+2$, then $\E \log^2 \det(d^{-1} M) \le 3k$.
    \end{enumerate}
\end{lemma}
\begin{proof}
    We will prove each part in turn. 
    
    \paragraph{Proof of (\ref{itm:general-ub-computational-trsq-centered}).}
    By standard facts about moments of Gaussians, we have that
    \[
        \E \Tr\lt(d^{-1} M - I_k\rt)^2
        =
        \E \lt[\f{1}{d^2} \lt(
            \sum_{i=1}^d
            \sum_{j=1}^k
            (X_{i,j}^2-1)
        \rt)^2\rt]
        =
        \f{2kd}{d^2}
        =
        \f{2k}{d}.
    \]

    \paragraph{Proof of (\ref{itm:general-ub-computational-trsq}).}
    Since $\E \Tr\lt(d^{-1} M - I_k\rt) = 0$, we have, using part (\ref{itm:general-ub-computational-trsq-centered}),
    \[
        \E \Tr\lt(d^{-1} M\rt)^2
        =
        \E \lt(\Tr\lt(d^{-1} M - I_k\rt) + k\rt)^2
        =
        k^2 +
        \E \Tr\lt(d^{-1} M - I_k\rt)^2
        =
        k^2 + \f{2k}{d}.
    \]
    
    \paragraph{Proof of (\ref{itm:general-ub-computational-trdelsq}, \ref{itm:general-ub-computational-vartrdelsq}).}
    Let $X_1,\ldots,X_k \in \bR^d$ be the columns of $X$.
    First, observe that
    \begin{equation}
        \label{eq:app-general-ub-computational-trdelsq-combinatorial-expansion}
        \Tr\lt((d^{-1}M - I_k)^2\rt)
        = 
        \norm{d^{-1}M - I_k}_{\HS}^2
        = 
        \f{1}{d^2}
        \sum_{i,j=1}^k
        \lt(\la X_i, X_j\ra - d\delta_{i,j}\rt)^2.
    \end{equation}
    Direct computations with moments of Gaussians yield that
    \begin{align*}
        \E[(\la X_i, X_i\ra - d)^2]
        &= 2d,
        & \text{} \\
        \E[\la X_i, X_j\ra^2]
        &= d,
        & \text{if $i\neq j$} \\
        \E[(\la X_i, X_j\ra - d)^4]
        &= 12d(d+4),
        &\text{} \\
        \E [
            (\la X_i, X_i\ra-d)^2
            (\la X_j, X_j\ra-d)^2
        ]
        &= 4d^2,
        &\text{if $i\neq j$} \\
        \E[\la X_i,X_j\ra^4]
        &= 3d(d+2),
        &\text{if $i\neq j$} \\
        \E [
            \la X_i, X_j\ra^2
            (\la X_i, X_i\ra-d)^2
        ]
        &= 2d(d+4),
        &\text{if $i\neq j$} \\
        \E[
            \la X_i,X_j\ra^2
            (\la X_k, X_k\ra - d)^2
        ]
        &= 2d^2,
        & \text{if $i,j,k$ are distinct} \\
        \E [
            \la X_i, X_j\ra^2
            \la X_i, X_k\ra^2
        ]
        &= d(d+2),
        &\text{if $i,j,k$ are distinct} \\
        \E[
            \la X_i, X_j\ra^2
            \la X_k, X_\ell \ra^2
        ]
        &= d^2.
        &\text{if $i,j,k,\ell$ are distinct}
    \end{align*}
    Therefore, by the expansion (\ref{eq:app-general-ub-computational-trdelsq-combinatorial-expansion}),
    \[
        \E \Tr\lt((d^{-1}M - I_k)^2\rt)
        = \f{k \cdot 2d + k(k-1) \cdot d}{d^2} = \f{k^2 + k}{d},
    \]
    verifying part (\ref{itm:general-ub-computational-trdelsq}).
    Now suppose $d\ge k$.
    By squaring (\ref{eq:app-general-ub-computational-trdelsq-combinatorial-expansion}) and combinatorially expanding, we have
    \begin{align*}
        \E \Tr\lt((d^{-1}M - I_k)^2\rt)^2
        &=
        \f{1}{d^4}
        \Bigg[
            k \cdot 12d(d+4) +
            2 \binom{k}{2} \cdot 4d^2 +
            4 \binom{k}{2} \cdot 3d(d+2) +
            8 \binom{k}{2} \cdot 2d(d+4) \\
        &\qquad +
            12 \binom{k}{3} \cdot 2d^2 +
            24 \binom{k}{3} \cdot d(d+2)
            24 \binom{k}{4} \cdot d^2
        \Bigg] \\
        &=
        \lt(\f{k^2 + k}{d}\rt)^2 +
        \f{1}{d^4}
        \lt[
            4k^2 d^2 +
            4k d^2 +
            8k^3 d +
            20k^2 d +
            20k d
        \rt] \\
        &\le
        \lt(\f{k^2 + k}{d}\rt)^2 +
        \f{48k^2}{d^2} +
        \f{8k^3}{d^3}
        \le 
        \lt(\f{k^2 + k}{d}\rt)^2 +
        \f{56k^2}{d^2}.
    \end{align*}
    This and part (\ref{itm:general-ub-computational-trdelsq}) imply part (\ref{itm:general-ub-computational-vartrdelsq}).

    \paragraph{Proof of (\ref{itm:general-ub-computational-det}).}
    Since $M$ is the $k\times k$ Wishart matrix with $d$ degrees of freedom, if $d\ge k$ then $M$ has probability density
    \[
        f_{d,k}(M) =
        \f{
            \det(M)^{(d-k-1)/2} \etr\lt(-M/2\rt)
        }{
            2^{kd/2} \pi^{k(k-1)/4} \prod_{i=1}^k \Gamma(\f{d+1-i}{2})
        }
    \]
    with respect to the Lebesgue measure on the positive semidefinite cone $\cS^{k\times k} \subseteq \bR^{k\times k}$.
    If $d\ge k+2$, then a Wishart matrix with $d-2$ degrees of freedom has density $f_{d-2, k}$.
    The fact that this density integrates to $1$ implies that (using $d\ge 2k+2$)
    \begin{align*}
        \E \det(d^{-1}M)^{-1}
        &=
        \int_{\cS^{k\times k}}
        \det(d^{-1} M)^{-1}
        \cdot
        \f{
            \det(M)^{(d-k-1)/2} \etr\lt(-M/2\rt)
        }{
            2^{kd/2} \pi^{k(k-1)/4} \prod_{i=1}^k \Gamma(\f{d+1-i}{2})
        }
        \diff{M} \\
        &=
        \lt(\f{d}{2}\rt)^k
        \f{
            \Gamma(\f{d-k}{2})\Gamma(\f{d-k-1}{2})
        }{
            \Gamma(\f{d}{2})\Gamma(\f{d-1}{2})
        }
        \int_{\cS^{k\times k}}
        \f{
            \det(M)^{(d-k-3)/2} \etr\lt(-M/2\rt)
        }{
            2^{k(d-2)/2} \pi^{k(k-1)/4} \prod_{i=1}^k \Gamma(\f{d-1-i}{2})
        }
        \diff{M} \\
        &=
        \lt(\f{d}{2}\rt)^k
        \f{
            \Gamma(\f{d-k}{2})\Gamma(\f{d-k-1}{2})
        }{
            \Gamma(\f{d}{2})\Gamma(\f{d-1}{2})
        }
        \le
        \lt(\f{d}{2}\rt)^k
        \lt(\f{d-k-1}{2}\rt)^{-k} \\
        &=
        \lt(
            1 + \f{k+1}{d-k-1}
        \rt)^k
        \le
        \exp\lt(\f{k(k+1)}{d-k-1}\rt)
        \le
        e^k.
    \end{align*}
    
    \paragraph{Proof of (\ref{itm:general-ub-computational-log2det}).}
    Given a positive definite matrix $Z\in \bR^{k\times k}$, let $\lambda_1, \ldots, \lambda_k > 0$ denote its eigenvalues.
    By Cauchy-Schwarz, we have that
    \[
        \log^2 \det(Z)
        =
        \lt(
            \sum_{i=1}^k \log \lambda_i
        \rt)^2
        \le
        k \sum_{i=1}^k
        \log^2 \lambda_i
        \le
        k \sum_{i=1}^k
        (\lambda_i + \lambda_i^{-1})
        =
        k\Tr(Z) + k\Tr(Z^{-1}),
    \]
    where the second inequality follows from the fact that $\log^2 x \le x + x^{-1}$ for all $x>0$.
    Since $d\ge k$, $M$ is positive definite almost surely.
    Combining this inequality with the standard facts that $\E M = dI_k$ and $\E M^{-1} = \f{1}{d-k-1} I_k$ yields that (using $d\ge 2k+2$)
    \[
        \E \log^2 \det(d^{-1}M)
        \le
        \f{1}{d} \E \Tr(M) + d \E \Tr(M^{-1})
        =
        k + \f{kd}{d-k-1}
        \le 3k.
    \]
\end{proof}

We now proceed to the proof of Proposition~\ref{prop:general-ub-kl-successive-differences}.

\begin{proof}[Proof of Proposition~\ref{prop:general-ub-kl-successive-differences}]
    We will prove each part in turn. 
    
    \paragraph{Proof of (\ref{itm:kl-conditioning-kla-klb}).}
    We can write
    \begin{eqnarray*}
        \KL_A
        &=&
        \E_{W\sim \muvl}
        \KL\lt(
            \E_{\Xvl \sim \gamma(W)} \muv(\Xvl)
            \parallel
            \nuv
        \rt), \\
        \KL_B
        &=&
        \E_{W\sim \muvltd}
        \KL\lt(
            \E_{\Xvl \sim \gamma(W)} \muv(\Xvl)
            \parallel
            \nuv
        \rt).
    \end{eqnarray*}
    Thus $\KL_A$ and $\KL_B$ are both nonnegative.
    By convexity of KL divergence,
    \begin{align*}
        \KL_A - \KL_B
        &=
        -\P(\Tdetc) \KL_B +
        \E_{\Xpvl}
        \ind{\Xpvl\in \Tdetc}
        \KL\lt(
            \E_{\Xvl \sim \gamma(W(\Xpvl))} \muv(\Xvl)
            \parallel
            \nuv
        \rt) \\
        &\le
        \E_{\Xpvl}
        \E_{\Xvl \sim \gamma(W(\Xpvl))}
        \ind{\Xpvl\in \Tdetc}
        \KL\lt(\muv(\Xvl) \parallel \nuv\rt) \\
        &\le
        \E_{\Xvl}
        \lt(
            \P_{\Xpvl \sim \gamma(W(\Xvl))}
            (\Xpvl \in \Tdetc)
        \rt)
        \KL\lt(\muv(\Xvl) \parallel \nuv\rt).
    \end{align*}
    Recall that $\muv(\Xvl) = \cN\lt(0, d^{-1} \Xbnt \Xbn\rt)$ and $\nuv = \cN\lt(0, I_{\downNv}\rt)$.
    By Lemma~\ref{lem:kl-gaussians},
    \[
        \KL\lt(\muv(\Xvl) \parallel \nuv\rt)
        =
        \f12 \lt[
            \Tr\lt(d^{-1} \Xbnt \Xbn - \Ibn\rt)
            - \log \det \lt(d^{-1} \Xbnt \Xbn\rt)
        \rt].
    \]
    It now follows that
    \begin{align*}
        \KL_A - \KL_B
        &\le
        \E_{\Xvl}
        \f12
        \lt(
            \P_{\Xpvl \sim \gamma(W(\Xvl))}
            (\Xpvl \in \Tdetc)
        \rt)
        \Tr \lt(d^{-1} \Xbnt \Xbn - \Ibn\rt) \\
        &\qquad +
        \E_{\Xvl}
        \f12
        \lt(
            \P_{\Xpvl \sim \gamma(W(\Xvl))}
            (\Xpvl \in \Tdetc)
        \rt)
        \lt(-\log\det \lt(d^{-1} \Xbnt \Xbn\rt)\rt).
    \end{align*}
    Call the last two expectations $\Gamma_1$ and $\Gamma_2$.
    By Cauchy-Schwarz,
    \begin{align*}
        \Gamma_1
        &\le
        \f12
        \lt[
            \E_{\Xvl}
            \lt(
                \P_{\Xpvl \sim \gamma(W(\Xvl))}
                (\Xpvl \in \Tdetc)
            \rt)^2
        \rt]^{1/2}
        \lt[
            \E_{\Xvl}
            \Tr \lt(d^{-1} \Xbnt \Xbn - \Ibn\rt)^2
        \rt]^{1/2} \\
        &\le
        \f12
        \lt[
            \E_{\Xvl}
            \lt(
                \P_{\Xpvl \sim \gamma(W(\Xvl))}
                (\Xpvl \in \Tdetc)
            \rt)
        \rt]^{1/2}
        \lt[
            \E_{\Xvl}
            \Tr \lt(d^{-1} \Xbnt \Xbn - \Ibn\rt)^2
        \rt]^{1/2} \\
        &=
        \f12
        \P(\Tdetc)^{1/2}
        (2d^{-1}\ddegv)^{1/2}
        \le
        \P(\Tdetc)^{1/2} n
    \end{align*}
    by Lemma~\ref{lem:general-ub-computational}(\ref{itm:general-ub-computational-trsq-centered}).
    Similarly, we have that
    \begin{align*}
        \Gamma_2
        &\le
        \f12
        \lt[
            \E_{\Xvl}
            \lt(
                \P_{\Xpvl \sim \gamma(W(\Xvl))}
                (\Xpvl \in \Tdetc)
            \rt)^2
        \rt]^{1/2}
        \lt[
            \E_{\Xvl} \log^2 \det\lt(d^{-1} \Xbnt \Xbn\rt)
        \rt]^{1/2} \\
        &\le
        \f12
        \lt[
            \E_{\Xvl}
            \lt(
                \P_{\Xpvl \sim \gamma(W(\Xvl))}
                (\Xpvl \in \Tdetc)
            \rt)
        \rt]^{1/2}
        \lt[
            \E_{\Xvl} \log^2 \det\lt(d^{-1} \Xbnt \Xbn\rt)
        \rt]^{1/2} \\
        &\le
        \f12
        \P(\Tdetc)^{1/2}
        (3\ddegv)^{1/2}
        \le
        \P(\Tdetc)^{1/2} n
    \end{align*}
    where the third inequality is by Lemma~\ref{lem:general-ub-computational}(\ref{itm:general-ub-computational-log2det}), whose condition $d \ge 2\ddegv + 2$ holds for sufficiently large $n$ because $d \gg \max_{v\in [n]} \ddegv$.
    Finally $\KL_A - \KL_B \le \Gamma_1 + \Gamma_2 \le 2\P(\Tdetc)^{1/2} n$.
    
    \paragraph{Proof of (\ref{itm:kl-conditioning-klb-klc}).}
    Let $\cL^{\downNv}$ denote Lebesgue measure on $\bR^{\downNv}$ scaled by $(2\pi)^{-\ddegv/2}$, i.e. for $\phi\in \bR^{\downNv}$, $\rn{\nuv}{\cL^{\downNv}} = \exp\lt(-\f12 \norm{\phi}_2^2\rt)$.
    Then, we have that
    \begin{align*}
        \KL_B - \KL_C
        &=
        \E_{W\sim \muvltd}
        \E_{\phi\sim \nuv}
        \lt(
            \E_{\Xvl\sim \gamma(W)}
            \ind{\Xvl\not\in \Tdet}
            \rn{\muv(\Xvl)}{\nuv}(\phi)
        \rt) \\
        &\qquad\qquad \times \log\lt(
            \E_{\Xvl\sim \gamma(W)}
            \rn{\muv(\Xvl)}{\nuv}(\phi)
        \rt) \\
        &=
        \E_{W\sim \muvltd}
        \E_{\phi\sim \nuv}
        \lt(
            \E_{\Xvl\sim \gamma(W)}
            \ind{\Xvl\not\in \Tdet}
            \rn{\muv(\Xvl)}{\nuv}(\phi)
        \rt) \\
        &\qquad\qquad \times \lt[
            \log
            \lt(
                \E_{\Xvl\sim \gamma(W)}
                \rn{\muv(\Xvl)}{\cL^{\downNv}}(\phi)
            \rt)
            - \log \rn{\nuv}{\cL^{\downNv}}(\phi)
        \rt] \\
        &=
        \E_{W\sim \muvltd}
        \E_{\phi\sim \nuv}
        \lt(
            \E_{\Xvl\sim \gamma(W)}
            \ind{\Xvl\not\in \Tdet}
            \rn{\muv(\Xvl)}{\nuv}(\phi)
        \rt) \\
        &\qquad\qquad \times \log
        \lt(
            \E_{\Xvl\sim \gamma(W)}
            \rn{\muv(\Xvl)}{\cL^{\downNv}}(\phi)
        \rt) \\
        &\qquad +
        \E_{W\sim \muvltd}
        \E_{\phi\sim \nuv}
        \lt(
            \E_{\Xvl\sim \gamma(W)}
            \ind{\Xvl\not\in \Tdet}
            \rn{\muv(\Xvl)}{\nuv}(\phi)
        \rt)
        \f12 \norm{\phi}_2^2.
    \end{align*}
    Let $\Gamma_3$ and $\Gamma_4$ denote these last two expectations.
    To bound $\Gamma_3$ we will use the following fact.
    At any $\phi\in \bR^{\downNv}$, the (scaled) density of $\muv(\Xvl)$ with respect to Lebesgue measure is upper bounded by
    \begin{align*}
        \rn{\muv(\Xvl)}{\cL^{\downNv}}(\phi)
        &=
        \det \lt(d^{-1} \Xbnt \Xbn\rt)^{-1/2}
        \exp\lt(-\f{1}{2} \phi^\top \lt(d^{-1}\Xbnt \Xbn\rt)^{-1} \phi\rt) \\
        &\le
        \det \lt(d^{-1} \Xbnt \Xbn\rt)^{-1/2}.
    \end{align*}
    Now, note that if $W = W(\Xpvl)$ for some $\Xpvl \in \Tdet \subseteq \Svdet$, which is true almost surely for $W\sim \muvltd$, then we have that
    \[
        \E_{\Xvl \sim \gamma(W)}
        \rn{\muv(\Xvl)}{\cL^{\downNv}}(\phi)
        \le
        \E_{\Xvl\sim \gamma(W)}
        \det\lt(d^{-1} \Xbnt \Xbn\rt)^{-1/2} \le e^n
    \]
    by the definition (\ref{eq:general-ub-def-svdet}) of $\Svdet$.
    Therefore,
    \begin{align*}
        \Gamma_3
        &\le
        \E_{W\sim \muvltd}
        \E_{\phi\sim \nuv}
        \lt(
            \E_{\Xvl \sim \gamma(W)}
            \ind{\Xvl\not\in \Tdet}
            \rn{\muv(\Xvl)}{\nuv}(\phi)
        \rt) n \\
        &\le
        \f{n}{\P(\Tdet)}
        \E_{W\sim \muvl}
        \E_{\phi\sim \nuv}
        \E_{\Xvl\sim \gamma(W)}
        \ind{\Xvl\not\in \Tdet}
        \rn{\muv(\Xvl)}{\nuv}(\phi) \\
        &=
        \f{\P(\Tdetc)n}{\P(\Tdet)}
        \le
        2\P(\Tdetc)n
        \le
        2\P(\Tdetc)^{1/2}n.
    \end{align*}
    Here we have used the fact that, by Proposition~\ref{prop:general-ub-s-hp}, $\P(\Tdet)\ge 1 - \P(T^c) - \P((\Svdet)^c) \ge 1 - 2n^{-20} + e^{-n/2} \ge \f12$ for sufficiently large $n$.
    Now, note that $\Gamma_4$ can be bounded by
    \begin{align*}
        \Gamma_4
        &\le
        \f{1}{\P(\Tdet)}
        \E_{W\sim \muvl}
        \E_{\phi\sim \nuv}
        \lt(
            \E_{\Xvl\sim \gamma(W)}
            \ind{\Xvl\not\in \Tdet}
            \rn{\muv(\Xvl)}{\nuv} (\phi)
        \rt)
        \f12 \norm{\phi}_2^2 \\
        &=
        \f{1}{2\P(\Tdet)}
        \E_{\Xvl}
        \ind{\Xvl\not\in \Tdet}
        \E_{\phi\sim \muv(\Xvl)}
        \norm{\phi}_2^2 \\
        &=
        \f{1}{2\P(\Tdet)}
        \E_{\Xvl}
        \ind{\Xvl\not\in \Tdet}
        \Tr\lt(d^{-1} \Xbnt \Xbn\rt),
    \end{align*}
    where we recall that $\Xbn$ is the submatrix of $\Xvl$ with columns indexed by $\downNv$ and $\muv(\Xvl) = \cN\lt(0, d^{-1} \Xbnt \Xbn\rt)$.
    By Cauchy-Schwarz and Lemma~\ref{lem:general-ub-computational}(\ref{itm:general-ub-computational-trsq}), we have that
    \begin{align*}
        \Gamma_4
        &\le
        \f{1}{2\P(\Tdetc)}
        \lt(
            \E_{\Xvl}
            \ind{\Xvl\not\in \Tdet}
        \rt)^{1/2}
        \lt(
            \E_{\Xvl}
            \Tr\lt(d^{-1} \Xbnt \Xbn\rt)^2
        \rt)^{1/2} \\
        &=
        \f{\P(\Tdetc)^{1/2}}{2\P(\Tdetc)}
        \lt(\ddegv^2 + 2d^{-1} \ddegv \rt)^{1/2}
        \le
        \P(\Tdetc)^{1/2} \ddegv
        \le \P(\Tdetc)^{1/2} n.
    \end{align*}
    Therefore, $\KL_B - \KL_C \le \Gamma_3 + \Gamma_4 \le 3\P(\Tdetc)^{1/2} n$.
    
    \paragraph{Proof of (\ref{itm:kl-conditioning-klc-kld}).}
    Using the inequality $\log(1+x)\le x$, we have
    \begin{align*}
        \KL_C - \KL_D
        &=
        \E_{W\sim \muvltd}
        \E_{\phi \sim \nuv}
        \lt(
            \E_{\Xvl\sim \gamma(W)}
            \ind{\Xvl \in \Tdet}
            \rn{\muv(\Xvl)}{\nuv}(\phi)
        \rt) \\
        &\qquad \times
        \log \lt[
            \P_{\Xvl\sim \gamma(W)}
            (\Xvl \in \Tdet)
            \cdot
            \lt(1 + \f{
                \E_{\Xvl\sim \gamma(W)} \ind{\Xvl\not\in \Tdet}
                \rn{\muv(\Xvl)}{\nuv} (\phi)
            }{
                \E_{\Xvl\sim \gamma(W)} \ind{\Xvl\in \Tdet}
                \rn{\muv(\Xvl)}{\nuv} (\phi)
            }\rt)
        \rt] \\
        &\le
        \E_{W\sim \muvltd}
        \E_{\phi \sim \nuv}
        \lt(
            \E_{\Xvl\sim \gamma(W)}
            \ind{\Xvl \in \Tdet}
            \rn{\muv(\Xvl)}{\nuv}(\phi)
        \rt) \\
        &\qquad \times
        \log \lt[
            1 + \f{
                \E_{\Xvl\sim \gamma(W)} \ind{\Xvl\not\in \Tdet}
                \rn{\muv(\Xvl)}{\nuv} (\phi)
            }{
                \E_{\Xvl\sim \gamma(W)} \ind{\Xvl\in \Tdet}
                \rn{\muv(\Xvl)}{\nuv} (\phi)
            }
        \rt] \\
        &\le
        \E_{W\sim \muvltd}
        \E_{\phi\sim \nuv}
        \lt(
            \E_{\Xvl\sim \gamma(W)}
            \ind{\Xvl\not\in \Tdet}
            \rn{\muv(\Xvl)}{\nuv}(\phi)
        \rt) 
        = \P(\Tdetc)
    \end{align*}
    which proves the proposition.

    \paragraph{Proof of (\ref{itm:kl-conditioning-kld-kle}).}
    We can write $\KL_D$ and $\KL_E$ in the form
    \begin{eqnarray*}
        \KL_D
        &=&
        \E_{W\sim \muvltd}
        \lt(
            \E_{\substack{
                \Xvl \sim \gamma(W) \\
                \Xvl \in \Tdet
            }}
            \muv(\Xvl)
            \parallel
            \nuv
        \rt)
        \lt(
            \P_{\Xvl\sim \gamma(W)}
            (\Xvl \in \Tdet)
        \rt) \\
        \KL_E
        &=&
        \E_{W\sim \muvltd}
        \KL\lt(
            \E_{\substack{
                \Xvl \sim \gamma(W) \\
                \Xvl \in \Tdet
            }}
            \muv(\Xvl)
            \parallel
            \nuv
        \rt).
    \end{eqnarray*}
    Since the KL is nonnegative and the probability above is at most $1$, this implies $\KL_D\le \KL_E$.
\end{proof}

\section{Subgraph Statistics}
\label{appsec:subgraph-statistics}

In this section, we compile some inequalities relating counts of various subgraphs, which we will use to control the subgraph counts arising in the proofs of the TV divergence results in Section~\ref{sec:tv-lower-bounds-proofs}.
We also present the deferred proof of Lemma~\ref{lem:bipartite-ub-hypothesis-translation}, which translates the hypotheses in Theorem~\ref{thm:bipartite-ub} into a more amenable form.

\begin{lemma}
\label{lem:subgraph-statistics}
    For all graphs $G$, the following inequalities hold.
    The $\lesssim$ hides only a constant factor independent of $G$.
    \begin{enumerate}[label=(\alph*), ref=\alph*]
        \item \label{itm:ss-c3} $\Num_G(C_3) \lesssim \Num_G(P_2)$.
        \item \label{itm:ss-c4} $\Num_G(C_4) \lesssim \Num_G(P_3)$.
        \item \label{itm:ss-k13-p3-c3+} $\Num_G(K_{1,3}, P_3, C_3^{+}) \lesssim \Num_G(P_2)^{3/2}$.
        \item \label{itm:ss-k13+-p4} $\Num_G(K_{1,3}^{+}, P_4) \lesssim \Num_G(P_2)^{3/2} + \Num_G(K_{1,4})$.
        \item \label{itm:ss-k23} $\Num_G(K_{2,3}) \lesssim \Num_G(C_4)^{3/2}$.
        \item \label{itm:ss-c42ev} $\Num_G(C_4^{2,ev}) \lesssim \Num_G(C_4, K_{2,4}, C_4^{2,e})$.
        \item \label{itm:ss-k24-c42e-c42v} $\Num_G(K_{2,4}, C_4^{2,e}, C_4^{2,v}) \le \Num_G(C_4)^2$.
\end{enumerate}

\end{lemma}
\begin{proof}
    We will separately prove each part.

    \paragraph{Proof of (\ref{itm:ss-c3}).}
    Each copy of $C_3$ in $G$ contains three copies of $P_2$, obtained by deleting an edge.

    \paragraph{Proof of (\ref{itm:ss-c4}).}
    Each copy of $C_4$ in $G$ contains four copies of $P_3$, obtained by deleting an edge.

    \paragraph{Proof of (\ref{itm:ss-k13-p3-c3+}).}
    Note that $\binom{a}{2}^{3/2}\ge \binom{a}{3}$ for all nonnegative integers $a$.
    By Cauchy-Schwarz,
    \begin{align*}
        \Num_G(P_2)^3
        &=
        \lt[\sum_{v\in G} \binom{\degv}{2}\rt]^3
        \ge
        \lt[\sum_{v\in G} \binom{\degv}{2}\rt]
        \lt[\sum_{v\in G} \binom{\degv}{2}^2\rt] \\
        &\ge
        \lt[\sum_{v\in G} \binom{\degv}{2}^{3/2}\rt]^2
        \ge
        \lt[\sum_{v\in G} \binom{\degv}{3}\rt]^2
        =
        \Num_G(K_{1,3})^2.
    \end{align*}
    Therefore, $\Num_G(K_{1,3}) \le \Num_G(P_2)^{3/2}$.
    By AM-GM,
    \begin{align*}
        \Num_G(P_3)
        &\le
        \sum_{(u,v)\in E(G)}
        (\degv-1) (\degu-1)
        \lesssim
        \sum_{(u,v)\in E(G)}
        \lt[(\degv-1)^2 + (\deg(u)-1)^2\rt] \\
        &\lesssim
        \sum_{v\in G}
        \degv (\degv-1)^2
        \lesssim
        \sum_{v\in G}
        \degv \lt[
            \binom{\degv-1}{2} +
            (\degv-1)
        \rt] \\
        &\lesssim
        \sum_{v\in G}
        \lt[
            \binom{\degv}{3} +
            \binom{\degv}{2}
        \rt]
        = \Num_G(K_{1,3}, P_2)
        \lesssim
        \Num_G(P_2)^{3/2}.
    \end{align*}
    Finally, each copy of $C_3^+$ in $G$ contains two copies of $P_3$, so $\Num_G(C_3^+) \lesssim \Num_G(P_3) \lesssim \Num_G(P_2)^{3/2}$.

    \paragraph{Proof of (\ref{itm:ss-k13+-p4}).}
    By AM-GM,
    \begin{align*}
        \Num_G(K_{1,3}^+)
        &\le
        \sum_{(u,v)\in E(G)} \binom{\degu-1}{2}(\degv-1) \\
        &\lesssim
        \sum_{(u,v)\in E(G)}
        \lt[\binom{\degu-1}{2}^{3/2} + (\degv-1)^{3}\rt] 
        \lesssim
        \sum_{v\in G} \degv (\degv-1)^3 \\
        &\lesssim
        \sum_{v\in G}
        \degv
        \lt[
            \binom{\degv-1}{3} + (\degv-1)
        \rt]
        \lesssim
        \sum_{v\in G}
        \lt[
            \binom{\degv}{4} +
            \binom{\degv}{2}
        \rt] \\
        &=
        \Num_G(K_{1,4}, P_2).
    \end{align*}
    Now, recall that $P_2(G)$ is the set of paths $(u,v,w)$ of length 2 in $G$. By AM-GM,
    \begin{align*}
        \Num_G(P_4)
        &\lesssim
        \sum_{(u,v,w)\in P_2(G)} (\degu-1) (\deg(w)-1)
        \lesssim
        \sum_{(u,,v,w)\in P_2(G)} \lt[(\degu-1)^2 + (\deg(w)-1)^2\rt] \\
        &\lesssim
        \sum_{(u,v,w)\in P_2(G)} (\degu-1)^2
        \lesssim
        \sum_{(u,v,w)\in P_2(G)} \lt[\binom{\degu-1}{2} + (\degu-1)\rt] \\
        &\lesssim \Num_G(K_{1,3}^+, P_3)
        \lesssim
        \Num_G(P_2)^{3/2} + \Num_G(K_{1,4}),
    \end{align*}
    where the last inequality follows from $\Num_G(K_{1,3}^+) \lesssim \Num_G(K_{1,4}, P_2)$ and part (\ref{itm:ss-k13-p3-c3+}).

    \paragraph{Proof of (\ref{itm:ss-k23}).}
    This proof is by the same idea as the upper bound for $\Num_G(K_{1,3})$ in part (\ref{itm:ss-k13-p3-c3+}).
    Note that $\binom{a}{2}^{3/2}\ge \binom{a}{3}$ for all nonnegative integers $a$.
    By Cauchy-Schwarz,
    \begin{align*}
        \Num_G(C_4)^3
        &\ge
        \lt[
            \sum_{u,v\in G, u<v}
            \binom{\deg(u,v)}{2}
        \rt]^3
        \ge
        \lt[
            \sum_{u,v\in G, u<v}
            \binom{\deg(u,v)}{2}
        \rt]
        \lt[
            \sum_{u,v\in G, u<v}
            \binom{\deg(u,v)}{2}^2
        \rt] \\
        &\ge
        \lt[
            \sum_{u,v\in G, u<v}
            \binom{\deg(u,v)}{2}^{3/2}
        \rt]^2
        \ge
        \lt[
            \sum_{u,v\in G, u<v}
            \binom{\deg(u,v)}{3}
        \rt]^2
        =
        \Num_G(K_{2,3})^2.
    \end{align*}

    \paragraph{Proof of (\ref{itm:ss-c42ev}).}
    For $k \in \{1,2,3\}$ and vertices $u,v\in G$, let $p_k(u,v)$ denote the number of paths of length $k$ in $G$, without repeated vertices, with endpoints $u,v$.
    (Note that $p_1(u,v)$ is the indicator for the edge $(u,v)$.)
    Note that $\binom{a}{2}^2\le 36\binom{a}{4} + 3\binom{a}{2}$ for all nonnegative integers $a$.
    By AM-GM,
    \begin{align*}
        \Num_G(C_4^{2,ev})
        &\lesssim
        \sum_{u,v\in G, u<v}
        p_1(u,v)p_3(u,v) \binom{p_2(u,v)}{2}
        \lesssim
        \sum_{u,v\in G, u<v}
        \lt[
            p_1(u,v)^2p_3(u,v)^2 +
            \binom{p_2(u,v)}{2}^2
        \rt] \\
        &\lesssim
        \sum_{u,v\in G, u<v}
        \lt[
            p_1(u,v)\binom{p_3(u,v)}{2} +
            p_1(u,v)p_3(u,v) +
            \binom{p_2(u,v)}{4} +
            \binom{p_2(u,v)}{2}
        \rt] \\
        &\lesssim
        \Num_G(C_4, K_{2,4}, C_4^{2,e}).
    \end{align*}

    \paragraph{Proof of (\ref{itm:ss-k24-c42e-c42v}).}
    Note that $\Num_G(K_{2,4}, C_4^{2,e}, C_4^{2,v})$ counts pairs of copies of $C_4$ in $G$ whose unions are the specified graphs, while $\Num_G(C_4)^2$ counts pairs of copies of $C_4$ in $G$ without restriction.
\end{proof}

\begin{proof}[Proof of Lemma~\ref{lem:bipartite-ub-hypothesis-translation}]
    For all nonnegative integers $a$, we have that $a^4 \ge 4^4 \binom{a}{4} + 3^3 a$.
    So,
    \begin{align*}
        \sum_{v\in V_L} \lt(
            \degv^4 + \degv \log^3 n
        \rt)
        &\lesssim
        \sum_{v\in V_L} \lt(
            \binom{\degv}{4} + \degv \log^3 n
        \rt) \\
        &\lesssim
        \oNum_G(\oK_{1,4}) + \Num_G(E) \log^3 n.
    \end{align*}
    The first conclusion now follows from hypothesis (\ref{eq:bipartite-ub-hypothesis-k14}).
    We can expand $\sum_{i,j \in V_R} \deg(i,j)^2$ by
    \begin{align*}
        \sum_{i,j \in V_R} \deg(i,j)^2
        &=
        \sum_{i,j\in V_R, i<j} 2\deg(i,j)^2
        +
        \sum_{i\in V_R} \deg(i)^2 \\
        &=
        \sum_{i,j\in V_R, i<j} \lt(
            4\binom{\deg(i,j)}{2} +
            2\deg(i,j)
        \rt) +
        \sum_{i\in V_R} \lt(
            2\binom{\deg(i)}{2} + \deg(i)
        \rt) \\
        &= 4\Num_G(C_4) + 2\Num_G(P_2) + \Num_G(E)
        \le
        4\Num_G(C_4, P_2, E).
    \end{align*}
    The second conclusion follows from hypothesis (\ref{eq:bipartite-ub-hypothesis-4cycles}).
    Define $\Gamma_1, \Gamma_2, \Gamma_3$ to be the three sums in the third conclusion, as in the equations below.
    Note that
    \begin{align}
        \notag
        \Gamma_1
        &= \sum_{i,j,k,\ell \in V_R}
        \deg(i,j,k,\ell)^2
        \le
        \sum_{i,j,k,\ell \in V_R}
        \deg(i,j) \deg(k,\ell)
        =
        \lt(\sum_{i,j \in V_R} \deg(i,j)\rt)^2
        \ll d^4, \\
        \label{eq:bipartite-ub-gamma2-ub}
        \Gamma_2
        &=
        \sum_{i,j,k\in V_R}
        \deg(i,j) \deg(i,k)
        \le
        \sum_{i,i',j,k\in V_R}
        \deg(i,j) \deg(i',k)
        \le
        \lt(\sum_{i,j \in V_R} \deg(i,j)\rt)^2
        \ll d^4, \\
        \label{eq:bipartite-ub-gamma3-ub}
        \Gamma_3
        &=
        \sum_{i,j,k,\ell \in V_R}
        \deg(i,j,k) \deg(i,j,\ell)
        \le
        \sum_{i,j,k,\ell \in V_R}
        \deg(i,k) \deg(j,\ell)
        \le
        \lt(\sum_{i,j \in V_R} \deg(i,j)\rt)^2
        \ll d^4.
    \end{align}
    This implies the third conclusion.
    To prove the fourth conclusion, first note that
    \begin{align*}
        \Gamma_1
        &=
        \sum_{i,j,k,\ell \in V_R}
        \deg(i,j,k,\ell)^2
        \lesssim
        \sum_{\substack{
            u,v\in V_L \\
            i,j,k,\ell \in V_R
        }}
        \ind{\text{$e\in E(G)$ for all $e\in \{u,v\}\times \{i,j,k,\ell\}$ }} \\
        &\lesssim
        \sum_{u,v \in V_L}
        \deg(u,v)^4 
        \lesssim
        \sum_{u\in V_L}
        \deg(u)^4 +
        \sum_{u,v\in V_L, u<v}
        \deg(u,v)^4 \\
        &\lesssim
        \sum_{u\in V_L} \lt[
            \binom{\deg(u)}{4} +
            \deg(u)
        \rt] +
        \sum_{u,v\in V_L, u<v} \lt[
            \binom{\deg(u,v)}{4} +
            \deg(u,v)
        \rt] \\
        &\lesssim
        \Num_G(E, P_2) + \oNum_G(\oK_{1,4}, \oK_{2,4}).
    \end{align*}
    Furthermore, we have
    \[
        \lt[
            \sum_{v\in V_L} \lt(
                \degv^3 + \degv \log^2 n
            \rt)
        \rt]^2
        \lesssim
        \lt(\sum_{v\in V_L} \degv^3\rt)^2
        +
        \Num_G(E)^2 \log^4 n.
    \]
    So,
    \begin{align}
        \notag
        &d^{-4}
        \lt[
            \sum_{v\in V_L} \lt(
                \degv^3 + \degv \log^2 n
            \rt)
        \rt]^2
        \lt[
            d^{-4} \Gamma_1 +
            d^{-5} \Gamma_2 +
            d^{-6} \Gamma_3
        \rt] \\
        \notag
        &\qquad \le
        d^{-4}
        \lt[
            \lt(\sum_{v\in V_L} \degv^3\rt)^2
            +
            \Num_G(E)^2 \log^4 n
        \rt] \\
        \notag
        &\qquad \qquad \times 
        \lt[
            d^{-4} \Num_G(E, P_2) + d^{-4} \oNum_G(\oK_{1,4}, \oK_{2,4}) +
            d^{-5} \Gamma_2 +
            d^{-6} \Gamma_3
        \rt] \\
        \notag
        &\qquad \le
        d^{-8}
        \lt[
            \lt(\sum_{v\in V_L} \degv^3\rt)^2
            +
            \Num_G(E)^2 \log^4 n
        \rt]
        \oNum_G(\oK_{2,4}) \\
        \notag
        &\qquad \qquad +
        d^{-4}\lt(\sum_{v\in V_L} \degv^3\rt)^2
        \lt[
            d^{-4} \Num_G(E, P_2) +
            d^{-4} \oNum_G(\oK_{1,4}) +
            d^{-5} \Gamma_2 +
            d^{-6} \Gamma_3
        \rt] \\
        \label{eq:bipartite-ub-hypothesis-translation-intermediate}
        &\qquad \qquad +
        d^{-4} \Num_G(E)^2 \log^4 n
        \lt[
            d^{-4} \Num_G(E, P_2) +
            d^{-4} \oNum_G(\oK_{1,4}) +
            d^{-5} \Gamma_2 +
            d^{-6} \Gamma_3
        \rt].
    \end{align}
    We will bound each the three summands in the bound (\ref{eq:bipartite-ub-hypothesis-translation-intermediate}).
    Becuase $a^3 \ge 3^3 \binom{a}{3} + 2^2 a$ for all nonnegative integers $a$, we have
    \begin{align*}
        \lt(\sum_{v\in V_L} \degv^3\rt)^2
        &\lesssim
        \lt(\sum_{v\in V_L} \binom{\degv}{3} + \degv\rt)^2
        \lesssim
        \lt(\oNum_G(\oK_{1,3}) + \Num_G(E)\rt)^2 \\
        &\lesssim
        \oNum_G(\oK_{1,3})^2 + \Num_G(E)^2.
    \end{align*}
    By hypothesis (\ref{eq:bipartite-ub-hypothesis-d8}), we have that
    \[
        d^{-8}
        \lt[
            \lt(\sum_{v\in V_L} \degv^3\rt)^2
            +
            \Num_G(E)^2 \log^4 n
        \rt]
        \oNum_G(\oK_{2,4})
        \ll 1.
    \]
    This bounds the first summand in (\ref{eq:bipartite-ub-hypothesis-translation-intermediate}).
    To bound the second summand, note that by Cauchy-Schwarz,
    \begin{align}
        \notag
        \lt(\sum_{v\in V_L} \degv^3\rt)^2
        &\le
        \lt(\sum_{v\in V_L} \degv^2\rt)
        \lt(\sum_{v\in V_L} \degv^4\rt) \\
        \notag
        &\lesssim
        \lt[\sum_{v\in V_L} \lt(\binom{\degv}{2} + \degv\rt)\rt]
        \lt[\sum_{v\in V_L} \lt(\binom{\degv}{4} + \degv\rt)\rt] \\
        \label{eq:bipartite-ub-hypothesis-translation-sum-delv3}
        &\lesssim
        \Num_G(P_2, E) \lt(\oNum_G(\oK_{1,4}) + \Num_G(E)\rt)
        \ll d^5,
    \end{align}
    where the last inequality follows from hypotheses (\ref{eq:bipartite-ub-hypothesis-4cycles}) and (\ref{eq:bipartite-ub-hypothesis-k14}).
    Moreover, by (\ref{eq:bipartite-ub-hypothesis-4cycles}), (\ref{eq:bipartite-ub-hypothesis-k14}), (\ref{eq:bipartite-ub-gamma2-ub}), and (\ref{eq:bipartite-ub-gamma3-ub}),
    \[
        d^{-4} \Num_G(E, P_2) +
        d^{-4} \oNum_G(\oK_{1,4}) +
        d^{-5} \Gamma_2 +
        d^{-6} \Gamma_3
        \ll
        d^{-1}.
    \]
    This implies that
    \[
        d^{-4}\lt(\sum_{v\in V_L} \degv^3\rt)^2
        \lt[
            d^{-4} \Num_G(E, P_2) +
            d^{-4} \oNum_G(\oK_{1,4}) +
            d^{-5} \Gamma_2 +
            d^{-6} \Gamma_3
        \rt]
        \ll
        1.
    \]
    This bounds the second summand of (\ref{eq:bipartite-ub-hypothesis-translation-intermediate}).
    By a counting argument, we have that
    \begin{align*}
        \Gamma_2
        &=
        \sum_{i,j,k\in V_R}
        \deg(i,j) \deg(i,k)
        =
        \sum_{\substack{
            u,v\in V_L \\
            i,j,k\in V_R
        }}
        \ind {(u,i), (u,j), (v,i), (v,k) \in E(G)} \\
        &\lesssim
        \sum_{\substack{
            v\in V_L \\
            i,j,k\in V_R
        }}
        \ind{(v,i), (v,j), (v,k) \in E(G)}
        +
        \sum_{\substack{
            u,v\in V_L, u\neq v \\
            i,j,k\in V_R
        }}
        \ind{(u,i), (u,j), (v,i), (v,k) \in E(G)} \\
        &\lesssim
        \sum_{v\in V_L}
        \deg(v)^3
        +
        \Num_G(C_4, P_3, P_2)
        +
        \oNum_G(\oP_4) \\
        &\lesssim
        d^{5/2} +
        \Num_G(P_2)^{3/2} +
        \oNum_G(\oP_4)
    \end{align*}
    The last inequality follows from (\ref{eq:bipartite-ub-hypothesis-translation-sum-delv3}) and Lemma~\ref{lem:subgraph-statistics}(\ref{itm:ss-c4},\ref{itm:ss-k13-p3-c3+}).
    Therefore, the third summand of (\ref{eq:bipartite-ub-hypothesis-translation-intermediate}) is bounded by
    \begin{align*}
        &d^{-4} \Num_G(E)^2 \log^4 n
        \lt[
            d^{-4} \Num_G(E, P_2) +
            d^{-4} \oNum_G(\oK_{1,4}) +
            d^{-5} \Gamma_2 +
            d^{-6} \Gamma_3
        \rt] \\
        &\qquad \lesssim
        d^{-4} \Num_G(E)^2 \log^4 n
        \bigg[
            d^{-4} \Num_G(E, P_2) +
            d^{-4} \oNum_G(\oK_{1,4}) +
            d^{-5/2} +
            d^{-5} \Num_G(P_2)^{3/2} \\
            &\qquad \qquad +
            d^{-5} \oNum_G(\oP_4) +
            d^{-6} \Gamma_3
        \bigg] \\
        &\qquad \lesssim
        d^{-4} \Num_G(E)^2 \log^4 n
        \lt[
            d^{-4} \Num_G(E, P_2) +
            d^{-5} \Num_G(P_2)^{3/2} +
            d^{-6} \Gamma_3 +
            d^{-5/2}
        \rt] \\
        &\qquad \qquad +
        d^{-8} \Num_G(E)^2 \oNum_G(\oK_{1,4}) \log^4 n
        +
        d^{-9} \Num_G(E)^2 \oNum_G(\oP_4) \log^4 n
    \end{align*}
    By hypotheses (\ref{eq:bipartite-ub-hypothesis-d8}) and (\ref{eq:bipartite-ub-hypothesis-d9}), the last two terms are $o(1)$.
    By hypothesis (\ref{eq:bipartite-ub-hypothesis-k14}), $\Num_G(E)^2 \log^4 n \le \Num_G(E)^2 \log^6 n \ll d^6$.
    By (\ref{eq:bipartite-ub-hypothesis-4cycles}) and (\ref{eq:bipartite-ub-gamma3-ub}),
    \[
        d^{-4} \Num_G(E, P_2) +
        d^{-5} \Num_G(P_2)^{3/2} +
        d^{-6} \Gamma_3 +
        d^{-5/2}
        \ll
        d^{-2}.
    \]
    Thus,
    \[
        d^{-4} \Num_G(E)^2 \log^4 n
        \lt[
            d^{-4} \Num_G(E, P_2) +
            d^{-5} \Num_G(P_2)^{3/2} +
            d^{-6} \Gamma_3 +
            d^{-5/2}
        \rt]
        \ll 1.
    \]
    This bounds the third summand of (\ref{eq:bipartite-ub-hypothesis-translation-intermediate}) and proves the fourth conclusion.
\end{proof}

\section{TV Convergence and Divergence for Random Masks}
\label{appsec:er-masks-proofs}

In this section, we will prove Theorems~\ref{thm:er-mask} and \ref{thm:bip-er-mask}, which show sharp phase transitions between the TV convergence and divergence regimes for Erd\H{o}s-R\'enyi $G$ and bipartite Erd\H{o}s-R\'enyi $G$.
We will prove these results by treating the various subgraph counts appearing in the TV upper and lower bounds in Sections~\ref{sec:tv-upper-bounds} and \ref{sec:tv-lower-bounds} as low-degree polynomials in the edges' indicator variables and showing that for the random graphs we consider, these polynomials concentrate near their expectations.
This implies that, for each $G = G_n$ in the sample path $G_1,G_2,\ldots$, the hypotheses of Theorems~\ref{thm:er-mask} and \ref{thm:bip-er-mask} imply the the appropriate theorem in Section~\ref{sec:tv-upper-bounds} or \ref{sec:tv-lower-bounds} with high probability over the randomness of $G$.
We then obtain the desired almost sure convergence by the Borel-Cantelli lemma.

Our main tools are the following two inequalities, by Kim-Vu and Janson, which bound the tails of polynomials of i.i.d. Bernoulli random variables.
Theorem~\ref{thm:kim-vu-ineq}, by Kim and Vu, controls both upper and lower tails, while Theorem~\ref{thm:janson-ineq} by Janson controls lower tails.
In the Erd\H{o}s-R\'enyi mask setting, we will use Theorem~\ref{thm:kim-vu-ineq} to control the upper tails of subgraph statistics and Theorem~\ref{thm:janson-ineq} to control the lower tails; in the bipartite Erd\H{o}s-R\'enyi mask setting, we will use Theorem~\ref{thm:kim-vu-ineq} to control both tails.

\begin{theorem}\cite{Vu02}
    \label{thm:kim-vu-ineq}
    Let $\xi_i$, $i\in \Gamma$ be a collection of i.i.d. Bernoulli variables, and let $Y$ be a polynomial in the $\xi_i$ of degree $k$ with coefficients in $[0,1]$.
    Define
    \[
        \bE_j Y
        =
        \max_{|A|\ge j} \E[\partial_A Y],
        \qquad
        M_j(Y)
        =
        \max_{|A|\ge j, \xi\in \{0,1\}^\Gamma}
        \partial_A Y(\xi),
    \]
    where, for a set $A\subseteq \Gamma$, $\partial_A$ denotes the partial derivative with respect to $\{\xi_i : i\in A\}$.
    (Note that $\bE_0 Y  = \E Y$.)
    Suppose $\ell \in \{1,\ldots,k\}$ satisfies $M_\ell(Y) \le 1$.
    For $\tau \ge 8k \log n$, define
    \begin{eqnarray*}
        \sE_0
        &=&
        \max\lt(\bE_0 Y, \tau \bE_1 Y, \tau^2 \bE_2 Y, \ldots, \tau^{\ell-1}\bE_{\ell-1} Y, \tau^{\ell}\rt), \\
        \sE_1
        &=&
        \max\lt(\bE_1 Y, \tau \bE_2 Y, \ldots, \tau^{\ell-2}\bE_{\ell-1} Y, \tau^{\ell-1}\rt).
    \end{eqnarray*}
    There exist constants $c_k,d_k > 0$, dependent only on $k$, such that $\P(|Y - \E Y| \ge c_k \sqrt{\tau \sE_0 \sE_1}) \le d_k \exp(-\tau/8)$.
\end{theorem}
This inequality follows from \cite[Theorem 4.2]{Vu02} by setting $\tau = 2\lambda$ for $\lambda \ge 4k \log n$.
We treat $k$ as a constant because the subgraph counts we are interested in are of constant size.

\begin{theorem}{\cite[Theorem 0]{JR02}}
    \label{thm:janson-ineq}
    Let $\xi_i$, $i\in \Gamma$ be a collection of i.i.d. Bernoulli variables.
    For a set $A\subseteq \Gamma$, let $\xi_A = \prod_{i\in A} \xi_i$.
    Let $Y = \sum_{A\in \cS} \xi_A$ for a set family $\cS \subseteq 2^{\Gamma}$, and define $\bar \Delta = \sum_{A,B\in \cS, A\cap B\neq \emptyset} \E\lt[I_A I_B\rt]$.
    Then, $\P(Y - \E Y \le -t) \le \exp(-t^2/2\bar\Delta)$.
\end{theorem}

\subsection{Erd\H{o}s-R\'enyi Masks}

In this section, we will show Theorem~\ref{thm:er-mask}, which gives conditions for TV convergence and divergence for Erd\H{o}s-R\'enyi masks.
In this setting, $G\sim \cG(n,p)$ is a sample from an Erd\H{o}s-R\'enyi graph.
We begin with the following two lemmas, which show high probability upper and lower bounds on the relevant subgraph statistics.
We will use these bounds in the proof of Theorem~\ref{thm:er-mask} to show that the appropriate theorems in Sections~\ref{sec:tv-upper-bounds} and \ref{sec:tv-lower-bounds} hold for $G$ with high probability.

\begin{lemma}
    \label{lem:er-subgraph-ubs}
    Let $G \sim \cG(n,p)$.
    For all sufficiently large $n$, the following inequalities each hold with probability at least $1 - n^{-10}$.
    The $\lesssim$ hides only a constant factor independent of $n,p$.
    \begin{enumerate}[label=(\alph*), ref=\alph*]
        \item \label{itm:er-ub-e} $\Num_G(E) \lesssim n^2p + \log n$.
        \item \label{itm:er-ub-p2} $\Num_G(P_2) \lesssim n^3p^2 + \log^2 n$.
        \item \label{itm:er-ub-k14} $\Num_G(K_{1,4}) \lesssim n^5p^4 + \log^4 n$.
        \item \label{itm:er-ub-k18} $\Num_G(K_{1,8}) \lesssim n^9p^8 + \log^8 n$.
        \item \label{itm:er-ub-c3} $\Num_G(C_3) \lesssim n^3p^3 + \log^2 n$.\footnote{ This is the famous upper tail problem for triangles. By the celebrated works of Chatterjee \cite{Cha12} and DeMarco and Kahn \cite{DK12}, we can show a stronger upper bound of $n^3p^3 + \log n$. Since the stronger upper bound does not lead to a stronger result in our application, we are content to use a cruder bound, which we can prove by general-purpose techniques.}
        \item \label{itm:er-ub-c4} $\Num_G(C_4) \lesssim n^4p^4 + \log^3 n$.
        \item \label{itm:er-ub-c32e} $\Num_G(C_3^{2,e}) \lesssim n^4p^5 + \log^4 n$.
        \item \label{itm:er-ub-c32v} $\Num_G(C_3^{2,v}) \lesssim n^5p^6 + \log^5 n$.
        \item \label{itm:er-ub-k24} $\Num_G(K_{2,4}) \lesssim n^6p^8 + \log^6 n + \min(n^8p^8, \log^8 n)$.
        \item \label{itm:er-ub-c42e} $\Num_G(C_4^{2,e}) \lesssim n^6p^7 + \log^6 n + \min(n^8p^8, \log^7 n)$.
        \item \label{itm:er-ub-c42v} $\Num_G(C_4^{2,v}) \lesssim n^7p^8 + \log^6 n + \min(n^8p^8, \log^8 n)$.
    \end{enumerate}
\end{lemma}

\begin{proof}
    These bounds are routine consequences of Theorem~\ref{thm:kim-vu-ineq}.
    We will prove parts (\ref{itm:er-ub-e}, \ref{itm:er-ub-p2}, \ref{itm:er-ub-k14}, \ref{itm:er-ub-k18}) in full detail and sketch the remaining parts.
    Throughout this proof, for each unordered pair $i,j\in [n]$ with $i\neq j$, let $\xi_{i,j}$ be the indicator of the edge $(i,j)$ in $G$.
    The $\xi_{i,j}$ are i.i.d. samples from $\Ber(p)$.
    Throughout this proof, the quantities $\ell, c_k, d_k, \sE_0, \sE_1$ are defined as in Theorem~\ref{thm:kim-vu-ineq}.

    \paragraph{Proof of (\ref{itm:er-ub-e}, \ref{itm:er-ub-p2}, \ref{itm:er-ub-k14}, \ref{itm:er-ub-k18}).}
    We can identify $E$ and $P_2$ with $K_{1,1}$ and $K_{1,2}$.
    We will show more generally that, for any constant $k$ and star graph $K_{1,k}$, $\Num_G(K_{1,k}) \lesssim n^{k+1}p^k + \log^k n$ with probability $1-n^{-10}$, where the constant hidden by the $\lesssim$ may depend on $k$.
    Let
    \[
        Y = \sum_{\substack{
            1\le j_1 < \cdots < j_k \le n \\
            i \in [n] \setminus \{j_1,\ldots,j_k\}
        }}
        \prod_{t=1}^k \xi_{i, j_t}
    \]
    be the sum of the indicators for each copy of $K_{1,k}$.
    This is a degree-$k$ polynomial in the variables $\xi_{i,j}$.
    Then, $M_k(Y) \le 1$ and $\bE Y = \bE_0 Y \asymp n^{k+1}p^k$, while for $1\le j\le k-1$, $\bE_j Y \asymp n^{k-j}p^{k-j}$.
    We set $\tau = \max(8k\log n, 80\log n + 8\log d_k)$.
    For each $1\le j\le k-1$, we have
    \[
        \tau^j \bE_j Y
        \le
        \max\lt(\tau^k, \lt(\bE_j Y\rt)^{k/(k-j)}\rt)
        \lesssim
        n^kp^k + \log^k n.
    \]
    Therefore, $\sE_0, \tau \sE_1 \lesssim n^{k+1}p^k + \log^k n$.
    By Theorem~\ref{thm:kim-vu-ineq}, we have
    \[
        Y
        \le
        \E Y + c_3 \sqrt{\tau \sE_0 \sE_1}
        \lesssim
        n^3p^3 + \log^2 n
    \]
    with probability $1-d_k \exp(-\tau/8) \ge 1 - n^{-10}$.

    In the following proofs, we always set $Y$ to be the sum of the indicators of the appropriate subgraph.
    This is a degree-$k$ polynomial in the variables $\xi_{i,j}$, where $k$ is the number of edges in the subgraph.
    We set $\tau = \max(8k\log n, 80\log n + 8\log d_k)$.

    \paragraph {Proof of (\ref{itm:er-ub-c3}).}
    Analogous to the above, with $M_2(Y) \le 1$, $\bE_0 Y \asymp n^3p^3$, and $\bE_1 Y \asymp np^2$.

    \paragraph {Proof of (\ref{itm:er-ub-c4}).}
    Analogous to the above, with $M_3(Y)\le 1$, $\bE_0 Y \asymp n^4p^4$, $\bE_1 Y \asymp n^2p^3$, and $\bE_2 Y \asymp np^2$.

    \paragraph{Proof of (\ref{itm:er-ub-c32e}).}
    Analogous to the above, with $M_4(Y)\le 2$, $\bE_0 Y \asymp n^4p^5$, $\bE_1 Y \asymp n^2p^4$, $\bE_2 Y \asymp np^3$, and $\bE_3 Y \asymp np^2$.
    Because we have $M_4(Y)\le 2$, we apply Theorem~\ref{thm:kim-vu-ineq} to $\f12 Y$ instead of $Y$, though of course this makes no difference in the resulting asymptotic.

    \paragraph{Proof of (\ref{itm:er-ub-c32v}).}
    Analogous to part (\ref{itm:er-ub-c3}), with $M_5(Y)\le 1$, $\bE_0 Y \asymp n^5p^6$, $\bE_1 Y \asymp n^3p^5$, $\bE_2 Y \asymp n^2p^4$, $\bE_3 Y \asymp n^2p^3$, and $\bE_4 Y \asymp np^2$.

    \paragraph{Proof of (\ref{itm:er-ub-k24}).}
    We will first show that $\Num_G(K_{2,4}) \lesssim n^6p^8 + \log^8 n$.
    The proof is analogous to the above, with $M_8(Y)\le 1$ , $\bE_0 Y \asymp n^6p^8$, $\bE_1 Y \asymp n^4p^7$, $\bE_2 Y \asymp n^3p^6$, $\bE_3 Y \asymp n^2p^5$, $\bE_4 Y \asymp n^2p^4$, $\bE_5 Y \asymp np^3$, $\bE_6 Y \asymp np^2$, and $\bE_7 Y \asymp p$.
    By Lemma~\ref{lem:subgraph-statistics}(\ref{itm:ss-k24-c42e-c42v}) and part (\ref{itm:er-ub-c4}), we also have
    \[
        \Num_G(K_{2,4})
        \lesssim
        \Num_G(C_4)^2
        \lesssim
        n^8p^8 + \log^6 n.
    \]
    So, we have that
    \[
        \Num_G(K_{2,4})
        \lesssim
        \min(n^6p^8 + \log^8 n, n^8p^8 + \log^6 n)
        \lesssim n^6p^8 + \log^6 n + \min(n^8p^8, \log^8 n).
    \]

    \paragraph{Proof of (\ref{itm:er-ub-c42e}).}
    We will first show that $\Num_G(C_4^{2,e}) \lesssim n^6p^7 + \log^7 n$.
    The proof is analogous to the above, with $M_7(Y)\le 1$, $\bE_0 Y \asymp n^6p^7$, $\bE_1 Y \asymp n^4p^6$, $\bE_2 Y \asymp n^3p^5$, $\bE_3 Y \asymp n^2p^4$, $\bE_4 Y \asymp n^2p^3$, $\bE_5 Y \asymp np^2$, and $\bE_6 Y \asymp p$.
    We also have $\Num_G(C_4^{2,e}) \lesssim \Num_G(C_4)^2$ by Lemma~\ref{lem:subgraph-statistics}(\ref{itm:ss-k24-c42e-c42v}).
    The result follows in the same way as in part (\ref{itm:er-ub-k24}).

    \paragraph{Proof of (\ref{itm:er-ub-c42v}).}
    We will first show that $\Num_G(C_4^{2,v}) \lesssim n^7p^8 + \log^8 n$.
    The proof is analogous to the above, with $M_8(Y)\le 1$, $\bE_0 Y \asymp n^7p^8$, $\bE_1 Y \asymp n^5p^7$, $\bE_2 Y \asymp n^4p^6$, $\bE_3 Y \asymp n^3p^5$,  $\bE_4 Y \asymp n^3p^4$, $\bE_5 Y \asymp n^2p^3$, $\bE_6 Y \asymp np^2$, and $\bE_7 Y \asymp p$.
    We also have $\Num_G(C_4^{2,v}) \lesssim \Num_G(C_4)^2$ by Lemma~\ref{lem:subgraph-statistics}(\ref{itm:ss-k24-c42e-c42v}).
    The result follows in the same way as in part (\ref{itm:er-ub-k24}).
\end{proof}

\begin{lemma}
    \label{lem:er-subgraph-lbs}
    Let $G \sim \cG(n,p)$.
    The following inequalities hold with high probability.
    \begin{enumerate}[label=(\alph*), ref=\alph*]
        \item \label{itm:er-lb-c3} $\Num_G(C_3) \ge \f12 \E \Num_G(C_3)$ with probability at least $1 - \exp(-\Omega(\min(n^3p^3, n^2p, n)))$.
        \item \label{itm:er-lb-c4p2e} $\Num_G(C_4, P_2, E) \ge \f12 \E \Num_G(C_4, P_2, E)$ with probability at least $1 - \exp(-\Omega(n^2p, n))$.
    \end{enumerate}
    \end{lemma}
\begin{proof}
    For each $(i,j)$ with $1\le i<j\le n$, let $\xi_{i,j}$ be the indicator of the edge $(i,j)$ in $G$.
    We will show these lower bounds with Theorem~\ref{thm:janson-ineq}.
    Note that applying Theorem~\ref{thm:janson-ineq} with $t = \f12 \E Y$ gives
    \begin{equation}
        \label{eq:janson-with-t-half-y}
        \P \lt(Y \le \f12 \E Y\rt)
        \le
        \exp\lt(-\f{(\E Y)^2}{8\bar \Delta}\rt).
    \end{equation}
    We will apply this probability bound with $Y$ equal to the sum of the indicators of the appropriate subgraphs.

    \paragraph{Proof of (\ref{itm:er-lb-c3}).}
    Let $Y = \sum_{i<j<k} \xi_{i,j}\xi_{j,k}\xi_{k,i}$ be the sum of the indicators for each triangle.
    Then $\E Y \asymp n^3p^3$, and
    \[
        \bar \Delta
        \asymp
        \E \Num_G(C_3, C_3^{2,e}, C_3^{2,v})
        \asymp
        n^3p^3 + n^4p^5 + n^5p^6.
    \]
    So,
    \[
        \f{(\E Y)^2}{8\bar \Delta}
        \asymp
        \f{n^6p^6}{n^3p^3 + n^4p^5 + n^5p^6}
        \asymp
        \min(n^3p^3, n^2p, n).
    \]
    The result follows from (\ref{eq:janson-with-t-half-y}).

    \paragraph{Proof of (\ref{itm:er-lb-c4p2e}).}
    Let $Y$ be the sum of the indicators 4-cycles, 2-paths, and edges in $G$.
    Then $\E Y \asymp n^4p^4 + n^3p^2 + n^2p$, and by an analogous computation
    \[
        \bar \Delta
        \asymp
        n^2p + n^3p^2 + n^4p^3 + n^5p^4 + n^6p^7 + n^7p^8.
    \]
    Note that $n^3p^2 \le \max(n^2p, n^5p^4)$, $n^4p^3 \le \max(n^2p, n^5p^4)$, and $n^6p^7 \le n^6p^6 \le \max(n^5p^4, n^7p^8)$.
    So, in fact $\bar \Delta \asymp n^2p + n^5p^4 + n^7p^8$.
    Hence,
    \[
    \f{(\E Y)^2}{8\bar \Delta}
    \asymp
    \f{n^4p^2 + n^6p^4 + n^8p^8}{n^2p + n^5p^4 + n^7p^8}
    \asymp
    \min(n^2p, n),
    \]
    where the second asymptotic equality can be verified by separately considering the cases $p\gtrsim n^{-1/2}$, $n^{-1}\lesssim p \lesssim n^{-1/2}$, and $p \lesssim n^{-1}$.
    The result now follows from (\ref{eq:janson-with-t-half-y}).
\end{proof}

The next two lemmas give conditions under which, for $G\sim \cG(n,p)$, the hypotheses of the appropriate theorems in Sections~\ref{sec:tv-upper-bounds} and \ref{sec:tv-lower-bounds} hold almost surely.

\begin{lemma}
    \label{lem:er-ubs-aux}
    Let $G = G_n \sim \cG(n,p)$, and let
    \[
        \Phi = n^3p^3 + n^{3/2} p + np^{1/2} + n^{1/2}p^{1/4} \log^2 n + \log^3 n.
    \]
    Over the randomness of the sample path $G_1,G_2,\ldots$, the following inequalities hold for all sufficiently large $n$ almost surely.
    The $\gtrsim$ hides an absolute constant factor.
    \begin{eqnarray*}
        \Phi &\gtrsim& \Num_G(C_3), \\
        \Phi^2 &\gtrsim& \Num_G(C_4, P_2, E), \\
        \Phi^4 &\gtrsim& \Num_G(K_{1,8}) + \log^8 n \Num_G(K_{1,4}, E).
    \end{eqnarray*}
\end{lemma}
\begin{proof}
    Because $\sum_{n\in \bN} n^{-10} < \infty$, by the Borel-Cantelli lemma the inequalities in Lemma~\ref{lem:er-subgraph-ubs} each holds for all sufficiently large $n$ almost surely.
    Suppose this is the case.
    By Lemma~\ref{lem:er-subgraph-ubs}(\ref{itm:er-ub-c3}),
    \[
        \Num_G(C_3) \lesssim n^3p^3 + \log^2 n \le \Phi.
    \]
    By Lemma~\ref{lem:er-subgraph-ubs}(\ref{itm:er-ub-e},\ref{itm:er-ub-p2},\ref{itm:er-ub-c4}),
    \[
        \Num_G(C_4, P_2, E)
        \lesssim
        n^4p^4 + n^3p^2 + n^2p + \log^3 n
        \le
        \Phi^2.
    \]
    In the last inequality, we use that $n^4p^4 \le n^{9/2}p^4 \le \max(n^6p^6, n^3p^2)$.
    \[
        \Num_G(K_{1,8}) + \log^8 n \Num_G(K_{1,4}, E)
        \lesssim
        n^9p^8 + n^5p^4\log^8 n + n^2p \log^8 n
        \le
        \Phi^4.
    \]
    In the last inequality, we use that $n^5p^4 \log^8 n \lesssim n^6p^4$ and $n^9p^8 \le \max(n^{12}p^{12}, n^6p^4)$.
\end{proof}

\begin{lemma}
\label{lem:er-lbs-aux}
    Let $G = G_n \sim \cG(n,p)$
    Over the randomness of the sample path $G_1,G_2,\ldots$, the following inequalities hold almost surely for sufficiently large $n$.
    The $\lesssim$ hides an absolute constant factor.
    \begin{enumerate}[label=(\alph*), ref=\alph*]
        \item \label{itm:er-lbs-aux-dense}
        If $p\gtrsim n^{-3/4}$, then
        \begin{eqnarray*}
            n^3p^3
            &\lesssim&
            \Num_G(C_3), \\
            \Num_G(C_3^{2,e}, C_3^{2,v})
            &\ll&
            \Num_G(C_3)^2.
        \end{eqnarray*}
        \item \label{itm:er-lbs-aux-sparse}
        If $p\gg n^{-2}\log^3 n$, then
        \begin{eqnarray*}
            n^4p^4 + n^3p^2 + n^2 p
            &\lesssim&
            \Num_G(C_4, P_2, E), \\
            \Num_G(K_{1,4}, K_{2,4}, C_4^{2,e}, C_4^{2,v})
            &\ll&
            \Num_G(C_4, P_2, E)^2.
        \end{eqnarray*}
    \end{enumerate}
\end{lemma}
\begin{proof}
    Because $\sum_{n\in \bN} n^{-10} < \infty$, by the Borel-Cantelli lemma the inequalities in Lemma~\ref{lem:er-subgraph-ubs} each holds for all sufficiently large $n$ almost surely.
    Throughout this proof, assume this is the case.

    \paragraph{Proof of (\ref{itm:er-lbs-aux-dense}).}
    Because $p\gtrsim n^{-3/4}$, we have that $\min(n^3p^3, n^2p, n) \gtrsim n^{3/4}$.
    For any $c>0$, we have $\sum_{n\in \bN} \exp(-cn^{3/4}) < \infty$.
    So, by the Borel-Cantelli lemma, the inequality in Lemma~\ref{lem:er-subgraph-lbs}(\ref{itm:er-lb-c3}) holds for all sufficiently large $n$ almost surely.
    Because $\E \Num_G(C_3) \asymp n^3p^3$, this implies that $\Num_G(C_3) \gtrsim n^3p^3$.
    Thus, Lemma~\ref{lem:er-subgraph-ubs}(\ref{itm:er-ub-c32e}) implies
    \[
        \Num_G(C_3^{2,e}, C_3^{2,v})
        \lesssim
        n^4p^5 + n^5p^6 + \log^5 n
        \ll
        n^6p^6
        \lesssim
        \Num_G(C_3)^{2}.
    \]

    \paragraph{Proof of (\ref{itm:er-lbs-aux-sparse}).}
    Because $p\gg n^{-2} \log^3 n$, we have that $\min(n^2p, n) \gg \log^3 n$.
    For any $c>0$, $\sum_{n\in \bN}\exp(-c\log^3 n) < \infty$.
    So, by the Borel-Cantelli lemma, the inequality in Lemma~\ref{lem:er-subgraph-lbs}(\ref{itm:er-lb-c4p2e}) holds for all sufficiently large $n$ almost surely.
    Because $\E \Num_G(C_4, P_2, E) \asymp n^4p^4 + n^3p^2 + n^2p$, this implies that $\Num_G(C_4, P_2, E) \gtrsim n^4p^4 + n^3p^2 + n^2p$.
    Lemma~\ref{lem:er-subgraph-ubs}(\ref{itm:er-ub-k14},\ref{itm:er-ub-k24},\ref{itm:er-ub-c42e},\ref{itm:er-ub-c42v}) implies that
    \begin{align*}
        \Num_G(K_{1,4}, K_{2,4}, C_4^{2,e}, C_4^{2,v})
        &\lesssim
        n^7p^8 + n^6p^7 + n^5p^4 + \log^6 n + \min(n^8p^8, \log^8 n) \\
        &\lesssim
        n^7p^8 + n^5p^4 + \log^6 n,
    \end{align*}
    where the last inequality follows from the bounds $n^6p^7 \le n^6p^6 \le \max(n^7p^8, n^5p^4)$ and
    \[
        \min(n^8p^8, \log^8 n)
        \le
        \sqrt{n^8p^8\cdot \log^8 n}
        =
        n^4p^4 \log^4 n
        \lesssim
        n^5p^4.
    \]
    Because $n^7p^8 \ll n^8p^8$, $n^5p^4 \ll n^6p^4$, and $\log^6 n \ll n^4p^2$ (recall that $p\gg n^{-2}\log^3 n$), we have that
    \[
        \Num_G(K_{1,4}, K_{2,4}, C_4^{2,e}, C_4^{2,v})
        \ll
        n^8p^8 + n^6p^4 + n^4p^4
        \lesssim
        \Num_G(C_4, P_2, E)^2.
    \]
\end{proof}

We now have the tools to prove Theorem~\ref{thm:er-mask}.

\begin{proof}[Proof of Theorem~\ref{thm:er-mask}]
    We will prove this theorem using Theorems~\ref{thm:general-ub}, \ref{thm:deg3-lb}, and \ref{thm:deg4-lb}, using the appropriate parts of Lemmas~\ref{lem:er-ubs-aux} and \ref{lem:er-lbs-aux} to show the hypotheses of these theorems hold almost surely.

    \paragraph{Proof of (\ref{itm:er-mask-ub}).}
    Define $\Phi$ as in Lemma~\ref{lem:er-ubs-aux}.
    By hypothesis (\ref{eq:er-mask-ub-hypothesis}), $d \gg \Phi$.
    By Lemma~\ref{lem:er-ubs-aux}, Theorem~\ref{thm:general-ub} holds, and so $\TV(W(G,d), M(G)) \to 0$.

    \paragraph{Proof of (\ref{itm:er-mask-lb}).}
    We will separately consider the cases $p\gtrsim n^{-3/4}$, where we will use Theorem~\ref{thm:deg3-lb}, and $n^{-2}\log^3 n \ll p\lesssim n^{-3/4}$, where we will use Theorem~\ref{thm:deg4-lb}.

    First, suppose $p\gtrsim n^{-3/4}$. Then, $n^3p^3 + n^{3/2} p + np^{1/2} \asymp n^3p^3$, so hypothesis (\ref{eq:er-mask-lb-hypothesis}) gives $d \ll n^3p^3$.
    By Lemma~\ref{lem:er-lbs-aux}(\ref{itm:er-lbs-aux-dense}), Theorem~\ref{thm:deg3-lb} holds, and so $\TV(W(G,d), M(G)) \to 1$.

    Otherwise, suppose $n^{-2}\log^3 n \ll p\lesssim n^{-3/4}$.
    Because $p\lesssim n^{-3/4}$, we can verify that
    \[
        n^3p^3 + n^{3/2}p + np^{1/2}
        \asymp
        n^{3/2}p + np^{1/2}
        \asymp
        n^2p^2 + n^{3/2}p + np^{1/2},
    \]
    so hypothesis (\ref{eq:er-mask-lb-hypothesis}) gives $d \ll n^2p^2 + n^{3/2} p + np^{1/2}$, or equivalently $d^2 \ll n^4p^4 + n^3p^2 + n^2p$.
    By Lemma~\ref{lem:er-lbs-aux}(\ref{itm:er-lbs-aux-sparse}), Theorem~\ref{thm:deg4-lb} holds, and so $\TV(W(G,d), M(G)) \to 1$.
\end{proof}

\subsection{Bipartite Erd\H{o}s-R\'enyi Masks}

In this subsection, we will show Theorems~\ref{thm:bip-er-mask}, which gives conditions for TV convergence and divergence for bipartite Erd\H{o}s-R\'enyi masks.
In this setting, $G\sim \cG(n,m,p)$ is a sample from a bipartite Erd\H{o}s-R\'enyi graph, where $n\ge m$.
Recall that $\cG(n,m,p)$ is the graph on $[n+m]$ where every edge between $\{1,\ldots,n\}$ and $\{n+1,\ldots,n+m\}$ occurs with independent probability $p$.
Throughout this subsection, we adopt the orientation $V_L = \{1,\ldots,n\}$ and $V_R = \{n+1,\ldots,n+m\}$ and define all oriented subgraph counts $\oNum_G(H)$ with respect to this labeling.
We choose this labeling because Theorem~\ref{thm:bipartite-ub} is stronger with this labeling than with the reverse.

We begin with the following two lemmas, which show high probability upper and lower bounds on various subgraph statistics and oriented subgraph statistics.
We will use these bounds to show that the appropriate theorems in Sections~\ref{sec:tv-upper-bounds} and \ref{sec:tv-lower-bounds} hold.

\begin{lemma}
    \label{lem:bip-er-subgraph-ubs}
    Let $G\sim \cG(n,m,p)$.
    For all sufficiently large $n$, the following inequalities each hold with probability at least $1-n^{-10}$.
    The $\lesssim$ hides only a constant factor independent of $n,m,p$.
    \begin{enumerate}[label=(\alph*), ref=\alph*]
        \item \label{itm:bip-er-ub-e} $\Num_G(E) \lesssim nmp + \log n$.
        \item \label{itm:bip-er-ub-p2} $\Num_G(P_2) \lesssim n^2mp^2 + \log^2 n$.
        \item \label{itm:bip-er-ub-c4} $\Num_G(C_4) \lesssim n^2m^2p^4 + \log^3 n + \min(np^2\log^2 n, \log^4 n)$.
        \item \label{itm:bip-er-ub-k14} $\Num_G(K_{1,4}) \lesssim n^4mp^4 + \log^4 n$.
        \item \label{itm:bip-er-ub-k24} $\Num_G(K_{2,4}) \lesssim n^4m^2p^8 + \log^6 n + \min(n^2m^2p^4\log^4 n, \log^8 n)$.
        \item \label{itm:bip-er-ub-c42e} $\Num_G(C_4^{2,e}) \lesssim n^3m^3p^7 + \log^6 n + \min(n^2m^2p^4\log^{9/2} n, \log^9 n)$
        \item \label{itm:bip-er-ub-c42v} $\Num_G(C_4^{2,v}) \lesssim n^4m^3p^8 + \log^6 n + \min(n^2m^2p^4\log^4 n, \log^8 n)$.
        \item \label{itm:bip-er-ub-ok13} $\oNum_G(\oK_{1,3}) \lesssim nm^3p^3 + \log^3 n$.
        \item \label{itm:bip-er-ub-ok14} $\oNum_G(\oK_{1,4}) \lesssim nm^4p^4 + \log^4 n$.
        \item \label{itm:bip-er-ub-ok24} $\oNum_G(\oK_{2,4}) \lesssim n^2m^4p^8 + \log^6 n + \min(n^2m^2p^4\log^4 n, \log^8 n)$.
        \item \label{itm:bip-er-ub-op4} $\oNum_G(\oP_4) \lesssim n^2m^3p^4 + \log^4 n$.
    \end{enumerate}
\end{lemma}
\begin{proof}
    These bounds are routine consequences of Theorem~\ref{thm:kim-vu-ineq}, proved in the same manner as Lemma~\ref{lem:er-subgraph-ubs}(\ref{itm:er-ub-e},\ref{itm:er-ub-p2},\ref{itm:er-ub-k14},\ref{itm:er-ub-k18}).
    We will only sketch the proofs.
    In the bipartite setting, the underlying random variables are the indicators $\xi_{i,j}$ for each edge $(i,j)$ with $i\in V_L$ and $j\in V_R$, which are i.i.d. samples from $\Ber(p)$.
    Throughout this proof, $\ell, c_k, d_k, \sE_0, \sE_1$ are defined as in Theorem~\ref{thm:kim-vu-ineq}.

    In the following proofs, we always set $Y$ to the sum of the indicators of the apporpriate subgraph.
    This is a degree-$k$ polynomial in the $\xi_{i,j}$, where $k$ is the number of edges in the subgraph.
    Except where indicated, we set $\tau = \max(8k\log n, 80\log n + 8\log d_k)$.
    When bounding the terms $\tau^j \bE_j Y$, we recall that $m\le n$.

    \paragraph{Proof of (\ref{itm:bip-er-ub-e}).}
    Analogous to the above, with $M_1(Y)\le 1$ and $\bE_0 Y \asymp nmp$.

    \paragraph{Proof of (\ref{itm:bip-er-ub-p2}).}
    Analogous to the above, with $M_2(Y)\le 1$, $\bE_0 Y \asymp n^2mp^2$, and $\bE_1 Y \asymp np$.

    \paragraph{Proof of (\ref{itm:bip-er-ub-c4}).}
    Analogously to the above, we can verify that $M_3(Y)\le 1$, $M_4(Y)\le 1$,
    We proceed analogously to part (\ref{itm:bip-er-ub-k14}), obtaining that $M_3(Y)\le 1$, $M_4(Y)\le 1$, ${\E}_0 Y \asymp n^2m^2p^4$, ${\E}_1 Y \asymp nmp^3$, ${\E}_2 Y \asymp np^2$, and ${\E}_3 Y \asymp p$.
    Applying Theorem~\ref{thm:kim-vu-ineq} with $\ell = 4$ yields that $\Num_G(C_4) \lesssim n^2m^2p^4 + \log^4 n$.
    If we apply Theorem~\ref{thm:kim-vu-ineq} with $\ell = 4$, we get that
    \[
        \sE_0, \tau \sE_1
        \lesssim
        \max(\E Y, \tau^2 \bE_2 Y, \tau^3)
        \lesssim
        n^2m^2p^4 + n^2p\log^2 n + \log^3 n.
    \]
    The extra term $\tau^2 {\E}_2 Y$ is necessary, because it does not necessarily hold that $\tau^2 \bE_2 Y \lesssim \max(\E Y, \tau^3)$.
    So, $\Num_G(C_4) \lesssim n^2m^2p^4 + np^2\log^2 n + \log^3 n$.
    Thus,
    \begin{align*}
        \Num_G(C_4)
        &\lesssim
        \min\lt(n^2m^2p^4 + \log^4 n, n^2m^2p^4 + np^2\log^2 n + \log^3 n\rt) \\
        &\lesssim
        n^2m^2p^4 + \log^3 n + \min(np^2\log^2 n, \log^4 n).
    \end{align*}

    \paragraph{Proof of (\ref{itm:bip-er-ub-k14}).}
    Analogous to the above, with $M_4(Y) \le 1$, $\bE_0 Y \asymp n^4mp^4$, $\bE_1 Y \asymp n^3p^3$, $\bE_2 Y \asymp n^2p^2$, and $\bE_3 Y \asymp np$.

    \paragraph{Proof of (\ref{itm:bip-er-ub-k24}).}
    We will first show that $\Num_G(K_{2,4}) \lesssim n^4m^2p^8 + \log^8 n$.
    The proof is analogous to the above, with $M_8(Y)\le 1$, ${\E}_0 Y \asymp n^4m^2p^8$, ${\E}_1 Y \asymp n^3mp^7$, ${\E}_2 Y \asymp n^3p^6$, ${\E}_3 Y \asymp n^2p^5$, ${\E}_4 Y \asymp n^2p^4$, ${\E}_5 Y \asymp np^3$, ${\E}_6 Y \asymp np^2$, and ${\E}_7 Y \asymp p$.
    By Lemma~\ref{lem:subgraph-statistics}(\ref{itm:ss-k24-c42e-c42v}) and part (\ref{itm:bip-er-ub-c4}), we also have
    \[
        \Num_G(K_{2,4})
        \le
        \Num_G(C_4)^2
        \lesssim
        n^4m^4p^8 + \log^6 n + \min(n^2p^4\log^4n, \log^8 n).
    \]
    So,
    \begin{align*}
        \Num_G(K_{2,4})
        &\lesssim
        \min\lt(
            n^4m^2p^8 + \log^8 n,
            n^4m^4p^8 + \log^6 n + \min(n^2p^4\log^4n, \log^8 n)
        \rt) \\
        &\lesssim
        n^4m^2p^8 + \log^6 n +
        \min(n^4m^4p^8, \log^8 n) +
        \min(n^2p^4\log^4n, \log^8 n).
    \end{align*}
    Finally, note that $\min(n^2m^2p^4\log^4 n, \log^8 n)$ is an upper bound for each of the last two terms, because $n^2m^2p^4\log^4 n$ is the geometric mean of $n^4m^4p^8$ and $\log^8 n$.

    \paragraph{Proof of (\ref{itm:bip-er-ub-c42e}).}
    We will first show that $\Num_G(C_4^{2,e}) \lesssim n^3m^3p^7 + \log^9 n$.
    The proof is analogous to the above, with the following modification.
    We treat $Y$, a degree-7 polynomial in the $\xi_{i,j}$, as a degree-9 polynomial, and set $\tau = \max (72\log n, 80\log n + 8\log d_9)$.
    The remaining argument is the same, with $M_9(Y)\le 1$, $\bE_0 Y \asymp n^3m^3p^7$, $\bE_1 Y \asymp n^2m^2p^6$, $\bE_2 Y \asymp n^2mp^5$, $\bE_3 Y \asymp n^2p^4$, $\bE_4 Y \asymp nmp^3$, $\bE_5 Y \asymp np^2$, $\bE_6 Y \asymp p$, $\bE_7 Y \asymp 1$, and $\bE_8 Y = 0$.
    We also have $\Num_G(C_4^{2,e}) \lesssim \Num_G(C_4)^2$ by Lemma~\ref{lem:subgraph-statistics}(\ref{itm:ss-k24-c42e-c42v}).
    The result follows in the same way as in part (\ref{itm:bip-er-ub-k24}).

    \paragraph{Proof of (\ref{itm:bip-er-ub-c42v}).}
    We will first show that $\Num_G(C_4^{2,v}) \lesssim n^4m^3p^8 + \log^8 n$.
    The proof is analogous to the above, with $M_8(Y)\le 1$ and $\bE_0 Y \asymp n^4m^3p^8$, $\bE_1 Y \asymp n^3m^2p^7$, $\bE_2 Y \asymp n^3mp^6$, $\bE_3 Y \asymp n^2mp^5$, $\bE_4 Y \asymp n^2mp^4$, $\bE_5 Y \asymp nmp^3$, $\bE_6 Y \asymp np^2$, and $\bE_7 Y \asymp p$.
    We also have $\Num_G(C_4^{2,e}) \lesssim \Num_G(C_4)^2$ by Lemma~\ref{lem:subgraph-statistics}(\ref{itm:ss-k24-c42e-c42v}).
    The result follows in the same way as in part (\ref{itm:bip-er-ub-k24}).

    \paragraph{Proof of (\ref{itm:bip-er-ub-ok13}).}
    Analogous to the above, with $M_3(Y)\le 1$, $\E Y = \bE_0 Y \asymp nm^3p^3$, $\bE_1 Y \asymp m^2p^2$, and $\bE_2 Y \asymp mp$.

    \paragraph{Proof of (\ref{itm:bip-er-ub-ok14}).}
    Analogous to the above, with $M_4(Y)\le 1$, $\bE_0 Y \asymp nm^4p^4$, $\bE_1 Y \asymp m^3p^3$, $\bE_2 Y \asymp m^2p^2$, and $\bE_3 Y \asymp mp$.

    \paragraph{Proof of (\ref{itm:bip-er-ub-ok24}).}
    We will first show that $\oNum_G(\oK_{2,4}) \lesssim n^2m^4p^8 + \log^8 n$.
    The proof is analogous to the above, with $M_8(Y) \le 1$, $\bE_0Y \asymp n^2m^4p^8$, $\bE_1Y \asymp nm^3p^7$, $\bE_2Y \asymp nm^2p^6$, $\bE_3Y \asymp nmp^5$, $\bE_4Y \asymp \max(n,m^2) p^4$, $\bE_5Y \asymp mp^3$, $\bE_6Y \asymp mp^2$, and $\bE_7Y \asymp p$.
    We also have $\Num_G(C_4^{2,e}) \lesssim \Num_G(C_4)^2$ by Lemma~\ref{lem:subgraph-statistics}(\ref{itm:ss-k24-c42e-c42v}).
    The result follows in the same way as in part (\ref{itm:bip-er-ub-k24}).

    \paragraph{Proof of (\ref{itm:bip-er-ub-op4}).}
    Analogous to the above, with $M_4(Y)\le 1$, $\bE_0 Y \asymp n^2m^3p^4$, $\bE_1 Y \asymp nm^2p^3$, $\bE_2 Y \asymp nmp^2$, and $\bE_3 Y \asymp mp$.
\end{proof}

\begin{lemma}
    \label{lem:bip-er-subgraph-lbs}
    Let $G\sim \cG(n,m,p)$, with $p \gg (nm)^{-1} \log^3 n$.
    For all sufficiently large $n$, the following inequalities hold with probability at least $1-n^{-10}$.
    \begin{enumerate}[label=(\alph*), ref=\alph*]
        \item \label{itm:bip-er-lb-e} $\Num_G(E) \ge \f12 \E \Num_G(E)$.
        \item \label{itm:bip-er-lb-c4p2e} $\Num_G(C_4, P_2, E) \ge \f12 \E \Num_G(C_4, P_2, E)$.
    \end{enumerate}
\end{lemma}
\begin{proof}
    For each $(i,j)$ with $i\in V_L, j\in V_R$, let $\xi_{i,j}$ be the indicator for the edge $(i,j)$ in $G$.
    The $\xi_{i,j}$ are i.i.d. samples from $\Ber(p)$.

    \paragraph{Proof of (\ref{itm:bip-er-lb-e}).}
    This is by a Chernoff bound.
    Since $\Num_G(E) = \sum_{i\in V_L, j\in V_R} \xi_{i,j}$, we have
    \[
        \P\lt(\Num_G(E) < \f12 \Num_G(E)\rt)
        \le
        \exp\lt(-\f18 \E \Num_G(E)\rt)
        =
        \exp\lt(-\f18 nmp\rt)
        \ll
        \exp\lt(-\log^3 n\rt)
        \ll
        n^{-10}.
    \]

    \paragraph{Proof of (\ref{itm:bip-er-lb-c4p2e}).}
    Let $Y$ be the sum of indicators of the copies of $C_4$ in $G$.
    This is a degree-4 polynomial in the $\xi_{i,j}$.
    By the same computations as in the proof of Lemma~\ref{lem:bip-er-subgraph-ubs}(\ref{itm:bip-er-ub-c4}), we have $M_3(Y)\le 1$, $\bE_0 Y \asymp n^2m^2p^4$, $\bE_1 Y \asymp nmp^3$, and $\bE_2 Y \asymp np^2$.
    Set $\tau = \max(32\log n, 80\log n + 8 \log (3d_4))$.
    Since
    \[
        \tau {\E}_1 Y
        \le
        \max(({\E}_1 Y)^{3/2}, \tau^3)
        \lesssim
        \max(\E Y, \tau^3),
    \]
    we have
    \begin{align*}
        \sE_0
        &\lesssim
        \max(\E Y, \tau^2 {\E}_2 Y, \tau^3)
        \lesssim
        \max(n^2m^2p^4, np^2\log^2 n, \log^3 n), \\
        \sE_1
        &\lesssim
        \max(\tau {\E}_1 Y, \tau^2 {\E}_2 Y, \tau^3)
        \lesssim
        \max(nmp^3, np^2\log^2 n, \log^3 n).
    \end{align*}
    So, by Theorem~\ref{thm:kim-vu-ineq}, there exists a constant $C$ such that with probability $1-\f13 n^{-10}$,
    \[
    \Num_G(C_4)
    \le
    \E \Num_G(C_4) -
    C \max(n^2m^2p^4, np^2\log^2 n, \log^3 n)^{1/2}
    \max(nmp^3, np^2\log^2 n, \log^3 n)^{1/2}.
    \]
    By similar analysis for $P_2$ and $E$, there exists a constant $C$ such that with probability $1-\f13 n^{-10}$,
    \begin{align*}
        \Num_G(P_2)
        &\le
        \E \Num_G(P_2)
        -
        C \max(n^2mp^2, \log^2 n)^{1/2}
        \max(np\log n, \log^2 n)^{1/2}, \\
        \Num_G(E)
        &\le
        \E \Num_G(E)
        -
        C \max(nmp, \log n)^{1/2}
        \log^{1/2} n.
    \end{align*}
    By a union bound, there is a constant $C$ such that with probability $1-n^{-10}$,
    \begin{align*}
        &\Num_G(C_4, P_2, E)
        \ge
        \E \Num_G(C_4, P_2, E) \\
        &\qquad -
        C\bigg(
            \max(n^2m^2p^4, np^2\log^2 n, \log^3 n)^{1/2}
            \max(nmp^3, np^2\log^2 n, \log^3 n)^{1/2} \\
            &\qquad \qquad +
            \max(n^2mp^2, \log^2 n)^{1/2}
            \max(np\log n, \log^2 n)^{1/2}
            + \max(nmp, \log n)^{1/2}
            \log^{1/2} n
        \bigg).
    \end{align*}
    We will show that the error term is $o(n^2m^2p^4 + n^2mp^2 + nmp)$.
    By the hypothesis $p\gg (nm)^{-1}\log^3 n$, we have that $nmp^3 \ll n^2m^2p^4$, $np^2\log^2 n \ll n^2mp^2$, and $\log^3 n \ll nmp$.
    So,
    \begin{align*}
        \max(n^2m^2p^4, np^2 \log^2 n, \log^3 n)
        &\lesssim
        n^2m^2p^4 + n^2mp^2 + nmp, \\
        \max(nmp^3, np^2 \log^2 n, \log^3 n)
        &\ll
        n^2m^2p^4 + n^2mp^2 + nmp.
    \end{align*}
    It follows that
    \[
        \max(n^2m^2p^4, np^2 \log^2 n, \log^3 n)^{1/2}
        \max(nmp^3, np^2 \log^2 n, \log^3 n)^{1/2}
        \ll
        n^2m^2p^4 + n^2mp^2 + nmp.
    \]
    Similarly, we have that $np\log n \ll n^2mp^2$ and $\log^2 n \ll nmp$, so
    \begin{align*}
        \max(n^2mp^2, \log^2 n)
        &\lesssim
        n^2m^2p^4 + n^2mp^2 + nmp, \\
        \max(np\log n, \log^2 n)
        &\ll
        n^2m^2p^4 + n^2mp^2 + nmp.
    \end{align*}
    Thus,
    \[
        \max(n^2mp^2, \log^2 n)^{1/2}
        \max(np\log n, \log^2 n)^{1/2}
        \ll
        n^2m^2p^4 + n^2mp^2 + nmp.
    \]
    Finally, as $\log n \ll nmp$, we have
    $\max(nmp, \log n) \lesssim n^2m^2p^4 + n^2mp^2 + nmp$, and so
    \[
        \max(nmp, \log n)^{1/2} \log^{1/2} n
        \ll
        n^2m^2p^4 + n^2mp^2 + nmp.
    \]
    Because $\E \Num_G(C_4, P_2, E) \asymp n^2m^2p^4 + n^2mp^2 + nmp$, the lemma follows.
\end{proof}

The next two lemmas give conditions under which, for $G\sim \cG(n,m,p)$, the hypotheses of the appropriate theorems in Sections~\ref{sec:tv-upper-bounds} and \ref{sec:tv-lower-bounds} hold almost surely.

\begin{lemma}
    \label{lem:bip-er-mask-ubs-aux}
    Let $G = G_n \sim \cG(n,m,p)$, and let
    \[
        \Phi =
        nmp^2 + nm^{1/2}p + (nmp)^{1/2} +
        (nmp)^{1/3}\log n +
        (nmp)^{1/4} \log^{5/4} n +
        \log^{3/2} n.
    \]
    Over the randomness of the sample path $G_1, G_2, \ldots$, the following inequalities hold almost surely for sufficiently large $n$.
    The $\lesssim$ hides an absolute constant factor.
    \begin{eqnarray*}
        \Phi^2
        &\gtrsim&
        \Num_G(C_4, P_2, E), \\
        \Phi^3
        &\gtrsim&
        \oNum_G(\oK_{1,4}) +
        \Num_G(E)\log^3 n, \\
        \Phi^8
        &\gtrsim&
        \oNum_G(\oK_{1,3})^2 \oNum_G(\oK_{2,4}) +
        \Num_G(E)^2 \oNum_G(\oK_{1,4},\oK_{2,4}) \log^4 n, \\
        \Phi^9 &\gtrsim& \Num_G(E)^2 \oNum_G(\oP_4)\log^4 n.
    \end{eqnarray*}
\end{lemma}
\begin{proof}
    Because $\sum_{n\in \bN} n^{-10} < \infty$, by the Borel-Cantelli lemma the inequalities in Lemma~\ref{lem:bip-er-subgraph-ubs} each hold for all sufficiently large $n$ almost surely.
    Throughout this proof, we suppose this is the case.
    By Lemma~\ref{lem:bip-er-subgraph-ubs}(\ref{itm:bip-er-ub-e},\ref{itm:bip-er-ub-p2},\ref{itm:bip-er-ub-c4}),
    \begin{align*}
        \Num_G(C_4, P_2, E)
        &\lesssim
        n^2m^2p^4 + n^2mp^2 + nmp + \log^3 n + \min(np^2\log^2 n, \log^4 n) \\
        &\lesssim
        n^2m^2p^4 + n^2mp^2 + nmp + \log^3 n
        \le \Phi^2
    \end{align*}
    where the second-last inequality follows from $\min(np^2\log^2 n, \log^4 n) \le np^2\log^2 n \lesssim n^2mp^2$.
    This proves the first conclusion.
    By Lemma~\ref{lem:bip-er-subgraph-ubs}(\ref{itm:bip-er-ub-e},\ref{itm:bip-er-ub-ok14}),
    \begin{align*}
        \oNum_G(\oK_{1,4}) + \Num_G(E) \log^3 n
        &\lesssim
        nm^4p^4 + nmp \log^3 n + \log^{4} n \\
        &\lesssim
        n^3m^3p^6 + n^3m^{3/2}p^3 + nmp \log^3 n + \log^{9/2} n
        \le
        \Phi^3,
    \end{align*}
    where the second-last inequality follows from $nm^4p^4 \le n^3m^2p^4 \le \max(n^3m^3p^6, n^3m^{3/2}p^3)$ (recall $m\le n$).
    This proves the second conclusion.
    Lemma~\ref{lem:bip-er-subgraph-ubs}(\ref{itm:bip-er-ub-e},\ref{itm:bip-er-ub-ok13},\ref{itm:bip-er-ub-ok24}) yields
    \begin{align*}
        &\lt(\oNum_G(\oK_{1,3})^2 + \Num_G(E)^2 \log^4 n\rt) \oNum_G(\oK_{2,4}) \\
        &\qquad \lesssim
        \lt(n^2m^6p^6 + n^2m^2p^2\log^4 n + \log^6 n\rt)
        \lt(n^2m^4p^8 + n^2m^2p^4 \log^4 n + \log^6 n\rt) \\
        &\qquad =
        n^4m^{10}p^{14} +
        n^4m^6p^{10} \log^4 n +
        n^2m^4p^8 \log^6 n +
        n^4m^8p^{10}\log^4 n +
        n^4m^4p^6\log^8 n \\
        &\qquad\qquad +
        n^2m^2p^4\log^{10} n +
        n^2m^6p^6 \log^6 n +
        n^2m^2p^2 \log^{10} n +
        \log^{12} n.
    \end{align*}
    These terms are bounded by the following inequalities.
    \begin{align*}
        n^4m^{10}p^{14}
        &\le n^8m^7p^{14}
        \le \max(n^8m^8p^{16}, n^8m^4p^8)
        \le \Phi^8, \\
        n^4m^6p^{10} \log^4 n
        &\lesssim n^8m^5p^{10}
        \le \max(n^8m^8p^{16}, n^8m^4p^8)
        \le \Phi^8, \\
        n^2m^4p^8 \log^6 n
        &\lesssim n^8m^4p^8
        \le \Phi^8, \\
        n^4m^8p^{10} \log^4 n
        &\lesssim n^8m^5p^{10}
        \le \max(n^8m^8p^{16}, n^8m^4p^8)
        \le \Phi^8, \\
        n^4m^4p^6\log^8 n
        &\lesssim n^6m^4p^6
        \le \max(n^8m^4p^8, n^4m^4p^4)
        \le \Phi^8, \\
        n^2m^2p^4\log^{10} n
        &\lesssim n^4m^4p^4
        \le \Phi^8, \\
        n^2m^6p^6 \log^6 n
        &\lesssim n^6m^4p^6
        \le \max(n^8m^4p^8, n^4m^4p^4)
        \le \Phi^8, \\
        n^2m^2p^2 \log^{10} n
        &\le \Phi^8, \\
        \log^{12} n
        &\le \Phi^8.
    \end{align*}
    Thus $\lt(\oNum_G(\oK_{1,3})^2 + \Num_G(E)^2 \log^4 n\rt) \oNum_G(\oK_{2,4}) \lesssim \Phi^8$.
    Moreover, Lemma~\ref{lem:bip-er-subgraph-ubs}(\ref{itm:bip-er-ub-e},\ref{itm:bip-er-ub-ok14}) implies
    \begin{align*}
        \Num_G(E)^2 \oNum_G(\oK_{1,4}) \log^4 n
        &\lesssim
        \lt(n^2m^2p^2 + \log^2 n\rt)
        \lt(nm^4p^4  + \log^4 n\rt)
        \log^4 n \\
        & =
        n^3m^6p^6 \log^4 n +
        n^2m^2p^2\log^8 n +
        nm^4p^4\log^6 n +
        \log^{10} n.
    \end{align*}
    These terms are bounded by the following inequalities.
    \begin{align*}
        n^3m^6p^6\log^4 n
        &\lesssim n^6m^4p^6
        \le \max(n^8m^4p^8, n^4m^4p^4)
        \le \Phi^8, \\
        n^2m^2p^2\log^8 n
        &\le
        n^2m^2p^2\log^{10} n
        \le
        \Phi^8, \\
        nm^4p^4\log^6 n
        &\lesssim n^4m^4p^4
        \le \Phi^8, \\
        \log^{10} n &\le \log^{12} n \le \Phi^8.
    \end{align*}
    Therefore, $\Num_G(E)^2 \oNum_G(\oK_{1,4}) \log^4 n \lesssim \Phi^8$.
    This proves the third conclusion.
    Finally by Lemma~\ref{lem:bip-er-subgraph-ubs}(\ref{itm:bip-er-ub-e},\ref{itm:bip-er-ub-op4}),
    \begin{align*}
        \Num_G(E)^2\oNum_G(\oP_4)\log^4 n
        &\lesssim
        (n^2m^2p^2 + \log^2 n)
        (n^2m^3p^4 + \log^4 n)
        \log^4 n \\
        &=
        n^4m^5p^6\log^4 n +
        n^2m^3p^4 \log^6 n +
        n^2m^2p^2 \log^8 n +
        \log^{10} n.
    \end{align*}
    These terms are bounded by the following inequalities.
    \begin{align*}
        n^4m^5p^6\log^4 n
        &\lesssim n^6m^{9/2}p^6
        \le \max(n^9m^{9/2}p^9, n^{9/2}m^{9/2}p^{9/2})
        \le \Phi^9, \\
        n^2m^3p^4 \log^6 n
        &\lesssim n^4m^4p^4\log^3 n
        \le \max(n^{9/2}m^{9/2}p^{9/2}, m^3n^3p^3 \log^9 n)
        \le \Phi^9, \\
        n^2m^2p^2 \log^8 n
        &\le n^2m^2p^2 \log^{23/2} n
        \le \max(n^{9/4}m^{9/4}p^{9/4} \log^{45/4} n, \log^{27/2} n)
        \le \Phi^9, \\
        \log^{10} n
        &\le \log^{27/2} n
        \le \Phi^9.
    \end{align*}
    Therefore $\Num_G(E)^2\oNum_G(\oP_4)\log^4 n \lesssim \Phi^9$.
    This proves the fourth conclusion.
\end{proof}

\begin{lemma}
    \label{lem:bip-er-mask-lbs-aux}
    Let $G = G_n \sim \cG(n,m,p)$.
    Over the randomness of the sample path $G_1, G_2, \ldots$, the following inequalities hold almost surely for sufficiently large $n$.
    The $\lesssim$ hides an absolute constant factor.
    \begin{enumerate}[label=(\alph*), ref=\alph*]
        \item \label{itm:bip-er-mask-lbs-aux-deg4}
        If $p\gg (nm)^{-1}\log^3 n$ and $m \gg 1$, then
        \begin{eqnarray*}
            n^2m^2p^4 + n^2mp^2 + nmp &\lesssim& \Num_G(C_4, P_2, E), \\
            \Num_G(K_{1,4}, K_{2,4}, C_4^{2,e}, C_4^{2,v})
            &\ll&
            \Num_G(C_4, P_2, E)^2.
        \end{eqnarray*}
        \item \label{itm:bip-er-mask-lbs-aux-maxdeg}
        If $m=1$, then $\Num_{G}(E) \gtrsim np$.
    \end{enumerate}
\end{lemma}
\begin{proof}
    Because all parts of the lemma assume $p\gg (nm)^{-1}\log^3 n$, Lemma~\ref{lem:bip-er-subgraph-lbs} applies.
    Because $\sum_{n\in \bN} n^{-10} < \infty$, by the Borel-Cantelli lemma the inequalities in Lemma~\ref{lem:bip-er-subgraph-ubs} and Lemma~\ref{lem:bip-er-subgraph-lbs} each hold for all sufficiently large $n$ almost surely.
    Throughout this proof, assume that this is the case.

    \paragraph{Proof of  (\ref{itm:bip-er-mask-lbs-aux-deg4}).}
    The first conclusion follows directly from Lemma~\ref{lem:bip-er-subgraph-lbs}, because $\E \Num_G(C_4, P_2, E) \asymp n^2m^2p^4 + n^2mp^2+ nmp$.
    By Lemma~\ref{lem:bip-er-subgraph-ubs}(\ref{itm:bip-er-ub-k14},\ref{itm:bip-er-ub-k24},\ref{itm:bip-er-ub-c42e},\ref{itm:bip-er-ub-c42v}), we have that
    \begin{align*}
        \Num_G(K_{1,4}, K_{2,4}, C_4^{2,e}, C_4^{2,v})
        &\lesssim
        n^4m^3p^8 + n^3m^3p^7 + n^4mp^4 + \log^6 n + \min(n^2m^2p^4\log^{9/2} n, \log^9 n) \\
        &\lesssim
        n^4m^3p^8 + n^4mp^4 + \log^6 n.
    \end{align*}
    The last inequality follows from the inequalities $n^3m^3p^7 \le n^4m^{5/2}p^7 \le \max(n^4mp^4, n^4m^3p^8)$ and
    \[
        \min(n^2m^2p^4\log^{9/2} n, \log^9 n)
        \le
        n^2m^2p^4\log^{9/2} n
        \lesssim
        n^4mp^4.
    \]
    Because $m\gg 1$ and $p\gg (nm)^{-1}\log^3 n$, we have that $n^4m^3p^8 \ll n^4m^4p^8$, $n^4mp^4 \ll n^4m^2p^4$, and $\log^6 n \ll n^2m^2p^2$.
    So,
    \[
        \Num_G(K_{1,4}, K_{2,4}, C_4^{2,e}, C_4^{2,v})
        \ll
        n^4m^4p^8 + n^4m^2p^4 + n^2m^2p^2
        \lesssim
        \Num_G(C_4, P_2, E)^2.
    \]

    \paragraph{Proof of (\ref{itm:bip-er-mask-lbs-aux-maxdeg}).}
    The claim follows immediately from Lemma~\ref{lem:bip-er-subgraph-lbs}(\ref{itm:bip-er-lb-e}), because $\E \Num_G(E) = np$.
\end{proof}

We now have the tools to prove Theorem~\ref{thm:bip-er-mask}.

\begin{proof}[Proof of Theorem~\ref{thm:bip-er-mask}]
    We will prove this theorem using Theorems~\ref{thm:bipartite-ub}, \ref{thm:deg4-lb}, and \ref{thm:maxdeg-lb}, using the appropriate parts of Lemmas~\ref{lem:bip-er-mask-ubs-aux} and \ref{lem:bip-er-mask-lbs-aux} to show these theorems hold almost surely.

    \paragraph{Proof of (\ref{itm:bip-er-mask-ub}).}
    Define $\Phi$ as in Lemma~\ref{lem:bip-er-mask-ubs-aux}.
    By hypothesis (\ref{eq:bip-er-mask-ub-hypothesis}), $d \gg \Phi$.
    By Lemma~\ref{lem:bip-er-mask-ubs-aux}, Theorem~\ref{thm:bipartite-ub} holds, and so $\TV(W(G,d), M(G)) \to 0$.

    \paragraph{Proof of (\ref{itm:bip-er-mask-lb}).}
    We will separately consider the cases $m\gg 1$, where we will use Theorem~\ref{thm:deg4-lb}, and $m=O(1)$, where we will use Theorem~\ref{thm:maxdeg-lb}.

    First, suppose $m \gg 1$.
    Then, hypothesis (\ref{eq:bip-er-mask-lb-hypothesis}) gives that $d^2 \ll n^2m^2p^4 + n^2mp^2 + nmp$.
    By Lemma~\ref{lem:bip-er-mask-lbs-aux}(\ref{itm:bip-er-mask-lbs-aux-deg4}), Theorem~\ref{thm:deg4-lb} holds, and so $\TV(W(G,d), M(G)) \to 1$.

    Otherwise, suppose $m = O(1)$.
    Then, the hypothesis $p \gg (nm)^{-1} \log^3 n$ implies $p \gg n^{-1} \log^3 n$, and hypothesis (\ref{eq:bip-er-mask-lb-hypothesis}) simplifies to $d \ll np$
    Let $G'$ denote the induced subgraph of $G$ on $\{1,\ldots,n+1\}$, containing all of $V_L = \{1,\ldots,n\}$ and the vertex $n+1$ from $V_R$.
    Lemma~\ref{lem:bip-er-mask-lbs-aux}(\ref{itm:bip-er-mask-lbs-aux-maxdeg}) applies to $G'$, and yields that $np \lesssim \Num_{G'}(E) \le \maxdeg(G)$.
    Thus $d \ll \maxdeg(G)$.
    By Theorem~\ref{thm:maxdeg-lb}, $\TV(W(G,d), M(G)) \to 1$.
\end{proof}

\end{document}